\documentclass{memo-l}

\usepackage{}
\usepackage{amsmath}
\usepackage{amsfonts}
\usepackage{amssymb}
\usepackage{graphicx}
\usepackage{hyperref}
\usepackage{mathrsfs}
\usepackage{pb-diagram}
\usepackage{epstopdf}
\usepackage{amsmath,amsfonts,amsthm,enumerate,amscd,latexsym,curves}
\usepackage{bbm}
\usepackage[mathscr]{eucal}
\usepackage{epsfig,epsf,epic}
\usepackage{caption}
\usepackage{subcaption}
\usepackage{multirow}
\usepackage{color}
\usepackage[all]{xy}
\usepackage{fourier}
\usepackage{comment}
\usepackage{chngcntr}
\counterwithout{figure}{chapter}

\newcommand{\rat}{\mathbb{Q}}

\newcommand{\Z}{\mathbb{Z}}
\newcommand{\Q}{\mathbb{Q}}
\newcommand{\R}{\mathbb{R}}
\newcommand{\C}{\mathbb{C}}

\newcommand{\bP}{\mathbb{P}}

\newcommand{\be}{\mathbf{1}}

\newcommand{\bS}{\mathbb{S}}

\newcommand{\bT}{\mathbf{T}}

\newcommand{\AI}{A_\infty}
\newcommand{\CC}{\mathcal{C}}

\newcommand{\consti}{\mathbf{i}\,}
\newcommand{\conste}{\mathbf{e}}

\newcommand{\WT}[1]{\widetilde{#1}}
\newcommand{\UL}[1]{\underline{#1}}

\newcommand{\CO}{\mathcal{O}}

\newcommand{\cA}{\mathcal{A}}
\newcommand{\CT}{\mathcal{T}}
\newcommand{\cF}{\mathcal{F}}

\newcommand{\End}{\mathrm{End}}

\newcommand{\Hom}{\mathrm{Hom}}

\newcommand{\GL}{\mathrm{GL}}
\newcommand{\SL}{\mathrm{SL}}

\newcommand{\Jac}{\mathrm{Jac}}

\newcommand{\Span}{\mathrm{Span}}

\newcommand{\one}{\mathbf{1}}
\newcommand{\One}{\mathbf{1}}

\newcommand{\CF}{\mathrm{CF}}
\newcommand{\Fuk}{\mathrm{Fuk}}
\newcommand{\dg}{\mathrm{dg}}
\newcommand{\MF}{\mathrm{MF}}
\newcommand{\bL}{\mathbb{L}}

\newcommand{\Tw}{\mathrm{Tw}}
\newcommand{\open}{\mathrm{open}}
\newcommand{\mir}{\mathrm{mir}}
\newcommand{\orb}{\mathrm{orb}}
\newcommand{\wrap}{\mathrm{wrap}}
\newcommand{\cO}{\mathcal{O}}

\newcommand{\Sky}{\mathrm{Sky}}

\newtheorem{theorem}{Theorem}[chapter]
\newtheorem{lemma}[theorem]{Lemma}

\theoremstyle{definition}
\newtheorem{definition}[theorem]{Definition}
\newtheorem{example}[theorem]{Example}

\newtheorem{assumption}[theorem]{Assumption}
\newtheorem{prop}[theorem]{Proposition}
\newtheorem{cor}[theorem]{Corollary}
\newtheorem{conjecture}[theorem]{Conjecture}

\theoremstyle{remark}
\newtheorem{remark}[theorem]{Remark}

\numberwithin{section}{chapter}
\numberwithin{equation}{chapter}

\makeindex

\begin{document}

\frontmatter

\title{Noncommutative Homological Mirror Functor}


\author[Cho]{Cheol-Hyun Cho}
\address{Department of Mathematical Sciences and Research Institute in Mathematics\\ Seoul National University\\ 1 Gwanak-Ro\\ Gwanak-Gu\\Seoul 08826\\ Korea}
\curraddr{}
\email{chocheol@snu.ac.kr}
\thanks{}


\author[Hong]{Hansol Hong}
\address{Department of Mathematics\\Yonsei University\\  
50 Yonsei-Ro \\Seodaemun-Gu \\ Seoul 03722 \\ Korea}
\email{hansolhong@yonsei.ac.kr}



\author[Lau]{Siu-Cheong Lau}
\address{Department of Mathematics and Statistics\\ Boston University\\111 Cummington Mall\\ Boston\\ MA 02215\\ USA}
\curraddr{}
\email{lau@math.bu.edu}
\thanks{}

\date{Aug. 2 2016}




\maketitle

\tableofcontents

\begin{abstract}
We formulate a constructive theory of noncommutative Landau-Ginzburg models mirror to symplectic manifolds based on Lagrangian Floer theory.  The construction comes with a natural functor from the Fukaya category to the category of matrix factorizations of the constructed Landau-Ginzburg model.  As applications, it is applied to elliptic orbifolds, punctured Riemann surfaces and certain non-compact Calabi-Yau threefolds to construct their mirrors and functors.  In particular it recovers and strengthens several interesting results of Etingof-Ginzburg, Bocklandt and Smith, and gives a unified understanding of their results in terms of mirror symmetry and symplectic geometry.  As an interesting application, we construct an explicit global deformation quantization of an affine del Pezzo surface as a noncommutative mirror to an elliptic orbifold.
\end{abstract}


\subsection*{Acknowledgment}

We express our gratitude to Kenji Fukaya, Conan Leung, Yong-Geun Oh, Hiroshi Ohta and Kaoru Ono for helpful discussions and encouragement. We also thank Raf Bocklandt and Eric Zaslow for explaining their works.  The first author thanks IBS-CGP, Institut Mittag-Leffler for their hospitality where some parts of this work was carried out. The third author thanks Shing-Tung Yau for his continuous support. He is grateful to Emanuel Scheidegger for inviting him to give a lecture series on this topic at University of Freiburg , and to Paul Aspinwall for discussions on the spacetime and worldsheet superpotentials on the B side. The work of C.H. Cho was supported by Samsung Science and Technology Foundation under Project Number SSTF-BA1402-05. The work of S.-C. Lau was supported by Harvard University and Boston University.

%
\mainmatter
%
%
%


\chapter{Introduction}

Mirror symmetry reveals a deep duality between symplectic geometry and complex geometry.  For a mirror pair of geometries $(X,\check{X})$, homological mirror symmetry conjecture proposed by Kontsevich \cite{Kontsevich-HMS} asserts that the (derived) Fukaya category of Lagrangian submanifolds of $X$ (symplectic geometry) is (quasi-)equivalent to the derived category of coherent sheaves of $\check{X}$ (complex geometry), and vice versa.

Homological mirror symmetry has been proved in several interesting classes of geometries, see for instance \cite{PZ-E,KS-T,Fukaya-Absurf,Se,Se2,Abouzaid-toric,FLTZ,Sheridan11,Sh,CHL13,Keating}.  A common approach to homological mirror symmetry begins with a given pair of mirror geometries.  Then one attempts to find generators of the categories on both sides, and compute the morphisms spaces between the generators case-by-case using highly non-trivial techniques in Lagrangian Floer theory and homological algebra.  Finally one needs to construct isomorphisms between the morphism spaces on both sides, often with the help of homological perturbation lemma.  This would give the required quasi-equivalence between the two categories. 

Strominger-Yau-Zaslow \cite{SYZ} proposed a uniform constructive approach to mirror symmetry.  They conjectured that a Calabi-Yau manifold has a special Lagrangian torus fibration, and its mirror can be constructed by taking dual of all the torus fibers.  Moreover it asserted that there should be a certain kind of Fourier-Mukai transform \cite{LYZ} which can produce a functor for realizing homological mirror symmetry.   In symplectic geometry, Fukaya \cite{Fukaya-family} proposed that a family version of Lagrangian Floer theory applied to the torus fibration would give a functor which establishes a quasi-equivalence between the two categories.  

The groundbreaking works of Kontsevich-Soibelman \cite{KS-T,KS} and Gross-Siebert \cite{GS07} formulated very successful algebraic versions of the SYZ program.  Moreover, the works \cite{auroux07, auroux09, CLL, CLLT12, CCLT12, CCLT13, AAK, Lau13} realized the SYZ construction in different geometric contexts using symplectic geometry.  Family Floer theory was further pursued by Tu \cite{Tu-reconstruction,Tu-FT} and Abouzaid \cite{Abouzaid-family,Abouzaid-faithful} and they studied the functor for torus fibrations without singular fibers.

In \cite{CHL13}, we made a formulation which uses weakly unobstructed deformations of a Lagrangian immersion to construct the mirror geometry (which is a Landau-Ginzburg model).  This is particularly useful when a Lagrangian fibration is not known or too complicated in practice.  In this formulation we can always construct a natural functor from the Fukaya category to the mirror B-model category (namely the category of matrix factorizations for a Landau-Ginzburg model).  The theory was applied to the orbifold $\bP^1_{a,b,c}$ to construct its mirror by using an immersed Lagrangian invented by Seidel \cite{Se}, and we showed that the functor derives homological mirror symmetry for elliptic and hyperbolic orbifolds $\bP^1_{a,b,c}$ (with $\frac{1}{a} + \frac{1}{b} + \frac{1}{c} \leq 1$).  The formulation can be regarded as a generalization of SYZ where Lagrangian immersions are used in place of a torus fibration.  Note that \emph{the original SYZ construction is not directly applicable to $\bP^1_{a,b,c}$}.

The natural functor in \cite{CHL13} can be understood as a \emph{formal version of the family Floer functor} studied by \cite{Fukaya-family,Tu-reconstruction,Tu-FT,Abouzaid-family,Abouzaid-faithful}.  Namely, we fix a reference Lagrangian $\bL$ and a space of odd-degree weakly unobstructed formal deformations $b$ of $\bL$.  The mirror functor, which is essentially a family version of curved Yoneda embedding, is constructed by using the (curved) Floer theory between the family $(\bL,b)$ and other Lagrangians in the Fukaya category.  For Lagrangian fibrations, the singular fibers and wall-crossing phenomenons make it very difficult to construct the family Floer functor in an explicit way.  On the other hand, our formal family theory allows us to start right at an immersed Lagrangian, and there is no wall-crossing occurring in our construction.  As a result the construction can be made very clean and explicit.

The idea of applying Yoneda functor to the study of Fukaya category was introduced in \cite[Section 1]{Kho-Sei} by Khovanov-Seidel.  They asserted that from Fukaya's construction in the exact case, there is an $A_\infty$-algebra $\cA$ associated to a fixed collection of Lagrangians $\bL := \{L_1,\ldots,L_N\}$, and there exists a functor $D^b \Fuk(X) \to D(\cA)$. 

But our family version of construction is somewhat different from these approaches. The non-commutative algebras in this paper are
obtained from Maurer-Cartan solutions rather than the whole Hom space $\textrm{Hom}(\bL,\bL)$.
Also, even though we use compact Lagrangians 
$\bL$ to define our functor in our examples, it can still be used to transform non-compact Lagrangians, which send their infinite dimensional hom spaces (between non-compact Lagrangians) into infinite dimensional hom's between matrix factorizations. For example, this happens in the case of punctured Riemann surfaces, where we use compact Lagrangians to send wrapped Fukaya category to a matrix factorization category  and this gives an embedding in good cases. 
Also we work out the construction of the mirror and functor in detail for possibly curved and non-exact situations.  It is important for us to consider weakly unobstructed odd-degree deformations in order to obtain the mirrors.  Moreover the class of geometries that we are interested in are generally non-Calabi-Yau.  Thus we need to consider curved $\Z_2$-graded Fukaya categories and categories of matrix factorizations in the mirror.

In this paper, we extend our study in \cite{CHL13} to noncommutative deformation directions. Such a generalization is essential and natural from the viewpoint of Lagrangian Floer theory as explained below. 

First, it is important to enlarge the deformation space by noncommutative deformations.  In \cite{CHL13}, we consider the solution space of the weak Maurer-Cartan equation associated to $\bL$ in the sense of \cite{FOOO}.  The solution space can be trivial (namely the zero vector space), which is indeed the case for a generic Lagrangian $\bL$.  The condition that the solution space being non-trivial requires certain symmetry on the moduli
of holomorphic discs bounded by $\bL$.  Roughly speaking it requires that the counting of holomorphic discs through a point $p \in \bL$ is independent of the choice of $p$.  \emph{By using  noncommutative deformations, we can construct a non-trivial solution space in generic situations.}
As an example, the Seidel Lagrangian in a three-punctured sphere (see Figure \ref{fig:333intronew} in Chapter 5) is generically non-symmetric about the equator.  There is no non-zero commutative solution to the weak Maurer-Cartan equation.  On the other hand the noncommutative solutions form an interesting space $\C\langle x,y,z \rangle / \langle xy - ayx, yz - azy, zx - axz \rangle$.  In Section \ref{sec:deform333}, we construct the Sklyanin algebras as mirrors to the elliptic orbifold $\bP^1_{3,3,3}$ using noncommutative deformations.

Second, allowing noncommutative deformations gives a practical method to globalize our mirror construction.  Namely, the mirror functor constructed in \cite{CHL13} is localized to the Lagrangian immersion $\bL$ fixed in the beginning. In this paper, we take a family of Lagrangian immersions $\bL = \{L_1,\ldots,L_k\}$ and apply them to mirror construction.  Using a collection of Lagrangian immersions can capture more about the global geometry of $X$ and hence producing a more global functor.  \emph{Dimer models in dimension one or skeletons of coamoebas in general dimensions provide methods to produce a collection of Lagrangians} in concrete situations \cite{FHKV,Ishii-Ueda,UY11,UY13,Bocklandt}.  For example,  we will construct the mirror of the elliptic orbifold $\bP^1_{2,2,2,2}$, which cannot be produced by \cite{CHL13} using only one Lagrangian immersion. 
 
In order to extend the construction from one Lagrangian to several Lagrangians, the use of noncommutative deformations is natural and unavoidable.  For the collection $\bL = \{L_1,\ldots,L_k\}$, the endomorphism algebra is a path algebra of a directed graph and is naturally noncommutative.  While one could always quotient out the commutator ideals to obtain a commutative space, this generally loses information and the solution space of weak Maurer-Cartan equation would become too small in general.

The construction is briefly summarized as follows.  The odd-degree Floer-theoretical endomorphisms of $\bL$ are described by a directed graph $Q$ (so-called a quiver).  The path algebra $\Lambda Q$ is regarded as the noncommutative space of formal deformations of $\bL$.  Each edge $e$ of $Q$ corresponds to an odd-degree Floer generator $X_e$ and a formal dual variable $x_e \in \Lambda Q$.  Consider the formal deformation $b = \sum_e x_e X_e$.  The obstruction is given by
$$ m_0^b = m(e^b) = \sum_{k \geq 0} m_k (b,\ldots,b). $$

A novel point is \emph{extending the notion of weakly unobstructedness by Fukaya-Oh-Ohta-Ono \cite{FOOO} to the noncommutative setting}.
The corresponding weak Maurer-Cartan equation is
$$ m_0^b = \sum_{i=1}^k W_i(b) \one_{L_i} $$
where $\one_{L_i}$ is the Floer-theoretical unit corresponding to the fundamental class of $L_i$
(we assume that Fukaya $\AI$-category is unital).
  The solution space is given by a quiver algebra with relations $\cA = \Lambda Q / R$ where $R$ is the two-sided ideal generated by weakly unobstructed relations.  The end product is a noncommutative Landau-Ginzburg model
$$ \left( \cA, W = \sum_i W_i \right). $$

We call this to be a generalized mirror of $X$, in the sense that there exists a natural functor from the Fukaya category of $X$ to the category of (noncommutative) matrix factorizations of $(\cA, W)$.  It is said to be `generalized' in two reasons.  First, the construction can be regarded as a generalization of the SYZ program where we replace Lagrangian tori by immersions.  Second, the functor needs not to be an equivalence, and so $(\cA,W)$ needs not to be a mirror of $X$ in the original sense.

\begin{theorem}[Theorem \ref{thm:F}]
There exists an $A_\infty$-functor $\cF^\bL: \Fuk(X) \to \MF(\cA,W)$, which is injective on $H^\bullet(\Hom(\bL, U))$ for any $U$.
\end{theorem}

An important feature is that the Landau-Ginzburg superpotential $W$ constructed in this way is automatically a central element in $\cA$.  In particular we can make sense of $\cA / \langle W \rangle$ as a hypersurface singularity defined by `the zero set' of $W$. 
\begin{theorem}[Theorem \ref{thm:central1} and \ref{thm:centralseveral}]
$W \in \cA$ is a central element.
\end{theorem}

Noncommutative formal space dual to an $A_\infty$-algebra was constructed by Kontsevich-Soibelman \cite{KS_Ainfty}.  The quiver algebra with relations $\cA$ constructed in this paper can be regarded as a subvariety in the formal space (since we restrict to weakly unobstructed odd-degree deformations).  Moreover we obtain a Landau-Ginzburg model instead of a CY algebra.  As we have emphasized several times, weak unobstructedness is crucial for the construction of the mirror functor.  In Chapter \ref{sec:ext-mir} we formulate a mirror construction using the whole formal dual space.  The resulting mirror would be a curved dg-algebra.  However it is very difficult to compute in practice, since it involves all the higher $m_k$ operations.

One interesting feature in this approach is multiple mirrors to a single symplectic geometry $X$.  Note that the construction of a Landau-Ginzburg mirror $(\cA_\bL,W_\bL)$ and the associated mirror functor $\mathcal{F}^\bL: \Fuk(X) \to \MF(\cA_\bL,W_\bL)$ depends on the choice of $\bL \subset X$.  For a different choice $\bL'$, the mirror $(\cA_{\bL'}, W_{\bL'})$ becomes different, even though their $B$-model categories $\MF(\cA_\bL,W_\bL)$ and $\MF(\cA_{\bL'}, W_{\bL'})$ can remain equivalent.
Similar phenomenon occurs in SYZ mirror symmetry for multiple Lagrangian fibrations \cite{LY-ell,LL-twin}.  In Chapter \ref{sec:333} and \ref{sec:4puncture2222}, we will construct a continuous family of Landau-Ginzburg models $(\cA_t,W_t)$ which are mirror to one single elliptic orbifold (with a fixed K\"ahler structure).

This noncommutative construction has interesting applications even beyond mirror symmetry.  First, it provides a
Floer-theoretical understanding of noncommutative algebras and their central elements,
which are fundamental objects in noncommutative algebraic geometry and mathematical physics. 
For instance, in many situations a noncommutative algebra arises as the Jacobi algebra of a quiver $Q$ with a potential $\Phi$.  
In our formulation, such a $\Phi$ naturally appears as a spacetime superpotential (see Definition \ref{def:Phi}), and its Jacobi algebra forms the solution space to the weak Maurer-Cartan equation.

In practice, a central element of a noncommutative algebra is rather hard to find, even with the help of computer programs.  In our construction we can always show that the Lagrangian Floer potential $W$ is always a central element
(see Theorem \ref{thm:central1}).  As a result, we can compute the central elements using Floer theoretical techniques.

More importantly, our construction provides an interesting link between Lagrangian Floer theory and deformation quantization in the sense of Kontsevich \cite{Kon1}.  The scenario is the following.  Suppose that for a given choice of $\bL$, the constructed mirror quiver algebra $\cA$ can be obtained from a exceptional collection of a variety $\check{X}$.  Then we obtain a commutative mirror $(\check{X},W)$ which is equivalent to the noncommutative model $(\cA,W)$.  Now consider a deformation family of Lagrangians $\bL_t$ with $\bL_0 = \bL$.  The deformations can be obtained by equipping $\bL$ with flat $U(1)$ connections or deforming $\bL$ using non-exact closed one-forms on $\bL$.  Our mirror construction then gives a family of noncommutative Landau-Ginzburg models $(\cA_t,W_t)$.  As a consequence, the noncommutative hypersurfaces $\cA_t / W_t$ should give deformation quantization of the hypersurface $\{W=0\} \subset \check{X}$.  We realize this idea for mirrors of the elliptic orbifolds $\bP^1_{3,3,3}$ and $\bP^1_{2,2,2,2}$.



Recall that the elliptic orbifold $X = E / \Z_3 = \bP^1_{3,3,3}$ has a commutative Landau-Ginzburg mirror $(\C^3, W)$ which was constructed in \cite{CHL13} by using the immersed Lagrangian $\bL$ given by Seidel \cite{Se}.  $\bL$ is symmetric under an anti-symplectic involution on $\bP^1_{3,3,3}$.  In Chapter \ref{sec:333}, we break this symmetry and consider a deformation family of immersed Lagrangians $\bL_{h}$ for $h \in (\bS^1)^2$ with $\bL_{h=0} = \bL$.  Applying our mirror construction to a fixed $\bL_{h} \subset X$, we obtain $(\cA_{h},W_{h})$ as the mirror, where $\cA_{h}$ is a Sklyanin algebra and $W_{h}$ is its central element (Theorem \ref{thm:mainthm333}).  For $h=0$, $(\cA_0,W_0) = (\C[x,y,z],W)$ recovers the commutative mirror.  We can show that the hypersurface $\cA_{h} / W_{h}$ gives a deformation quantization of the affine del Pezzo surface $\C[x,y,z] / W$ (Theorem \ref{dq333}), which was studied by Etingof-Ginzburg \cite{EG} using the Poisson bracket induced from the function $W$.  Note that the deformation quantization we construct here is \emph{global}, in the sense that the deformation parameter $h$ belongs to a compact manifold $(\bS^1)^2$.

Sklyanin algebra is one of the most well-studied noncommutative algebras (see for instance \cite{AS87}, \cite{ATV}).  It takes the form
\begin{equation}\label{eqn:introskyw}
\cA := \frac{ \Lambda <x,y,z> }{ \big( axy + byz +cz^2, ayz+bzy+cx^2, azx+bxz+cy^2 \big)}
\end{equation}
for certain fixed $a,b,c \in \C$.  Note that $\cA=\C[x,y,z]$ when $a=-b, c=0$.
As a graded algebra $\cA$ gives a noncommutative deformation of $\mathbb{P}^2$.  Auroux-Katzarkov-Orlov (\cite{AKO06}, \cite{AKO08}) proved that non-exact deformations of the Fukaya-Seidel category associated with
$$\left(\{(x,y,z) \in (\C^*)^3 \mid xyz=1\},f=x+y+z\right)$$ (given by fiberwise compactifications of the Lefschetz fibration of $f$) is equivalent to derived category of modules over a Sklyanin algebra $\cA$.  In parallel to their results, we construct the noncommutative Landau-Ginzburg model $(\cA,W)$ as the mirror of $f = 0$ (or more precisely its $\Z_3$ quotient).


Zaslow \cite{Zaslow} and Aldi-Zaslow \cite{AZ} also obtained Sklyanin algebras from elliptic curves in the following way.  An ample line bundle gives an embedding of an elliptic curve to $\bP^2$.  To reconstruct $\bP^2$ (or its noncommutative deformations), they considered the following graded ring 
$$R = \bigoplus_{k=0}^\infty \mathcal{O}_Y(k) = \bigoplus_{k=0}^\infty \Hom ( \mathcal{O}, \mathcal{O}_Y(k)) = \bigoplus \Hom_{\Fuk}(S, \rho^k(S)),$$
where $S$ is the Lagrangian mirror to $\mathcal{O}_Y$ and $\rho$ is the symplectomorphism mirror to the tensor operation $\otimes \mathcal{O}_Y(1)$.
Tensoring $\otimes \mathcal{O}_Y(1)$ is the monodromy around the large complex structure limit point.  In the case of elliptic curve, the mirror symplectomorphism is given by the triple Dehn twist $\rho$.

Aldi-Zaslow \cite{AZ} considered a more general symplectomorphism $\rho$ by composing with a translation.  As a result they obtained a (noncommutative) twisted homogeneous coordinate ring of $\bP^2$ by computing the $m_2$-products of the Fukaya category. They find relations between products of degree-one generators, which gives rise to the Sklyanin algebra relations.

We emphasize that our construction systematically defines the mirror coordinate ring $\cA_{\mathbb{L}_{h}}$ by generators and relations coming from the Maurer-Cartan equations.  Moreover in addition to the Sklyanin algebra $\cA_{\mathbb{L}_{h}}$, we obtain its central element $W_{\mathbb{L}_{h}}$ as the Lagrangian Floer potential for $\mathbb{L}_{h}$.  Indeed they are constructed as a global family over $h \in (\bS^1)^2$.  Furthermore from our general theory there exists a family of $\AI$-functors $\mathcal{F}^{\mathbb{L}_{h}}$ from the Fukaya category to the matrix factorization categories of  $W_{\mathbb{L}_{h}}$.\footnote{In fact the derived category of matrix factorizations obtained in this example are all equivalent to each other by Theorem 3.6.2 of Etingof and Ginzburg \cite{EG},
which are also equivalent to the derived category of coherent sheaves on the mirror elliptic curve.}  Thus our work significantly strengthens the results of Aldi-Zaslow and Auroux-Katzarkov-Orlov.

Analogously, the mirror and its noncommutative deformations of $\bP^1_{2,2,2,2}$ can also be constructed using our method (Chapter \ref{sec:4puncture2222}).  In this case we need to take $\bL = \{L_1,L_2\}$ rather than a singleton.  Thus the mirror quiver consists of two vertices.  By taking $\bL$ to be symmetric under an anti-symplectic involution, we obtain the mirror Landau-Ginzburg model $(\check{X},W)$, where $\check{X}$ is the resolved conifold and $W$ is an explicit holomorphic function on $\check{X}$.  The mirror map can also be derived in this approach, see Theorem \ref{thm:open=mir}.  By considering a deformation family $\bL_h$ which breaks the symmetry of $\bL$, we again obtain a deformation quantization $(\cA_h,W_h)$ of the affine del Pezzo surface $\{W=0\} \subset \check{X}$.  The graded algebra $\cA_h$ can be identified as noncommutative quadrics.  The readers are referred to \cite{BP,VdB,OU} for detail on noncommutative quadrics.


The noncommutative construction is applied to punctured Riemann surfaces in Chapter \ref{sec:Rs}.  This is largely motivated from the work of Bocklandt \cite{Bocklandt}, who studied homological mirror symmetry for punctured Riemann surfaces using dimer models. (Lee \cite{HLee16} and Sibilla-Pascaleff \cite{PS16}  proved homological mirror symmetry for punctured Riemann surfaces using different approaches.)
Our construction provides a conceptual explanation on the construction of the mirror objects and also generalize Bocklandt's results.
Given a polygon decomposition of a Riemann surface (with vertices of the decomposition being the punctures), we associate it with a collection of Lagrangian circles and construct a quiver with potential.  When the polygon decomposition is given by a dimer, we will show that our mirror functor in this case agrees with the one constructed combinatorially in \cite{Bocklandt}.  In particular the functor gives an embedding of the wrapped Fukaya category to the category of (noncommutative) matrix factorizations of the dual dimer when it is zig-zag consistent.


In the last chapter, we apply our general theory to construct noncommutative mirrors of non-compact Calabi-Yau threefolds associated to Hitchin systems constructed by Diaconescu-Donagi-Pantev \cite{DDP}.  We restrict ourselves to the $\mathrm{SL}(2,\C)$ case in this paper, in which the Fukaya category was computed by Smith \cite{Smith}.  Applying the results from \cite{Smith}, we deduce that the mirror functor gives an equivalence between the derived Fukaya category generated by a WKB collection of Lagrangian spheres to the derived category of Ginzburg algebra (Theorem \ref{thm:WKB}).  The mirror functor in this paper is in the reverse direction of the embedding constructed in \cite{Smith}.  By the beautiful work of Bridgeland-Smith \cite{BS}, the functor is a crucial object for studying stability conditions and Donaldson-Thomas invariants for the Fukaya category.  Stability conditions of Fukaya category in the surface case were recently understood by \cite{HKK}.  There is a deep relation between stability conditions and cluster algebras discovered by Kontsevich-Soibelman \cite{KS-stab}.  The theory of cluster algebras was well-studied by Fock-Goncharov \cite{FG1,FG2}.  SYZ mirror symmetry for moduli spaces of Hitchin systems with structure groups $\SL_r$ and $\bP\SL_r$ (which are Langlands dual) was found by Hausel-Thaddeus \cite{HT}.

Finally, we remark that our construction is algebraic and can be applied to several different versions of Lagrangian Floer theory.  We assume that the Fukaya category is unital (for weakly unobstructedness) and
cyclic symmetric on Maurer-Cartan elements (to define spacetime superpotential $\Phi$), and $\CF^\bullet(\bL,\bL)$ is finite dimensional.  In the examples that are considered in this paper, we use a Morse-Bott version of Lagrangian Floer theory \cite{Se}, \cite{Seidel-book}, \cite{Sheridan11} to achieve them.  Unitality and finite dimensionality can be also achieved by considering minimal models.

The organization of the article is as follows.  From Chapter \ref{sec:algpre} to \ref{sec:singlelagfunctor}, we introduce the noncommutative mirror and the functor for a single Lagrangian immersion.  The construction for several Lagrangians is given in Chapter \ref{sec:MFseverallag}.  We discuss the mirror construction for finite group quotient (not necessarily Abelian) and Calabi-Yau grading in Chapter \ref{sec:fingpsymm}.  The extended mirror functor using deformations in all degrees is given in Chapter \ref{sec:ext-mir}.  Applications to concrete geometries are given in Chapter \ref{sec:333}, \ref{sec:4puncture2222}, \ref{sec:Rs} and \ref{sec:CY3}. We thank the anonymous referee for valuable suggestions which improved the exposition of the paper.

\chapter{A-infinity algebra over a noncommutative base}\label{sec:algpre}

We generalize the deformation-obstruction theory of $\AI$-algebra by Fukaya-Oh-Ohta-Ono \cite{FOOO} in noncommutative directions.   In this chapter, after reviewing the filtered $\AI$-category, we introduce noncommutative deformations for $\AI$-algebras, and define nc-weak Maurer-Cartan elements and the nc-potential (which will be also called a \emph{worldsheet superpotential} in this paper).

\section{Review of filtered A-infinity category and weak Maurer-Cartan elements}
We recall the notion of filtered $\AI$-algebra (category) from \cite{fukaya_mirror2}, \cite{FOOO} and set up the notations.
The Novikov ring $\Lambda_0$ is defined as
$$\Lambda_0 = \left\{ \left. \sum_{i=1}^\infty a_i \bT^{\lambda_i} \,\, \right| \,\,\lambda_i \in \R_{\geq 0}, a_i \in \C, \lim_{i \to \infty} \lambda_i = \infty \right\} $$
and its field of fraction $\Lambda_0[\bT^{-1}]$ is called the universal Novikov field, denoted by $\Lambda$.
$$F^\lambda \Lambda = \left\{ \left. \sum_{i=1}^\infty a_i \bT^{\lambda_i} \in \Lambda \,\, \right| \,\, \lambda_i \geq \lambda \right\}$$
defines a filtration on $\Lambda$ (as well as $\Lambda$-modules) and we set $F^+\Lambda = \bigcup_{\lambda >0} F^\lambda \Lambda$.
\begin{definition}
A filtered $\AI$-category $\mathcal{C}$ consists of a collection of objects $Ob(\mathcal{C})$,
torsion-free filtered graded $\Lambda_0$-module $\mathcal{C}(A_1, A_2)$ for each pair of objects $A_1,A_2 \in Ob(\mathcal{C})$,
equipped with a family of degree one operations 
$$m_k: \mathcal{C}[1](A_0,A_1) \otimes \cdots  \mathcal{C}[1](A_{k-1},A_k) \to  \mathcal{C}[1](A_0,A_k)$$
for  $A_i \in Ob(\mathcal{C})$, $k=0,1,2,\cdots$.
Each $\AI$-operation $m_k$ is assumed to respect the filtration, and satisfies the $\AI$-equation:
for $v_i \in \mathcal{C}[1](A_i, A_{i+1})$, 
$$ \sum_{k_1+k_2=n+1} \sum_i (-1)^{\epsilon_1} m_{k_1}(v_1,\cdots, m_{k_2}(v_{i},\cdots, v_{i+k_2-1}),v_{i+k_2},\cdots v_n)=0.$$
We denote by $|v|$ the degree of $v$ and  by $|v|' = |v|-1$ the shifted degree of $v_j$.  Then $\epsilon_1 = \sum^{i-1}_{j=1} (|v_j|')$.
\end{definition}
A filtered $\AI$-category with one object is called an $\AI$-algebra. If $m_{\geq 3} = m_0=0$, it
gives a structure of a filtered differential graded category,
where $m_1$ is the differential and $m_2$ is the composition (up to sign: see (2.3) \cite{CHL13}).

Let $A$ be an $\AI$-algebra. When $m_0 \neq 0$, $m_1$ may not define a differential, which
can be seen in the following $\AI$-equation:
\begin{equation}\label{aieq3}
0 = m_1^2(v_1) + m_2(m_0,v_1) + (-1)^{|v|'} m_2(v_1,m_0) 
\end{equation}
 The obstruction and deformation theory of such $\AI$-algebras have been studied by Fukaya-Oh-Ohta-Ono\cite{FOOO},
 who introduced the notion of weak bounding cochains (weak Maurer-Cartan elements).
 
For this purpose, recall that an element $\be_{A} \in \mathcal{C}^0(A,A)$ is called a {\em unit} if it satisfies
$$\begin{cases}
 m_2(\be_{A},v) = v & v \in  \mathcal{C}(A,A') \\
  (-1)^{|w|}m_2(w,\be_A) = w & w \in  \mathcal{C}(A',A) \\
 m_k(\cdots,\be_A, \cdots) =0 & \textrm{otherwise}.
 \end{cases}
$$ 
Note that if $m_0$ is a constant multiple of a unit, then the latter two terms of \eqref{aieq3} vanishes by the
property of a unit. This happens for $\AI$-algebras of monotone Lagrangians. 
In general, a boundary deformation  of a given $\AI$-algebra via an weak Maurer-Cartan element $b$ can be used to
define a deformed $\AI$-algebra $\{m_k^b\}$ such that $m_0^b$ becomes a multiple of a unit.
 Let us use the notation 
$$b^l = \underbrace{b\otimes \cdots \otimes b}_{l},  \;\; e^b:= 1+ b + b^2 + b^3 + \cdots. $$
\begin{definition}
An element $b \in F^{+}\mathcal{C}^1(A,A)$ is a {\em weak Maurer-Cartan } element if
$m(e^b) := \sum_{k=0}^\infty m_k(b,\cdots, b)$ is a multiple of the unit, i.e.
$$m(e^b) = PO(A,b) \cdot \be_A, \;\;\; \textrm{for some}\; PO(A,b) \in \Lambda$$
Denote by $\widetilde{\mathcal{M}}^+_{weak}(A)$ the set of all weak Maurer-Cartan elements.
\end{definition}
Thus $PO(A,b)$ defines a function from $\widetilde{\mathcal{M}}^+_{weak}(A)$ to $\Lambda$.
The positivity assumption of the energy of $b$ is to make sure that the infinite sum in $m(e^b)$ makes sense.
In many important cases,  $PO(A,b)$ indeed extends to the energy zero elements, and such an extension is
also denoted by $PO(A,b): \widetilde{\mathcal{M}}_{weak}(A) \to \Lambda$.
 
In general one should quotient out gauge equivalence between different Maurer-Cartan elements. 
In the examples we consider in this paper, different immersed generators are not gauge equivalent to each other, and hence
we do not quotient by gauge equivalence.

The notion of weak Maurer-Cartan element $b$, which provides a canonical way to deform the given $\AI$-structure, plays an important role both in Lagrangian Floer theory and mirror symmetry and we refer readers to 
\cite{FOOO} for a systematic introduction.
\begin{definition}\label{def:bdeformAI}
Given $b \in F^{+}\mathcal{C}^1(A,A)$,
we define the deformed $\AI$-operation $m_k^b$ as
$$m_k^b(v_1,\cdots, v_k) = \sum_{l_0,\cdots, l_k \geq 0} m_{k+l_0+\cdots+l_k}(b^{l_0},v_1,b^{l_1},v_2,\cdots, v_k,b^{l_k})$$
$$ = m(e^b,v_1,e^b,v_2,\cdots,e^b, v_k,e^b).$$
Then $\left\{m_k^b \right\}$ defines an $\AI$-algebra.
In general, given $b_0,\cdots, b_k \in F^{+}\mathcal{C}^1(A,A)$,
we define 
$$m_k^{b_0,\cdots,b_k}(v_1,\cdots, v_k)= m(e^{b_0},v_1,e^{b_1},v_2,\cdots,e^{b_{k-1}}, v_k,e^{b_k}).$$
Note that we have $m_k^b = m_k^{b,b,\cdots,b}$.
\end{definition}
Given a weak Maurer-Cartan element $b$, we have $m_0^b =   PO(A,b) \cdot \be_A$, and one can check that $(m_1^b)^2 =0$.
And if $PO(A,b_0) = PO(A,b_1)$, we have $(m_1^{b_0,b_1})^2=0$.

\section{$\AI$-functors and Yoneda embedding}
We briefly recall the notion of $\AI$-functors and Yoneda embedding for $\AI$-categories following Fukaya \cite{fukaya_mirror2}
to fix the convention. We will assume that our $\AI$-category is a filtered $\AI$-category
in the sense of \cite{FOOO}.
For an $\AI$-category $\CC$, let us denote 
$$B_k\mathcal{C}(A,B) = \bigoplus_{A=A_0,\cdots, A_k=B} \CC[1](A_0,A_1) \otimes \cdots \otimes \mathcal{C}[1](A_{k-1},A_k).$$
$$B\CC = \bigoplus_{A,B} B\CC(A,B) = \bigoplus_{A,B} \bigoplus_{k=0}^\infty B_k\CC(A,B).$$
The completion of $B\CC$ is denoted as $\widehat{B}\CC$.
We have the coderivation $\widehat{d}$ on $B\CC$ induced from $\AI$-operations $\{m_k\}$, such that 
the $\AI$-equations can be succintly written as  a codifferential equation $\widehat{d}^2=0$
(this naturally extends to  $\widehat{B}\CC$).
\begin{definition}
Given two $\AI$-categories $\CC_1,\CC_2$, an $\AI$-functor $\mathcal{F}:\CC_1 \to \CC_2$ is given by the map on objects and by collection of maps $\mathcal{F}_k$
for $k \in \mathbb{Z}_{\geq 0}$:
\begin{enumerate}
\item A map between objects $\mathcal{F}: Ob(\CC_1) \to Ob(\CC_2)$.
\item A sequence of degree 0 maps (for morphisms)  for any two objects $A_1,A_2 \in Ob(\CC_1)$
$$\mathcal{F}_k: B_k\CC_1(A_1,A_2) \to \CC_2[1](\mathcal{F}(A_1), \mathcal{F}(A_2))$$
\item These should satisfy the $\AI$-functor equations
$$ \widehat{d} \circ \widehat{\mathcal{F}} = \widehat{\mathcal{F}} \circ \widehat{d}$$
where $\widehat{\mathcal{F}}$ is defined by
$$\widehat{\mathcal{F}}(v_1 \otimes \cdots \otimes v_k) =\sum \mathcal{F}(v_1 \otimes \cdots  \otimes v_{i_1}) \otimes 
\mathcal{F}(v_{i_1+1} \otimes \cdots  \otimes v_{i_2})\otimes
\cdots \otimes
\mathcal{F}(v_{i_j+1} \otimes \cdots \otimes v_k)$$
\end{enumerate}
An $\AI$-functor is said to be {\em strict} if $\mathcal{F}_0 =0$. The $\AI$-functors that we construct in this paper will be always strict.
\end{definition}

Now, let $(\CC, \{m_k\})$ be an $\AI$-category with $m_0=0$.
Given an object $A \in Ob(\CC)$, Fukaya \cite{fukaya_mirror2} defined a hom functor $\textrm{Hom}(A,\cdot)$ so that it becomes
an $\AI$-functor from the $\AI$-category $\CC$ to the differential graded category of chain complexes $\mathcal{CH}$.
Here, an object of  $\mathcal{CH}$ is a chain complex over $\Lambda$, and morphism is given by the maps between
two chain complexes. For a morphism $x$ between chain complexes $(C,d), (C',d')$, 
the differential is defined as $m_1(x) = d' \circ x - (-1)^{|x|} x \circ d$, and the composition $m_2(x,y) = (-1)^{|x|(|y|+1)} y \circ x$.
Putting $m_k=0$ for $k\geq 3$, this defines an $\AI$-category (or dg-category) $\mathcal{CH}$.

The $\AI$-Yoneda embedding was introduced by Fukaya  \cite{fukaya_mirror2} as follows.
 (Here we give a slightly modified version from \cite{CHL13}.)
 \begin{definition}
Given an object $A \in Ob(\CC)$, the $\AI$-functor $\mathcal{F}^A : \mathcal{C} \to \mathcal{CH}$ is given by the following maps:
\begin{enumerate}
\item For $B \in Ob(\CC)$,   the corresponding object $\mathcal{F}^A(B)$ is a chain complex defined as 
$$\mathcal{F}^A(B) = (\CC(A,B),m_1)$$
\item For $B_1,B_2, \cdots, B_{k+1} \in Ob(\CC_1)$,
and morphisms $x_i \in \CC(B_i, B_{i+1})$ for $i=1,\cdots, k$, $\AI$-functor map
 $\mathcal{F}_k^A(x_1,\cdots, x_k) \in \mathcal{CH}(\mathcal{F}^A(B_1), \mathcal{F}^A(B_{k+1}))$
is defined as 
$$\mathcal{F}_k^A(x_1,\cdots, x_k)(y) = (-1)^{|y|'(|x_1|' + \cdots + |x_{k}|')} m_{k+1}(y,x_1,\cdots,x_k)$$
for $y \in \CC(A,B_1)$. (Here $k \geq 1$.)
\end{enumerate}
\end{definition}
Note that $\mathcal{CH}$ has $m_{\geq 3}=0$ (dg-category), and the higher $\AI$-morphisms $\mathcal{F}_k$ are needed in general
to define an $\AI$-functor between an $\AI$-category and a dg-category.

\section{$\AI$-algebra over a noncommutative base and weakly unobstructedness} \label{sec:base}
In this section we perform a base change of an $\AI$-algebra $A$.  Originally $A$ is over the Novikov ring, and we enlarge the base to be a noncommutative algebra.  This is an important step for deformations.

Let $K$ be a noncommutative algebra over $\Lambda_0$.
Consider a filtered $\AI$-algebra $(A,\{m_k\})$ over $\Lambda_0$.
We will consider an induced $\AI$-algebra structure on the completed tensor product $ K \widehat{\otimes}_{\Lambda_0} A$
where we take a completion with respect to the energy, namely the power of the formal variable $\bT$.
\begin{definition}\label{defn:nccoeff}
We define an $\AI$-structure on
\begin{equation}\label{eqn:anctensor} 
\tilde{A}_0 :=K \widehat{\otimes}_{\Lambda_0} A  
\end{equation}
For $f_i \in K, e_i  \in A \; i=1,\cdots, k$,  
the $\AI$-operation is defined as
\begin{equation}\label{eqn:pulloutrule}
m_k(f_1 e_{1},f_2 e_{2},\cdots, f_k e_{k} ) := f_k f_{k-1} \cdots f_2 f_1\cdot m_k(e_{1},\cdots, e_{k}).
\end{equation}
Then we extend it linearly to define the $A_\infty$-structure on $\tilde{A}_0$ and also tensor over $\Lambda$ to get the $A_\infty$-structure on
$\tilde{A} := \tilde{A}_0 \otimes \Lambda$.
\end{definition}


\begin{remark}
The order of $f_i$ in \eqref{eqn:pulloutrule} is crucial for the definition of the mirror functor in Chapter \ref{sec:singlelagfunctor} and \ref{sec:MFseverallag}. Also, the variables $f_i$ have degree zero. Later we consider extended deformations with variables $f_i$ of non-trivial gradings, then there will
be additional signs (see \eqref{eqn:pulloutrule2}, \eqref{eqn:pulloutrule3}).
\end{remark}
\begin{remark}
There is a well-known construction of the dual of an $\AI$-algebra, which is a noncommutative
formal space with a vector field $Q$ satisfying $[Q,Q]=0$ (see Kontsevich-Soibelman \cite{KS_Ainfty}). 
Hence. it is not surprising that we can introduce noncommutative coefficients for an $\AI$-algebra.
The relation between \cite{KS_Ainfty} and our work will become clearer for extended deformations in Chapter \ref{sec:ext-mir}.  

In several recent mirror symmetry literatures, the $\AI$-category or $\AI$-algebra itself is regarded as a part of non-commutative geometry. But the definition of  ``non-commutativity'' in our non-commutative mirror symmetry framework has more classical meaning. Namely, we obtain the classical non-commutative algebras via the study of  weakly unobstructedness of the deformations \eqref{eqn:singlelagwmc} and Maurer-Cartan equations.
\end{remark}

One can easily check that
\begin{lemma}\label{lem:ncainfty}
$(\tilde{A}, \{m_k\})$ satisfies the $\AI$-equation.
\end{lemma}
\begin{proof}
From linearity, it is enough to prove it when inputs are multiples of basis elements.
Namely, we consider the expansion of $m(\widehat{m}(f_1 e_{1},f_2 e_{2},\cdots, f_n e_{n}))$  which are given by
$$\sum_{k_1+k_2 =n+1} m_{k_1}(f_1 e_{1},\cdots, m_{k_2}(f_{j+1}e_{j+1}, \cdots ), \cdots, f_n e_{n}).$$
Here, $\widehat{m}$ is the coderivation corresponding to $m$. From  the $\AI$-equation of $A$, we have
$$ f_n f_{n-1} \cdots f_2 f_1 \sum_{k_1+k+2 =n+1}m_{k_1}(e_{1},\cdots, m_{k_2}(e_{j+1}, \cdots ), \cdots, e_{n})=0.$$
\end{proof}
The unit $\be_A$ of $A$ is also the unit of $\tilde{A}$.
Thus the noncommutative version of the weak Maurer-Cartan equation makes sense.
\begin{definition}\label{def:wmc}
An element $b \in F^+\tilde{A}^1$ is said to be a noncommutative weak Maurer-Cartan element if it satisfies
\begin{equation}\label{eqn:singlelagwmc}
m(e^b)= PO(\tilde{A},b) \cdot \be_A  \;\;\; \textrm{for some}\; PO(\tilde{A},b) \in  \Lambda_+K,
\end{equation}
where $m(e^b)$ is defined from the $\AI$-operations of $\tilde{A}$.
The potential  $PO(\tilde{A},b)$ is called a noncommutative potential.
\end{definition}

Later on $PO(\tilde{A},b)$ will also be denoted as $W$ and referred as the worldsheet superpotential.

\chapter{Noncommutative mirror from a single Lagrangian and the centrality theorem} \label{sec:nc-deform}
In the previous chapter, we extend the base of an $A_\infty$-algebra $A$ to be a noncommutative space.  In this chapter, we take the base to be the algebra of odd-degree (or degree-1 in $\Z$-graded case) weakly unobstructed deformations in the dual of $A$.  The construction can be regarded as a noncommutative version of that in \cite{CHL13}.  In \cite{CHL13}, we considered only the
cases that solutions of  Maurer-Cartan equation form a  vector space, but in this paper, we consider a more general setting that
non-trivial Maurer-Cartan equation defines an interesting algebraic variety, which is not necessarily commutative.

The key theorem we derive in this chapter is Theorem \ref{thm:central1}, which states that the potential function $W$ is automatically a central element. This already makes the theory rather interesting since finding central elements of a noncommutative algebra is in general challenging and done by computer calculations, while our construction automatically provides  a central element (which is a non-trivial element in non-Calabi-Yau situation).

For the purpose of mirror construction, the $A_\infty$-algebra is taken to be the Floer chain complex of a single (immersed) Lagrangian in a K\"ahler manifold $X$.  Then the base space together with the potential $W$ we construct here is regarded as a mirror (noncommutative Landau-Ginzburg model) of $X$ in a generalized sense. Namely, we will show in Chapter \ref{sec:singlelagfunctor} that there exists a natural functor from the Fukaya category of $X$ to the category of matrix factorizations of $W$.  Thus symplectic geometry of $X$ is (partially) reflected from the (noncommutative) algebraic geometry of $W$.  However the functor is not necessarily an equivalence in general. 
We remark that the above construction is algebraic, and hence can be applied to any $\AI$-algebras or $\AI$-categories, not necessarily a Fukaya category.

In Chapter \ref{sec:MFseverallag} we shall perform a more powerful construction using finitely many Lagrangians (instead of a single Lagrangian), and the corresponding base space will form a quiver algebra with relations.  This provides a useful approach to make the generalized mirror more global and capture more geometric information of $X$.  It will be used to study the mirror of punctured or orbifold Riemann surfaces and certain non-compact Calabi-Yau threefolds constructed by Smith \cite{Smith}.

We may also consider the whole dual algebra of $A$ instead of taking only odd-degree noncommutative weakly unobstructed deformations.   This would result in curved dg-algebras which will be constructed in Chapter \ref{sec:ext-mir}.  We will mostly stick with the smaller deformation space though, since it can be better understood and is more practical to compute.

\section{Noncommutative family of Maurer-Cartan elements}
Let $\CC$ be an $\AI$-category over $\Lambda$, and we fix an object $\bL$ of $\CC$, which is called a reference object.
We will consider a family of Maurer-Cartan equations for $\bL$ as follows.

Let $\{X_e\}$ be a set of  odd-degree generators in $A:=\textrm{Hom}(\bL,\bL)$, where $e$ runs in a label set $E$ (which we assume to be finite).  One may take all odd-degree generators or its certain subset (to make the deformation space smaller). 
Let $x_e$ be the corresponding formal dual variables, and take $K = \Lambda_0 \ll  x_e : e \in E \gg$
be the tensor algebra on the subspace of $A$ corresponding to $E$.
Then we obtain an $A_\infty$-algebra $\tilde{A}:= K \widehat{\otimes}_{\Lambda_0} A$ with coefficients in the noncommutative ring $K$.

To consider a family of Maurer-Cartan elements, we take the formal sum
$$b = \sum_{e} x_e X_e.$$

For $\Z_2$ grading, we  set $x_e$ to have degree zero so that $b$ has  total degree one. 
When we consider extended deformations later, we set the degree as follows.
\begin{definition} \label{def:deg}
The degree of $x_e$ is defined to be $|x_e| = 1 - |X_e|$, so that $b$ has  total degree one.
\end{definition}

We  find the condition of $\{x_e\}$ that makes $b$ a Maurer-Cartan element
(Def. \ref{def:wmc}).
$$m(e^b) := m_0(1)+ m_1(b) + m_2(b,b) + \cdots = PO(\tilde{A},b) \cdot \one_\bL.$$
Since $m(e^b)$ has even degree, we may write
\begin{equation}\label{eq:wmcbl}
m(e^b) = W_\bL \cdot \one_\bL + \sum_{f} P_{f} X_{f}
\end{equation}
where $\one_\bL$ is the unit of  $\textrm{Hom}(\bL,\bL)$ and $X_f$ runs over all the even-degree generators of 
$\textrm{Hom}(\bL,\bL)$ other than $\one_\bL$ (and $f$ runs in a certain index set $F$).
Now, we will redefine our coefficient ring so that $b$ become Maurer-Cartan elements.
\begin{definition} \label{def:mir-sp}
Define the ring 
$$ \cA := \Lambda_0 \ll x_e : e \in E \gg \big/ \langle P_f: f \in F \rangle,$$
where $\langle P_f: f \in F \rangle$ is a completed two sided ideal generated by $P_f$'s. 
We call $\{P_f\}$ the weakly unobstructed relations or
Maurer-Cartan equations.

The potential $W = W_\bL$ (which is the coefficient of $\one_\bL$ in $m_0^b$) can be regarded as an element in $\cA$.  
\end{definition}
We regard $W_\bL$ as an element of $\cA$, since the Maurer-Cartan equation holds only over the ring $\cA$.
In geometric situations, the coefficient of the unit $\one_\bL$ in $m(e^b)$ can be read off  by taking a pairing with
a point class of $\bL$. If $b$ is
a Maurer-Cartan element, then this number is independent of the choice of a point.

Classically, Maurer-Cartan elements contain the deformation information of $\bL$. Hence the above construction may be interpreted as considering formal (non-commutative) deformations of $\bL$ and the deformation space is given by the non-commutative algebra $\cA$, together with an invariant $W_\bL(b)$ for each $b$.

\begin{remark}
One can check that given an $\AI$-homomorphism $A \to B$, there is a map of associative algebras $\mathcal{B} \to \mathcal{A}$.
\end{remark}

\section{Noncommutative mirrors using a single immersed Lagrangian} \label{sec:single}
Let us apply the previous algebraic formalism to an immersed Lagrangian.
Let $X$ be a K\"ahler manifold, and $\bL \stackrel{\iota}{\rightarrow} X$ be a compact spin oriented unobstructed Lagrangian immersion with transverse self-intersection points.  (By abuse of notations we also denote its image by $\bL$.)  

By \cite{AJ} (generalizing that of \cite{FOOO} for the case of Lagrangian submanifolds), it is associated with an $A_\infty$-algebra whose underlying vector space is the space of cochains 
$$\CF^\bullet(\bL,\bL) := \CF^\bullet(\bL \times_\iota \bL) = C^\bullet(\bL) \oplus \bigoplus_p \Span \{(p_- ,p_+), (p_+,p_-) \},$$ where $p$ runs over the immersed points and $p_-,p_+$ are its preimages under $\iota$.  We call $(p_-,p_+)$ and $(p_+,p_-)$ \emph{complementary immersed generators}.  The vector space is $\Z_2$-graded in general, and is $\Z$-graded when $X$ is Calabi-Yau and $\bL \subset X$ has Maslov class zero. 

Recall that the Lagrangian Floer complex $CF(L_1,L_2)$ for two different transversally intersecting Lagrangians are generated by their intersection points. Intuitively, at each self intersection point $p$ of the immersed Lagrangian $\bL$, two branches of the Lagrangian $\bL$, say $L_{+},L_{-}$, meet transversally, and two generators $(p_-,p_+)$ and $(p_+,p_-)$ correspond to the Floer generators  $CF(L_-,L_+)$ and $CF(L_+,L_-)$ respectively. 

In fact, as we have seen the construction  is algebraic and hence can be applied to any version of Fukaya $\AI$-category
as defined by Fukaya-Oh-Ohta-Ono \cite{FOOO}, Akaho-Joyce \cite{AJ},  Seidel \cite{Seidel-book}, Seidel-Sheridan \cite{Sheridan11,Sh} or Abouzaid-Seidel \cite{AS10}.  The space of cochains $C^\bullet(\bL)$ (and its $\AI$-operations) depends on the chosen theory.  Here we assume that if necessary, a canonical model is taken so that $C^\bullet(\bL)$ is finite dimensional.  Moreover the $A_\infty$-algebra is assumed to be unital, and the unit is represented by the fundamental class of $\bL$.


Now we consider noncommutative formal deformations by the immersed generators.  Let $\{X_e\}$ be the set of all odd-degree immersed generators in $CF^\bullet(\bL, \bL)$ where $e$ runs in a label set $E$ of immersed points (which is a finite set).
One can also include odd degree generators which is not immersed generator, but this case will not be considered in this paper.
As before, let $x_e$ be the corresponding formal dual variables, and take $K = \Lambda_0 \ll  x_e : e \in E \gg$ and
we obtain an $A_\infty$-algebra $\tilde{A} = K \widehat{\otimes}_{\Lambda_0} A$.


%

The formal bounding cochain
$$b = \sum_{e} x_e X_e$$
represents formal smoothing deformation of $\bL$ at each intersection point $e$.

We can consider the deformed $\AI$-algebra  $\tilde{A}^b$ with $\AI$-operations  $\left\{m_k^b \right\}$ from Definition \ref{def:bdeformAI}.
Weak Mauer-Cartan equation for $b$ is the condition for $m_0^b$ to be a multiple of the unit class $\one_\bL$ is the unit of the Floer theory of $\bL$.
From the weak MC equation \eqref{eq:wmcbl}, and Definition \ref{def:mir-sp}, we obtain the non-commutative ring $\cA$
of weak Maurer-Cartan relations.

%

\begin{definition}
The potential $W = W_\bL$ (which is the coefficient of $\one_\bL$ in $m_0^b$)  as an element of $\cA$ is called a \emph{generalized mirror} of $X$ (with respect to the choice of an immersed Lagrangian $\bL$).
$W$ is also called the \emph{worldsheet superpotential}. 
\end{definition}
Recall that $m_0$ counts holomorphic discs bounded in $\bL$ passing through a generic point.
The worldsheet superpotential $W_\bL$  counts holomorphic discs bounded in $\bL$ passing through a generic point with boundary insertions of $b$'s. Namely, holomorphic discs are now allowed to turn corners at $X_e$'s and the names of the corners are recorded in the variables $x_e$'s.
We also remark that even in the case of $m_0=0$, the above construction may provide non-trivial $m_0^b$,
and this is the case in most of the Riemann surface examples in this paper.

Note that since $m_0^b$ has degree two, and hence $W$ has degree two.  Later in this section, we show that $W \in \cA$ is always central.  Thus $W$ can be regarded as a function on the noncommutative space defined by $\cA$.

\begin{remark}\label{rmk:wcountst}
In actual computations, we will use the Morse complex of $\bL$ (induced from a Morse function) which contains one maximum point $\one_{L}$ and one minimum point $T_{L}$.  Concretely $W$ counts the number of polygons bounded by $\bL$ whose boundaries pass through $T_{L}$, which is the Poincar\'e dual of $\one_{L}$.
\end{remark}


\section{Spacetime superpotential and Maurer-Cartan equation}
There is another type of superpotential different from $W$, called the \emph{spacetime superpotential}, which gives an effective way to describe the weakly unobstructed relations under good conditions.  This employs the intersection pairing $\langle \cdot,\cdot \rangle$ on the space $\CF^*(\bL, \bL)$, which is graded commutative (before shifting degrees) and has degree $- \dim \bL$. 

The following cyclic potential is defined as an element in $K$ or (its projection to) the cyclic quotient $K/[K,K]$ by (completed) commutators.
\begin{definition} \label{def:Phi}
The spacetime superpotential is defined to be
\begin{equation}\label{eq:cysp}
\Phi = \Phi_{\mathbb{L}}:=\sum_{k} \frac{1}{k+1} \langle m_k (b,\cdots,b), b \rangle \in K/[K,K] \end{equation}
\end{definition}

We make the following assumption whenever spacetime superpotential $\Phi$ is discussed.
\begin{assumption}
Assume that the $A_\infty$-algebra associated to $\bL$ is cyclic, namely
$$ \langle m_k(v_1,\cdots,v_k), v_{k+1} \rangle = (-1)^{|v_1|'(|v_2|'+ \cdots + |v_{k+1}|')}\langle m_k(v_2,\cdots, v_{k+1}), v_1 \rangle $$
for any $v_1,\cdots, v_{k+1} \in C^*(\bL \times_\iota \bL)$.
\end{assumption}
We remark that cyclic symmetric $\AI$-algebra for a Lagrangian submanifold is constructed by Fukaya \cite{Fukaya_cyclic}.
In fact, for the definition of the spacetime supterpotential (and for the proposition below), the cyclic symmetry of the expression in \eqref{eq:cysp} is enough, which holds all of our examples.

Due to the assumption, every term of $\Phi$ appears with its cyclic permutation in the variables $x_e$.  For such a cyclic element of $\Lambda_0 \langle \langle x_e : e \in E \rangle \rangle$, we have the cyclic derivatives which are defined by
\begin{equation}\label{def:pd}
\partial_{x_e} (x_{l_1} \cdots x_{l_k}) = \sum_{j=1}^k \partial_{x_e}(x_{l_j}) \cdot x_{l_{j+1}} \cdots x_{l_k} x_{l_1} \cdots x_{l_{j-1}}
\end{equation}
where $\partial_{x_e}(x_{l}) = \delta_{e, l}$ (see \cite{RSS} or \cite{DWZ} for more details).

The following proposition is similar to the one for special Lagrangian submanifolds
in Calabi-Yau 3-folds in  \cite{FOOO} (3.6.49) and \cite{Fukaya_cyclic}.

\begin{prop} \label{prop:mccd}
Partial derivatives $\partial_{x_e} \Phi$ for $e \in E$ are (part of) weakly unobstructed relations. 
Furthermore, suppose $\dim_\C X=n$ is odd, and $\bL$ is a Lagrangian sphere with the non immersed sector ($C^\bullet (\bL)$) of $\CF^\bullet(\bL, \bL)$ isomorphic to the Morse complex of $S^n$ equipped with the standard height function. Then the algebra  in Definition \ref{def:mir-sp} equals to 
$$\cA = \Jac (\Phi) = \Lambda_0 \langle \langle x_e : e \in E \rangle \rangle /( \partial_{x_e} \Phi: e \in E).$$
\end{prop}

\begin{proof}
 From the cyclic symmetry and the definition of $\partial_{x_a}$, we have
\begin{eqnarray*}
\partial_{x_e} \Phi_\mathbb{L} &=& \sum_k \langle m_k (b, \cdots, b),  X_{e} \rangle \\
&=& \langle P_{\bar{e}} \cdot X_{\bar{e}}, \, X_{e} \rangle = \pm P_{\bar{e}}.
\end{eqnarray*}
where $X_{\bar{e}}$ is the complementary immersed generator to $X_e$.
Thus $\partial_{x_e} \Phi_\mathbb{L}$ for each $e$ gives a part of weakly unobstructed relation.

When $\dim_\C X$ is odd, $X_{\bar{e}}$ for every odd-degree immersed generator $X_e$ has even degree. 
Suppose further that $\bL$ is a sphere
and $\CF^\bullet(\bL, \bL)$ is generated by exactly one maximum point and one minimum point. The minimum point $T$ have $\deg T = \dim \bL$ which is odd, and hence cannot be a term in $m_0^b$.  Therefore $m_0^b$ is a linear combination of the even-degree immersed generators, and all the weakly unobstructed relations are $\partial_{x_e} \Phi_\mathbb{L}$ for some $e$.
Hence $\cA = \Jac (\Phi)$ in this case.
\end{proof}
\begin{remark}
The idea of relating critical points of spacetime superpotential and Maurer-Cartan equation appeared
 in \cite[(3.6.49)]{FOOO} for the case of  special Lagrangian submanifolds in Calabi-Yau 3-folds.
\end{remark}

\section{Centrality}

The following centrality theorem is an important feature of our noncommutative mirror construction.  
Note that our approach gives a conceptual way to produce a central element of a non-commutative algebra, and moreover the proof is surprisingly neat which employs weak unobstructedness in a systematic way.


\begin{theorem}[Centrality of the potential] \label{thm:central1}
The noncommutative potential $W_\bL \in \cA$ is central, namely
$$W_\bL \cdot x_e =  x_e \cdot W_\bL $$
for all $x_e$.
\end{theorem}
\begin{proof}
The key is that we have $m_0^b = W_{\bL} \one_\bL$ over $\cA$.
Let us denote by $\hat{m}$ the coderivation induced from $m$, and the $\AI$-equation
can be rephrased as $m \circ \hat{m} =0$ on tensor coalgebra of the $A_\infty$-algebra.
Then
\begin{eqnarray*}
0 = m\circ \hat{m} (e^b) &= & m(e^b, \hat{m}(e^b),e^b) \\
&=& m_2(b,W_{\mathbb{L}} \one_\bL) + m_2(W_{\mathbb{L}} \one_\bL,b) \\
&=& \sum_{e} (W_{\mathbb{L}} \cdot x_e - x_e \cdot W_{\mathbb{L}}) X_e.
\end{eqnarray*}
The second and third equalities follows from the property of the unit $\one_\bL$.
\end{proof}

\begin{remark}
Centrality is also needed to have $(m_1^{b,b})^2 =0$.  For any $x$, the centrality implies the cancellation of the last two terms in the following $\AI$-equation 
$$(m_1^{b,b}(x))^2  + m_2(W_\mathbb{L} \be, x) + (-1)^{|x|'} m_2(x,  W_\mathbb{L}\be) = 0.$$
\end{remark}

We conclude this section with the following.

\begin{definition}
A triple $(\Lambda_0 \ll x_e: e \in E \gg, \Phi, W)$ is called a {\em noncommutative Landau-Ginzburg model}
if the noncommutative Jacobi algebra 
$$\cA = \Jac(\Phi) := \Lambda_0 \ll x_e: e \in E \gg \big/ (\partial_{x_e} \Phi)$$
has $W$ in its center.
\end{definition}

\begin{cor}
Under the conditions in Proposition \ref{prop:mccd}, the generalized mirror 
$$(\Lambda_0 \ll x_e: e \in E \gg, \Phi, W)$$ 
associated to $\bL$ defined in Section \ref{sec:single} is a noncommutative Landau-Ginzburg model.
\end{cor}

In the next chapter we transform Lagrangian submanifolds in $X$ to matrix factorizations of $W$.  Thus symplectic geometry naturally corresponds to noncommutative algebraic geometry in our generalized mirror.

%
%
%


\chapter{The mirror functor from a single Lagrangian} \label{sec:singlelagfunctor}
In this chapter, we construct a canonical $\AI$-functor from the Fukaya category of $X$ to
the category of matrix factorizations of its generalized mirror $(\cA,W)$ constructed in the previous chapter.  
It explains the fundamental reason why symplectic geometry of $X$ should correspond to algebraic geometry of its mirror.  Most of the proofs in this chapter is similar to those in \cite{CHL13} in the commutative setting, and so we will be brief.  $\Fuk(X)$ is the Fukaya category of unobstructed Lagrangians with a bounding cochain in this chapter. 
If $L$ is unobstructed,  then, either $L$ has a vanishing $m_0$ or has a bounding cochain $b_0$ such that $m_0^{b_0} =0$. 
For simplicity, we assume that $b_0=0$ (i.e. the first case)  since it is straightforward to modify the construction nontrivial $b_0 \neq 0$ (see \cite{CHL13} for more details).

First let us recall the category of matrix factorizations of a central element $W$ in a ring $\cA$ which is not necessarily commutative.

\begin{definition}
A matrix factorization of $(\cA,W)$ consists of two projective left $\cA$-modules $M_0, M_1$ with
maps $\delta_0:M_0 \to M_1, \delta_1:M_1 \to M_0$ such that 
$$\delta_0 \delta_1 = \delta_1 \delta_0 = W \cdot \mathrm{Id}.$$
Alternatively it can be defined as a $\Z_2$-graded projective $\cA$-module $M = M_0 \oplus M_1$ with an odd endomorphism $\delta = \delta_0 + \delta_1$ satisfying $\delta^2 = W \cdot \mathrm{Id}$.
\end{definition}

\begin{remark}
We may also consider $\Z$-graded matrix factorizations in Calabi-Yau setting.  This will be discussed in Section \ref{sec:CY-grading}.
\end{remark}

\begin{definition} \label{def:MF}
The differential graded category of matrix factorizations of $(\cA,W)$ is defined as follows.
A morphism between two matrix factorizations $(M,\delta_M)$ and $(N,\delta_N)$ is simply an $\cA$-module homomorphism $f:M \to N$.
The morphism space is $\Z_2$-graded and is equipped with a differential $m_1$ (or $d$) defined by
$$m_1(f) = \delta_N \circ f - (-1)^{deg(f)} f \circ \delta_M.$$
The composition of morphisms is denoted by $m_2$.
\end{definition}

Now we are ready to construct the functor.  Recall that $\cA$ is the space of weakly unobstructed odd-degree noncommutative deformations of the reference Lagrangian $\bL$, and $W$ is the counting of holomorphic discs bounded by $\bL$ whose boundary pass through a generic marked point (see Definition \ref{def:mir-sp} and Remark \ref{rmk:wcountst}).

\begin{definition} \label{def:mirMF}
Let $\Fuk(X)$ be the $\Z_2$-graded Fukaya category of unobstructed Lagrangian submanifolds of $X$.  For an object $U$ of $\Fuk(X)$, its mirror matrix factorization of $(\cA,W)$ is defined as 
\begin{equation}\label{eqn:objcorr1}
\cF(U) := \left(\cA \widehat{\otimes}_{\Lambda_0} \CF^{\bullet}((\bL,b), U), d = (-1)^{\deg (\cdot)} m_1^{b,0}(\cdot) \right).
\end{equation}
\end{definition}
\begin{remark}
The above functor can be defined for any $\Fuk_\lambda$ for $\lambda \in \C$ (which maps to $MF(W-\lambda)$), where $\Fuk_\lambda$ is the subcategory
whose objects are weakly unobstructed Lagrangians with potential value $\lambda$. See \cite{CHL13} for more details. The same applies throughout the paper, but we will consider only unobstructed ones just for simplicity.
\end{remark}

There are different ways of realizing Lagrangian Floer theory, and $\CF^{\bullet}((\bL,b), U)$ depends on the choice.  In the theory developed by Fukaya-Oh-Ohta-Ono \cite{FOOO}, $U$ can have clean intersections with $\bL$, and $\CF^{\bullet}((\bL,b), U)$ is the direct sum of the spaces of singular chains (or differential forms) in the intersections.  In the theory of Seidel \cite{Seidel-book}, $U$ is made to have transverse intersections with $\bL$, and $\CF^{\bullet}((\bL,b), U)$ is spanned by the intersection points.  We will fix any one of the choices to carry out the constructions.

Recall that the Floer theory operation $m_1^{b,0}$ on $\cA \otimes_{\Lambda_0} \CF^{\bullet}((\mathbb{L},b), U)$ is
$$ m_1^{b,0}(p) = \sum_{k=0}^\infty m_{k+1}(\underbrace{b,\cdots,b}_k,p)$$
for $p \in \CF^{\bullet}((\mathbb{L},b), U)$.  Intuitively the mirror matrix factorization is counting pseudo-holomorphic strips bounded by $\bL$ and $U$, where arbitrary number of boundary insertions $b$'s are allowed along the upper boundary mapped in $\bL$ (see Figure \ref{fig:stripb}).  The matrix coefficients are noncommutative power series, each term of which records the boundary insertions along the upper boundary. 

Recall that the intersection points $\bL \cap U$ has a well-defined $\Z_2$-grading given by
the signed intersection number.
The differential $m_1^{b,0}$ is  a matrix of decorated counting from even to odd, or odd to even intersection points. 


\begin{figure}[h]
\begin{center}
\includegraphics[height=0.8in]{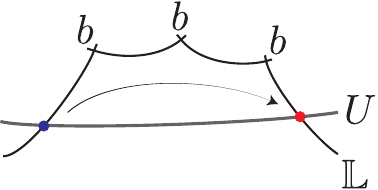}
\caption{Holomorphic strip contributing to $m_1^{b,0}$}\label{fig:stripb}
\end{center}
\end{figure}

The key observation is that the above really defines a matrix factorization of $(\cA,W)$, whose proof is essentially the same as that of \cite[Theorem 2.19]{CHL13}.

\begin{prop}
In Definition \ref{def:mirMF}, 
$$\delta^2 = W_{\mathbb{L}} \cdot \mathrm{Id}.$$
\end{prop}

\begin{proof}
Recall that $\cA = \Lambda_0 \ll x_e: e \in E \gg \big/ \langle P_f: f \in F \rangle$ where $P_f$ are the coefficients of $X_f$ in $m_0^b$.
Consider the $A_\infty$-equation
$$m_1^{b,0} \circ m_1^{b,0} (p) + m_2^{b,0}(m_{0,\bL}^b, p) + (-1)^{\deg p - 1} m_2^{b,0}(p,m_{0,U}) = 0.$$
Since $U$ is unobstructed, $m_{0,U} = 0$.  Moreover $m_{0,\bL}^b = W \cdot \one_\bL$ as an element in $\cA \otimes_{\Lambda_0} \CF^{\bullet}(\bL,\bL)$ by the definition of $\cA$.
Hence the second term equals to $W \cdot p$.  Hence $\delta^2 (p) = (-1)^{\deg p} (-1)^{\deg p - 1} m_1^{b,0} \circ m_1^{b,0} (p) = W \cdot \one_\bL$.

\end{proof}
The essence of the above proof is the following. In Floer theory, the differential  (in our case, decorated by $b$'s) may not satisfy $\delta^2=0$. The failure of $\delta^2=0$ comes from disc bubblings. The disc bubbling with boundary on $\bL$ (with $b$-decorations) provide the term $W_{\mathbb{L}} \cdot \mathrm{Id}$, whereas the disc bubbling with boundary on $U$ is assumed to vanish.
Hence we obtain the equation $\delta^2 = W_{\mathbb{L}} \cdot \mathrm{Id}.$ Therefore, the above proposition claims that appropriate decorated counting for Floer differential $\delta$ provides a matrix factorization of the superpotential $W_{\bL}$.



The correspondence between objects defined above extends to be an $A_\infty$-functor $\cF = (\cF_k)_{k \geq 1}^\infty$ as follows.   Given an intersection point $q$ between two Lagrangians $U_1$ and $U_2$ in $\Fuk(X)$, which can be regarded as a morphism from $U_1$ to $U_2$, we have the morphism 
$$(-1)^{(\deg q - 1)( \deg (\cdot) - 1)} m_2^{b,0,0}(\cdot, q): \cF(U_1) \to \cF(U_2).$$
Namely, given an intersection point $p \in \mathbb{L} \cap U_1$, the expression
$$(-1)^{(\deg q - 1)( \deg p - 1)} m_2^{b,0,0}(p,q) = (-1)^{(\deg q - 1)( \deg p - 1)} \sum_{k=0}^\infty m_{k+2}(\underbrace{b,\cdots,b}_k,p,q)$$
defines an element in the underlying module of $\cF(U_2)$.  This defines a map 
$$\cF^{\bL}_1: \Hom (U_1,U_2) \to \Hom(\cF(U_1),\cF(U_2)).$$

The higher parts of the functor is given by $m_{l+1}^{b,0,\cdots, 0}$ up to signs.  Namely, given $(l+1)$-tuple of Lagrangians $U_1, \cdots, U_{l+1}$ ($l \geq 1$) and a sequence of morphisms 
$$ (q_1, \cdots, q_l) \in \Hom (U_1, U_2) \otimes \cdots \otimes \Hom (U_l, U_{l+1}),$$
we define 
$$\cF_l (q_1, \cdots, q_l) \in \Hom (\cF(U_1) , \cF(U_{l+1}))$$
as a map which sends 
$$p \in \mathbb{L} \cap U_1 \mapsto (-1)^{\epsilon} m_{l+1}^{b,0,\cdots, 0} (p, q_1, \cdots, q_l)$$
 where $\epsilon= (\deg p - 1) \left( \sum_{i=1}^l (\deg q_i - 1) \right)$.

The proof of $\cF$ being an $\AI$-functor is essentially the same as Theorem 2.18 \cite{CHL13} and so the detail is omitted here.  The key feature is that the underlying ring is taken to be $\cA = \Lambda_0 \ll x_e: e \in E \gg \big/ \langle P_f: f \in F \rangle$ rather than $\Lambda_0$.  Over this ring we have $m_{0,\bL}^b = W \cdot \one_\bL$.  Moreover $\cF_l(q_1,\ldots,q_l)$ defined above are truly $\cA$-module homomorphisms, since from Equation \eqref{eqn:pulloutrule} we have
$$m_{k+l+1}(b,b,\cdots,b, g \cdot p,q_1,\ldots,q_l) = g \cdot m_{k+l+1}( b,b,\cdots, b, p, q_1,\ldots,q_l). $$
for any $g \in \cA$ (where recall that $b = \sum_{e} x_e X_e$).

\begin{theorem} \label{thm:F}
$\cF = \cF^{\bL} = (\cF_i)_{i=1}^\infty$ defines an $\AI$-functor.
\end{theorem} 

The following proposition transforms Hamiltonian isotopty into homotopy.
\begin{prop}\cite[Proposition 5.4]{CHL-toric}
If $L_1$ and $L_2$ are Hamiltonian isotopic, then the resulting matrix factorizations $\cF^\bL(L_1)$ 
and $\cF^\bL(L_2)$ are homotopic to each other.
\end{prop}
The proof is the same as in \cite{CHL-toric} and we omit the proof.

The following theorem shows that our functor is injective on a certain class of Hom spaces related to $\bL$.
\begin{theorem}[Injectivity]\label{thm:injthmsingle}
If the $\AI$-category is unital, then the  mirror $\AI$ functor $\mathcal{F}^\bL$ is injective
on $H^\bullet(\Hom(\bL, U))$ (and also on  $\Hom(\bL, U)$) for any Lagrangian $U$.
\end{theorem}
\begin{proof}
We construct a right inverse $\Psi$ to $\mathcal{F}^\bL_1$.
Consider the matrix factorizations $\mathcal{F}(\bL), \mathcal{F}(U)$ corresponding to $\bL, L$ respectively.
Recall that 
$$\mathcal{F}(\bL) = \big( \mathcal{A} \otimes_{\Lambda_0} \CF^{\bullet}((\bL,b), \bL), m_1^{b,0} \big).$$
Let $\phi$ be a morphism in $\Hom(\mathcal{F}(\bL), \mathcal{F}(U))$.
Since $\be_\bL$ is an element of  $\CF^{\bullet}((\bL,b), \bL) =\CF^{\bullet}(\bL, \bL) $, we may take $\phi(\be_\bL)$
which lies in $\mathcal{A} \otimes_{\Lambda_0} \CF^{\bullet}(\bL, U)$.
We define $\Psi(\phi)$ as
$$\Psi(\phi) = \big( \phi (\be_\bL) \big)|_{b=0} \in \CF^{\bullet}(\bL, U).$$
We first show that $\Psi$ is a chain map:
\begin{eqnarray*}
\Psi(m_1(\phi)) &=& \Psi(m_1^{b,0}\circ \phi)) - (-1)^{|\phi|} \Psi(\phi \circ m_1^{b,0})\\
&=& \big( m_1^{b,0}(\phi(\be_\bL)) \big)|_{b=0} -  (-1)^{|\phi|} \big(  \phi ( m_1^{b,0} (\be_\bL)) \big)|_{b=0} \\
&=& \big( m_1^{0,0}(\phi(\be_\bL)|_{b=0}) \big) - (-1)^{|\phi|}\big(  \phi (-b) \big)|_{b=0} \\
&=& m_1^{0,0}( \Psi (\phi)).
\end{eqnarray*}
Now, let us show that $\Psi$ is the right inverse of $\mathcal{F}^\bL_1$:
$$\big( \Psi \circ \mathcal{F}^\bL_1  \big) (p) =  \big( \mathcal{F}^\bL_1(p) (\be_\bL) \big)|_{b=0}
= \big( m_2^{b,0,0}(\be_\bL, p) \big)|_{b=0} = m_2^{0,0,0}(\be_\bL,p) =p.$$
\end{proof}

%
%
%


\chapter{Elliptic curves and deformation quantizations} \label{sec:333}
In this chapter, we apply the general construction in Chapter \ref{sec:nc-deform} and \ref{sec:singlelagfunctor} to the elliptic curve quotient $E/\Z_3 = \bP^1_{3,3,3}$ and obtain a family of noncommutative mirrors.  As explained in the Introduction, Zaslow \cite{Zaslow} and Aldi-Zaslow \cite{AZ} obtained noncommutative mirrors of the elliptic curve by computing the Floer operations for a Lagrangian section and its Dehn twists.  In this paper we provide a general systematic construction of noncommutative mirrors.  (In Section \ref{sec:4puncture2222} we work out a similar construction for $\bP^1_{2,2,2,2}$ which involves quiver algebras.)

We consider a family of reference Lagrangians $(\mathbb{L}_t,\lambda)$, which produces the mirror functors from $\Fuk (\bP^1_{3,3,3})$ to a family of noncommutative Landau-Ginzburg models. The resulting LG-models turn out to be (3-dimensional) Sklyanin algebras with central elements.  Surprisingly, the quotients of the Sklyanin algebras by the central elements gives rise to the deformation quantization of affine Del Pezzo surface $W_0(x,y,z) =0$ in $\C^3$ where $W_0$ \eqref{eqn:mir333t=0} is the  commutative potential obtained from $(\mathbb{L}_{t=0}, \lambda=-1)$.

The main result of the chapter can be summarized as follows.

\begin{theorem}\label{thm:mainthm333}
There is a $T^2$-family $(\mathbb{L}_t,\lambda)$ of Lagrangians decorated by flat $U(1)$ connections in $\bP^1_{3,3,3}$
for $(t-1) \in \R/2\Z$ and $\lambda \in U(1)$ whose corresponding generalized mirror $(\cA_{(\lambda,t)},W_{(\lambda,t)})$ satisfies the following.
\begin{enumerate}
\item The noncommutative algebras $\cA_{(\lambda,t)}$ are Sklyanin algebras, which are  of the form
\begin{equation}\label{eq:skly}
\cA_{(\lambda,t)} := \frac{ \Lambda <x,y,z> }{ \big( axy + byx +cz^2, ayz+bzy+cx^2, azx+bxz+cy^2 \big)}
\end{equation}
for $a=a(\lambda,t), b=b(\lambda,t), c=c(\lambda,t) \in \Lambda$. We have $\cA_0:= \cA_{(-1,0)} = \Lambda[x,y,z].$
\item $W_{(\lambda,t)}$ lies in the center of $\cA_{(\lambda,t)}$ for all $(\lambda,t)$. We denote $W_0 = W_{(-1,0)}$.
\item The coefficients $(a:b:c)$ are given by theta functions, which define an embedding $T^2 \to \mathbb{P}^2$ onto the mirror elliptic curve 
$$\check{E} = \{(a:b:c) \in \mathbb{P}^2 \mid W_0(x,y,z)=0\}.$$
\item For each $(\lambda,t)$, there exists a $\Z_2$-graded $\AI$-functor 
$$ \cF^{(\bL_t,\lambda)}: \Fuk(\bP^1_{3,3,3}) \to \MF(\cA_{(\lambda,t)}, W_{(\lambda,t)}).$$
Upstairs there is a $\Z$-graded $\AI$-functor
$$ \cF^{(\tilde{\bL}_t,\lambda)}: \Fuk^\Z(E) \to \MF^{\Z}(\hat{\cA}_{(\lambda,t)}, \hat{W}_{(\lambda,t)}).$$
When $t=0, \lambda=-1$, they give derived equivalences.
\item The graded noncommutative algebra $\cA_{(\lambda,t)} / \langle W_{(\lambda,t)} \rangle$ is a twisted homogeneous coordinate ring of $\check{E}$.
\item The family of noncommutative algebras  $\cA_{(\lambda,t)} / \langle W_{(\lambda,t)} \rangle$ near $t=0, \lambda=-1$
gives a quantization of the affine del Pezzo surface defined by $W_0(x,y,z)=0$ in $\C^3$.
\end{enumerate}
\end{theorem}
$\Fuk^\Z$ denotes the Fukaya category equipped with $\Z$-grading. The category $\MF^{\Z}(\hat{\cA}_{(\lambda,t)}, \hat{W}_{(\lambda,t)})$ and the functor upstairs are to be defined in \ref{sec:MFupstairs}.  The grading on $\cA_{(\lambda,t)}$ is simply given by $\deg x = \deg y = \deg z =1$.


\begin{remark}
When $t=0$ and $\lambda = -1$, we have $a = -b, c=0$ so that the relations in \eqref{eq:skly} become commutators among three variable. Thus the mirror constructed is commutative as in \cite{CHL13}. 
\end{remark}

\section{Deforming the reference Lagrangian} \label{sec:deform333}

Let $E$ be an elliptic curve $\mathbb{C}/\mathbb{Z} \oplus \mathbb{Z}\langle e^{2\pi i/3} \rangle$, and $\Z_3$ act on $\mathbb{C}$ (and $E$) by $2 \pi/3$ rotation around the origin.
In the quotient orbifold $E/\Z_3 = \bP^1_{3,3,3}$, we consider the family  of Lagrangians $\bL_t$ for $ -1<t<1$.
A reference Lagrangian $\bL_t$ (for a fixed $t$) will be the projection of vertical line $\widetilde{\bL}_{t,0}$ in $\mathbb{C}$ passing through the point $\left( \frac{-t+1}{4},0 \right)$. (See (a) of Figure \ref{fig:333intronew} for the original Lagrangian $\widetilde{\bL}_{0,0}$ (red dotted line)
and its translated copy $\widetilde{\bL}_{t,0}$ (blue solid line).)
Applying $\Z_3$-action to $\widetilde{\bL}_{t,0}$, we obtain $\widetilde{\bL}_{t,1}, \widetilde{\bL}_{t,2}$ which intersect each other at three points. (Assume $t \neq \pm 1/3$ to avoid triple intersections.) They project down to a single immersed circle $\bL$ in the quotient $E/\Z_3$ with three immersed points Let us denote the associated degree odd immersed Floer generators by $X, Y, Z$. (See (b) of Figure \ref{fig:333intronew}.)

\begin{figure}
\begin{center}
\includegraphics[height=2.2in]{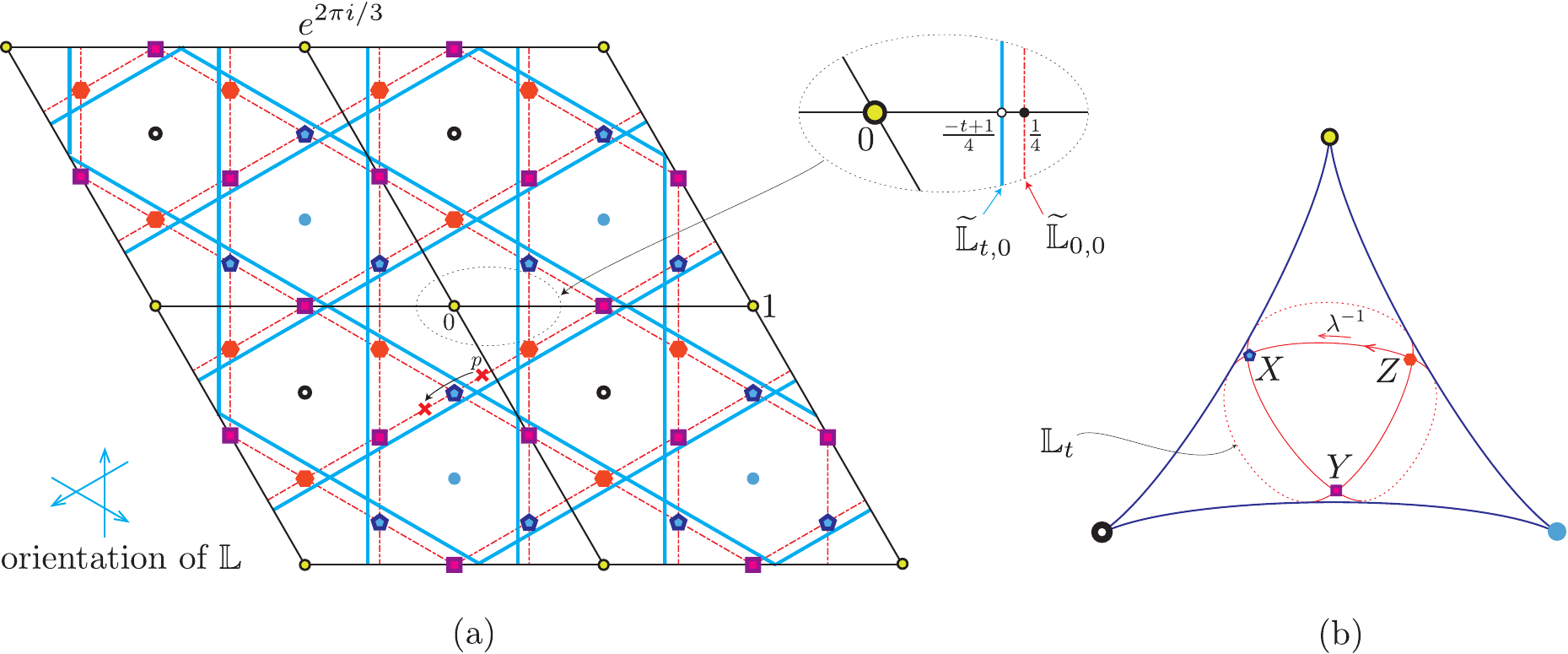}
\caption{(a) $(E, \widetilde{\bL})$ and  (b) $(\mathbb{P}^1_{3,3,3}, \bL)$}\label{fig:333intronew}
\end{center}
\end{figure}

An additional holonomy parameter $\lambda$ is introduced as follows.
We further equip the Lagrangian $\mathbb{L}_t$ with a flat line bundle whose holonomy is $\lambda=e^{2\pi i s} \in U(1)$. We represent this flat bundle by the trivial line bundle on $\mathbb{L}_t$ together with the flat connection whose parallel transport is proportional to the length of a curve. More precisely, if a curve in $\bL$ has reverse orientation to $\bL$ and length $l$, the parallel transport along this curve is given by the multiplication of $\lambda^{\frac{l}{l_0}}=e^{\frac{2 \pi i s l}{l_0}}$ where $l_0$ denotes the total length of $\mathbb{L}_t$.

We will apply our construction in Chapter \ref{sec:nc-deform} and \ref{sec:singlelagfunctor} using each $(\mathbb{L}_t, \lambda)$ as a reference Lagrangian. For this, we use the Morse-Bott type construction of $\AI$-category for surfaces following Seidel \cite{Se}, Sheridan \cite{Sheridan11} and we refer readers to these for more details (see also \cite{CHL13}).
In fact, we make the following changes (from that of \cite{CHL13}) for the computations.
First, we will choose a Morse function on $\mathbb{L}_t$ to have one maximum and one minimum on each segment of $\mathbb{L}_t$ (here, immersed points of $\mathbb{L}_t$ divides $\mathbb{L}_t$ into 6 segments). 
Second, instead of representing the non-standard spin structure of $\mathbb{L}_t$ by a point in  $\mathbb{L}_t$, we will consider 
a flat line bundle with holonomy $(-1)$, which is evenly distributed along $\mathbb{L}_t$. These changes will make the computations to be more symmetric.

Given the Lagrangian $\bL_t$, we study the Maurer-Cartan equation of $b$ where $b$ is given by the formal linear combination of immersed generators $b= xX + yY + zZ$.  Note that in the case of $t=0$, the Figure \ref{fig:333intronew}  becomes symmetric. Namely, the configuration of $\widetilde{\bL}_0$ is invariant under the reflection along the line which passes through three fixed points of $\Z_3$-action.

It was shown in  \cite{CHL13} that if we further set $\lambda =-1$ for $\bL_0$, any such $b$ satisfies the weak Maurer-Cartan
equation (where $x,y,z$ are regarded as commutative variables). 
More precisely, when $t=0$ and $\lambda=-1$ (which is  gauge equivalent to choosing a non-trivial spin structure on $\bL$), the following is proved in \cite{CHL13}.
\begin{theorem}[\cite{CHL13}]\label{thm33313}
When $\bL_0$ is equipped with a non-trivial spin structure, $b = xX +yY+zZ$ is a weak Maurer-Cartan solution for any $x,y,z \in \C$.  
The mirror LG superpotential $W_0$, after a rescaling on $x,y,z$, takes the form
\begin{equation}\label{eqn:mir333t=0}
W_0=x^3+y^3+z^3 - \sigma(q_\orb) xyz
\end{equation}
where $q_\orb=\bT^{\omega( \bP^1_{3,3,3}) }$ is the K\"ahler parameter of $\bP^1_{3,3,3}$ and $\sigma(q_\orb)$ is the inverse mirror map
\begin{equation} \label{eqn:sigma}
\sigma(q_\orb) = -3 - \left( \frac{\eta(q_\orb)}{\eta(q_\orb^9)} \right)^3.
\end{equation}
$\eta$ above denotes the Dedekind eta function.
\end{theorem}
Weakly unobstructedness of $( \bL_0, \lambda=-1)$ is mainly due to the symmetry of $\bL_0$ under the anti-symplectic involution together with certain sign computations.
The potential $W_0$ was computed by counting infinite series of triangles passing through a given point class.
(In fact, with our new formulation of non-standard spin structure,  the coordinate change $y \to \tilde{y}$ in \cite{CHL13} is not necessary. More details will be given below.) 

After deformation of $(\bL_0, \lambda=-1)$, the anti-symplectic involution no longer preserves $\bL_t$ and $b$ does {\em not} satisfy the weak Maurer-Cartan equation for $(\bL_t, \lambda)$ anymore in the sense of \cite{CHL13}.
In this article, we allow $x,y,z$ to be noncommutative formal variables (subject to weak Maurer-Cartan relations) so that $b$  still solves the noncommutative version of weak Maurer-Cartan equation for $(\bL_t, \lambda)$. This enables us to construct a noncommutative Landau-Ginzburg model for any $t, \lambda$.

\section{Sklyanin algebras as Weak Maurer-Cartan relations}\label{subseck:weakMC333hol}
In this section we show that the weak Maurer-Cartan relations of $(\bL_t,\lambda)$ produce Sklyanin algebras,
which are important examples of noncommutative algebraic geometry (see for instance \cite{ATV, AV, AS87}).  It was discussed in \cite{AKO08, AZ} that they are closely related with mirrors of elliptic curves and $\bP^2$.
\begin{definition}[Sklyanin algebra]
The Sklyanin algebra is defined as
$$ \Sky_3 (a,b,c) := \frac{ \Lambda <x,y,z> }{ \big( axy + byz +cz^2, ayz+bzy+cx^2, azx+bxz+cy^2 \big)} $$
for $a,b,c \in \Lambda$.
\end{definition}




Let us now compute weak Maurer-Cartan relations explicitly. Formally,  we need to compute $m_0^b = m_1 (b) + m_2 (b,b) + \cdots$ for $(\mathbb{L}_t,\lambda)$ where  $m_0^b$ takes the form $h_{\bar{X}} \bar{X} + h_{\bar{Y}} \bar{Y} + h_{\bar{Z}} \bar{Z} + W \one_\bL$. 

\begin{prop}
Weak Maurer-Cartan relations for general $\lambda, t$ are given by $h_{\bar{X}} =h_{\bar{Y}} = h_{\bar{Z}} = 0$, where
\begin{equation}\label{eqn:firstcompABC}
h_{\bar{X}} = A yz +B zy + C x^2, 
h_{\bar{Y}} = A zx + B xz +C y^2, 
h_{\bar{Z}} = A x y + B yx + C z^2,
\end{equation}
and
\begin{equation}\label{eqn:333ABCbef}
A = \lambda^{ \frac{1+3t}{2}}  \left(\sum_{k \in \Z} \lambda^{3k} q_0^{(6k+1+3t)^2} \right),
B =  \lambda^{ \frac{1+3t}{2}} \left(\sum_{k \in \Z} \lambda^{3k+2} q_0^{(6k+5+3t)^2} \right), 
C=  \lambda^{ \frac{1+3t}{2}} \left(\sum_{k \in \Z} \lambda^{3k+1} q_0^{(6k+3+3t)^2} \right).
\end{equation}
Here, $q_0 = \exp(-\textrm{Min})$ where $\textrm{Min}$ is the symplectic area of the minimal non-trivial triangle bounded by $\bL_0$. (Note that we use $e^{-1}$ instead of the Novikov parameter $\bT$.)
\end{prop}

The above proposition can be checked by direct but tedious calculations, and we  provide a brief sketch and omit the details.
For example,  the coefficient of $yz$ in $h_{\bar{X}}$ comes from the contribution of $m_2(zZ,yY)$ with an output $\bar{X}$, which means that we count triangles with corners $Z,Y,X$ in a counter-clockwise order, with weights given by its areas and flat connections.
The contribution comes in two (infinite) sums  of triangles (which is combined nicely)

\begin{equation}\label{eqn:Ztrick}
\sum_{k \geq 0} \left(\lambda^{-1}\right)^{3k+\frac{5}{2}-\frac{3t}{2}} q_0^{(6k+5-3t)^2}+ \sum_{k \geq 0} \lambda^{3k+\frac{1}{2} + \frac{3t}{2} } q_0^{(6k+1+3t)^2} = \sum_{k \in \Z} \lambda^{3k+\frac{1+3t}{2}} q_0^{(6k+1+3t)^2}.
\end{equation}
Observe that the (exponent of) holonomy part is proportional to the boundary length of a triangle. The reason for having $\lambda^{-1}$ in the first term but $\lambda^{+1}$ in the second is because the orientations of triangles in two sequences are opposite to each other. 

The triangle for the $k$-th term in the first summand of \eqref{eqn:Ztrick} is shown in Figure \ref{fig:333wmck}. As the boundary length of the triangle is $\left(3k + \frac{2}{5} - \frac{3t}{2} \right)l_0$ ($l_0=$ total length of $\bL_t$) and the holonomy effect is proportional to length, we obtain the first factor $\left(\lambda^{-1}\right)^{3k + \frac{2}{5}- \frac{3t}{2} }$. Also, it is easy to see from the ratio between edge lengths that the area of this triangle is $(6k+5-3t)^2$-times the area of the minimal $xyz$-triangle for $\bL_0$, and so comes the second factor $q_0^{(6k+5-3t)^2}$.

\begin{figure}[h]
\begin{center}
\includegraphics[height=1.5in]{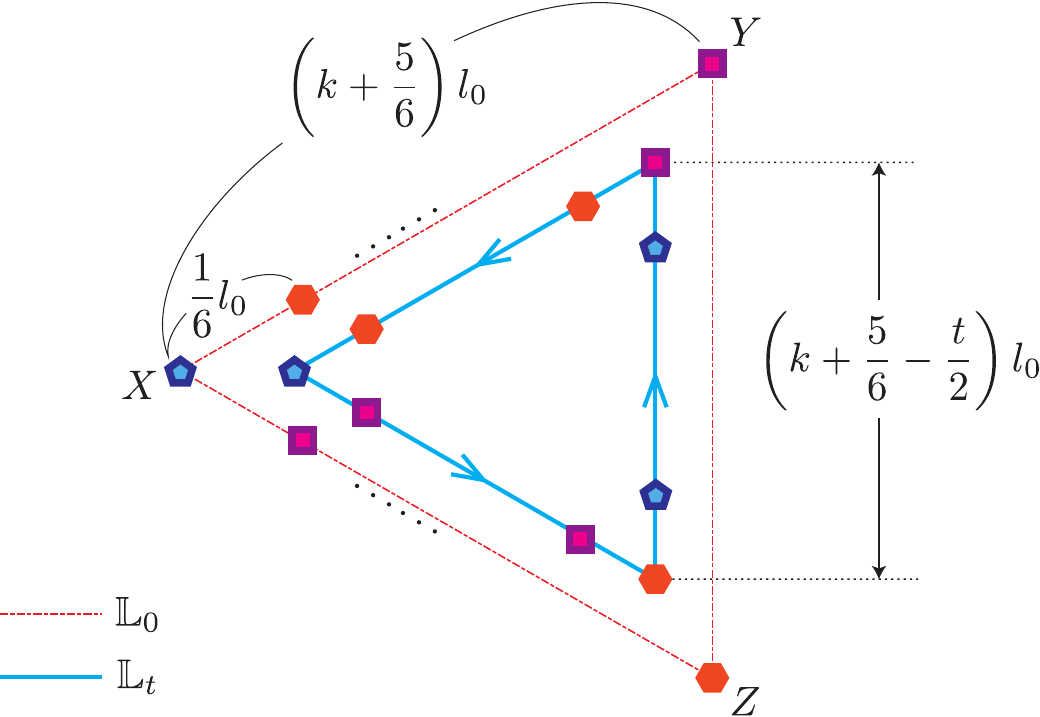}
\caption{A triangle contributing to $m_2 (zZ,yY)$}\label{fig:333wmck}
\end{center}
\end{figure}

As dividing $A,B$,$C$ by a common factor does not affect the ideal generated by weak Maurer-Cartan relations, we set 
\begin{equation}\label{eqn:333abcre}
a= q_0^{-(3t+1)^2} \lambda^{-\frac{1+3t}{2}} \cdot A, \quad b= q_0^{-(3t+1)^2} \lambda^{-\frac{1+3t}{2}} \cdot B,\quad c=q_0^{-(3t+1)^2} \lambda^{-\frac{1+3t}{2}} \cdot C
\end{equation}
which are indeed theta functions in a holomorphic variable $u$ defined as follows.

As in Theorem \ref{thm33313} we put $q_{\orb}=q_0^8$ which is the exponentiated symplectic area of $\mathbb{P}^1_{3,3,3}$ (here we put $\bT = e^{-1}$), and $e^{2 \pi i \tau}:=q_{\orb}^3$, where the right hand side is the exponentiated symplectic area of the elliptic curve $E$. (The dependence of $\tau$ on $q_{\orb}$ is the inverse mirror map.) We write $\lambda = e^{2 \pi i s}$ (recall $\lambda \in U(1)$), and define 
\begin{equation} \label{eq:w}
u:= s+ \tau\frac{t}{2} + \frac{\tau}{6}.
\end{equation}
Since $s$ and $t$ have periodicity $1$ and $2$ respectively, $u$ can be thought of as a (holomorphic) coordinate function on $E_\tau:= \mathbb{C} / \mathbb{Z} \oplus \mathbb{Z} \langle \tau \rangle$.

\begin{prop}
The function $a,b,c$ in Equation \eqref{eqn:333abcre} are holomorphic functions in the coordinate $u$ on $E_\tau$, and can be expressed as
\begin{equation}\label{eqn:thetaabc}
a(u) = \Theta_0 (u, \tau)_3, \quad b(u) = \Theta_2 (u, \tau)_3, \quad c(u) = \Theta_1 (u, \tau)_3.
\end{equation}
In particular $a,b,c$ are convergent over $\C$. Here, $\Theta_j (u, \tau)_s := \theta \left[ \frac{j}{s}, 0 \right] ( s u, s \tau)$ (see appendix for more details).
\end{prop}
\begin{proof}
See Appendix \ref{appsubsec:333theta}.
\end{proof}
Such theta functions were also obtained in \cite{AZ} in a different setting.

Note that a priori the weakly unobstructed relations from Lagrangian Floer theory are only defined around the large volume limit $\tau \to i \infty$.  It follows from the above proposition that they are well-defined over the global moduli.
We will prove that $(a:b:c)$ satisfies the mirror elliptic curve equation $W_0=0$ later.


Consequently, the mirror noncommutative algebra is the Sklyanin algebra with coefficients $a,b,c$ and, 
at $u_0:=\frac{1}{2} + \frac{\tau}{6}$ (i.e. $t=0$ and $\lambda=-1$) 
the algebra recovers the usual polynomial algebra, since $a(u_0) = - b(u_0)$ and $c (u_0)=0$. This finishes the proof of (1) of Theorem \ref{thm:mainthm333}.

We finish the section by explaining briefly a geometric meaning of the weal Maurer-Cartan relations.
Intuitively speaking, weak Maurer-Cartan relations are needed to ensure the well-definedness of the potential $W$.
Namely, if $W$ is given by the counting of polygons passing through a generic point $p$, then we want the counting
to be independent of the position of $p$. In actual computation of $W$ (Section \ref{subsec:wspot333}), we will use a different type of representatives of a generic point, but here we temporarily choose a point $p$ as in Figure \ref{fig:333intronew} to illustrate the geometric meaning of weak Maurer-Cartan relation.

When $p$ moves from one segment to the other along a straight line, passing through a corner labelled $X$, the set of polygons
that intersect $p$ before, get changed into another family of polygons afterward and their difference gives the coefficient $h_{\bar{X}}$ below. For $\bL_0$, these two sets are related by anti-symplectic involution at $X$, and actually the signed counting remains invariant. But  if we consider $\bL_t$ (blue lines), and move $p$ as before, it is easy to see that the corresponding two set of polygons cannot be identified. For example, two small blue triangles sharing a vertex $X$, have different areas, and hence cannot be identified.
But in a noncommutative setting,  the total difference across $X$, given by  $h_{\bar{X}}$, will define a weak Maurer-Cartan relation which is set to be zero after taking quotient.

\section{Noncommutative potential}\label{subsec:wspot333}

We next compute the (noncommutative) potential function $W_{(\lambda,t)}$ for $(\mathbb{L}_t,\lambda)$,
by suitable counts of (holomorphic) triangles. From the weak Maurer-Cartan equation, the potential function $W_{(\lambda,t)}$
can be read off by intersecting the contributing holomorphic polygons (in this case triangles) with the point class.

 We will pick a representative of the point class of $\mathbb{L}_t$ and we count the number of triangles passing through this representative (hence one triangle can contribute several times depending on the position of intersection).
Note that the configuration of Lagrangians in the universal cover changes for various choice of $t$'s (even topologically).
Let us first assume $-\frac{1}{3}< t <\frac{1}{3}$ so that the topological type of $\bL_t$ remains the same as in Figure \ref{fig:333intronew}.

To represent the point class, we put a marked point on each minimal segment of $\mathbb{L}_t$ bounded by two self intersection points. Since we have six such segments, each marked point represents $1/6$ of the point class, i.e. $\frac{1}{6} PD(\be_\bL)$. Thus total contribution of a triangle is $\frac{1}{6}$ of the number of marked points, or equivalently, $\frac{1}{6}$ of the number of segments in the boundary of the triangle. See Figure \ref{fig:onesixth}.

\begin{figure}[h]
\begin{center}
\includegraphics[height=1.5in]{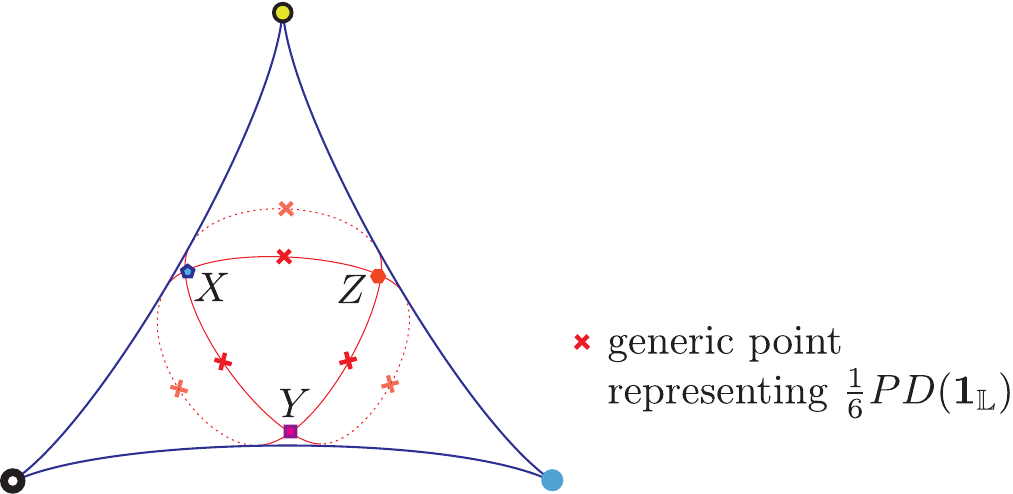}
\caption{Each time a polygon passes through $\color{red} \times$, it constributes by $\frac{1}{6}$.}\label{fig:onesixth}
\end{center}
\end{figure}

\begin{remark}
The choice of the Morse function (hence the marked points for the point class) here is different from that of \cite{CHL13}, but the resulting potential for the case of $t=0$ and $\lambda=-1$  is the same as the one in \cite{CHL13}.
\end{remark}

Direct counting gives the following $x^3,y^3,z^3$-terms in the potential. 
\begin{equation}\label{eqn:potxcubeterm}
- \lambda^{\frac{1+3t}{2}}  \left( \sum_{k \in \Z} \left(k + \frac{1}{2}\right) \lambda^{3k+1} q_0^{(6k+3+3t)^2} \right) (x^3 +y^3+z^3) 
=  -  \dfrac{q_0^{(3t+1)^2} \lambda^{\frac{1+3t}{2}}}{6\pi i}  \left(  c'(u)  +  \pi i c(u) \right) (x^3 +y^3+z^3).
\end{equation}
In actual counting, each coefficient is given as two sums, both of which are taken over $\{ k \geq 0 \}$ contributed by two sequences of triangles. We apply the same trick as in \eqref{eqn:Ztrick} to reduce to the above expression. 

For example, $x^3$-term comes from the coefficients of $\be_\bL$ in $m_3 (xX,xX,xX)$, and the triangles with all corners labelled by $X$ contribute to it. Therefore, the counting is similar to that for $C$ \eqref{eqn:333ABCbef}, but the only difference is that we count the number of marked points (representing $1/6 PD(\be_\bL)$) this time. 
The $k$-th $x^3$-triangle has $6k+3$ minimal segments (with end points at self intersections) in one edge whose counting gives rise to $(6k+3) (\frac{1}{6} \be_\bL) = (k+ \frac{1}{2}) \be_\bL$. There are three such edges, and hence, a priori, we get $(3k+ \frac{3}{2}) \be_\bL$. However, each $x^3$-triangle has $\Z_3$-rotation symmetry, so in view of the orbifold $\mathbb{P}^1_{3,3,3}$, the total contribution should be $\frac{1}{3} \times (3k+ \frac{3}{2}) \be_\bL = (k+ \frac{1}{2}) \be_\bL$.

Observe that each summand in \eqref{eqn:potxcubeterm} is the same as the summand in $C$ \eqref{eqn:333ABCbef} except $\left( k +\frac{1}{2} \right)$ factor in the coefficient. This is why the derivative $c'(u)$ appears on the right hand side of \eqref{eqn:potxcubeterm}.



Likewise, the triangles with $X$, $Y$ and $Z$ as corners in counterclockwise order contribute
\begin{equation*}
-  \lambda^{\frac{1+3t}{2}} \sum_{k \in \Z}  \left( k+\frac{5}{6} \right) \lambda^{3k+2} q_0^{(6k+5+3t)^2} =- \dfrac{q_0^{(3t+1)^2} \lambda^{\frac{1+3t}{2}}}{6\pi i}  \left(  b'(u)  +  \pi i b(u) \right)
\end{equation*}
to the coefficients of $zyx, xzy, yxz$ respectively.
(Recall that $m_3 (xX, yY, zZ) = zyx\cdot m_3(X,Y,Z)$ according to our convention. See \eqref{eqn:pulloutrule}.) 

On the other hand, triangles with $X$, $Y$ and $Z$ as corners in clockwise order contribute
\begin{equation*}
-\lambda^{\frac{1+3t}{2}}  \sum_{k \in \Z} \left( k  +\frac{1}{6} \right) \lambda^{3k} q_0^{(6k+1+3t)^2}= -  \dfrac{q_0^{(3t+1)^2} \lambda^{\frac{1+3t}{2}}}{6\pi i}  \left(  a'(u)  +  \pi i a(u) \right)
\end{equation*}
to the coefficients of $xyz, zxy, yzx$ respectively.

If $t \leq \frac{-1}{3}$ or $t \geq \frac{1}{3}$, topological type of $\bL$ changes, but we still put one marked point at each minimal segment bounded by self intersection points, and each marked point represents $\frac{1}{N} \be_\bL$ where $N$ is the total number of minimal segments that form $\bL_{t}$. It is easy to check that $N=3$ for $t=\pm \frac{1}{3}$ and $N=6$ otherwise.

In each case, one can obtain the the following three coefficients for $W_{\lambda,t}$ by counting triangles in a similar fashion:
$$-  \dfrac{q_0^{(3t+1)^2} \lambda^{\frac{1+3t}{2}}}{6\pi i}  \left(  a'(u)  +  r \pi i a(u) \right) ,\quad  - \dfrac{q_0^{(3t+1)^2} \lambda^{\frac{1+3t}{2}}}{6\pi i}  \left(  b'(u)  +  r \pi i b(u) \right),\quad- \dfrac{q_0^{(3t+1)^2} \lambda^{\frac{1+3t}{2}}}{6\pi i}  \left(  c'(u)  +  r \pi i c(u) \right)$$
where
\begin{equation*}
r= \left\{
\begin{array}{ccc}
-1 &\mbox{for} &-1<t< -\frac{1}{3}\\
0 &\mbox{for} &  t=-\frac{1}{3} \\
1 &\mbox{for} & -\frac{1}{3} < t < \frac{1}{3} \\
2 &\mbox{for} & t= \frac{1}{3} \\
3 &\mbox{for} & \frac{1}{3} < t < 1
\end{array}
\right. .
\end{equation*}

By using weak Maurer-Cartan relations,
we obtain the following formula which is valid for all $t$ (i.e. the formula is valid for any  topological type of $\bL_t$).

\begin{prop}\label{eq:nclgw333}
The noncommutative LG superpotential $W_{(\lambda,t)} \in \Sky(a(\lambda,t),b(\lambda,t),c(\lambda,t))$ for $(\mathbb{L}_t,\lambda)$ is given as 
\begin{equation}\label{eq:noncomm333pot}
\begin{array}{lcl}
W_{(\lambda,t)} &=& - \frac{ q_0^{(3t+1)^2} \lambda^{\frac{1+3t}{2}}}{6 \pi i} \left( a'(u) (xyz + zxy + yzx) +   b'(u)  (zyx + xzy + yxz) + c'(u) (x^3 + y^3 + z^3) \right) .
\end{array}
\end{equation}
In particular $W_{(\lambda,t)}$ is convergent over $\C$.
\end{prop}

As a direct consequence from Theorem \ref{thm:central1}, 
\begin{cor}[(2) of Theorem \ref{thm:mainthm333}] \label{cor:central333}
$W_{(\lambda,t)}$ given above is a central element of $\Sky(a(\lambda,t),b(\lambda,t),c(\lambda,t))$.
\end{cor}
Historically computer program has been used to
find  centers in these  noncommutative algebras, but our method provides a more conceptual approach.

\section{$(a,b,c)$ and the mirror elliptic curve}
We use a geometric idea to prove that 
\begin{theorem}[(3) of Theorem \ref{thm:mainthm333}]\label{abceq}
The coefficients $a(\lambda,t),b(\lambda,t),c(\lambda,t)$ of Maurer-Cartan relations (using the coordinate \eqref{eq:w})
give an embedding  $E_\tau \to \bP^2$ whose image is isomorphic to the mirror elliptic curve  $\check{E} = \{W_0 = 0\}$. In particular, we have
$$ a(\lambda,t)^3 + b(\lambda,t)^3+ c(\lambda,t)^3 - \sigma(q_\orb) a(\lambda,t)b(\lambda,t)c(\lambda,t)=0,$$
where $\sigma(q_\orb)$ is given in Theorem \ref{thm33313}.
\end{theorem}
\begin{remark}
From the theta function expression of $(a(\lambda,t),b(\lambda,t),c(\lambda,t))$, we know that it defines a projective embedding onto the Hesse elliptic curve $\{a^3 + b^3 + c^3 - \gamma abc = 0\} \subset \mathbb{P}^2$ for some $\gamma \in \C$ (See for e.g., \cite[Theorem 3.2]{Dol}).
The above theorem identifies this with the mirror elliptic curve, namely we have $\gamma = \sigma$.
\end{remark}

\begin{remark}
The above also works for $\Lambda$-valued noncommutative deformations.  Namely, we can take flat $\Lambda^*$-connections parametrized by $\lambda \in \Lambda^*$ (in place of flat $U(1)$-connections) on the family of Lagrangians $\bL_t$ (for $t \in \R/2\Z$).  Flat $\Lambda^*$-connections were used in the work of Abouzaid-Smith \cite{AS-torus}.

We have the global parameter $U = T^{A_E t/2} \cdot \lambda \in \Lambda^* / \langle T^{A_E} \rangle$ which plays the role of $e^{2\pi i u}$ (see Equation \eqref{eq:w}), where $A_E$ is the area of the elliptic curve $E$.  Expressions in \eqref{eqn:333ABCbef} for $A,B,C$ with $q_0 = T^{A_E/24}$ (and $\lambda \in \Lambda^*$) are still valid for the weakly unobstructed relations, and Expression \eqref{eqn:333abcre} for $a,b,c$ still makes sense and can be written in terms of the global parameter $U$.  (Novikov convergence for $A,B,C$ is automatic since for each $t$ there are only finitely many polygons bounded by $\bL_t$ with areas smaller than a fixed number.)  $(a,b,c)$ gives a map $\Lambda^* / \langle T^{A_E} \rangle \to \Lambda\bP^2=(\Lambda^3-\{0\})/\Lambda^\times$.  Since the proof to the above Theorem \ref{thm:mainthm333} is purely in Lagrangian Floer theory, it still works to prove that $(a,b,c)$ satisfies the same equation (with $q_{\mathrm{orb}} = T^{A_E/3}$).  As in \cite[Definition 6.3]{AS-torus}, under this map $\Lambda^* / \langle T^{A_E} \rangle$ should be identified with the variety defined by $W_0=0$ in $\Lambda\bP^2$.
\end{remark}

\begin{proof}
The main idea of the proof is to relate the Maurer-Cartan coefficients  $a(\lambda,t),b(\lambda,t),c(\lambda,t)$ with
the coefficients of a particular matrix factorization of the mirror potential $W_0$.

This is based on the following geometric observation.
\begin{lemma}
Consider a Lagrangian $\mathbb{L}_{3t}$, equipped with a flat $U(1)$ connection with holonomy $\lambda_0 = -1$ with a small $t \neq 0$.
Mirror functor $\mathcal{F}^{(\bL_0, -1)}$ (as in \cite[Section 6.3]{CHL13}) maps $\bL_{3t}$ with holonomy $\lambda_0$ to a $(3\times 3)$ matrix factorization
$$M_0 \cdot M_1 = M_1 \cdot M_0 = W_0 \cdot I_{3\times3},$$
where $M_0$ (the matrix with linear terms) is of the form
\begin{equation}\label{abcm0}
-M_0 = \left(\begin{array}{ccc} Ax & Bz & Cy \\ Cz & Ay & Bx \\ By & C x & A z\end{array}\right).
\end{equation}
(``$-$" in \eqref{abcm0} is due to our sign convention \eqref{eqn:objcorr1}.)
Then, we have 
\begin{equation}\label{abcabc}
(A,B,C) = \big( A(\lambda_0,t),B(\lambda_0,t),C(\lambda_0,t)\big).
\end{equation}
\end{lemma}
\begin{proof}[proof of Lemma]
We first review the computation of the matrix factorization which is essentially the same as the one for the long diagonal  in \cite{CHL13}, which is the Floer differential on $\CF ((\mathbb{L}_0,b), (\bL_{3t},\lambda))$
Note that  there are six intersection points between $\mathbb{L}_0$ and $\mathbb{L}_{3t}$.  The ($b$-deformed) holomorphic strips from odd generators (between $\mathbb{L}_0$ and $\mathbb{L}_{3t}$) to even generators are all triangles, and that from even to odd generators are all trapezoids.  Since two vertices of these strips are inputs and outputs, entries of $M_0$ consist of linear terms in $x,y$ and $z$, and those of $M_1$ are all quadratic,
and one can check that $M_0$ is of the form given in \eqref{abcm0}. 

Then, the relation \eqref{abcabc} can be checked by hand. We only give a few helpful comments, and leave it as an exercise.
Consider the  vertical Lagrangian $\bL_{t,0}$  and its $\Z_3$-rotation images $\bL_{t,1}, \bL_{t,2}$. If one takes two
of these Lagrangians, say $\bL_{t,0}$ and $\bL_{t,1}$, then it is easy to observe that they are just simple translates of 
$\bL_{0,0}, \bL_{0,1}$. Then one can check that the decomposition of the elliptic curve by  
three lines $(\bL_{t,0}, \bL_{t,1}, \bL_{t,2})$  and that by the three lines $(\bL_{0,0}, \bL_{0,1}, \bL_{3t,2})$  
are the same (related by a simple translation). 
The former  is used for Maurer-Cartan relations for $\bL_{t}$ and the latter  is used for Matrix Factorization of $\bL_{3t}$ with reference $\bL_0$. This is the reason why we take $\bL_{3t}$ in this lemma. 
One can check the labeling of corners to show the rest of the assertion.
\end{proof}

Now to finish the proof of the theorem, we take the determinant of the equation $M_0 \cdot M_1 = W_0 \cdot \mathrm{Id}$.  Thus $\det (M_0)$ is a cubic polynomial that divides $W_0^3$. Since $W_0$ is irreducible, we should have $\det (M_0) = k \cdot W_0$ for some constant $k \not= 0$.  Then 
\begin{equation} \label{eq:det(P)}
\det (M_0) = -ABC ( x^3 + y^3 + z^3) + (A^3 + B^3 + C^3) xyz = k \cdot W_{0},
\end{equation}
or equivalently, $\frac{A^3 + B^3 + C^3}{ABC} =\sigma$, which implies that $W_0(A,B,C)=0$ for $(\lambda_0,t)$ with $t$ small.
As $A,B,C$ are holomorphic functions (after overall scaling), this shows that in fact $A,B,C$ satisfy the above equation even for large $t$ and all $\lambda$.
%
\end{proof}
\section{Relation to the quantization of an affine del Pezzo surface}\label{subsec:defquant333}
Recall that deformation quantization of a (commutative) Poisson algebra is a formal deformation into a noncommutative associative unital algebra whose commutator in the first order is  the Poisson bracket. 

For any $\phi \in \C[x,y,z]$, the following brackets on coordinated functions extends to a Poisson struncture on $\C[x,y,z]$:
\begin{equation}\label{eqn:poipotphi}
\{x,y\} = \frac{\partial \phi}{\partial z}, \quad\{y,z\} = \frac{\partial \phi}{\partial x}, \quad  \{z,x\} = \frac{\partial \phi}{\partial y}.
\end{equation}
One can check that $\phi$ itself Poisson commute with any other element, and hence the above Poisson structure
descend to the quotient $\C[x,y,z]/(\phi)$ by the principal ideal generated by $\phi$, which is denoted as $\mathcal{B}_\phi$.

Our theory provides the quantization of the affine del Pezzo surface in the sense of \cite{EG}.
We write $v=u -u_0$ so that now $v=0$ corresponds to the commutative point  (see the last paragraph of \ref{subseck:weakMC333hol}).
\begin{theorem} \label{dq333}
The family of noncommutative algebra $\mathcal{A}_{v}/(W_{v})$ near $v=0$ gives a quantization of the affine del Pezzo surface, given by the mirror elliptic curve equation $W_0(x,y,z)=0$ in $\C^3$ in place of $\phi$ in \eqref{eqn:poipotphi}.
\end{theorem}
\begin{proof}
As $v \to 0$, we have $$\mathcal{A}_{v}/(W_{v}) \to \C[x,y,z]/(W_0).$$
So it remains to check that the first order term of the commutator of $A_{v}$ is equivalent to the Poisson bracket induced by $W_0$.
We may write $a(v)$ as ( similarly for $b(v),c(v)$)
$$a(v) = a(0) + a'(0)v + o(v).$$
(Here, $a(0)$ and $a'(o)$ means $a(v)$ and $a'(v)$ evaluated at $v=0$.)
Since $a(0) = -b(0)$ and $c(0) =0$, the Sklyanin algebra relation becomes
$$ -r (xy-yx) + v ( a'(0)xy+b'(0)yx + c'(0) z^2) + o(v) =0.$$
where we set $r:=b(0) =-a(0) \in \C$. Therefore,
$$ \lim_{v \to 0} \frac{xy - yx}{v} = \frac{1}{r} ( a'(0)xy+b'(0)yx + c'(0) z^2).$$
(On the right hand side, $x,y,z$ are commutative variables since the algebra itself is commutative at $v=0$.)
By Proposition \ref{eq:nclgw333}, $a'(0)xy+b'(0)yx + c'(0) z^2$ is a multiple of $\frac{\partial W_0}{\partial z}$, so the result follows.
\end{proof}

Etingof-Ginzburg \cite{EG} constructed a deformation quantization of such  affine del Pezzo surfaces via different method,
but  the resulting noncommutative algebra is of the same form.
Namely,  in \cite{EG}, the deformation quantization of $\C[x,y,z]/(\phi)$ is obtained in a two step process, where
first, $\C^3$ with the Poisson structure given by $\phi$ is deformed into a Calabi-Yau algebra $A$, and the Poisson center
$\phi$ then gets deformed into a central element $\Psi$ of this algebra $A$.
The Calabi-Yau algebra $A$ is defined as a Jacobi algebra of some superpotential $\Phi$, and the existence
of the center $\Psi$ uses a deep theorem of Kontsevich on deformation quantization \cite{Kon1}.

In our approach, $\Phi$ is given by the spacetime superpotential which is related to the Maurer-Cartan equation,
and the central element $W$ is obtained as a worldsheet superpotential, where both $\Phi$ and $W$ are
obtained by the elementary counting of triangles in elliptic curves.


It is known by \cite{ATV, St1} that the quotient algebra of Sklyanin algebra by its central element can be written as a homogeneous
coordinate ring. In our setting, it can be written as follows.
\begin{theorem}[\cite{ATV, St1}] \label{thm:tw-ring333}
Let $W$ be a non-zero central element of the Sklyanin algebra $\Sky_3(a,b,c)$.  Then $\cA_{(\lambda,t)}/ (W_{(\lambda,t)}) =\Sky_3 (a,b,c) / ( W )$ is isomorphic to the twisted homogeneous coordinate ring 
$$\bigoplus_{m\geq 1} \Gamma(\check{E},\mathcal{L} \otimes \sigma^* \mathcal{L} \otimes \cdots \otimes (\sigma^{m-1})^* \mathcal{L}),$$
where $\check{E}$ is the mirror elliptic curve defined by $\{W_0 = 0\} \subset \bP^2$, $\mathcal{L}$ is $\mathcal{O}(1)$ and $\sigma$ is an automorphism of
$\check{E}$ given by the group structure of $\check{E}$.
\end{theorem}

\begin{theorem}[Theorem 3.6.2 of \cite{EG}]
We have a derived equivalence for each $(\lambda,t)$
\begin{equation}\label{eqn:qequiEG}
D^b \mathrm{Coh} ( \check{E} ) \simeq D^b \MF^\Z (\cA_{\bL_{(\lambda,t)}}, W_{(\lambda,t)}).
\end{equation}
\end{theorem}

It is show in \cite{EG} that both categories in \eqref{eqn:qequiEG} are equivalent to a certain triangulated category canonically constructed from the twisted homogenous coordinate ring of $\check{E}$.
It is important in the above two theorems that the coefficients $(a,b,c)$ of the Sklyanin algebra satisfies 
the mirror elliptic curve equation (Theorem \ref{abceq}).


Recall that $\mathbb{P}^1_{3,3,3} = [E/ \Z_3]$.  
From the construction to be given in Chapter \ref{sec:fingpsymm}, we will have an upstair functor from the Fukaya category of $E$ to the {\em graded} matrix factorization category of $W_{(\lambda,t)}$. (See Proposition \ref{prop:F-graded}.)  In the commutative case (where $(\lambda,t)=(-1,0)$), upstair functor gives an equivalence.

\begin{prop} The $\Z$-graded functor
\begin{equation}\label{eqn;upsfunt333}
\mathcal{F}^{\tilde{\bL}_0} : \Fuk^\Z(E) \to MF^\Z (\cA_0,W_0)
\end{equation}
induces a derived equivalence.  ($\cA_0=\Lambda[x,y,z]$.)
\end{prop}
\begin{proof}
Recall that in \cite[Theorem 7.24]{CHL13}  we have proved that the $\Z_2$-graded functor gives a derived equivalence
from  $D^b \Fuk (\bP^1_{3,3,3})$ to $D^b \MF (\cA_0, W_0)$, by showing that the functor takes the split generator $\bL_0$ to a
split generator which is a wedge-contraction type matrix factorization $\xi$ \cite{Dy}.
We will show in Proposition \ref{prop:F-graded} that the corresponding functor upstairs preserves the $\Z$-gradings.  The derived Fukaya category of the torus $E$
is split generated by the three Lagrangian lifts $\bL^{1}, \bL^{g}, \bL^{g^2}$ for $\Z_3=\{1,g,g^2\}$ with gradings given by
$$\theta_0=\frac{1}{2},\quad \theta_1=\frac{1}{2} + \frac{2}{3},\quad \theta_2=\frac{1}{2} + \frac{4}{3}.$$
Under the mirror functor, they are sent to $\xi, \xi[-\frac{2}{3}], \xi[-\frac{4}{3}]$ 
respectively, which split-generate the category of graded matrix factorizations $MF^\Z (\cA_0,W_0)$ by Tu \cite[Theorem 6.2]{Tu2}.  Hence the functor derives an equivalence.
\end{proof}

By composing the upstair functor with \eqref{eqn:qequiEG}, we obtain
\begin{cor}\label{cor:etatwt}
There exists a family of functors parametrized by $(\lambda,t)$
\begin{equation}\label{eqn:upstairLGCY}
D^b \Fuk^\Z (E) \longrightarrow D^b \MF^\Z (\cA_{(\lambda,t)}, W_{(\lambda,t)}) \stackrel{\simeq}{\longrightarrow} D^b \mathrm{Coh} (\check{E}).
\end{equation}
\end{cor}
We expect that these functors are all equivalent to each other, and in particular equivalent to the commutative one when $(\lambda,t)=(-1,0)$.  This is still unknown to us for general $(\lambda,t)$ since noncommutative matrix factorizations are not well understood yet. In the commutative case of $\bL_0$, Lee \cite{swlee} showed that \eqref{eqn:upstairLGCY} agrees with the equivalence constructed by Polishchuk-Zaslow \cite{PZ-E} by further analyzing our construction in \cite{CHL13}.

%
%
%


\chapter{Mirror construction using several Lagrangians and quiver algebras} \label{sec:MFseverallag}

In this chapter, we take the reference $\bL$ to be a set of Lagrangians, instead of a single Lagrangian immersion.  Allowing $\bL$ to be a set of Lagrangians provides a way to construct global mirrors. For instance, one may take $\bL$ to be a set of generators of the Fukaya category.  If we take their union as a single Lagrangian and carry out the construction in Chapter \ref{sec:nc-deform}, the noncommutative algebra obtained does not know about the sources and targets of the morphisms since it does not record the individual branches.  In this section we use quiver algebra to take care of this point.

The result of the construction is a quiver $Q$ with relations $R$, together with a central element $W$ of the quiver algebra with relations $\Lambda Q /R$. Under good conditions, the relations can be obtained as partial derivatives of a cyclic  element $\Phi$ in the path algebra $\Lambda Q$.  We will regard $(\Lambda Q / R, W)$ as a generalized mirror of $X$, in the sense that there exists a natural functor $\Fuk (X) \to \MF(\Lambda Q / R, W)$.



The quiver $Q = Q^{\mathbb{L}}$ is a finite oriented graph, whose vertices correspond to elements in $\mathbb{L}$, and whose arrows  correspond to odd-degree immersed generators of $\bL$.
The path algebra of $Q$ is taken to be the base of the Fukaya subcategory $\mathbb{L}$.  Alternatively $\mathbb{L}$ can be associated with an $\AI$-{\em algebra} over the semi-simple ring $\Lambda^\oplus$ \cite{Se2}, which will be explained shortly. The occurrence of idempotents $\pi_i's$ is a key feature when we use a family of Lagrangians as a reference, and the subalgebra generated by $\pi_i's$ is isomorphic to $\Lambda^\oplus$. The path algebra of $Q^\mathbb{L}$ can be regarded as an algebra over this semi-simple ring. 

%


The relations $R$ of the quiver algebra come naturally from the weakly unobstructed condition on the formal deformations of $\mathbb{L}$.  Thus $\Lambda Q / R$ is essentially the space of weakly unobstructed noncommutative deformations of $\bL$.  As before, the counting of holomorphic discs passing through a generic marked point on each member of $\bL$ defines several superpotentials $W_i$.
Then $$W_{\mathbb{L}} := \sum_i W_i \in \Lambda Q / R.$$
We will show that the centrality theorem continues to hold in this setting, namely the potential $W_{\mathbb{L}}$ defined above lies in the center of $\Lambda Q / R$.

\section{Path algebra and the semi-simple ring}
We first construct the quiver $Q^\mathbb{L}$ in detail.
Let $X$ be a K\"ahler manifold, and $\bL = \{L_1, \dots, L_k\}$ be a set of spin Lagrangian immersions of $X$ with transverse self-intersections.  (In case there is no confusion, $L_i$ may refer to the immersion map, the domain or the image of the immersion, depending on situations.) Suppose that any two Lagrangians in this family intersect transversely with each other. We regard $\bigcup_{1 \leq i \leq k} L_i$ as an immersed Lagrangian. 

\begin{definition}\label{def:mirrorquiverseveral}
$Q = Q^\mathbb{L}$ is defined to be the following graph.  
Each vertex $v_i$ of $Q$ corresponds to a Lagrangian $L_i \in \bL$.  Thus the vertex set is
$$Q^\mathbb{L}_0 = \{v_1,\cdots, v_k\}.$$
Each arrow from $v_i$ to $v_j$ corresponds to odd-degree Floer generator in $\CF^\bullet (L_i, L_j)$ (which is an intersection point between $L_i$ and $L_j$).
In particular for $i=j$, we have loops at $v_i$ corresponding to
the odd-degree immersed generators of $L_i$.

$Q$ will be sometimes called the \emph{endomorphism quiver} of $\bL$.
\end{definition}
\begin{figure}[htp]
\begin{center}
\includegraphics[scale=0.6]{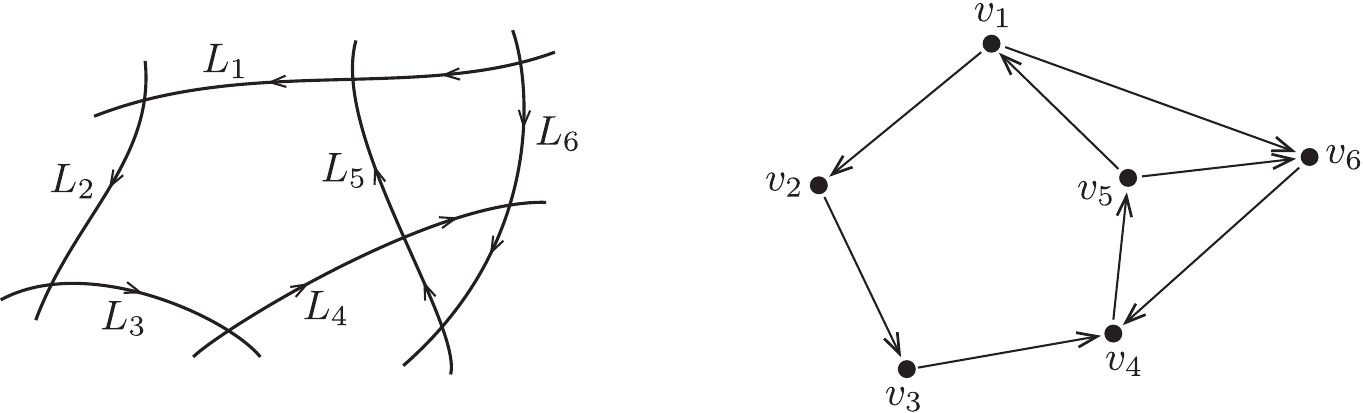}
\caption{Endomorphism quiver $Q^\mathbb{L}$}\label{fig:lagrquiv}
\end{center}
\end{figure}

Given a quiver $Q$, we can define a path algebra $\Lambda Q$ spanned by paths in $Q$.  We will use the following convention for the path algebra.

\begin{definition}
For a quiver $Q$, with the set of vertices $Q_0$ and the set of arrows $Q_1$,  we have maps
$h, t : Q_1 \to Q_0$ which assign head and tail vertices to each arrow respectively. 
Each vertex $v$ corresponds to a trivial path $\pi_v$ (of length zero).  
A sequence of arrows $p = a_{k-1}a_{k-2}\cdots a_1 a_0$ is called to be a path of length $k$ if 
$$ h (a_{i}) = t (a_{i+1}), \;\; \textrm{for each} \;\; i.$$
We set $h(p):=h(a_{k-1}), t(p):= t(a_{0})$, and we call $p$ a cyclic path if $h(p) = t(p)$.
We can concatenate two paths $p$ and $q$ to get the path $pq$ if $t(q)=h(p)$. Otherwise, we set $pq =0$.
\end{definition}

\begin{definition}
A path algebra $\Lambda_0 Q$ over the ring $\Lambda_0$ for a quiver $Q$ is defined as follows. As a $\Lambda_0$-module $\Lambda_0 Q$ is spanned by paths in $Q$.
The product is defined by linearly extending the concatenation of paths.  We take a completion $\widehat{\Lambda_0 Q}$ of $\Lambda_0 Q$ using a filtration defined by the minimal energy of the coefficients of elements in $\Lambda_0 Q$. In fact, all quiver path algebra in this paper will be the completed ones, and we will simply write
them as $\Lambda_0 Q$ without the hat notation.
We set $\Lambda Q := \Lambda \otimes_{\Lambda_0} \Lambda_0 Q$.
\end{definition}

For a vertex $v$, $\pi_v$ is an idempotent in $\Lambda Q^{\mathbb{L}}$.  They generates a semi-simple ring $\Lambda^\oplus$ defined below and
$\Lambda Q^{\mathbb{L}}$ can be regarded as a bimodule over $\Lambda^\oplus$.
\begin{definition}
The ring $\Lambda^\oplus$ is defined as
\begin{equation}\label{eqn:semisimNov}
\Lambda^\oplus = \Lambda \pi_1 \oplus \cdots \oplus \Lambda \pi_k,
\end{equation}
with $\pi_i \cdot \pi_i = \pi_i$ and $\pi_i \cdot \pi_j = 0$ for $i \neq j$.
\end{definition}

It is well-known that a category is like an algebra over a semi-simple ring.  For instance, Seidel used the semi-simple ring $\Lambda^\oplus$ to regard an $\AI$-category as an $\AI$-algebra.
We have a  $\Lambda^\oplus$-bimodule 
$$\mathcal{C} = \cF(\mathbb{L}, \mathbb{L}):=  \bigoplus_{i,j}  \cF(L_i, L_j),$$
where $\cF(L_i,L_j) = \Span_{\Lambda} (L_i \cap L_j)$ for $i\not=j$ and $\cF(L_i,L_i) = \CF^\bullet (L_i \times_\iota L_i)$, and
the module structure is given by
$$\pi_i \cdot \mathcal{C} \cdot \pi_j = \cF(L_i, L_j).$$
Then, the tensor product of $\mathcal{C}$ over $\Lambda^\oplus$  boils down to composable morphisms and can be identified as
$$\mathcal{C}^{\otimes r} = \bigoplus_{i_1,\cdots, i_k} \cF(L_{i_1},L_{i_2}) \otimes_{\Lambda} 
\cdots \otimes_\Lambda  \cF(L_{i_{k-1}}, L_{i_k}).$$

In this setting, $\mathcal{C}$ is said to be unital if each $\CF(L_i,L_i)$ has a unit $\be_{L_i}$ and
if $\be_{\mathbb{L}} := \oplus \be_{L_i}$ becomes a unit of $\mathcal{C}$.


\section{Mirror construction}
As in Section \ref{sec:base}, first we perform a base change for the $A_\infty$-algebra:
$$\tilde{A}^{\bL} :=  \Lambda Q^{\mathbb{L}} \widehat{\otimes}_{\Lambda^\oplus} \CF(\mathbb{L}, \mathbb{L}).$$
Due to the bimodule structure, an expression $p X_e := p \otimes X_e$ for a path $p$ and $X_e \in \CF(L_i,L_j)$ is non-zero if and only if $t(p) = i$.
We use Definition \ref{defn:nccoeff} to extend the $\AI$-structure on $\CF(\mathbb{L}, \mathbb{L})$ to $\tilde{A}^{\bL}$.

Denote the formal variable in $\Lambda Q$ associated to each arrow $e$ of $Q$ by $x_e$, and denote the corresponding odd-degree immersed generator in $\CF(\bL,\bL)$ by $X_e$.  Now take the linear combination
\begin{equation*}
b = \sum_e x_e X_e \in \tilde{A}^{\bL}. 
\end{equation*}
As in Definition \ref{def:deg}, $\deg x_e := 1 - \deg X_e$ so that $b$ has degree one.  In particular $\deg x_e$ is even. 
We define nc-weak Maurer-Cartan relations in the following way (assuming the Fukaya category $\mathcal{C}$ is unital).

\begin{definition} \label{def:mir-several}
The coefficients $P_f$ of the even-degree generators $X_f$ of \, $\CF(\bL,\bL)$ (other than the fundamental classes $\one_{L_i}$) in
$$m_0^b = m(e^b) = \sum_i W_i \one_{L_i} + \sum_f P_f X_f$$
are called the nc-weak Maurer-Cartan relations.  Let $R$ be the completed two-sided ideal generated by $P_f$.  Then define the noncommutative ring
$$\cA := \Lambda Q/ R.$$
$W_\bL = \sum_i W_i$ is called the worldsheet superpotential of \, $\bL$.
\end{definition}
We regard $\cA$ to be the space of noncommutative weakly unobstructed deformations of $\bL$.  
Instead of working on $\tilde{A}^\bL = \Lambda Q^{\mathbb{L}} \widehat{\otimes}_{\Lambda^\oplus} \CF(\mathbb{L}, \mathbb{L})$,
we define $\cA = \Lambda Q^{\mathbb{L}} / R$ as above and work on
$$A^\bL = \cA \widehat{\otimes}_{\Lambda^{\oplus}} \CF(\mathbb{L}, \mathbb{L}).$$ 
Now, the induced $\AI$-structure on  $A^\bL$ satisfies
\begin{equation}\label{eqn:defpotseveral}
m_0^b = W_\mathbb{L} \cdot \one_\mathbb{L} = \sum_i W_i \cdot \one_{L_i} \in \cA \widehat{\otimes}_{\Lambda^\oplus} \CF (\bL,\bL)
\end{equation}
where $W_\bL = \sum_i W_i$ and $\one_\mathbb{L} := \sum_i \one_{L_i}$.  We have $W_i = W_\mathbb{L} \cdot \pi_i$ where $\pi_i$ is the idempotent corresponding to the vertex $v_i$ (which is just the trivial path at $v_i$).
Geometrically $W_i$ counts the holomorphic polygons bounded by $\bL$ passing through a generic marked point in $L_i$.

$(\cA, W)$ is called a \emph{generalized mirror} of $X$, and we will construct a natural mirror functor from $\Fuk(X)$ to $\MF(\cA,W)$.

%
%
%
%

\section{Centrality and the spacetime superpotential}
As in the construction for a single Lagrangian, the worldsheet superpotential $W_\bL$ constructed in our scheme is automatically central.  The proof is similar and we include it here for readers' convenience.

\begin{theorem}[Centrality]\label{thm:centralseveral}
The worldsheet superpotential $W_\mathbb{L} \in \cA$ given in Definition \ref{def:mir-several} is central.
\end{theorem}

\begin{proof}
Over $\cA$, we have $m_0^b = W_{\bL} \one_\bL = \sum_i W_i \one_{L_i}$.
Therefore
\begin{eqnarray*}
0 = m\circ \hat{m} (e^b) &= & m(e^b, \hat{m}(e^b),e^b) \\
&=& m_2(b,\sum_i W_i \one_{L_i}) + m_2(\sum_i W_i \one_{L_i},b) \\
&=& \sum_{i,e} \left( W_i x_e \cdot m_2(X_e, \one_{L_i}) + x_e W_i \cdot m_2(\one_{L_i},X_e) \right) \\
&=& \sum_{i,e} \left( W_i x_e \cdot \delta_{t(e),v_i}) - x_e W_i \cdot \delta_{v_i,h(e)} \right) X_e \\
&=& \sum_{e} (W_{\mathbb{L}} \cdot x_e - x_e \cdot W_{\mathbb{L}}) X_e.
\end{eqnarray*}
where 
$\delta_{vw} = 1$ if $v=w$ and zero otherwise.
\end{proof}

In good situations the quiver algebra $\cA$ can be neatly written in the form of the Jacobi algebra with respect to a spacetime superpotential $\Phi_{\mathbb{L}} \in  \Lambda Q$.  Recall that the Jacobi algebra of a quiver $Q$ with a cyclic potential $\Phi$ in the path algebra is defined as follows:
\begin{definition}\label{defn:jacobiquiver1}
Let $Q$ be a quiver and $\Phi \in \Lambda Q$ be cyclic.  The Jacobi algebra $\Jac (Q, \Phi)$ is defined as
$$ \Jac (Q,\Phi) = \frac{\Lambda Q}{(\partial_{x_e} \Phi: e \in E)}$$
where $x_e$ are generators of the path algebra corresponding to the arrows of $Q$, partial derivatives $\partial_{x_e}$ are defined in \eqref{def:pd}, and $(\partial_{x_e} \Phi: e \in E)$ denotes the completed two sided ideal generated by $\partial_{x_e} \Phi$.
\end{definition}


Definition \ref{def:Phi} and Proposition \ref{prop:mccd} for the spacetime superpotential $\Phi = \Phi_\bL$ extends naturally to the current setting.  The proof is exactly the same and hence omitted.

\begin{prop}\label{prop:jac6}
Partial derivatives $\partial_{x_e} \Phi$ for $e \in E$ are noncommutative weak Maurer-Cartan relations.  Suppose $\dim_\C X$ is odd. If further every $L_i \in \bL$ is a Lagrangian sphere equipped with the standard height function, then
$$\cA = \Jac (Q, \Phi).$$
\end{prop}

%
%


\section{The mirror functor}\label{sec:FuktoMF}

In this section, we construct an $A_\infty$-functor
\begin{equation}\label{eqn:LMfunct}
\cF = \cF_\mathbb{L} : \Fuk (X) \to \MF(\Jac(Q^\mathbb{L}, \Phi_{\mathbb{L}}); W_{\mathbb{L}}).
\end{equation}
The construction is essentially the same as the one given in Chapter \ref{sec:singlelagfunctor}, except that now we are working over the semi-simple ring $\Lambda^\oplus$.  We give a detailed description of the object level functor $(\cF_\mathbb{L})_0$ which gives a concrete manifestation of the role of $\Lambda^\oplus$. 

Recall that we have 
$A^\bL = \cA \widehat{\otimes}_{\Lambda^{\oplus}} \CF(\mathbb{L}, \mathbb{L})$ and $\cA = \Lambda Q^{\mathbb{L}} / R$, where $R$ denotes the ideal generated by weak Maurer-Cartan relations.
Then the object level correspondence
$$ \mathrm{Ob} (\Fuk (X) ) \to \mathrm{Ob} (\MF (W)) $$
is defined by sending an unobstructed Lagrangian $U$ to the following projective $\Z_2$-graded $\cA$-module with an odd endomorphism $\delta_U$.

\begin{equation}\label{eqn:defPL}
\cF(U) := (\mathcal{A} \otimes_{\Lambda^{\oplus}} \CF(\mathbb{L}, U), \delta) \cong \left(\bigoplus_i \bigoplus_{v \in L_i \cap U} \cA \cdot \pi_i, \delta \right) 
\end{equation}
where $\Lambda^\oplus$-modules structure on $\CF(\bL,U)$ is defined by
$\pi_i \cdot v = 1$ for $v \in L_i \cap U$ and $\pi_i \cdot v= 0$ otherwise.
In the last identification, $\pi_i$ is the idempotent corresponding to the vertex $v_i$ so that $\cA \cdot \pi_i$ consists of (classes of) paths which start from the vertex $v_i$.


The $\Z_2$-grading naturally comes from the standard Lagrangian Floer theory.
For $v \in \mathbb{L} \cap U$, $\delta(v)=(-1)^{\deg v - 1} m_1^{b,0}$ which counts holomorphic strips bounded by $\mathbb{L}$ and $U$.  This defines $\delta$ by $A$-linearity.


\begin{prop}
$\cF(U)$ defined above is a matrix factorization of $(\cA,W_\mathbb{L})$, namely $\delta^2 = W \cdot \mathrm{Id}$.
\end{prop}

\begin{proof}
%
For $v \in L_i \cap U$, we have
\begin{eqnarray*}
\delta^2 (v) &=& (-1)^{|v|}(-1)^{|v|+1} m_1^{b,0} \circ m_1^{b,0} (v)\\
&=& m_2^{b,b,0}(m^b_{0,\mathbb{L}}, v) \\
&=& \sum_{j=1}^{k} m_2 ( W_j {\bf 1}_{L_j} , v)
\end{eqnarray*}
by the property of the unit. Since $v \in L_i \cap U$, $\pi_j \cdot v =0$ for $j \neq i$ so that
$$\sum_{j} m_2 (W_j {\bf 1}_{L_j} , v) = W_i \cdot m_2 ({\bf 1}_{L_i} , v) = W_i \cdot v.$$ 
Using the same property again, we have
$$W_i \cdot v = (W_1 \pi_1 + \cdots +W_k \pi_k) \cdot v = W \cdot v$$ 
and hence $\delta^2 v = W \cdot v$.
\end{proof}

The rest of the construction of $\cF_\mathbb{L}$ is similar to that in Chapter \ref{sec:singlelagfunctor} and we do not repeat here. We finish the section with the proof of the injectivity of our functor restricted to $\Hom(\bL,U)$ for any $L \in \Fuk(X)$ which is a generalization of Theorem \ref{thm:injthmsingle}.

\begin{theorem}[Injectivity] \label{thm:inj}
If the $\AI$-category is unital, then the  mirror $\AI$ functor $\mathcal{F}^\bL$ is injective on $H^\bullet(\Hom(\bL, U)$ (and also on  $\Hom(\bL, U)$) for any $U$.
\end{theorem}
\begin{proof}
The only difference is that now we are working over the semi-simple ring $\Lambda^\oplus$. Recall that
$\cF(U) =  (\mathcal{A} \otimes_{\Lambda^{\oplus}} \CF(\mathbb{L}, U), (-1)^{\deg (\cdot)} m_1^{b,0} (\cdot))$.
For a morphism $\phi \in \Hom(\mathcal{F}(\bL), \mathcal{F}(U))$, we define $\Psi (\phi)$ by 
$$\Psi (\phi) := \phi(\be_{\bL})|_{b=0} = \sum_{i=1}^k \phi ( \be_{L_i})|_{b=0} \in \CF(\bL, U) = \oplus_{i=1}^k \CF(L_i, U)$$
which is a chain map by the same argument as in the proof of Theorem \ref{thm:injthmsingle}. 

Now, for $p \in L_i \cap U (\subset \bL \cap U)$, 
$$\big( \Psi \circ \mathcal{F}^\bL_1  \big) (p) =  \big( \mathcal{F}^\bL_1(p) (\be_\bL) \big)|_{b=0}
= \big( m_2^{b,0,0}(\be_\bL, p) \big)|_{b=0} = m_2^{0,0,0}(\be_{L_i},p) =p$$
which implies that $\Psi$ is a right inverse of $\mathcal{F}^{\bL}_1$.
\end{proof}

%
%
%


\chapter{Finite group symmetry and graded mirror functors}\label{sec:fingpsymm}
In this chapter, we consider the noncommutative mirror construction for a K\"ahler manifold with a finite group symmetry
and define $\Z$-graded mirror functors in the Calabi-Yau cases.  The former is in parallel with the commutative construction in \cite[Section 5]{CHL13}. One advantage of the current setup is that it naturally allows non-Abelian groups. 

Finite group symmetry was used in the study of mirror symmetry since the earliest work of Candelas et.al. \cite{candelas91}.  It plays an important role in recent proofs of homological mirror symmetry for Fermat-type hypersurfaces by Seidel and Sheridan \cite{Se2},\cite{Se},\cite{Sh}.  Quotienting out a finite group symmetry often reduces to simpler geometries (for instance orbifold projective spaces) where homological mirror symmetry can be proved.  Then one can go back to the original manifold by incorporating the equivariant and/or $\Z$-graded constructions. 

\begin{example}
$\mathbb{P}^1_{3,3,3}$ is the $(\Z_3)^2$ quotient of the hypersurface 
$$M=\left\{[x,y,z] \in \mathbb{P}^2 \mid x^3+y^3+z^3=0 \right\}$$
where $(\Z_3)^2 \cong (\Z_3)^{3} / (1,1,1)$ acts on $[x,y,z]$ by componentwise multiplication.
The quotient of $M$ by $\Z_3$ (generated by $(1,-1,0)$) is the elliptic curve $E$ with complex multiplication by the cube root of unity.  By the equivariant construction in this section, a $\Z_2$-graded mirror of the Fukaya category of $E$ (resp.  $M$) is given as the $\Z_3$-equivariant (resp.  $(\Z_3)^2$-equivariant) matrix
factorization category of $W_{3,3,3}$. On the other hand both $M$ and $E$ are Calabi-Yau.
But  $\Z_3$-action on $E$ (whose quotient is $\mathbb{P}^1_{3,3,3}$ ) does not preserve the holomorphic volume form.  Thus, the mirror functor from orbi-sphere is not $\Z$-graded, but
upstairs functors can be made $\Z$-graded. Namely, from our graded construction in this section, we obtain a $\Z$-graded mirror of the Fukaya category of $E$ (resp.  $M$), which is given as the $\Z$-graded (resp. $\Z$-graded $\Z_3$-equivariant) matrix factorization category of $W_{3,3,3}$.
\end{example}

\begin{example}\label{ex:mainfgs}
Consider the superpotential $W = xyz$ defined on $\C^3$.  This is mirror to the pair-of-pants (see \cite{Se},  \cite{AAEKO}).  Let $Q$ be the quiver with only one vertex and three arrows labeled by $x,y,z$.  Take $\cA = \Lambda Q / R = \C[x,y,z]$ where $R$ is generated by the commutator relations $xy-yx,xz-zx,yz-zy$.  Then $\C^3 = \mathrm{Spec} \, \cA$.

First take the following $\Z_2$-grading.  The variables $x,y,z$ have degree $0$.  (The dual generators $X,Y,Z$ have degree $1$ so that $xX+yY+zZ$ has degree $1$.)  
Let $M_0 =\cA $ and $ M_1 = \cA [-1]$ regarded as free $\cA$-modules.  Then the left multiplications $x: M_0 \to M_1$ and $yz: M_1 \to M_0$ (which are of degree $1$ due to the shifting) give a $\Z_2$-graded matrix factorization of $W$.

We can refine it by taking a fractional grading of $\cA$, where $x,y,z$ have degree $2/3$.  This is mirror to the fact that the pair-of-pants equals to a $\Z_3$-quotient of a punctured elliptic curve (which is $\Z$-graded).  Then $W$ has degree $2$. (See Example \ref{ex:mainref2} below.) The previous factorization  can be modified so that 
the following left multiplications 
$$x: \cA \to \cA[-1/3], \quad yz:  \cA[-1/3] \to \cA$$
both are degree one maps. In this way,  they give a $\frac{1}{3} \cdot \Z$-graded matrix factorization of $W$.

Equivalently we can make it integrally graded by multiplying the fractional degree by $3/2$.  Then $x,y,z$ have degree 1 and $W$ has degree 3.  We take $x: \cA \to \cA[-2]$
which has degree $3$ and $yz: \cA[-2] \to \cA$
which has degree $0$.  This gives a graded matrix factorization in the sense of Orlov.

The above gives the $\Z$-graded Landau-Ginzburg model for the quotient $(\C^3/\Z_3,W=xyz)$, which is mirror to the punctured elliptic curve as a 3-to-1 cover of the pair-of-pants.  We can introduce a quiver in order to describe $\Z_3$-equivariant modules.  The quiver is depicted in Figure \ref{fig:quiver-C3modZ3}, where the arrows are labeled by $x_i,y_i,z_i$ for $i \in \Z_3$ (see Example \ref{ex:mainref1} for more details).  $\Z_3$-equivariant modules are given as modules over $\hat{\cA}$ which is the path algebra of $Q \# \Z_3$ quotient by the commutator relations $x_{i+1}y_{i}-y_{i+1}x_i,x_{i+1}z_i-z_{i+1}x_i,y_{i+1}z_i-z_{i+1}y_i$ for $i\in \Z_3$.  The lifted superpotential is $\hat{W} = \sum_{i \in \Z_3} x_i y_{i+1} z_{i+2}$.  We have the corresponding matrix factorization $(z_2y_1+ z_0 y_2 + z_1y_0  ) \cdot (x_0+ x_1 + x_2)$.
It can be made into $\Z$-graded by taking $\deg x_i = \deg y_i=\deg z_i=0$ for $i=0,1$ and $\deg x_2 = \deg y_2=\deg z_2=2$.  $W$ has degree $2$ and the differential in the matrix factorization has degree $1$ (by suitable degree shiftings). See Examples \ref{ex:intdegups} and \ref{ex:intdegups1} for more details.
\end{example}

\begin{figure}[h]
\begin{center}
\includegraphics[scale=0.45]{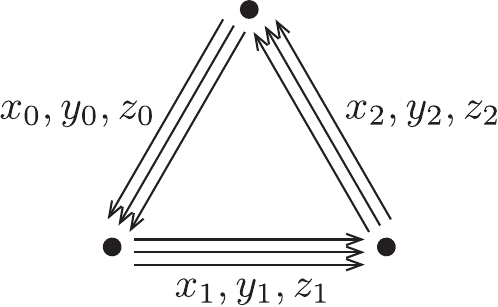}
\caption{The quiver corresponding to $\C^3/\Z_3$.}\label{fig:quiver-C3modZ3}
\end{center}
\end{figure}

The above example manifests some interesting aspects of group action and grading, and we will study them in a general framework.  Proposition \ref{prop:hatA} provides the mirror quotient geometry.  The mirror $B$-model category is given in Proposition \ref{prop:MF}.  If in addition the manifold is Calabi-Yau and the group action preserves the holomorphic volume form up to scaling, then the mirror quotient geometry is equipped with a natural $\Z$-grading and the functor respects the grading (Proposition \ref{prop:F-graded}).

\section{K\"ahler manifolds with finite group symmetry}
Let $(\widetilde{X},\omega)$ be a K\"ahler manifold where $\omega$ denotes a K\"ahler form.  Suppose we have a finite group $G$ acting on $\widetilde{X}$ such that for each $g \in G$, $g^* \omega = \omega$.  Then $X = \widetilde{X}/G$ is a K\"ahler orbifold.  At this point we \emph{do not assume that $G$ is Abelian.}

Let $\bL = \bigcup_i L_i \subset X^0$ be a compact spin oriented weakly unobstructed immersed Lagrangian with clean self-intersections, where 
$$X^0 := (X - \{p \in X: G_p \not= \{1\}\})/G,$$
and assume that each $L_i$ lifts to a smooth compact unobstructed Lagrangian submanifold $\tilde{L}^1_i \subset \widetilde{X}$ (and we fix this choice of lifting).  We have $|G|$ choices of a lifting, which are denoted as $\bL^g := g \cdot \bL^1$ where $\bL^1 = \bigcup_i \tilde{L}^1_i$.  The union $\bigcup_g \bL^g$ is denoted as $\tilde{\bL}$.  By assumption $\bL^g$ intersects with $\bL^h$ cleanly for $g \not= h$ (and it is assumed that $\bL^g \not= \bL^h$).  We assume that the Fukaya $\AI$-category of $\widetilde{X}$
is strictly $G$-equivariant (i.e. we have $m_k(gw_1,\cdots, gw_k) = gm_k(w_1,\cdots,w_k)$ for any $g\in G$) and unital.  Then the Fukaya category of $X$ is defined as the $G$-invariant part of the Fukaya category of $\widetilde{X}$.

We carry out the construction described in Chapter \ref{sec:MFseverallag} for $\bL \subset X$ downstairs and obtain $(\cA, W)$ where $\cA = \Lambda Q/R$ is a quiver algebra and $W$ is a central element of $\cA$.  We have a functor $\cF:\Fuk(X) \to \MF(\cA,W)$.  Similarly we can carry out the mirror construction upstairs for $\tilde{\bL} \subset \widetilde{X}$.  The aim is to understand the relation between the mirror and the functor of $\widetilde{X}$ upstairs and that of $X$ downstairs.

\subsection{Action by the character group}\label{sec:dualgponmir}

We show that there exists a canonical action of $\hat{G}$ on $\cA$, where $\hat{G} = \Hom(G,U(1))$ denotes the character group of $G$.  Moreover we have $W \in \cA^{\hat{G}}$, the $\hat{G}$-invariant subalgebra.  When $G$ is Abelian, the $\hat{G}$ quotient $(\cA^{\hat{G}},W)$ is essentially the mirror of $\widetilde{X}$.  (See \ref{sec:G-MF} for the precise mirror category.)  

On the other hand taking the character group of $G$ loses information when $G$ is non-Abelian.  In the next subsection we formulate the notion of a `formal dual group quotient' (Definition \ref{def:formal-quot}).  We shall formulate the mirror of $\widetilde{X}$ as a formal dual-$G$ quotient of $(\cA,W)$ for a general finite group $G$.

\begin{definition}
Fix a smooth lift $\bL^1 \subset \widetilde{X}$ of $\bL \subset X$.  An immersed generator $x$ of $\bL$ is lifted to a Floer morphism from $\bL^1$ to $\bL^g$ for some $g\in G$.  This defines a map $\mathfrak{f}: x \mapsto g$ from the set of immersed generators of $\bL$ to $G$. 
\end{definition}

$\mathfrak{f}$ extends to a map $\Lambda Q \to G$.  A loop in $Q$ is mapped to $1 \in G$.

\begin{remark}
Denote the lift of $x$ by $\hat{x}$.  If we change the lift to $h\cdot\bL^1$ instead of $\bL^1$, then $x$ is lifted to $h\cdot\hat{x}$ instead, which is a morphism from $h\cdot\bL^1$ to $h\cdot g\cdot \bL=(h\cdot g\cdot h^{-1}) \cdot (h\cdot\bL^1)$.  Thus the function $\mathfrak{f}$ changes to $h \cdot \mathfrak{f} \cdot h^{-1}$.  Throughout the section we have fixed the choice of lift $\bL^1$.
\end{remark}

\begin{definition} \label{lem:g_x}
We define an action of $\hat{G}$ on the path algebra $\Lambda Q$ as follows.
Let $\chi \in \hat{G}$, and $x_p\ldots x_1 \in \Lambda Q$.  Let $g_i = \mathfrak{f}(x_i)$.  Then
\begin{equation} \label{eq:G-on-x}
\chi \cdot x_p\ldots x_1 := \chi(g_1)\ldots \chi(g_p) x_p\ldots x_1
\end{equation}
which extends linearly to give an action of $\hat{G}$ on $\Lambda Q$.
\end{definition}

The following proposition ensures that $(\cA,W)$ is preserved by the action of $\hat{G}$.  The mirror of $\tilde{X}$ can be taken to be the $\hat{G}$ quotient of $(\cA,W)$ when $G$ is Abelian.

\begin{prop}\label{prop:G-on-A}
$\hat{G}$ acts on each weakly unobstructed relation by an overall $U(1)$-scaling.  Thus the two-sided ideal $R$ generated by the weakly unobstructed relations remains invariant under $\hat{G}$, and hence the action of $\hat{G}$ descends to $\cA = \Lambda Q/R$.  Moreover $W \in \Lambda Q/R$ is invariant under $\hat{G}$.
\end{prop}
\begin{proof}
Consider each term $P_{\bar{X}} \cdot \bar{X}$ of $m_0^{(\bL,b)}$ where $\bar{X}$ is an even-degree morphism.\footnote{In Chapter \ref{sec:333} and \ref{sec:4puncture2222}, $\bar{X}$ is used to denote the even-degree complementary generator to an odd-degree immersed generator $X$.  In general complementary generators can have either even or odd degrees depending on the dimension of the Lagrangian.
In this section $\bar{X}$ just refers to an arbitrary even-degree generator.}  Each $P_{\bar{X}}$ is a weakly unobstructed relation.   Note that by assumption $m_0^{(\bL,b)}$ is $G$-invariant.  Fix a smooth lift $\bL^1$ of $\bL$.  $\bar{X}$ corresponds to certain morphism from $\bL^1$ to $\bL^g$ for some $g \in G$.  Each term in $P_{\bar{X}}$ corresponds to a polygon with ordered inputs $x_1,\ldots,x_l$ and the output $\bar{X}$, which has a lift as it is null-homologous. Observe that $\mathfrak{f}(x_l)\ldots \mathfrak{f}(x_1)=\mathfrak{f}(\bar{X})=g$.  Hence $\chi \cdot P_{\bar{X}} = \chi(g)P_{\bar{X}}$ for all $\chi \in \hat{G}$.  Moreover when $\bar{X}$ is the unit of a branch $L_i$ for some $i$, we have $g = 1$ and hence $\chi(g)=1$.  Thus $W$ is invariant under $\hat{G}$.
\end{proof}

\subsection{Formal dual group quotient of an LG model}
The mirror of the geometry $\widetilde{X}$ upstairs should be given by the quotient of the LG model $(\cA,W)$ of $X$ by the `dual group' of $G$.  When $G$ is Abelian, its dual is simply the character group $\hat{G}$, and we have the quotient of $(\cA,W)$ by $\hat{G}$ from the last section.  However we do not have a direct definition of the dual group when $G$ is non-Abelian.  Nevertheless we can formally define the dual group quotient of $(\cA,W)$ using G-grading and smash product specialized to quivers.

First we need an action on $(\cA,W)$.  It is given by a $G$-grading on the quiver algebra $\cA = \Lambda Q / R$ (\cite{GrM}) such that $W$ has degree $1 \in G$.  Concretely it is a function $\mathfrak{f}$ from the set of arrows of $Q$ to $G$ satisfying the properties in the following definition.  We call it to be a formal dual group action on $(\cA,W)$.

\begin{definition}[Formal dual group action] \label{def:dual-group}
Let $Q$ be a quiver and $\cA = \Lambda Q / \langle p_1,\ldots,p_K \rangle$ be a path algebra with relations.   A $G$-grading on the algebra $\cA$ is given by a function $\mathfrak{f}$ from the set of arrows of $Q$ to $G$ (which induces a map from the set of paths to $G$) such that 
 relation $R$ is generated by homogeneous elements ( i.e. relations $p_l$ can be chosen so that for each $l$, each term of the relation $p_l$ is mapped to the same $g_l \in G$),

A formal dual-$G$ action on $(\cA,W)$ is a $G$-grading on the algebra $\cA$
such that $W$ has grading $1 \in G$ (i.e. it has a representative in $\Lambda Q$ whose every term is mapped to $1 \in G$).
\end{definition}

When $G$ is Abelian, the function encodes an action of the character group $\hat{G}$ on the path algebra of $Q$, namely $\chi \in \hat{G}$ acts by $\chi \cdot x = \chi (\mathfrak{f}(x)) x$ for each edge $x$.  We may also denote $\mathfrak{f}(x)$ by $g_x$.  In our context, the function is given by Definition \ref{lem:g_x} from the $A$-side geometry.

The proof of Proposition \ref{prop:G-on-A} gives the following.

\begin{prop}
Let $(\cA,W)$ be a generalized mirror of $X$ constructed in Chapter \ref{sec:MFseverallag}.
The function given in Definition \ref{lem:g_x} defines a formal dual-$G$ action on $(\cA, W)$.
\end{prop}


Then we define the quotient algebra $\hat{\cA}$ essentially as the smash product of $\cA$ with $G$.  In the following we specialize to our case of quiver algebras and make a concrete description in Definition \ref{def:formal-quot}.

\begin{definition} \label{Q*G}
Given a quiver $Q$ and a formal dual group action $\mathfrak{f}$ of $G$ on its path algebra, define the quiver $Q \# G$ as follows.  The number of vertices in $Q \# G$ is $|G|$ times the number of vertices in $Q$.  Each vertex is labeled as $v^g$, where $v$ is a vertex in $Q$ and $g \in G$.  The number of arrows in $Q \# G$ from $u^{g_1}$ and $v^{g_2}$ is defined as the number of arrows $x$ in $Q$ from $u$ to $v$ with $\mathfrak{f}(x) = g_2 g_1^{-1}$.  Each corresponding arrow in $Q \# G$ is labeled by $x^{g_1}$ (whose source vertex is $u^{g_1}$).  
\end{definition}

\begin{example}\label{ex:mainref1}
In the case of Example \ref{ex:mainfgs}, the downstair quiver $Q$ has one vertex $v$ with three arrows $x,y,z$ together with commuting relations among three arrows. The corresponding path algebra with relations is nothing but $\C[x,y,z]$. 
\begin{figure}[h]
\begin{center}
\includegraphics[scale=0.45]{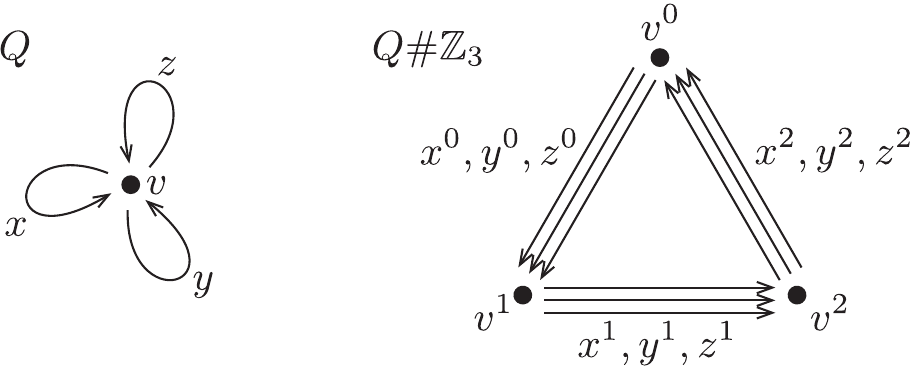}
\caption{The quiver $Q$ corresponding to $\C^3$ and $Q \# \Z_3$.}\label{fig:quiver-C3modZ3ref}
\end{center}
\end{figure}

We consider the $\Z_3$-grading induced by $\mathfrak{f}(x) = \mathfrak{f}(y) = \mathfrak{f}(z) = 1 \in \Z_3$. Then $Q \# \Z_3$ has three vertices $v^0,v^1,v^2$, and there are three arrows from $v^{i}$ to $v^{i+1}$ labeled by $x^i,y^i,z^i$ for $i \in \Z_3$. See Figure \ref{fig:quiver-C3modZ3ref}.
\end{example}

There is a canonical correspondence between paths in $Q \# G$ and that in $Q$.

\begin{definition} \label{def:lift-path}
Fix $g \in G$.  For a path $x_k \ldots x_1$ in $Q$, define a path in $Q \# G$ to be $x_k^{g_k} \ldots x_2^{g_2}x_1^{g_1}$, where $g_1 = g$ and $\mathfrak{f}(x_i) = g_{i+1} g_i^{-1}$ for $i = 1,\ldots,k-1$.  It gives a map $L_g: \Lambda Q \to \Lambda (Q \# G)$ by extending linearly.  

Moreover we have a canonical map $\Lambda (Q \# G) \to \Lambda Q$ by forgetting the group elements, namely 
$$x_k^{g_k} \ldots x_2^{g_2}x_1^{g_1} \mapsto x_k \ldots x_1$$
if $\mathfrak{f}(x_i) = g_{i+1} g_i^{-1}$ for $i = 1,\ldots,k-1$, and zero otherwise.
\end{definition}

Now we define the formal dual $G$-quotient of $(\cA,W)$.

\begin{definition} \label{def:formal-quot}
Let $Q$ be a quiver and $\cA = \Lambda Q / \langle P_1,\ldots,P_K \rangle$ be a path algebra with relations.  Suppose there is a formal dual group action $\mathfrak{f}$ on $\cA$ by $G$.  Then we define the quotient to be the algebra $\hat{\cA} = \Lambda (Q \# G) \big/ \langle L_g(P_i): g \in G, i = 1,\ldots,K \rangle$, where $L_g: \Lambda Q \to \Lambda (Q \# G)$ is given in Definition \ref{def:lift-path}.

If further we have $W \in \cA$ which is invariant under the action, define $\hat{W} := \sum_{g\in G} W^g$ where $W^g$ are the lifts of $W$.
\end{definition}

\begin{prop}
In the above definition, if $W$ is a central element of $\cA$, then $\hat{W}$ is a central element of $\hat{\cA}$.
\end{prop}

\begin{proof}
We need to prove that $x^g \cdot \sum_{g'} W^{g'} = \left( \sum_{g''} W^{g''} \right) \cdot x^g$ for all $x^g$ upstairs.  Since $W$ is invariant, each term in $W^{g'}$ (which is a path in $Q\#G$) has its final vertex to be $v^{g'}$ for some $v$.  Thus for the left hand side, only the term with $g' = g$ survives.  Similarly for the right hand side, $W^{g''} \cdot x^g \not=0$ only when $\mathfrak{f}(x) = g^{-1} g''$.  Thus the left hand side is $x^g \cdot W^g$ and the right hand side is $W^{g''} \cdot x^g$ where $\mathfrak{f}(x) = g^{-1} g''$.  The initial vertex is $u^g$ and the final vertex is $v^{g''}$ for both sides, for some $u$ and $v$.

Downstairs we already know from Theorem \ref{thm:centralseveral} that $x \cdot W = W \cdot x$.  Lifting this by the group element $g$, we obtain $x^g \cdot W^g = W^{g''} \cdot x^g$ as desired.
\end{proof} 

Now we give a description of the category $\MF (\hat{\cA}, \hat{W})$ from the formal dual group quotient as follows.

\begin{definition} \label{def:A*G}
$\cA \# G$ is defined to be the category whose set of objects is given by $Q_0 \times G$, where $Q_0$ denotes the vertex set of $Q$.  We denote the objects by $A_{(v,g)}$.  The morphism space
$\Hom_{\cA \# G} (A_{(v_1,g_1)},A_{(v_2,g_2)})$ is defined to be the subspace of $\cA$ spanned by paths $x_p \ldots x_1$ in $Q$ from $v_1$ to $v_2$ with
\begin{equation} \label{eq:A*G}
\mathfrak{f}(x_p) \ldots \mathfrak{f}(x_1) = g_2 g_1^{-1}.
\end{equation}

The category $\Tw(\cA\#G,W)$ consists of objects of the form 
$$\left(\bigoplus_{i \in I} A_{(v(i),g(i))}[\sigma_i],\delta\right)$$
where $\sigma_i \in \Z_2$, $\delta$ is an odd-degree endomorphism of $\bigoplus_{i \in I} A_{(v(i),g(i))} [\sigma_i]$ satisfying $\delta^2 = W \cdot \mathrm{Id}$.  It is an analog of twisted complexes.

The morphism space from $\left(\bigoplus_{i \in I} A_{(v(i),g(i))}[\sigma_i],\delta_1 \right)$ to $\left(\bigoplus_{k \in K} A_{(u(k),f(k))}[\tau_k],\delta_2 \right)$ is defined as
$$\bigoplus_{i,k} \mathrm{Hom}_{\cA\# G}(A_{(v(i),g(i))},A_{(u(k),f(k))})[\tau_k-\sigma_i]$$ equipped with a differential $d$ defined by $d u = \delta_2 \circ u - (-1)^{\deg u} u \circ \delta_1$.
\end{definition}

\begin{remark}
In other words, objects in $\cA \# G$ are simply vertices of $Q \# G$.  The morphism space $\Hom_{\cA \# G} (A_{(v_1,g_1)},A_{(v_2,g_2)})$ is $v_2^{g_2} \cdot \hat{\cA} \cdot v_1^{g_1}$ spanned by paths from $v_1^{g_1}$ to $v_2^{g_2}$.
\end{remark}

\begin{prop} \label{prop:MF}
$\MF (\hat{\cA}, \hat{W})$ is isomorphic to $\Tw(\cA\#G,W)$.
\end{prop}

\begin{proof}
An object in $\MF (\hat{\cA}, \hat{W})$ is a projective $\hat{\cA}$-module of the form $\bigoplus_{i \in I} P_{v(i)^{g(i)}}[\sigma_i]$ together with an odd endomorphism $\hat{\delta}$ with $\hat{\delta}^2 = \hat{W}$, where $\sigma_i \in \Z_2$ and $P_{v^g}$ is the projective $\hat{\cA}$-module $v^g \cdot \hat{\cA}$ which is spanned by all paths in $Q\#G$ ending at $v^g$ (recall that vertices of $Q\#G$ are labeled by $v^g$).  The morphism space from $\left( \bigoplus_{i \in I} P_{v(i)^{g(i)}}[\sigma_i],\delta_1 \right)$ to $\left( \bigoplus_{k \in K} P_{u(k)^{f(k)}}[\tau_k],\delta_2 \right)$ is
$$ \bigoplus_{i,k} \Hom (P_{v(i)^{g(i)}}, P_{u(k)^{f(k)}})[\tau_k - \sigma_i] $$
where $\Hom (P_{v^{g}}, P_{u^{f}})$ is spanned by all paths starting from $v^g$ to $u^{f}$.

We define the object in $\Tw(\cA\#G)$ corresponding to $\bigoplus_{i \in I} P_{v(i)^{g(i)}}[\sigma_i]$ to be $\bigoplus_{i \in I} A_{(v(i),g(i))}[\sigma_i]$.  We need to identify $\Hom_{\cA\# G}(A_{(v,g)},A_{(u,f)})$ with $\Hom (P_{v^g}, P_{u^f})$.  Given a path $x_p \ldots x_1$ in $Q$ from $v$ to $u$ and the group element $f \in G$, by Definition \ref{def:lift-path} we have a path $x_p^{g_p} \ldots x_1^{g_1}$ in $Q\#G$ from $v^{g}$ to $u^{g_{p+1}}$, where $g_1=g$, $\mathfrak{f}(x_i) = g_{i+1} g_i^{-1}$ for $i=1,\ldots,p$.  Thus $g_{p+1} = f$ if and only if $\mathfrak{f}(x_p) \ldots \mathfrak{f}(x_1) = f g^{-1}$.  

From the above identification between the morphism spaces, the endomorphism $\hat{\delta}$ upstairs corresponds to an endomorphism $\delta$ of $\bigoplus_{i \in I} A_{(v(i),g(i))}[\sigma_i]$.  From the identification between paths upstairs and downstairs by Definition \ref{def:lift-path}, it is clear that $\hat{\delta}^2 = \hat{W}$ if and only if $\delta^2 = W$.
Thus the two categories are isomorphic.
\end{proof}

\subsection{The mirror and the functor upstairs} \label{sec:mir-up}
Now we carry out the mirror construction upstairs for $\tilde{\bL} \subset \widetilde{X}$ and identify it with the formal dual group quotient $(\hat{A},\hat{W})$ defined in the previous subsection.

First of all, it is clear from definition that

\begin{prop}
The quiver $\hat{Q}$ formed by odd-degree generators of $\tilde{\bL}$ equals to $Q \# G$, where $Q$ is the quiver of $\bL$ and the $G$-grading $\mathfrak{f}$ in Definition \ref{Q*G} is taken to be the one given in Definition \ref{lem:g_x}.
\end{prop}

Then we consider the weakly unobstructed relations and worldsheet superpotential.  By using the correspondence between holomorphic polygons bounded by $\tilde{\bL}$ upstairs and those bounded by $\bL$ downstairs, one easily obtains the following proposition.

\begin{prop} \label{prop:hatA}
Let $P_{\bar{X}} \in \Lambda Q$ be the weakly unobstructed relations in the definition of $\cA$, where $\bar{X}$ runs over the even-degree generators of $\bL$ other than the fundamental classes of $L_i$.  Then the weakly unobstructed relations for $\tilde{\bL}$ are given by $P^g_{\bar{X}}$ which are the lifting of $P_{\bar{X}} \in \Lambda \hat{Q}$ by $g \in G$ given in Definition \ref{def:lift-path}.  Thus the quiver algebra with relations upstairs is given by $\hat{\cA} = \Lambda \hat{Q} \big/ \langle P^g_{\bar{X}} \rangle$.

Let $W \in \cA$ be the worldsheet superpotential for $\bL \subset X$.  Then the worldsheet superpotential upstairs for $\tilde{\bL} \subset \widetilde{X}$ is given by $\hat{W} = \sum_{g} W^g \in \hat{\cA}$.
\end{prop}

\begin{proof}
The weakly unobstructed relations for $\tilde{\bL}$ are given by counting holomorphic polygons with output $\bar{X}^g$.
The moduli space of stable holomorphic polygons in $\widetilde{X}/G$ with inputs $(X_1,\ldots,X_k)$ and the output $\bar{X}$ is isomorphic to the moduli space of stable discs in $\widetilde{X}$ with the corresponding inputs $(X_1^{g_1},\ldots,X_k^{g_k})$ and the output $\bar{X}^g$, where $g_1 = g$ and $g_{x_i} =g_{i+1} g_i^{-1}$ for $i=1,\ldots,k-1$.  Thus the counting with output $\bar{X}^g$ equals to $P^g_{\bar{X}}$.

Similarly by considering holomorphic discs with output $\one_i^g$ (i.e. those passing through Poincar\'e dual of $\one_i^g$) instead of $\bar{X}^g$, we obtain that the superpotential for $\widetilde{X}$ is $\sum_{g} W^g$.
\end{proof}

In conclusion, we have a mirror functor $\Fuk (\widetilde{X}) \to \MF (\hat{\cA}, \hat{W}) = \Tw(\cA\#G,W)$ upstairs.

\subsection{Abelian case: $\hat{G}$-equivariant matrix factorizations} \label{sec:G-MF}



Let's examine the above notions in the case that $G$ is Abelian.

\begin{prop}
Suppose $G$ is Abelian.  The action of $\hat{G}$ on the path algebra of $Q$ by a $U(1)$-scaling on each generator $x$ is equivalent to a formal dual group action of $G$.
\end{prop}

\begin{proof}
Given  the action of $\hat{G}$, for each arrow $x$ of the quiver, we have $\chi \cdot x = \mathfrak{f}(\chi) x$ where $\mathfrak{f}(\chi) \in U(1)$ for every $\chi \in \hat{G}$.  This defines a character of $\hat{G}$.  Since $G$ is Abelian, $\widehat{\hat{G}} = G$ and hence it defines an element $g_x \in G$.
Conversely given a function $\mathfrak{f}$ from the arrow set of $Q$ to $G$, we have the action of $\chi$ by $\chi \cdot x = \chi(\mathfrak{f}(x)) \cdot x$.
\end{proof}

We have learned from Proposition \ref{prop:MF} that there is a natural identification between $\MF(\hat{\cA},\hat{W})$ and $\Tw(\cA\#G,W)$, no matter whether $G$ is Abelian or not.  On the other hand, when $G$ is Abelian, we already know that the geometry upstairs is essentially the $\hat{G}$-quotient of $(\cA,W)$, and the mirror category should consist of $\hat{G}$-equivariant matrix factorizations.  This motivates the following definition.

\begin{definition}
Suppose $G$ is Abelian.  The category of $\hat{G}$-equivariant matrix factorizations is defined as $\MF_{\hat{G}}(\cA,W) := \Tw(\cA\#G,W)$.
\end{definition}

The objects in $\Tw(\cA\#G,W)$ can be identified with $\hat{G}$-equivariant twisted modules over $\cA$, and so the above is the same as the standard definition of $\hat{G}$-equivariant matrix factorizations.  Intuitively speaking, the objects in $\MF_{\hat{G}}(\cA,W)$ are endomorphisms $\delta$ on orbi-bundles over the quotient geometry described by the invariant subalgebra $\cA^{\hat{G}}$ with $\delta^2 = W$.
We can describe the morphism spaces more explicitly as follows.

\begin{prop}
Suppose $G$ is Abelian.  The morphism space
$\Hom_{\cA \# G} (A_{(v_1,h_1)},A_{(v_2,h_2)}) $
is given by the invariant subspace of $v_2 \cdot \cA \cdot v_1$ under the $(h^{-1}_1,h_2)$-twisted action of $\hat{G}$.  The twisted action of $\chi \in \hat{G}$ on a generator $x_1 \ldots x_p$ is defined as
$$ \chi \cdot_{(h^{-1}_1,h_2)} (x_1\ldots x_p) := \chi(h^{-1}_1 h_2 \mathfrak{f}(x_1) \ldots \mathfrak{f}(x_p)) x_1 \ldots x_p. $$
\end{prop}
\begin{proof}
Since $G$ is Abelian, Condition \ref{eq:A*G} is equivalent to
$\chi(h^{-1}_1 h_2 \mathfrak{f}(x_1) \ldots \mathfrak{f}(x_p)) = 1$ for all $\chi$.  

Thus $\Hom_{\cA \# G} (A_{(v_1,g_1)},A_{(v_2,g_2)})$ admits the above alternative description.
\end{proof}

In the case when the quiver $Q$ has only one vertex and $G = \Z_d$, the definition agrees with the one given in \cite{CT,AAEKO}.  By the construction in Section \ref{sec:FuktoMF}, we have a functor from $\Fuk(\widetilde{X})$ to $\MF_{\hat{G}}(\cA,W)$.

\section{Calabi-Yau manifolds with finite-group symmetry} \label{sec:CY-grading}


Let $(\widetilde{X},\omega,\Omega)$ be a Calabi-Yau manifold, where $\omega$ denotes a K\"ahler form and $\Omega$ denotes a holomorphic volume form\footnote{We do not require the metric condition $\omega^n = c \Omega \wedge \bar{\Omega}$ for a constant $c \in \C$.}.  Suppose we have a finite group $G$ acting on $\widetilde{X}$ such that for each $g \in G$, $g^* \omega = \omega$ and $g^* \Omega = c_g \Omega$ for some $c_g \in U(1)$.
$X = \widetilde{X}/G$ is a K\"ahler orbifold, while the Calabi-Yau form $\Omega$ does not necessarily descend to $\widetilde{X}/G$.  Let $m$ be the smallest integer such that $c_g^m = 1$ for all $g \in G$.

The association $g \mapsto c_g$ defines a homomorphism $G \to U(1)$, and we denote its kernel by $K$.  The image of $G$ is $\Z_m$ in $U(1)$, and hence we have the exact sequence
$$ 0 \to K \to G \to \Z_m \to 0. $$

As in the previous subsection, let $\bL = \bigcup_i L_i \subset X^0$ be a compact spin oriented weakly unobstructed immersed Lagrangian with clean self-intersections, where $X^0 := (X - \{p \in X: G_p \not= \{1\}\})/G$, and assume that each $L_i$ lifts to a smooth compact unobstructed Lagrangian submanifold $\tilde{L}^1_i \subset \widetilde{X}$ (and we fix this choice of lifting).  We have $|G|$ choices of a lifting, which are denoted as $\bL^g := g \cdot \bL^1$.  The union $\bigcup_g \bL^g$ is denoted as $\tilde{\bL}$.  By assumption $\bL^g$ intersects with $\bL^h$ cleanly for $g \not= h$.  We additionally assume that each $\tilde{L}^1_i$ has zero Maslov class, such that it can be equipped with a $\Z$-grading with respect to $\Omega$.

Downstairs we have constructed a functor $\cF$ from the $\Z_2$-graded Fukaya category of $X$ to the $\Z_2$-graded category of (noncommutative) matrix factorizations of the superpotential $W \in \cA = \Lambda Q/R$ associated to $\bL$. (The objects in $\Fuk(X)$ are Lagrangian immersions in $X^0$ which lift to compact unobstructed Lagrangian submanifolds in $\widetilde{X}$.  Morphism spaces are the $G$-invariant parts of the corresponding morphism spaces in $\mathrm{Fuk}(\widetilde{X})$.  $A_\infty$ operations $m_k$ for $\mathrm{Fuk}(\widetilde{X})$ restrict to give an $A_\infty$ structure on $\Fuk(X)$.)  Upstairs we have defined a functor $\tilde{\cF}$ from the $\Z_2$-graded Fukaya category of $\widetilde{X}$ to the $\Z_2$-graded category $\Tw(\cA\#G, W)$.

In this section, we will equip both sides downstairs with $\frac{1}{m} \Z$-gradings and enhance $\cF$ to a $\frac{1}{m} \Z$-graded functor (Proposition \ref{prop:F^1/m}).  Moreover, by studying the relation between the fractional grading of $X$ and the integer grading of $\widetilde{X}$, we will define a $\Z$-grading on the category of equivariant matrix factorizations $\Tw(\cA\#G, W)$ and enhance $\tilde{\cF}$ to a $\Z$-graded functor (Proposition \ref{prop:F-graded}). 

$\Z$-graded matrix factorizations were introduced by Orlov \cite{Orlov}.  The formulation we make here is similar to that in C\u{a}ld\u{a}raru-Tu \cite{CT} when $G = \Z_d$ and $\cA$ is the usual polynomial ring.  Fractional grading has been used by Abouzaid et al. \cite{AAEKO} for the mirror symmetry of cyclic covers of a punctured sphere.

\subsection{Fractional grading downstairs}

The fractional grading on $\Fuk(X)$ is defined as follows.  Consider $\Omega^{\otimes m}$ on $\widetilde{X}$ which is invariant under $G$, namely $g^* \Omega^{\otimes m} = \Omega^{\otimes m}$.  

\begin{definition}
A $\frac{1}{m} \Z$-graded immersed Lagrangian in $X^0$ is an immersed Lagrangian $L$ equipped with a phase function $\theta_L: \tilde{L} \to \R$ such that $\Omega^{\otimes m}(T_p \tilde{L}) = \conste^{\pi \consti \theta_L(p)}$ for all $p \in \tilde{L}$ (where $\tilde{L}$ denotes the normalization of $L$).  

$L$ is called to be a special Lagrangian with respect to $\Omega^{\otimes m}$ if it can be equipped with a constant phase function.
\end{definition}

\begin{definition}
Let $S$ be a clean intersection between two $\frac{1}{m} \Z$-graded Lagrangians submanifold $L_1$ and $L_2$.  The degree of $\one_S$, the fundamental class of $S$ as a morphism from $L_1$ to $L_2$, is defined as
\begin{equation} \label{eq:deg^{1/m}}
\deg^{1/m}(\one_S) := \frac{1}{m}(\theta_{L_2}(p) - \theta_{L_1}(p) + \angle^m(\widearc{L_2L_1})|_p) \in \frac{1}{m} \Z
\end{equation}
where $p \in S$ is any chosen point, $\pi \cdot \angle^m(\widearc{L_2L_1})|_p$ is the phase angle of the positive-definite path $\widearc{L_2L_1}$ from $T_pL_2$ to $T_pL_1$ measured by $\Omega^{\otimes m}(p)$.

For a general degree $l$ element  $a \in H^l(S)$, $\deg^{1/m}(a) := \deg^{1/m}(\one_S) + l$.
\end{definition}

Note that the definition is independent of the choice of $p$ since all quantities involved are continuous in $p$ while the whole expression is a rational number.

The following relates the $\Z_2$ grading coming from the orientation and the fractional grading.

\begin{lemma} \label{lem:parity}
If $m$ is odd, then the parity of a morphism $a$ (under the $\Z_2$ grading) equals to $m \cdot \deg^{1/m}(a) \mod 2$.  If $m$ is even, then $m \cdot \deg^{1/m}(a)$ is always even.
\end{lemma}
\begin{remark}
If $m$ is even, $\Omega^{\otimes m} (T_p\tilde{L})$ becomes independent of the orientation of $L$.
So, the second statement of the lemma is natural.
\end{remark}

\begin{proof}
Taking a lifting of $L_i$ in $\tilde{X}$, and correspondingly a lifting $\tilde{a}$ of the morphism $a$.  $(\theta_{L_i} + 2k_i \pi)/m$ for some $k_i \in \Z$ gives a $\Z$-grading of the liftings under $\Omega$.  
Note that $m \cdot \deg^{1/m}(a) = m \cdot \deg^{\Z}(\tilde{a}) + 2k$ under the $\Z$-grading for some $k \in \Z$.  This is always even if $m$ is even.  

The parity of $a$ equals to that of $\tilde{a}$ under the $\Z_2$ grading, which equals to $\deg^{\Z}(\tilde{a})$.  If $m$ is odd, the parity equals to $m \cdot \deg^{\Z}(\tilde{a}) + 2k$.
\end{proof}

\begin{prop} \label{prop:deg^1/m}
Suppose $L$ is an immersed special Lagrangian with respect to $\Omega^{\otimes m}$ in $X^0$.  Then the degree of $\one_S$ as a morphism from one branch $L_1$ to another branch $L_2$ at a clean self-intersection $S$ is given by $\angle(\widearc{L_2L_1})$, where $\pi \angle(\widearc{L_2L_1})$ is the phase angle of the positive-definite path $\widearc{L_2L_1}$ from $L_2$ to $L_1$ at any point in $S$ measured by $\Omega$\footnote{Note that $\angle(\widearc{L_2L_1})$ is well-defined since the phase angles measured by $\Omega$ and $g^* \Omega = \conste^{-\pi \consti \alpha_g} \Omega$ are the same for any $g$.}.
\end{prop}
\begin{proof}
Since $L$ is special, at each self-intersection point, the $\frac{1}{m}\Z$-gradings of the two branches are equal.  Then by Equation \ref{eq:deg^{1/m}}, $\deg^{1/m}(\one_S) = \frac{1}{m}(\theta_{L_2} - \theta_{L_1} + \angle^m(\widearc{L_2L_1})) = \frac{1}{m} \angle^m(\widearc{L_2L_1}) = \angle(\widearc{L_2L_1})$.
\end{proof}

Now let $\Fuk^{\frac{1}{m}\Z} (X)$ be the full subcategory of $\Fuk (X)$ whose objects are oriented and $\frac{1}{m} \Z$-graded.  Then the morphism spaces are $\frac{1}{m} \Z$-graded, namely a transverse intersection point between $L_1$ and $L_2$ (regarded as a morphism from $L_1$ to $L_2$) can be associated with a fraction in $\frac{1}{m} \Z$.  Moreover from the orientations we have a $\Z_2$ grading on the morphism spaces.

\begin{prop}
The $A_\infty$-multiplications $m_k$ has degree $2-k$ under the $\frac{1}{m}\Z$ grading.
\end{prop}
This follows from the  index formula (see \cite{Seidel-book}, \cite{Se}).
Namely,  one can measure the Maslov index using the quadratic volume form $\Omega^{\otimes 2}$
by combining the effects of grading along the Lagrangians $\Omega^{\otimes 2}(T_p \tilde{L_i}) = \conste^{2\pi \consti  \theta_{L_i}(p)/m}$ and the angles of the intersections $\angle(\widearc{L_{i}L_{i+1}})$ for each $i$.

Let us go back to the immersed Lagrangians $\bL = \bigcup_i L_i \subset X^0$ that we begin with.  Since it is assumed that the lifting $\bL^1 \subset \widetilde{X}$ is equipped with a grading $\theta$ with respect to $\Omega$, $\bL$ can also be equipped with a grading $m\theta$ with respect to $\Omega^{\otimes m}$.  Thus $\bL$ is $\frac{1}{m}\Z$-graded.

Recall that we consider formal deformations $b = \sum_{i} x_i X_i$ of $\bL$ where $X_i$ are odd-degree generators (under the canonical $\Z_2$ grading) and $x_i \in \Lambda Q$.  Using the $\frac{1}{m}\Z$-grading, each $X_i$ is associated with a fraction $\deg^{1/m} X_i$.

\begin{definition} \label{def:grading-A}
Define a grading on the path algebra $\Lambda Q$ generated by $x_i$ by
$\deg x_i := 1 - \deg^{1/m} X_i.$
\end{definition}

\begin{example}\label{ex:mainref2}
Let us consider the immersed Lagrangian $\mathbb{L}$ discussed in Chapter \ref{sec:333}, regarded as an immersion in the pair of pants (by removing three singular point of $\mathbb{P}^1_{3,3,3}$). It is easy to see that the resulting potential is given by $W=xyz$. The $1/3$-grading of the self-intersection point $X$ (see Figure \ref{fig:1o3gr}) can be computed as follows:
$$ \deg^{1/3} (X) =  \frac{1}{3} \left(\angle^3 (\widearc{L_2L_1})|_X \right) = \frac{1}{3}.$$
where $L_1$ and $L_2$ are local branches of $\mathbb{L}$ which intersect at $X$.
(Here, $\theta_L$ is a constant on $\mathbb{L}$, and hence the first two terms in \eqref{eq:deg^{1/m}} cancel out.  See also Proposition \ref{prop:deg^1/m}.)
\begin{figure}[h]
\begin{center}
\includegraphics[scale=0.4]{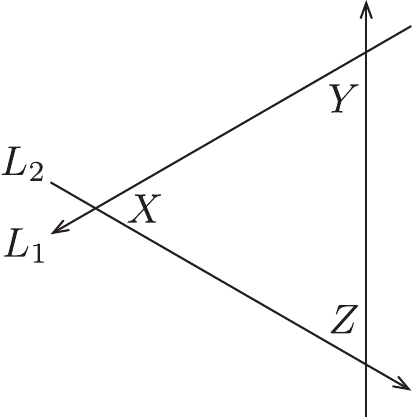}
\caption{Two branches meeting at the immersed generator $X$ in $CF(\mathbb{L},\mathbb{L})$ drawn in the $3:1$ cover.}\label{fig:1o3gr}
\end{center}
\end{figure}
Therefore the corresponding variable $x$ has $\deg x = 1- \deg^{1/3} (X) = \frac{2}{3}$. Likewise, we have $\deg y = \deg z =\frac{2}{3}$. Observe that $\deg (W) = \deg (xyz) =2$.
\end{example}

By definition $b = \sum_{i} x_i X_i$ always has degree one.  Note that $\deg x_i > 0$ if $\deg^{1/m} X_i < 1$.  In this case a homogeneous series in $x_i$ must have finitely many terms.

\begin{lemma} \label{lem:2k/m}
$\deg x_i = 2k_i / m$ for some $k_i \in \Z$.
\end{lemma}
\begin{proof}
By Lemma \ref{lem:parity}, $m \cdot \deg^{1/m}(X_i)$ is always even if $m$ is even.  When $m$ is odd it equals to the parity of $X_i$ (which is odd).  Thus $m \deg x_i = m - m \cdot \deg^{1/m}(X_i)$ is again even.
\end{proof}

Note that $m=1$ is the case that the Fukaya category is $\Z$-graded.  Then $X_i$ are taken to have odd degree, and so $x_i$ must have even degree such that $x_i X_i$ has degree $1$.

\begin{prop}
Under the above grading on $x_i \in \frac{1}{m} \Z$, the weakly unobstructed relations in $R$ are homogeneous.  Hence the grading descends from $\Lambda Q$ to $\cA = \Lambda Q / R$.  Moreover $W$ has degree $2$.
\end{prop}
\begin{proof}
Since $b = \sum_{i} x_i X_i$ has degree one, $m_0^b$ has degree 2.  A weakly unobstructed relation is a coefficient in $m_0^b$, and hence is homogeneous.  Moreover $\one_{L_i}$ has degree zero, and hence its coefficient $W_i$ in $m_0^b$ has degree 2, and so $W = \sum_i W_i$ has degree 2.
\end{proof}

\begin{cor}
If $\deg^{1/m} X_i < 1$ for all $i$, the mirror superpotential $W$ has only finitely many terms.
\end{cor}

Now we consider the mirror B-side.  

\begin{definition} \label{def:fracMF}
Given a $\frac{1}{m} \Z$-graded algebra $\cA$ with all homogeneous elements having degree $2k/m$ for some $k\in\Z$, and a central element $W \in \cA$ of degree $2$, the category $\MF^{\frac{1}{m}\Z} (\cA,W)$ consists of $(M,\delta)$, where $M$ is a $\frac{1}{m} \Z$-graded and $\Z_2$-graded projective $\cA$-module, and $\delta$ is a degree-one endomorphism (with respect to both the $\frac{1}{m} \Z$ and $\Z_2$ gradings) of $M$ satisfying $\delta^2 = W \cdot \mathrm{Id}$.  When $m$ is even, it is required that every homogeneous element of $M$ has degree $2k/m$ for some $k$.  When $m$ is odd, it is required that every homogeneous element of $M$ has degree $k/m$ where $k \mod 2$ is the parity of the element under the $\Z_2$ grading.  The morphism spaces are just the usual ones between $\cA$-modules.  (See Definition \ref{def:MF}.)
\end{definition}

\begin{prop} \label{prop:F^1/m}
$\cF$ defines a functor $\Fuk^{\frac{1}{m}\Z} (X) \to \MF^{\frac{1}{m}\Z} (\cA,W)$.
\end{prop}
\begin{proof}
We already have a $\Z_2$-graded functor from Section \ref{sec:FuktoMF}.  We just need to take care of the fractional grading.  Since $b$ has degree $1$, $m_k^{(\bL,b),L_1,\ldots,L_k}$ has degree $2-k$ for every $k$.  In particular the matrix factorization $\delta = m_1^{(\bL,b),L}$ corresponding to $L$ has degree $1$, and the functor on the morphism level given by $m_2^{(\bL,b),L_1,L_2}$ has degree zero and hence preserves grading.  Similarly the higher parts of the functor also have the correct degrees.  The last requirement follows from Lemma \ref{lem:parity}.
\end{proof}

\begin{example}\label{ex:mainref3}
We continue Example \ref{ex:mainref2} for the pair-of-pants.  Let's take the Lagrangian $L$ depicted by the red vertical line in Figure \ref{fig:1o3grmod}.  $L$ goes from one puncture to another and is a noncompact Lagrangian in the pair-of-pants.  

Let's fix the $\frac{1}{3}\Z$-grading on $\bL$ to be the constant phase function $3/2$ (and so $\Omega^{\otimes 3}(T_p\tilde{\bL})=\conste^{3\pi \consti /2}$ for all $p$).  Fix the $\frac{1}{3}\Z$-grading on $L$ to be the constant phase function $1/2$ (and so $\Omega^{\otimes 3}(T_pL)=\conste^{\pi \consti /2}$).

$L$ intersects $\mathbb{L}$ in two points $p$ and $q$ shown in the picture.  Then
$$\deg^{1/3} (p) = \frac{1}{3}(1/2-3/2) + \angle_p(\widearc{L\bL}) = 0, \quad \deg^{1/3} (q) = \frac{1}{3}(1/2-3/2) + \angle_q(\widearc{L\bL}) =\frac{1}{3}.$$
\begin{figure}[h]
\begin{center}
\includegraphics[scale=0.4]{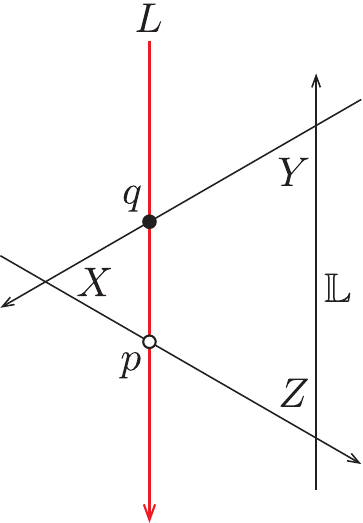}
\caption{The noncompact Lagrangian $L$ in the pair-of-pants drawn in the $3:1$ cover.}\label{fig:1o3grmod}
\end{center}
\end{figure}

Therefore the resulting matrix factorization is
$$x: M_0 \to M_1 , \qquad yz:  M_1  \to M_0 \qquad (M_0 =\C[x,y,z], \quad M_1 = \C[x,y,z][-1/3])$$
where $M_0$ is spanned by $p$ and $M_1$ is by $q$. Notice that this is precisely the one introduced in Example \ref{ex:mainfgs}.

We can also shift the phase function of $L$ by $2k$.  Then the corresponding matrix factorization has grading shifted by $2k/3$.  They correspond to the $\Z$-graded Lagrangians in the punctured elliptic curve which have the same image as that of $L$ in the pair-of-pants.  (Those shifted by $2k/3\in\Z$ correspond to $L$ in the punctured elliptic curve with different $\Z$-gradings.)
\end{example}

\begin{remark} \label{rem:grMF}
Suppose the algebra $\cA$ is connected (that is, there are no non-trivial idempotent elements).  Then the above definition of $\MF^{\frac{1}{m}\Z} (\cA,W)$ is an equivalent formulation of Orlov's graded matrix factorizations 
$$(P_0,P_1,p_0: P_0 \to P_1,p_1: P_1 \to P_0) $$
for $(\hat{\cA},\hat{W})$.  Namely, we multiply every fractional degree by $m/2$ and get an integer grading on $\cA$.  $W$ then has degree $m$.  

For a $\frac{1}{m} \Z$-graded and $\Z_2$ graded $\cA$-module $M$ as in Definition \ref{def:fracMF}, we set $P_0$ and $P_1$ to be the even and odd parts of $M$ with respect to the $\Z_2$ grading respectively; $\delta$ is sending $P_0$ to $P_1$ and $P_1$ back to $P_0$.  $P_0$ and $P_1$ are $\frac{1}{m} \Z$-graded, and we multiply all the degrees by $m/2$, so that they become $\frac{1}{2}$-graded, and $\delta$ becomes degree $m/2$.  We shift the grading on $P_1$ by $m/2$.  Then $p_0:=\delta: P_0 \to P_1$ has degree $m$, and $p_1:=\delta: P_1 \to P_0$ has degree $0$.  If $m$ is even, both $P_0$ and $P_1$ are $\Z$-graded (since elements in $M$ are required to have degree $2k/m$ for some $k \in \Z$ under the fractional grading).  If $m$ is odd, recall that elements in $M$ are required to have fractional degree $k/m$ where $k \mod 2$ equals to the parity in $\Z_2$ grading.  After multiplying by $m/2$, $P_0$ is $\Z$-graded; further shifting by $m/2$ (where $m$ is odd) makes $P_1$ also $\Z$-graded.  

Since $\delta^2 = W \cdot \mathrm{Id}$, we have a corresponding quasi-periodic sequence.  The above process is reversible, and hence it is equivalent to Orlov's graded matrix factorization.  (Here the setting is a bit more general: the quiver algebra $\cA$ is not connected when the quiver has more than one vertices.)
\end{remark}

\subsection{Integer grading on equivariant matrix factorizations} \label{sec:MFupstairs}
In this subsection we purely work on grading on the B-side.  Let $W$ be a central element of a quiver algebra with relations $\cA = \Lambda Q/R$, and suppose we have a formal dual group action of $\hat{G}$ on $\cA$ (denoted by $x \mapsto g_x$) such that $W$ is invariant (see Definition \ref{def:dual-group}).  Then we have the formal dual group quotient $(\hat{\cA},\hat{W})$ from Definition \ref{def:formal-quot}.

We fix the following additional grading information and consider $\Z$-graded matrix factorizations of $(\hat{\cA},\hat{W})$, which are $\frac{1}{m} \Z$-graded matrix factorizations (Definition \ref{def:fracMF}) for $m=1$.  Under our mirror construction, this information is provided by Definition \ref{def:grading-A} from the A-side.

\begin{definition} \label{def:grading-B}
A grading on the formal dual $G$-quotient $(\hat{\cA},\hat{W})$ is a choice of a character $G \to U(1)$ together with a factional grading on $\cA$, namely an assignment $\deg x \in \Q$ for each generator $x$ of $\cA$, such that the character maps $g_x$ to $\conste^{\pi \consti \deg x}$.  Moreover $W$ is required to be homogeneous of degree 2 under the grading.
\end{definition}

The image of a character $G \to U(1)$ equals to $\Z_m$ for some $m \geq 1$.  Thus we have the exact sequence 
\begin{equation} \label{eq:K}
0 \to K \to G \to \Z_m \to 0
\end{equation}
where $K$ is the kernel of $G \to U(1)$.  Then automatically $\deg x \in \frac{2}{m}\Z$ in the above definition.  (This matches with Definition \ref{def:fracMF}.)  Moreover the character maps each $g$ to $\conste^{\pi \consti \alpha_g}$ for a unique $\alpha_g \in [0,2)$.  In particular we have $\alpha_{g_x} \equiv \deg x \mod 2\Z$ for every $x$.  

We use the above information to give a $\Z$-grading on the category $\cA\#\hat{G}$ given in Definition \ref{def:A*G}.

\begin{prop}
An element in $\Hom_{\cA\#\hat{G}}(A_{(v,g)},A_{(u,h)})$ (as a series) has every term with $\deg \equiv \alpha_h - \alpha_g \mod 2\Z$.
\end{prop}
\begin{proof}
Recall from Definition \ref{def:A*G} that an element in $\Hom_{\cA\#\hat{G}} (A_{(v,g)},A_{(u,h)})$ has every term $x_p \ldots x_1$ with $g_{x_p} \ldots g_{x_1} = h g^{-1}$.  
Hence $e^{\pi \consti \left(\alpha_g - \alpha_h + \sum_{i=1}^p \alpha_{g_i}\right)} = 1$.  Thus
$$ \alpha_g - \alpha_h + \sum_{i=1}^p \alpha_{g_i} \equiv \alpha_g - \alpha_h + \deg (x_1 \ldots x_p) \equiv 0 \mod 2\Z.$$
\end{proof}

\begin{definition}
For a homogenous polynomial $f \in \mathrm{Hom}_{\cA\#\hat{G}}(A_{(v,g)},A_{(u,h)})$, its twisted degree is defined as
$$ \widetilde{\deg} f := \deg f + (\alpha_g - \alpha_h) \in 2\Z. $$
\end{definition}

From Proposition \ref{prop:MF}, $\MF(\hat{\cA},\hat{W})$ is identified with $\Tw(\cA\#\hat{G})$.  Now we define the $\Z$-graded enhancement $\Tw^\Z(\cA\#\hat{G})$ and denote it by $\MF_K^\Z(\cA,W)$, where $K$ is given in the exact sequence \eqref{eq:K}.  (The notation $\MF_K^\Z(\cA,W)$ emphasizes that the choice of $\Z$-grading depends on $K$.)  Similar to Proposition \ref{prop:MF}, $\Tw^\Z(\cA\#\hat{G})$ is equivalent to $\mathrm{MF}^\Z(\hat{A},\hat{W})$, where $\MF^\Z$ is given in Definition \ref{def:fracMF} with $m=1$.  

\begin{definition} \label{def:MF_graded}
The category $\Tw^\Z(\cA\#\hat{G})$ consists of objects of the form $\left(\bigoplus_{i \in I} A_{(v(i),g(i))} [\sigma_i],\delta\right)$ where $\sigma_i \in \Z$, $\delta$ is an endomorphism of $\bigoplus_{i \in I} A_{(v(i),g(i))} [\sigma_i]$ satisfying $\delta^2 = W \cdot \mathrm{Id}$.  The graded morphism space is
$$ \mathrm{Hom} \left( \bigoplus_{i \in I} A_{(v(i),g(i))} [\sigma_i], \bigoplus_{k \in K} A_{(u(k),h(k))} [\tau_k] \right) = \bigoplus_{i,k} \mathrm{Hom}_{\cA\# \hat{G}}(A_{(v(i),g(i))},A_{(u(k),h(k))}) [\tau_k - \sigma_i]$$
where on the right hand side, $\mathrm{Hom}_{\cA\# \hat{G}}(A_{(v(i),g(i))},A_{(u(k),h(k))})$ is a graded vector space by using the above twisted degree, and $[\tau_k - \sigma_i]$ means to shift this grading by $\tau_k - \sigma_i$.  The differential $\delta$ above is required to have degree one with respect to this grading.  
\end{definition}

\begin{remark}
Given a fractional grading on $\cA$, we have Orlov's graded matrix factorizations.  (See Definition \ref{def:fracMF} and Remark \ref{rem:grMF}.)  However $\MF_K^\Z(\cA,W)$ is more than this, namely we need the group $G$, and the $G$-grading on $\cA$: $x \mapsto g_x$.

Suppose $K=\{1\}$ (and in particular $G \cong \Z_m$ is cyclic).  Then $\MF_K^\Z(\cA,W)$ is an equivalent formulation of $\MF^{\frac{1}{m}\Z}(\cA,W)$ (which is an equivalent formulation of Orlov's graded matrix factorizations).  The integer degree of $x^{[h]} \in \hat{\cA}$ (where $[h] \in \Z_m$) is written in terms of the fractional degree of $x \in \cA$ as $\widetilde{\deg}\, x^{[h]} = \deg x + (\alpha_{2h/m} - \alpha_{2h/m + \deg x})$. Here, the subscripts on the right hand side are the elements of $\Z_m$ identified with $e^{\pi i (2h/m)}$ and $e^{\pi i (2h/m + \deg x)}$ in $\Z_m \subset U(1)$. The fractionally-graded projective $\cA$-module $P_v[-k/m]$ (for $k\in\Z$ is identified with the projective $\hat{\cA}$-module $\hat{P}_{v^{[k]}}[-\lfloor k/m \rfloor]$ where $[k] \in \Z_m$.  (Recall that $P_v$ and $\hat{P}_{v^{[k]}}$ denote $v \cdot \cA$ and $v^{[k]} \cdot \hat{A}$ respectively.)
\end{remark}

\begin{example}\label{ex:intdegups}
Let us consider the pair-of-pants in Example \ref{ex:mainref3} and its 3-fold cover. The variable $x$  downstairs gives rise to three different variables $x^{[0]},x^{[1]},x^{[2]}$ for $Q \# \Z_3$ ($[i] \in \Z_3$). More precisely, $x^{i}$ is taken to be the dual variable corresponding to the lift of $X$ in $\tilde{\mathbb{L}}^i \cap \tilde{\mathbb{L}}^{i+1}$ (see Figure \ref{fig:1o3lift}). From the above discussion, their integer degrees are given by
\begin{equation*}
\begin{array}{l}
\WT{\deg} x^{[0]} = \deg^{1/3} x + \left( \frac{2}{3} - \frac{4}{3} \right) =0\\
\WT{\deg} x^{[1]} = \deg^{1/3} x +  \left( 0 - \frac{2}{3} \right)= 0\\
\WT{\deg} x^{[2]} = \deg^{1/3} x +  \left( \frac{4}{3} - 0 \right)   = 2.
\end{array}
\end{equation*}
(One can check this directly using the standard volume $dz$ form upstairs.)
The integer gradings of $y^{[i]}$ and $z^{[i]}$ are given similarly.
\end{example}

\subsection{Graded functor upstairs}
Now we show that the functor upstairs respects the gradings on $\Fuk(\widetilde{X})$ and $\MF_K^\Z(\cA,W)$.  

Recall that for each $g \in G$, we have $g^* \Omega = \conste^{\pi \consti \alpha_g} \Omega$ for some $\alpha_g \in [0,2)$.
We have equipped the lifting $\bL^{g=1}$ with the grading $\theta$ with respect to $\Omega$, and have equipped $\bL$ with the grading $m\theta$ with respect to $\Omega^{\otimes m}$.
Since $g^* \Omega = \conste^{\pi\consti \alpha_g} \Omega$, another lifting $\bL^g$ can be equipped with the grading $\theta + \alpha_g$.  We stick with this setting throughout this section.

The following two lemmas relate the gradings upstairs and downstairs.

\begin{lemma}\label{lem:Lliftup1}
Suppose an $\Omega$-graded Lagrangian of $L \subset \widetilde{X}$ intersects $\bL^g$ at a point $S$ transversely.  Then
$$ \deg S = \deg^{1/m} S - \alpha_g $$
where $S$ is identified as a morphism from $\bL^g$ to $L$ in $\Fuk(\widetilde{X})$ on the left hand side, and as a morphism from $\bL$ to $L$ in $\Fuk(X)$ on the right hand side.
\end{lemma}
\begin{proof}
It follows from the equalities
$$ \deg^{1/m} S = \frac{1}{m} \left(m \theta_L(S) - m \theta(S) + m \angle(\widearc{LL^g}) \right) = \theta_L(S) - \theta(S) + \angle(\widearc{LL^g}) $$
and 
$$\deg S = \theta_L(S) - (\theta(S) + \alpha_g) + \angle(\widearc{LL^g})$$
where $\theta_L$ denotes the $\Omega$-grading of $L$, $\theta$ is the $\Omega$-grading of the lifting $\bL^{g=1}$, and $\angle(\widearc{LL^g})$ denotes the phase angle of the positive-definite path $\widearc{LL^g}$ measured by $\Omega$ at $S$.
\end{proof}

\begin{lemma} \label{lem:deg}
Suppose $S \subset \widetilde{X}$ is a clean intersection between $\bL^g$ and $\bL^h$ for $g \not= h \in G$.  Then 
$$ \deg (\one_S) = \deg^{1/m} (\one_S) + (\alpha_h-\alpha_g) $$
where $\one_S$ is regarded as a morphism in $\Hom(\bL^g,\bL^h)$ of $\,\Fuk(\widetilde{X})$ and as an endomorphism in $\Hom(\bL,\bL)$ of $\,\Fuk(X)$ on the left and right hand sides respectively.
\end{lemma}
\begin{proof}
Pick a point in $S$, which corresponds to $p \in \bL^g$ and $q \in \bL^h$.  Then
$$ \deg (\one_S) = (\theta(q) + \alpha_h) - (\theta(p) + \alpha_g) + \angle(\widearc{\bL^h\bL^g})|_{(p,q)} $$
and
$$ \deg^{1/m} (\one_S) = \frac{1}{m} (m \theta(q) - m \theta(p) + m \angle(\widearc{\bL^h\bL^g)})|_{(p,q)} =  \theta(q) - \theta(p) + \angle(\widearc{\bL^h\bL^g}|)_{(p,q)}.$$
Thus the equality follows.
\end{proof}

For the category of matrix factorization, recall from Definition \ref{def:grading-B} that the grading information consists of a character of $G$ and a degree $\deg x \in \rat$ for each generator $x$ of $\cA$.  
We take the character of $G$ to be the one defined by $g^* \Omega = \conste^{\pi \consti \alpha_g} \Omega$ for $\alpha_g \in [0,2)$, and we take the degree $\deg x$ to be the one defined by Definition \ref{def:grading-A}.  The following lemma verifies the compatibility between the character and the degree.

\begin{lemma}
$\alpha_{g_x} \equiv \deg x \mod 2\Z$ for each generator $x$.
\end{lemma}
\begin{proof}
$x$ corresponds to an immersed generator $X$ of $\bL$, and by definition $\deg x = 1 - \deg^{1/m} X$.  Lifting to upstairs, $X$ is lifted to a morphism from $\bL^1$ to $\bL^{g_x}$.  By Lemma \ref{lem:deg}, $\deg X = \deg^{1/m} X + \alpha_{g_x}$.  Moreover $\deg X$ is odd since we only take odd-degree deformations in our construction.  It follows that  $\deg x - \alpha_{g_x} = 1 - \deg X$ is even.
\end{proof}

Finally we show that the mirror functor upstairs is $\Z$-graded.

\begin{prop} \label{prop:F-graded}
$\cF$ defines a graded functor $\Fuk^{\Z} (\widetilde{X}) \to \MF_K^\Z(\cA,W)$ where $\Fuk^{\Z} (\widetilde{X})$ is the $\Z$-graded Fukaya category of $\widetilde{X}$.
\end{prop}
\begin{proof}
Let $U_1,\ldots,U_k \subset \widetilde{X}$ be $\Omega$-graded Lagrangians intersecting each other transversely and also intersecting $\tilde{\bL}$ transversely.  Let $\theta_i$ be the $\Omega$-grading on $U_i$.  The Lagrangians downstairs have corresponding fractional grading $m\theta_i$ with respect to $\Omega^{\otimes m}$.

Since downstairs $m_k^{(\bL,b),U_1,\ldots,U_k}$ has degree $2-k$, we have
$$ \deg^{1/m} m_k^{(\bL,b),U_1,\ldots,U_k} (\iota,a_1,\ldots,a_{k-1}) = 2-k + \deg^{1/m}(\iota) + \sum_{i=1}^{k-1} \deg^{1/m}(a_i). $$
Now
\begin{align*}
m_k^{(\bL,b),U_1,\ldots,U_k} (\iota,a_1,\ldots,a_{k-1}) &= \sum_{p=0}^\infty m_{p+k}^{\bL,\ldots,\bL,U_1,\ldots,U_k} (b,\ldots,b,\iota,a_1,\ldots,a_{k-1}) \\
&= \sum_{p=0}^\infty \sum_{|I|=p} x_{i_p}\ldots x_{i_1} m_{p+k}^{\bL,\ldots,\bL,U_1,\ldots,L_k} (X_{i_1},\ldots,X_{i_p},\iota,a_1,\ldots,a_{k-1}).
\end{align*}
This implies
$$\deg(x_{i_p}\ldots x_{i_1}) + 2-k-p + \sum_{l=1}^p \deg^{1/m} X_{i_l} = 2- k$$
for consecutive morphisms $X_{i_l} \in \Hom(L^{g_{i_{l-1}}},L^{g_{i_l}})$, $l=1,\ldots,p$, contributing to $m_k^{(\bL,b),U_1,\ldots,U_k}$.  

Since $\deg^{1/m} X_{i_l} = \deg X_{i_l} - (\alpha^{g_{i_l}} - \alpha^{g_{i_{l-1}}})$, we have
$$\widetilde{\deg}(x_{i_p} \ldots x_{i_1}) := \deg (x_{i_p} \ldots x_{i_1}) + (\alpha^{g_{i_0}} - \alpha^{g_{i_p}}) = \sum_{l=1}^p (1-\deg X_{i_l}) \in 2\Z$$
where the last expression is even because each $\deg X_{i_l}$ is odd (we always choose odd-degree deformations).
Then for the degree upstairs,
\begin{align*}
&\deg(x_{i_p}\ldots x_{i_1}) m_{p+k}^{\bL,\ldots,\bL,U_1,\ldots,U_k}(X_{i_1},\ldots,X_{i_p},\iota,a_1,\ldots,a_{k-1}) \\
=& \widetilde{\deg}(x_{i_p}\ldots x_{i_1}) + 2-(p+k) + \sum_{l=1}^p \deg X_{i_l} + \deg \iota + \sum_{j=1}^{k-1} \deg a_j \\
=& \deg f + 2-p-k + \sum_{l=1}^p \deg^{1/m} X_{i_l} + \deg \iota + \sum_{j=1}^{k-1} \deg a_j \\
=& 2-k + \deg \iota + \sum_{j=1}^{k-1} \deg a_j.
\end{align*}
Thus $m_k^{(\tilde{\bL},b),U_1,\ldots,U_k}$ also has degree $2-k$ upstairs since we use the twisted degree $\widetilde{\deg}$.

For an $\Omega$-graded Lagrangian $U \subset \widetilde{X}$ which intersects $\tilde{\bL}$ transversely, the functor $\cF$ maps it to
$\cF(L) = \bigoplus_{g \in G} \bigoplus_{p \in \bL^g \cap U} A_g [-\deg p]$ together with $m_1^{(\tilde{\bL},b),U}$,
where $p \in \bL^g \cap U$ is identified as a morphism from $\bL^g$ to $U$.  We need to show that $m_1^{(\tilde{\bL},b),U}$ belongs to $\Hom(\cF(U),\cF(U))$ in the sense defined in \ref{sec:MFupstairs}, and it has degree one.  From the above calculation, the component $f$ of $m_1^{(\tilde{\bL},b),U}$ which maps $\bigoplus_{p \in \bL^g \cap U} A_g [-\deg p]$ to $\bigoplus_{q \in \bL^h \cap U} A_h [-\deg q]$ has
$ \deg f + (\alpha^{h} - \alpha^g) \in 2\Z $
and hence $m_1^{(\tilde{\bL},b),U}$ belongs to $\Hom(\cF(U),\cF(U))$.  Also the above shows that $m_1^{(\tilde{\bL},b),U}$ has degree one.
  
Similarly, the functor on morphism level $\Hom(U_1,U_2) \to \Hom(\cF(U_1),\cF(U_2))$ is defined by $m_2^{(\tilde{\bL},b),U_1,U_2}$ and hence has degree $0$, and the higher part $\cF_k: \Hom(U_1,U_2) \otimes \ldots \otimes \Hom(U_{k-1},U_k) \to \Hom(\cF(U_1),\cF(U_k))$ is defined by $m_{k}^{(\tilde{\bL},b),U_1,\ldots,U_k}$ and hence has degree $2-k$.
\end{proof}

\begin{example}\label{ex:intdegups1}
We take the lift $\tilde{L}$ of $L$ in Example \ref{ex:mainref3} as shown in Figure \ref{fig:1o3lift}. By Lemma \ref{lem:Lliftup1}, the liftings $\tilde{p},\tilde{q}$ of $p,q \in \mathbb{L} \cap L$ have the following integer gradings
$$\deg \tilde{p} = \deg^{1/3} (p) + 0 = 0  ,\quad \deg \tilde{q} = \deg^{1/3} (q) + \frac{2}{3} =  1.$$
Here, we choose a grading of $\tilde{L}$ to be $1/6$ so that it is compatible with  $\frac{1}{3} \Z$-grading downstairs (which was $3 \times 1/6 = 1/2$).
\begin{figure}[h]
\begin{center}
\includegraphics[scale=0.45]{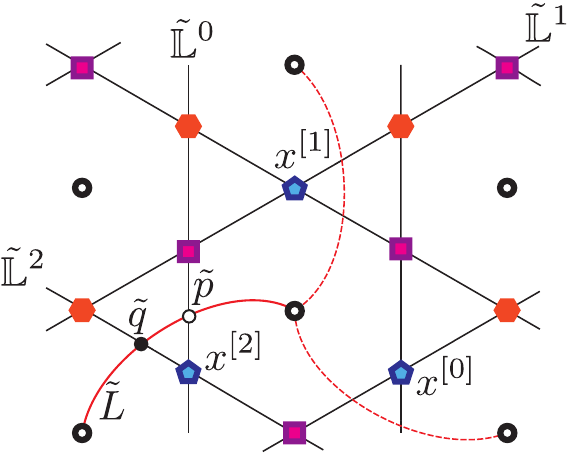}
\caption{3-fold cover of the pair-of-pants}\label{fig:1o3lift}
\end{center}
\end{figure}

Therefore the resulting matrix factorization is 
$$ y^{[0]} z^{[1]} \cdot x^{[2]}$$
for $[i] \in \Z_3$. Note that this indeed acts on $\tilde{p}$ as $\hat{W} \cdot \mathrm{Id}$ (where $W= \sum_i x^{[i]} y^{[i+1]} z^{[i+2]}$), since $\tilde{p} \in CF(\tilde{\mathbb{L}}, \tilde{L})$ is supported at $v^{[2]}$ of the quiver. Likewise, the $\Z_3$-family $\cup_{[i] \in \Z_3} [i] \cdot \tilde{L}$ transforms into the matrix factorization given at the end of Example \ref{ex:mainfgs}
\end{example}

In summary, given a K\"ahler manifold $\tilde{X}$ with a finite group symmetry $G$, we can first construct the mirror of the quotient $X = \tilde{X}/G$ by consider an immersed Lagrangian $\bL$ in $X$.  The geometrically constructed mirror $(\cA,W)$ automatically admits a (formal) dual-$G$ action, and the mirror of $\tilde{X}$ is given as the formal dual-$G$ quotient $(\hat{\cA},\hat{W})$.  If $\tilde{X}$ has a holomorphic volume form and $G$ preserves the holomorphic volume form up to rescaling by constants, then we have a canonical way to enhance the mirror $(\hat{\cA},\hat{W})$ and the mirror functor $\cF$ to be $\Z$-graded.

%
%
%


\chapter{4-punctured spheres and the pillowcase}\label{sec:4puncture2222}
In this chapter we apply our theory to construct the mirror of the elliptic orbifold $E / \Z_2$, where $\Z_2$ acts on $\C$ by $z \mapsto -z$ and the action descends to an elliptic curve $E$.  Topologically the quotient is a sphere, and there are four orbifold points which are locally isomorphic to $\C / \Z_2$.  Thus the orbifold is denoted as $\bP^1_{2,2,2,2}$, which is also known as the pillowcase.

We begin with a set of Lagrangian immersions $\bL_0 := \{ L_{1}, L_{2}\}$ in $\bP^1_{2,2,2,2}$ as shown in Figure \ref{fig:2222basic} and apply the construction in Chapter \ref{sec:MFseverallag}.  The following shows the main result.

\begin{theorem} \label{thm:P2222}
The generalized mirror of $E/\Z_2 = \bP^1_{2,2,2,2}$ corresponding to $\bL_0$ is given by $\left(\cA_0, W_0 \right)$ where
\begin{enumerate}
\item $Q$ is the directed graph with two vertices $v_1, v_2$, two arrows $\{y,w\}$ from $v_1$ to $v_2$ and two arrows $\{x,z\}$ from $v_2$ to $v_1$.  
\item $\cA_0 = \cA(Q, \Phi_0)$ is the noncommutative resolution of the conifold $\frac{ \Lambda Q }{ \left( \partial \Phi_0 \right)}$, 
where $\Phi_0: = xyzw -wzyx$.
\item $$W_0=\phi(q_\orb^{\frac{1}{4}}) ((xy)^2 + (xw)^2 + (zy)^2 + (zw)^2 + (yx)^2 + (wx)^2 + (yz)^2 + (wz)^2) + \psi(q_\orb^{\frac{1}{4}}) (xyzw + wzyx)$$
lies in the center of $\cA_0$, where $\phi$ and $\psi$ are given in Theorem \ref{thm:mir-2222} and $q_\orb = \exp \left(- \int_{\bP^1_{2,2,2,2}} \omega \right)$ is the K\"ahler parameter.
\item $\psi(q_\orb^{\frac{1}{4}})\left(\phi(q_\orb^{\frac{1}{4}})\right)^{-1}$ equals to the mirror map of $\bP^1_{2,2,2,2}$.
\item
We have  $\Z_2$-graded, and $\Z$-graded $\AI$-functors
$$ \mathcal{F}^{\mathbb{L}_0}: \Fuk(\bP^1_{2,2,2,2}) \to \MF(\mathcal{A}_0, W_0), \;\;\;  \mathcal{F}^{\mathbb{L}_0}: \Fuk^\Z(E) \to \MF^{\Z}(\cA_0, W_0).$$
\end{enumerate}
\end{theorem}

In the later part of this chapter, we deform $\bL_0$ to $\bL_t$, and the generalized mirror is then deformed to $(\cA_t,W_t)$.  In this way we obtain an interesting class of noncommutative algebras with central elements.

\begin{theorem}\label{thm:deform2222coni}
There is a $T^2$-family $(\mathbb{L}_t,\lambda)$ of Lagrangians decorated by flat $U(1)$-connections for $(t-1) \in \R/2\Z$ and $\lambda \in U(1)$ such that
\begin{enumerate}
\item the corresponding mirror noncommutative algebras $\cA_{(\lambda,t)}$ takes the form
$$\cA_{(\lambda,t)} :=  \frac{ \Lambda Q }{ \left( \partial \Phi_{(\lambda,t)} \right)}, 
 \Phi = a(\lambda,t) \, xyzw  +  b(\lambda,t)\, wzyx 
+ \frac{1}{2} c(\lambda,t)\, ((wx)^2 + (yz)^2) + \frac{1}{2} d(\lambda,t)\, ( (xy)^2   + (zw)^2)$$
($\cA_{(1,0)}$ is the same as $\mathcal{A}_0$ in Theorem \ref{thm:P2222} (2))
\item Coefficients $(a:b:c:d)$ defines an embedding $T^2 \to \mathbb{P}^3$ onto the complete intersection of two quadrics given as follows. For $[x_1,x_2,x_3,x_4] \in \bP^3$ and $\sigma=\frac{\psi}{\phi}$,
\begin{eqnarray*}
x_1 x_3 = x_2 x_4 \\ 
 x_1^2 + x_2^2 +x_3^2 +x_4^2 + \sigma x_1 x_3 = 0.
\end{eqnarray*}
This is isomorphic to the mirror elliptic curve $\check{E}$ given by Hesse cubic in Theorem \ref{thm:mainthm333}.
\item The family of noncommutative algebra $\mathcal{A}_{\lambda, t}/(W_{\lambda,t})$ near $t=0, \lambda=1$ gives a quantization of the complete intersection given by above two quadratic equations in $\C^4$ in the sense of \cite{EG}.
\end{enumerate}
\end{theorem}


\section{Mirror of the four-punctured sphere}
More general case of punctured Riemann surfaces will be discussed in Chapter \ref{sec:Rs}.  We make it into an independent chapter to give a nice illustration of our construction and to compactify it as the orbifold $\bP^1_{2,2,2,2}$.

Let $p_1, p_2, p_3, p_4$ be four distinct points in $\bP^1$.  We take $p_1, p_2,p_3 , p_4$ to lie in the real locus $\R\bP^1 \subset \bP^1$.  This fixes the complex structure of $X = \bP^1-\{p_1,p_2,p_3,p_4\}$.  The K\"ahler structure on $X$ is given by the restriction of the Fubini-Study metric $\omega$ on $\bP^1$ (which can be rescaled).  We have the anti-symplectic involution $\iota$ defined by $[z_1:z_2] \mapsto [\overline{z_1}:\overline{z_2}]$ on $X$.

To carry out our construction, we fix the Lagrangian immersion $\bL_0$ as depicted in Figure \ref{fig:2222basic}, which consists of two circles $L_1$ and $L_2$.  Note that $L_1 \cup L_2$ is invariant under the anti-symplectic involution $\iota$.  There is a quadrilateral contained in the upper hemisphere $\{[z_1:z_2] \in \bP^1: \mathrm{Im}(z_1/z_2) > 0\} \subset X$ bounded by $L_1 \cup L_2$, and let $q_d = \exp (-A)$ where $A$ denotes its area.  There is also a quadrilateral contained in the lower hemisphere $\{[z_1:z_2] \in \bP^1: \mathrm{Im}(z_1/z_2) < 0\} \subset X$ bounded by $L_1 \cup L_2$ which has the same area due to the anti-symplectic involution.


Equip $L_1$ and $L_2$ with the orientations indicated in Figure \ref{fig:2222basic} and the trivial spin structure.  There are four immersed points which are denoted as $X,Y,Z,W$ using counterclockwise convention.  We denote the corresponding odd-degree immersed generators also by $X,Y,Z,W$, and the corresponding even-degree immersed generators by $\bar{X},\bar{Y},\bar{Z},\bar{W}$.  $X$ and $Z$ are morphisms from $L_2$ to $L_1$, while $Y$ and $W$ are morphisms from $L_1$ to $L_2$.  Thus
\begin{lemma}[(1) of Theorem \ref{thm:P2222}]
The corresponding endomorphism quiver $Q=Q^\bL$ in Definition \ref{def:mirrorquiverseveral} is given as in 
(a) of Figure \ref{fig:4puncturemq}.
\end{lemma}

\begin{figure}[h]
\begin{center}
\includegraphics[scale=0.6]{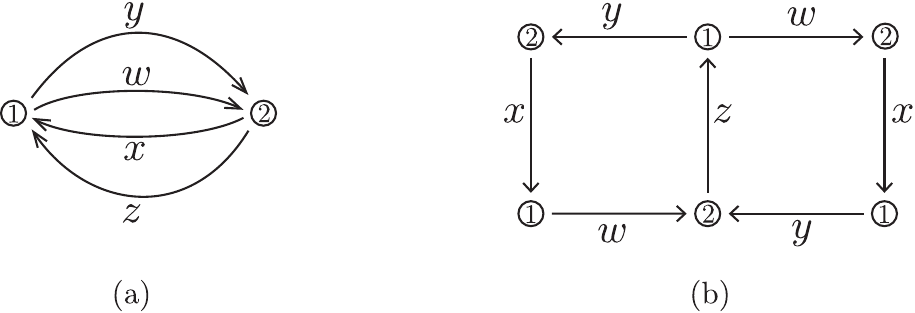}
\caption{Dimer structure on the endomorphis quiver $Q$}\label{fig:4puncturemq}
\end{center}
\end{figure}

Let $\widehat{\Lambda Q}$ be the completed path algebra of $Q$ which is generated by $x,y,z,w$, and consider the infinitesimal deformation $b = xX + yY + zZ + wW$.  The $A_\infty$-algebra $(\CF^*(L_1 \cup L_2),(m_k)_{k=0}^\infty)$ is deformed by $b$ to $(\CF^*(L_1 \cup L_2),(m^b_k)_{k=0}^\infty)$.  A basis of $\CF^*(L_1 \cup L_2)$ is given by
$$\{X,Y,Z,W,\bar{X},\bar{Y},\bar{Z},\bar{W},\One_{L_1},\One_{L_2},T_{L_1},T_{L_2}\},$$
where $\One_{L_i}$ and $T_{L_i}$ are the fundamental class and point class of $L_i$ respectively (for $i=1,2$).  They can be represented by the maximum and minimum points of a Morse function, and here we take the Morse function to be the standard height function on the circle $L_i$.  $\One_{L_i}$ for $i=1,2$ serve as units of the Floer complex of $L_i$.  $T_{L_i}$ for $i=1,2$ are marked as crosses in Figure \ref{fig:2222basic}.

The only holomorphic polygons contributing to $m_0^b$ are the two quadrilaterals contained in the upper and lower hemispheres mentioned above.  Hence we have the following.
\begin{prop}
For $X = \bP^1-\{p_1,p_2,p_3,p_4\}$ and an odd-degree immersed deformation $b = xX + yY + zZ + wW$ of the immersed Lagrangian $L_1 \cup L_2$, we have
\begin{equation}\label{eqn:m0b2222sym}
m_0^b = q_d \left((xyz - zyx) \bar{W} + (yzw - wzy) \bar{X} + (zwx - xwz) \bar{Y} + (wxy - yxw) \bar{Z} +  (xyzw \One_{L_1} + yxwz \One_{L_2})\right)
\end{equation}
where $\One_{L_i}$ denotes the unit corresponding to the Lagrangian $L_i$ for $i=1,2$ respectively.
\end{prop}

According to our construction scheme given in Chapter \ref{sec:MFseverallag}, we quotient the path algebra $\widehat{\Lambda Q}$ by the weakly unobstructed relations and define $W$ to be the sum of the coefficients of $\One_{L_i}$ in $m_0^b$.  The weakly unobstructed relations are given by the coefficients of even-degree generators other than $\One_{L_i}$ in $m_0^b$.  Note that the weakly unobstructed relations can be obtained as cyclic derivatives of the spacetime superpotential $\Phi = xyzw - wzyx$.  
Thus we obtain the mirror given below.

\begin{theorem} \label{thm:mir-S-4}
The generalized mirror of $\bP^1-\{p_1,p_2,p_3,p_4\}$ is the noncommutative Landau-Ginzburg model $(\cA_0,W^{\textrm{punc}})$, where $\cA_0$ is the algebra stated in Theorem \ref{thm:P2222}, and $W^{\textrm{punc}}$ is the central element of $\cA_0$ defined by
$$ W^{\textrm{punc}} = xyzw + wzyx. $$
\end{theorem}

From the relations we have $xyzw = yzwx$, $wzyx = xwzy$ and so on.
The quiver $Q$ in this case can be embedded into a torus in the way depicted by (b) of  Figure \ref{fig:4puncturemq}.
It gives a dimer which will be discussed in Chapter \ref{sec:Rs} in detail.

The algebra $\cA_0$ is a noncommutative crepant resolution of the conifold $Y$ \cite[Section 5.1]{Aspinwall-Katz}.  Namely, the derived category of left $\cA_0$-modules is equivalent to the derived category of coherent sheaves over the commutative crepant resolution of the conifold.  Below we provide a brief review and establish explicit relations between the coordinates of $\cA_0$ and $Y$.

The conifold is the affine variety
$$Y = \{(Z_1,Z_2,Z_3,Z_4) \in \C^4: Z_1 Z_3 = Z_2 Z_4\} $$
which has an isolated singularity at the origin.  It admits a crepant resolution by the total space of the vector bundle $\hat{Y} = \CO_{\bP^1}(-1) \oplus \CO_{\bP^1} (-1)$, which can be written as the GIT quotient $(\C^4-V) / \C^\times$, where $(x,y,z,w)$ are the standard coordinates of $\C^4$, $V = \{x = z = 0\}$, and $\C^\times$ acts by $\lambda \cdot (x,y,z,w) = (\lambda x, \lambda^{-1} y, \lambda z, \lambda^{-1} w)$. \footnote{Here we use the same notations for the coordinates of $\C^4$ and the generators of the quiver $Q$, since there is a natural correspondence between them, by identifying them as sections of $\CO(1)$ and $\CO(-1)$ over $\tilde{Y}$.}  The resolution map $\hat{Y} \to Y$ is given by $Z_1=xy,Z_2=yz,Z_3=zw,Z_4=wx$.

The variables $x,z$ have degree $1$ and can be regarded as generating sections of the line bundle $\CO(1)$ over $\hat{Y}$.  Similarly $y,w$ have degree $-1$ and can be regarded as generating sections of the line bundle $\CO(-1)$ over $\hat{Y}$.  The trivial line bundle $\CO$ has generating sections $Z_1=xy,Z_2=yz,Z_3=zw,Z_4=wx$.

$\CT = \{\CO, \CO(1)\}$ forms a full strongly exceptional collection of the derived category of coherent sheaves over $\hat{Y}$.  The morphism spaces of $\CT$ are represented by the quiver $Q$.  Moreover, since multiplications of the sections $x,y,z,w$ are commutative, it follows that the relations between morphisms are exactly given by the cyclic derivatives of $\Phi = xyzw - wzyx$.

By the result of Bondal \cite{Bondal}, the derived category of left $\cA_0$-modules is equivalent to the derived category $D^b(\hat{Y})$.  Given an $\cA_0$-module $M$, we obtain the complex $\CT \otimes_{\cA_0} M$ which is the corresponding object in $D^b(\hat{Y})$.  The reverse functor is given by $Hom(\CT,-)$.

The superpotential $W \in \cA_0$ in Theorem \ref{thm:mir-S-4} can be lifted to a function over $\hat{Y}$, which is given as $W = 2 Z_1Z_3 = 2 Z_2Z_4$.  

In conclusion, the constructed mirror of the punctured sphere $X$ is equivalent to $(\hat{Y}, Z_1Z_3)$, which agrees with the one used in \cite{AAEKO}.

\section{Mirror of the pillowcase}\label{subsec:centralfiber2222}
Now we construct the mirror of the elliptic orbifold $\bP^1_{2,2,2,2}$, which is obtained by compactifying the punctured sphere $\bP^1 - \{p_1,p_2,p_3,p_4\}$ by four $\Z_2$-orbifold points.  We also denote these four orbifold points by $p_1,p_2,p_3,p_4$.  Again we have the anti-symplectic involution $\iota$ on $\bP^1_{2,2,2,2}$.

The elliptic orbifold $\bP^1_{2,2,2,2}$ is the quotient of an elliptic curve $E = \C / \Lambda$ by $\Z_2$, where $\Lambda \cong \Z^2$ is a lattice in $\C$ and $[1] \in \Z_2$ acts by $[z] \mapsto -[z]$ on $E$.  Since we focus on the K\"ahler structures, let's fix the complex structure of $E$, say taking $\Lambda = \Z\langle 1, i \rangle \subset \C$.  The universal cover of $E$ and the lifts of the Lagrangians $L_1,L_2$ (given in the previous section) are depicted by Figure \ref{fig:2222basic}.

\begin{figure}
\begin{center}
\includegraphics[height=2.7in]{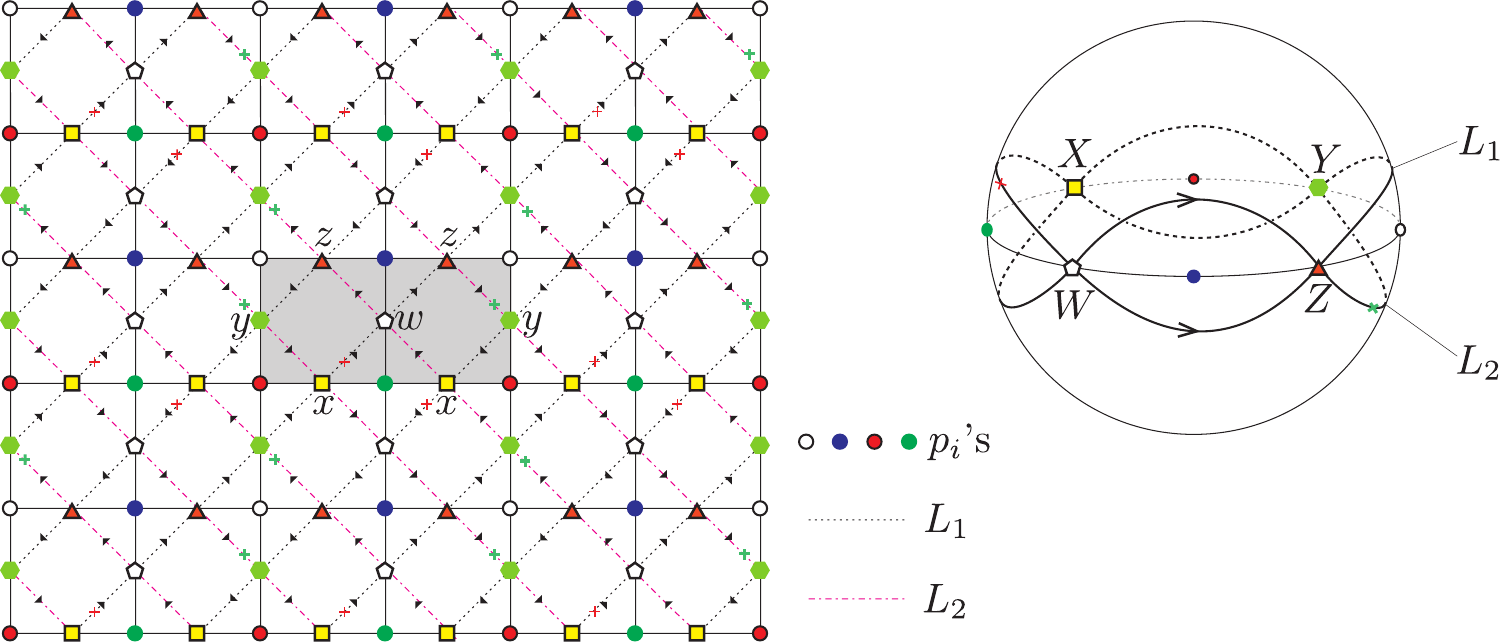}
\caption{The universal cover of $\bP^1_{2,2,2,2}$.  The bold lines show the lifts of the equator of $\bP^1_{2,2,2,2}$, and the dotted lines show the lifts of the Lagrangians $L_1$ and $L_2$.}\label{fig:2222basic}
\end{center}
\end{figure}

We use the formal deformation parameter $b = xX+yY+zZ+wW$ for $\bL_0 = L_1 \cup L_2$ as in the previous section.  In particular we have the same quiver $Q$ and its path algebra $\widehat{\Lambda Q}$.  On the other hand, the obstruction term $m_0^b$ is different since there are more holomorphic polygons in $\bP^1_{2,2,2,2}$.  In general the weakly unobstructed relations could become much more complicated.  It turns out that in this particular case, the weakly unobstructed relations remain to be the same.  The main reason is that the immersed Lagrangian $\bL_0$ is taken to be invariant under the anti-symplectic involution $\iota$. (In \ref{subsec:2222pointmod}, we will study the case where the symmetry breaks.)

\begin{lemma}
The weakly unobstructed relations for $\bL_0 \subset \bP^1_{2,2,2,2}$ are the cyclic derivatives of $\Phi = xyzw - wzyx$.
\end{lemma}
\begin{proof}
The weakly unobstructed relations are given by the coefficients of $\bar{W}, \bar{X}, \bar{Y}, \bar{Z}$ in $m_0^b$.  We just consider $\bar{W}$ and the others are similar.

The polygons contributing to $m_0^b$ appear in pairs due to the anti-symplectic involution $\iota$, and the polygons in each pair have the same area and opposite signs.  Moreover the corresponding monomials for each pair are related to each other by reversing the order (since the anti-symplectic involution flips the orientation of a polygon).
As a result, the coefficient of $\bar{W}$ takes the form $f(q_d) (xyz - zyx)$ where $f(q_d) = q_d + \ldots$ is a non-zero series in $q_d$.  Thus it gives the weakly unobstructed relation $xyz \sim zyx$.

Hence we conclude that the weakly unobstructed relations remain to be the same as those for the four-punctured sphere, namely $xyz \sim zyx, yzw \sim wzy, zwx \sim xwz, wxy \sim yxw$.
\end{proof}

Then we directly compute the worldsheet superpotential $W=W_1+W_2$, where $W_i$ are the coefficients of $\one_{L_i}$ for $i=1,2$  They are computed by counting holomorphic polygons bounded by $\bL$ with corners at the immersed generators $X,Y,Z,W$, which are quadrilaterals in this case, see Figure \ref{fig:2222basic}.  More precisely, we count the number of $T_i=PD(\one_{L_i})$ on the boundary of such polygons. We omit the detailed computations here and only state the result as below.

\begin{theorem}[(2) and (3) of Theorem \ref{thm:P2222}] \label{thm:mir-2222}
The generalized mirror of $\bP^1_{2,2,2,2}$ is $(\cA_0, W_0)$, where
\begin{equation} \label{eq:W2222}
W_0 =\phi(q_d) ((xy)^2 + (xw)^2 + (zy)^2 + (zw)^2 + (yx)^2 + (wx)^2 + (yz)^2 + (wz)^2) + \psi(q_d) (xyzw + wzyx)
\end{equation}
which is a central element of $\cA_0$, and
\begin{align*}
\phi(q_{d})&=
\sum_{k,l\geq 0}^{\infty}(k+l+1)q_{d}^{(4k+1)(4l+3)},\\
\psi(q_{d})&=\sum_{k,l\geq 0}^{\infty}(4k+1)q_{d}^{(4k+1)(4l+1)}+\sum_{k,l\geq 0}^{\infty}(4k+3)q_{d}^{(4k+3)(4l+3)}.
\end{align*}
\end{theorem}

As explained in the previous section, $\cA_0$ is a noncommutative resolution of the conifold.  $W_0$ corresponds to a unique element in the conifold algebra
$$2\left(\phi(q_d) (Z_1^2 + Z_2^2 + Z_3^2 + Z_4^2) + \psi(q_d) Z_1Z_3\right) $$
which lifts to a holomorphic function on the commutative resolution $\CO_{\bP^1}(-1) \oplus \CO_{\bP^1}(-1)$.
By a simple change of coordinates on $(Z_1,Z_2,Z_3,Z_4)$, we can change the expression to
$$ W^{\open} = Z_1^2 + Z_2^2 + Z_3^2 + Z_4^2 + \frac{\psi(q_d)}{\phi(q_d)} Z_1Z_3. $$
In this way the generalized mirror should be equivalent to $(\CO_{\bP^1}(-1) \oplus \CO_{\bP^1}(-1), W^\open)$.  We will leave HMS for future work.

The quotient construction by $\Z_2$ given in Chapter \ref{sec:fingpsymm} can be applied to obtain the generalized mirror $(\hat{A}_0,\hat{W}_0)$ of the covering elliptic curve $E$.  Recall from Definition \ref{def:formal-quot} that $\hat{\cA}_0 = \Lambda (Q \# \Z_2) \big/ \langle P^g_i: g \in G, i = 1,\ldots,K \rangle$, where $P^g_i$ are the lifts of the weakly unobstructed relations $P_i$.  $Q \# \Z_2$ is the quiver with vertices $(i,[j]) \in \{0,1\} \times \Z_2$, and with arrows $x^{[i]}, y^{[i]},z^{[i]},w^{[i]}$  (where the upper-index $[i] \in \Z_2$ indicates the second component of the head vertex).  $\hat{W}_0$ is the lift of $W_0$ downstairs according to Definition \ref{def:lift-path}.)  We conclude by the following.
\begin{prop}
The generalized mirror of $E$ corresponding to $\tilde{\bL}_0$ is $(\hat{\cA}_0,\hat{W}_0)$.
\end{prop}

\section{Open mirror theorem}\label{subsec:openmirthm2222}
We compare the above worldsheet superpotential $W_0$ with the family
$$ W^{\mir} := Z_1^2 + Z_2^2 + Z_3^2 + Z_4^2 + \sigma Z_1Z_3$$
parametrized by $\sigma$.  By direct computation, the critical locus of $W^{\mir}$ is the zero section $\bP^1 \subset \CO_{\bP^1}(-1) \oplus \CO_{\bP^1}(-1)$ instead of a point.  

We wants to compare our generalized mirror with the mirror map determined by Frobenius structure.  Unfortunately the Frobenius structure on the universal deformation space of $W^{\mir}$ is not known since Saito's theory is not yet known for non-isolated singularities.  

On the other hand, $W^\mir$ descends to the quotient of $\CO_{\bP^1}(-1) \oplus \CO_{\bP^1}(-1)$ by $\Z_2$.  After blowing up along $\bP^1$, it is the total space of canonical line bundle $K_{\bP^1 \times \bP^1}$.  The critical locus of $W^\mir$ in $K_{\bP^1 \times \bP^1}$ is the elliptic curve $\{W^\mir=0\} \subset \bP^1 \times \bP^1$, where $(x:z,y:w)$ are the standard homogeneous coordinates of $\bP^1 \times \bP^1$.  It can also be embedded into $\bP^3$ via Segre embedding
\begin{equation}\label{eqn:segrecoord}
x_1 = xy, x_2 = xw, x_3 = zw, x_4 = zy.
\end{equation}
Then we obtain the following family of elliptic curves, which are complete intersections
$$\{x_1 x_3 = x_2 x_4\} \cap \{x_1^2 + x_2^2 +x_3^2 +x_4^2 + \sigma x_1 x_3 = 0\} \subset \bP^3.$$

The flat coordinate $q$ of moduli of elliptic curves is classically well-known, and can be found explicitly by embedding the above family into $\bP^2$.  This gives a way to obtain the flat coordinate $q_\orb$ of $W^{\mir}$, which is related to $q$ by $q = q_\orb^2$.

Explicitly the change of coordinates between $\sigma$ and $q$ is given by
$$ \sigma(q) = \frac{\eta^{12}(q)}{\eta^{8}(q^{2})\eta^{4}(q^{\frac{1}{2}})}$$
where $\eta(q) = q^{1/24}\prod_{n=1}^\infty (1-q^n)$ is the Dedekind $\eta$ function.

Now we can compare $W^\open$ with $W^\mir$.  The following can be verified by direct calculations.
\begin{theorem}[(4) of Theorem \ref{thm:P2222}] \label{thm:open=mir}
$$W^\open (q_d) = W^\mir(\sigma(q))$$
where $q_\orb = q_d^4$, and $q = q_\orb^2 = q_d^8$.
\end{theorem}

As a consequence, the elliptic curve family defined by the complete intersection 
$$\left\{W^{\open} = Z_1^2 + Z_2^2 + Z_3^2 + Z_4^2 + \frac{\psi(q^{\frac{1}{8}})}{\phi(q^{\frac{1}{8}})} Z_1Z_3 = 0 \right\} \cap \{Z_1 Z_3 = Z_2 Z_4\} \subset \bP^3 $$
equals to the Hesse elliptic curve family
$$ \{ x^3 + y^3 + z^3 + \left(-3 - \left( \frac{\eta(q^{\frac{1}{3}})}{\eta(q^3)} \right)^3\right) xyz \} \subset \bP^2 $$
given in Theorem \ref{thm33313}.  This is a natural expectation from mirror symmetry, since both are the mirror elliptic curve family.

More explicitly, we can state the above in terms of the $j$-invariant.  The $j$-invariant of the elliptic curve family
$$\{((xy)^2 + (xw)^2 + (zy)^2 + (zw)^2) + \sigma xyzw=0\} \subset \bP^1 \times \bP^1$$
is given by \cite{LZ}
\begin{equation}
j(\sigma)=\frac{\left(\sigma^4-16 \sigma^2+256\right)^3}{\sigma^4 \left(\sigma^2-16\right)^2}
\end{equation}
which can be obtained by constructing an explicit isomorphism between the Hesse elliptic curve $\{x^3 + y^3 + z^3 + s xyz = 0\}$ and the above family.  Then we have
$$ j\left(\frac{\psi(q_d)}{\phi(q_d)}\right) =  \frac{1}{q} + 744 + 196884 q + 21493760 q^{2} + 864299970 q^{3} + \ldots $$
(recall that $q = q_d^8$) equals to the $j$-function in flat coordinate.

\section{Deformation of the reference Lagrangian $\bL$}\label{subsec:2222pointmod}
In the rest of this chapter, we consider a geomtric deformation $\mathbb{L}_t$ of $\mathbb{L}_0$, and the mirror construction will produce a family of noncommutative algebras.  The weak Maurer-Cartan relations become much more complicated.  Miraculously we will show that coefficients in the relations determines a point in the mirror elliptic curve.  The construction here is parallel to that in Chapter \ref{sec:333} for $\mathbb{P}^1_{3,3,3}$.

\begin{figure}
\begin{center}
\includegraphics[height=1.8in]{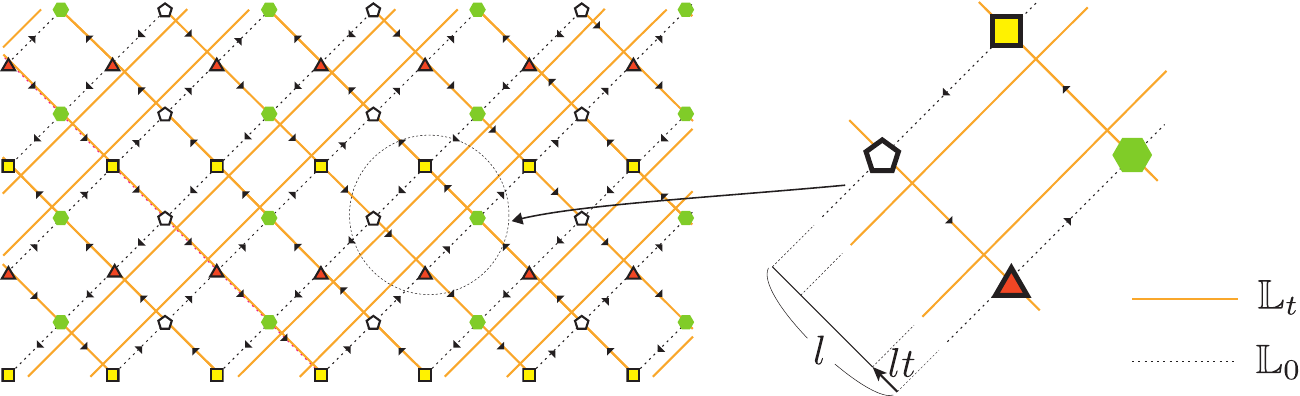}
\caption{Translation to obtain $\bL_t$}\label{fig:2222deform}
\end{center}
\end{figure}

We define $\mathbb{L}_t$ as follows. Let $l$ be the length of minimal $xyzw$-rectangle for the original reference Lagrangian $\mathbb{L}=\mathbb{L}_0$. We translate $L_1$ by $lt$ to get $\left(L_1 \right)_t$ while $L_2$ being fixed. (See Figure \ref{fig:2222deform}.)  As in \ref{subseck:weakMC333hol}, we also deform the flat line bundle on $\left(L_1\right)_t$ to have holonomy $\lambda$, and we represent this flat bundle by the trivial line bundle on $\left(L_1\right)_t$ together with a flat connection whose parallel transport is proportional to the length of a curve.
%
%
The line bundle on $L_2$ remains unchanged with trivial holonomy. 

We now consider the (noncommutative) mirror induced by formal deformation of $(\mathbb{L}_t := \left(L_1 \right)_t \cup L_2, \lambda)$ with $b= xX +yY + zZ +w W$. 
For $(\bL_t, \lambda)$, the coefficients of $\bar{X}, \tilde{Y},\bar{Z},\bar{W}$ in $m_0^b$ are given as
%
\begin{equation}\label{eqn:idealLt2222}
\begin{array}{lcl}
h_{\bar{X}} &= & A\, yzw + B\, wzy + C\, wxw  + D\, yxy\\
h_{\bar{Y}} &=& A\, zwx +  B\, xwz + C\, zyz +  D\, xyx \\
h_{\bar{Z}} &=&  A\, wxy + B\, yxw + C\, yzy +  D\, wzw\\
h_{\bar{W}}  &=& A\, xyz + B\, zyx + C\, xwx + D\, zwz.
\end{array}
\end{equation}
These relations generate the ideal $R_{(\lambda,t)}$ of the weak Maurer-Cartan relations in $\widehat{\Lambda Q}$ (with $Q$ given as in Figure \ref{fig:4puncturemq}). Corresponding spacetime potential $\Phi$ is given by
\begin{equation*}
 \Phi = A \, (xyzw) + B\, (wzyx)  + \frac{1}{2} C\, ((wx)^2  +  (yz)^2) + \frac{1}{2} D\, ( (xy)^2 +(zw)^2).
 \end{equation*}
Thus the mirror algebra $\mathcal{A}_{(\lambda,t)}$ arising from $(\bL_t,\lambda)$ can is the Jacobi algebra of $\Phi$, $\cA_{(\lambda,t)} =  \frac{ \Lambda Q }{ \left( \partial \Phi_{(\lambda,t)} \right)}$. This proves (1) of Theorem \ref{thm:deform2222coni}.

One can express $A$, $B$, $C$ and $D$ in terms of theta series with deformation parameters $(\lambda,t)$. By direct counting from the picture, we see that
\begin{equation}\label{eqn:ABCD2222bef}
\begin{array}{ll}
A(\lambda,t) &= \displaystyle\sum_{k,l \geq 0}^{\infty} \lambda^{-2l-\frac{1}{2}}   q_d^{(4k+1 -2t_1)(4l+1 -2 t_2)}- \sum_{k,l \geq 0}^{\infty} \lambda^{2l+\frac{3}{2}} q_d^{(4k+3 +2t_1)(4l+3 +2t_2)}   \\
B(\lambda,t) &= \displaystyle\sum_{k,l \geq 0}^{\infty} \lambda^{-2l-\frac{3}{2}}   q_d^{(4k+3 -2t_1)(4l+3 -2t_2)}- \sum_{k,l \geq 0}^{\infty} \lambda^{2l+\frac{1}{2}}  q_d^{(4k+1 +2t_1)(4l+1 +2t_2)}\\
C(\lambda,t) &= \displaystyle\sum_{k,l \geq 0}^{\infty} \lambda^{-2l-\frac{1}{2}}   q_d^{(4k+3 -2t_1)(4l+1-2t_2)}- \sum_{k,l \geq 0}^{\infty} \lambda^{2l+\frac{3}{2}}   q_d^{(4k+1 +2t_1)(4l+3 +2t_2)}\\
D(\lambda,t) &= \displaystyle\sum_{k,l \geq 0}^{\infty}\lambda^{-2l-\frac{3}{2}}   q_d^{(4k+1 -2t_1)(4l+3-2t_2)}- \sum_{k,l \geq 0}^{\infty} \lambda^{2l+\frac{1}{2}}   q_d^{(4k+3 +2t_1)(4l+1+2t_2)}.
%
%
\end{array}
\end{equation}
For instance, holomorphic polygons with small areas for the coefficients $A$, $B$, $C$ and $D$ are shown in Figure \ref{fig:2222mfell}.

Recall $q_{orb} = q_d^4$, and set $q_{orb}^2 = e^{2\pi i \tau}$, $\lambda=e^{2\pi is}$.
Up to a common factor, $A,B,C,D$ are holomorphic functions in a variable $u := -s - \tau \frac{t}{2} - \frac{1}{4}$ which can be thought of as a coordinate on $E_\tau:= \C / (\Z \oplus \Z \langle \tau \rangle)$.
 To see this, we multiply $K = \theta_0 (K')^{-1}$ to $A,B,C,D$ to get
\begin{equation}\label{eqn:KAKB}
 K A =   \dfrac{ \theta [0,0] ( 2u , 2\tau) }{ \theta \left[\frac{3}{4}, 0 \right] (2u , 2\tau)   }, \quad
 K B=  -i  \dfrac{ \theta [0,0] ( 2u , 2\tau) }{ \theta \left[\frac{1}{4}, 0 \right] (2u , 2\tau)   }, \quad
K C=   i \dfrac{ \theta \left[\frac{2}{4}, 0 \right]  ( 2u , 2\tau) }{ \theta \left[\frac{1}{4}, 0 \right] (2u , 2\tau)   }, \quad
K D=    \dfrac{ \theta \left[\frac{2}{4}, 0 \right]  ( 2u , 2\tau) }{ \theta \left[\frac{3}{4}, 0 \right] (2u , 2\tau)   }   
\end{equation}
where $K':=\frac{1}{2 \pi} e^{\frac{\pi i \tau}{2}} \theta' ( (2\tau+1)/2, 2 \tau)$ and $\theta_0:= \theta \left[\frac{3}{4}, 0 \right] (-1/2,2\tau)$. The proof of these identities will be given in Appendix \ref{appsubsec:theta2222}.
Denote the above four functions $K A, K B, K C, K D$ by $a,b,c,d$. By our choice of the holomorphic coordinate $u$, $u_0:=-\frac{1}{4}$ gives back the algebra $\mathcal{A}_0$ studied in Section \ref{subsec:centralfiber2222}. Namely, $a(u_0) = -b(u_0)$ and $c(u_0)=d(u_0)=0$.

It is easy to see that $a,b,c,d$ define a map
$(a:b:c:d) : E_\tau \to \mathbb{P}^3$. Indeed, it gives an embedding of $E_\tau$ into $\mathbb{P}^3$ whose proof will be given in Appendix \ref{appsubsec:theta2222e}. The image of this map in $\mathbb{P}^3$ is a complete intersection of two quadrics, one of which is $\{ac+bd=0\}$. One can check this relation directly from the expression \eqref{eqn:KAKB}. Defining equation of the other quadric can be also found explicitly as follows.
%
%

We apply the same trick as in Chapter \ref{sec:333} to obtain a matrix factorization of $W (=W_{\mathbb{L}_0})$ one of whose matrices has $a$, $b$, $c$, $d$ as coefficients. Namely, we translate $L_1$ further by $2t$ and compute the mirror matrix factorization of $L:=\left(L_1 \right)_{2t}$ which is by definition $\CF(\mathbb{L}_0, L )$ with $m_1^{b,0}$. Here, we set the flat connection on $L$  to have trivial holonomy ($\lambda=1$) and only $t$ varies. So, both $\mathbb{L}_0$ and $L$ have trivial holonomies in what follows. 

Note that $L$ intersects only one branch, $L_2$, of $\mathbb{L}_0$. Thus, we restrict ourselves to the subalgebra $\pi_2 \cdot \left(\widehat{\Lambda Q} / R_0 \right) \cdot \pi_2$ of $\widehat{\Lambda Q} / R_0$ which are generated by four loops based at $\pi_2$. We write these four loops as
$$ \alpha =yz,\quad  \beta =wz, \quad \gamma =wx,\quad  \delta =yx.$$
It is easy to see that these variables commute to each other. For example, 
$$ \alpha \beta = yzwz  \stackrel{\eqref{eqn:m0b2222sym}}{=} wzyz = \beta \alpha. $$
One can deduce $\alpha \gamma = \beta \delta$ similarly so that
\begin{equation*}\label{eqn:2ndvertsub}
\pi_2 \cdot \left(\widehat{\Lambda Q} / R_0 \right) \cdot \pi_2 \cong \C[\alpha, \beta, \gamma,\delta ] / \alpha \gamma - \beta \delta
\end{equation*}
which defines a conifold in $\C^4$.

Also, the potential $W$ reduces to
\begin{equation}\label{eqn:redpot2}
 \UL{W}(\alpha, \beta, \gamma,\delta) = \phi  (\alpha^2 + \beta^2 + \gamma^2 + \delta^2) + \psi \alpha \gamma
\end{equation}
Finally, $L$ intersects $L_2$ at four points denoted by $p_1$, $p_2$ with odd degree and $q_1$, $q_2$ with even degree (see Figure \ref{fig:2222mfell}), and our functor produces a 2 by 2 matrix factorization of $ \UL{W}(\alpha, \beta, \gamma,\delta)$ over the commutative algebra $\C[\alpha, \beta, \gamma,\delta ] / \alpha \gamma - \beta \delta$ (after taking $\pi_2 \cdot (-) \cdot \pi_2$)
\begin{equation}\label{eqn:2222tranMF}
P_{odd \to even} \cdot P_{even \to odd} = \UL{W} \cdot Id_2.
\end{equation}

\begin{figure}
\begin{center}
\includegraphics[height=4.5in]{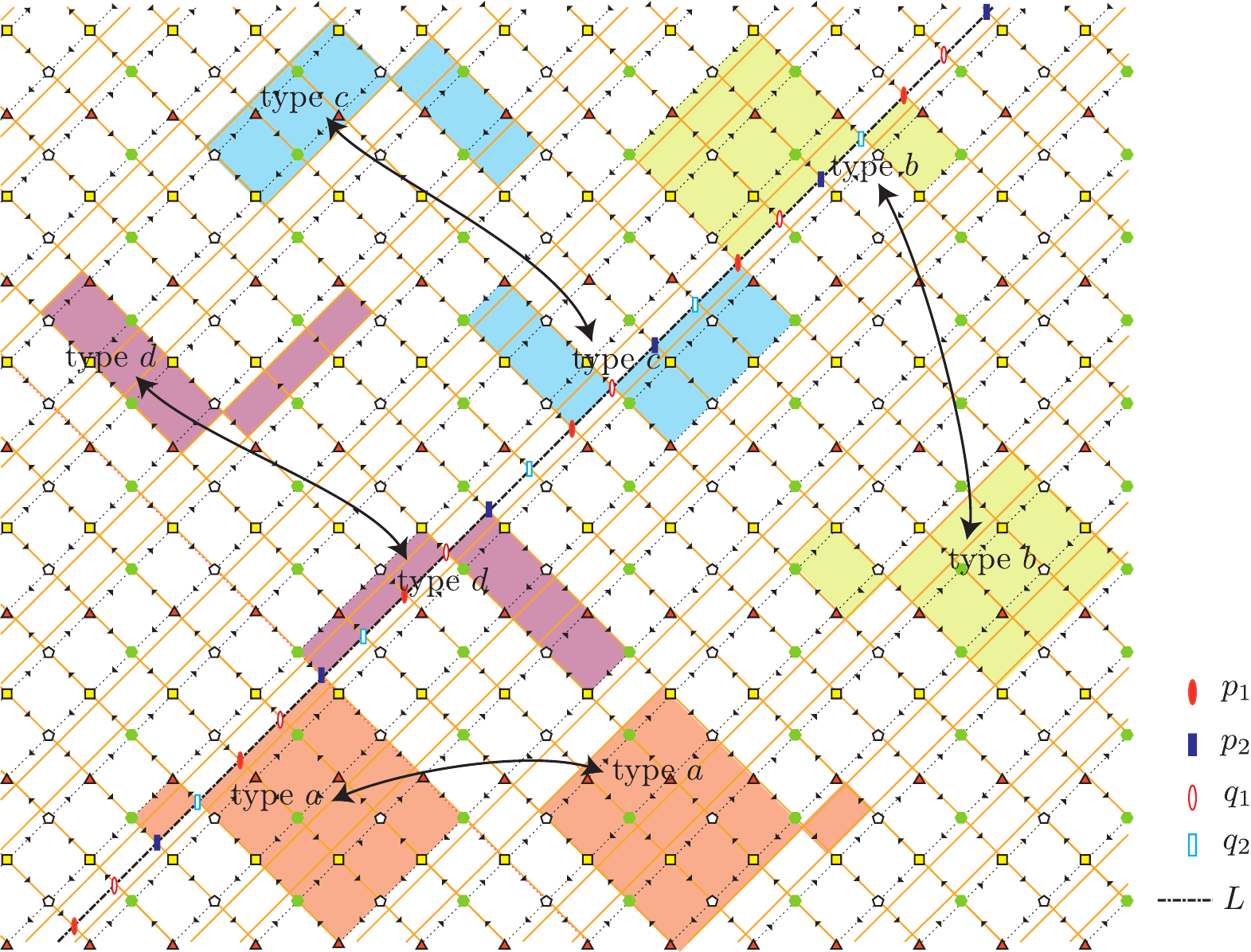}
\caption{Correspondence between $\{$polygons contributing to $a$, $b$, $c$, $d$$\}$ and $\{$holomorphic strips contributing to $P_{odd \to even}$$\}$}\label{fig:2222mfell}
\end{center}
\end{figure}

One can see from Figure \ref{fig:2222mfell} that there is an obvious one-to-one correspondence between ``polygons for $a$, $b$, $c$, $d$ (or $A,B,C,D$)" and ``holomorphic strips contributing to $P_{odd \to even}$", which is area preserving. In the figure, holomorphic strips for $\CF(\mathbb{L}_0, L)$ are arranged along the diagonal dashed line representing $L$. By direct computation we get 
\begin{equation*}
P_{odd \to even}=
\bordermatrix{ & p_1 & p_2 \cr
q_1 & a\, wx + c\, yz & b\, yx + d\, wz \cr
q_2 & b\, wz + d\, yx & a\, yz + c\, wx \cr              } =
\bordermatrix{ &  &  \cr
 &  a\, \gamma + c\, \alpha & b\, \delta + d\, \beta  \cr
 &  b\, \beta + d\, \delta  &  a\, \alpha+  c\, \gamma \cr              }
\end{equation*}
(up to scaling) where $a,b,c,d$ above are evaluated at $(\lambda,t)=(1,t)$.
Since we are now working over the commutative algebra, it makes sense to take the determinant of the both sides of the equation \ref{eqn:2222tranMF},
$$\det P_{odd \to even} = ac (\alpha^2 + \gamma^2 ) - bd (\beta^2 + \delta^2) + (a^2 +c^2  - b^2 - d^2) \alpha \gamma$$
which should be a constant multiple of $\UL{W}$ by the (generic) irreducibility of $\UL{W}$ as in the proof of Theorem \ref{abceq}. 
Comparing with \eqref{eqn:redpot2}, we see that
$$ ac +bd =0,\quad a^2 + c^2 - b^2 - d^2  - \frac{\psi}{\phi} ac =0$$
hold for $(\lambda,t) = (1,t)$. As $a,b,c,d$ are holomorphic functions in $s + \frac{\tau t}{2}$, this relations extend to hold over whole $E_\tau$. 

We conclude that for all $(\lambda,t) \in E_\tau$, $( \tilde{a} :b : \tilde{c} : \tilde{d}) := (-ia : b : ic : -d) \in \mathbb{P}^3$ lies on the complete intersection 
$$ \{  \tilde{a} \tilde{c} -b\tilde{d} =0\} \cap \{\tilde{a}^2 + b^2 + \tilde{c}^2 + \tilde{d}^2 +(\psi / \phi) \tilde{a}\tilde{c}=0 \} \subset \mathbb{P}^3$$
which is precisely the mirror elliptic curve in Section \ref{subsec:openmirthm2222}.

For later use, we provide the formula of the worldsheet potential $W_{(\lambda,t)}$ (only whose first component will be given). Since we have two branches $L_1$ and $L_2$,
$$ W_{(\lambda,t)} = (W_{(\lambda,t)})_1 {\bf 1}_{\mathbb{L}_1} + (W_{(\lambda,t)})_2 {\bf 1}_{\mathbb{L}_2}.$$
By the similar computation as in Section \ref{subsec:wspot333}, one can check that the coefficients of $W_1$ are $u$-derivatives of $A,B,C,D$:
\begin{equation}\label{eqn:w12222def}
\begin{array}{lcl}
(W_{(\lambda,t)})_1 &=& \left(\frac{K}{4 \pi i} a'(u) + \frac{1}{8} A(u) \right) (xyzw + zwxy )
+ \left(\frac{K}{4 \pi i}  b'(u) + \frac{1}{8} B(u) \right) (xwzy + zyxw ) \\
&& + \left(\frac{i K}{4 \pi i}  c'(u)  + \frac{1}{8} C(u) \right) (xwxw + zyzy )
+ \left(\frac{i K}{4 \pi i}  d'(u)  + \frac{1}{8} D(u) \right) (xyxy + zwzw)\\
&=& \frac{K}{4 \pi i} \left(  a'(u)  (xyzw + zwxy) +  b'(u)  (xwzy + zyxw) +  c'(u)  ( (xw)^2+ (zy)^2) + d'(u) ( (xy)^2 + (zw)^2) \right) 
\end{array}
\end{equation}
where we used the weak Maurer-Cartan relations to get the simpler expression in the second line.

In order to obtain the similar result for $W_2$, one needs to deform the second Lagrangian $L_2$ as well so that $a,b,c,d$ now become two variable functions. Then the coefficients of $W_2$ will be partial derivatives of $a,b,c,d$ with respect to the second variable. We omit this computation here as it will not be used in this paper.

\section{Relation to the quantization of an intersection of two quadrics in $\C^4$}

In parallel with Section \ref{subsec:defquant333}, we relate the algebra $\mathcal{A}_{(\lambda,t)}$ obtained in the previous section with a deformation quantization of a hypersurface in the conifold. 

We first restrict $\mathcal{A}_{(\lambda,t)}$ to the subalgebra generated by the loops based at the first vertex of $Q$ ((a) of Figure \ref{fig:4puncturemq}). Let us denote this algebra by $\underline{\mathcal{A}}$.
We take $ x_1:=xy, x_2:=xw, x_3:=zw, x_4:=zy$ to be the generators of $\underline{\mathcal{A}}$ similarly to \eqref{eqn:segrecoord}.
Multiplying a suitable variable to front or back of the original relations \eqref{eqn:idealLt2222}, we obtain relations among the generating loops $x_1,x_2,x_3,x_4$:
%
\begin{equation}\label{eqn:rel1}
\begin{array}{lcl}
K  \cdot x h_{\bar{X}} &= & a\, x_1  x_3 + b\, x_2  x_4 + c\, x_2^2  + d\, x_1^2\\
K  \cdot h_{\bar{Y}}y &=& a\, x_3  x_1 +  b\, x_2  x_4 + c\, x_4^2 +  d\, x_1^2 \\
K   \cdot z h_{\bar{Z}} &=&  a\, x_3  x_1 + b\, x_4 x_2 + c\, x_4^2 +  d\, x_3^2\\
K  \cdot h_{\bar{W}} w &=& a\, x_1  x_3 + b\, x_4  x_2 + c\, x_2^2 + d\, x_3^2.
\end{array}
\end{equation}
(There is one redundant relation in \eqref{eqn:rel1}.)
\begin{equation}\label{eqn:rel2}
\begin{array}{lcl}
K  \cdot z h_{\bar{X}} &= & a\, x_4  x_3 + b\, x_3  x_4 + c\, x_3  x_2  + d\, x_4  x_1\\
K  \cdot h_{\bar{Y}} w &=& a\, x_3  x_2 +  b\, x_2  x_3 + c\, x_4  x_3 +  d\, x_1  x_2 \\
K  \cdot x h_{\bar{Z}} &=&  a\, x_2  x_1 + b\, x_1  x_2 + c\, x_1  x_4 +  d\, x_2  x_3\\
K  \cdot h_{\bar{W}} y  &=& a\, x_1  x_4 + b\, x_4  x_1 + c\, x_2  x_1 + d\, x_3  x_4.
\end{array}
\end{equation}
The subalgebra $\underline{\mathcal{A}}$ of $\mathcal{A}$ is obviously the quotient of $\C \{x_4, x_3, x_2, x_1\}$ by the ideal generated by these eight relations. Recall that $a,b,c,d$ are functions in 
$$u=-s-\tau \frac{ t}{2} - \frac{1}{4}$$
(and $\lambda=e^{2\pi \consti s}$).
To emphasize the dependence of $\underline{\mathcal{A}}$ in $u$, we write $\underline{\mathcal{A}}_u$ from now on. 

Note that $\underline{\mathcal{A}}_{u_0}$ represents the commutative conifold since $a(u_0) = - b(u_0)$ and $c(u_0)=d(u_0)=0$:
\begin{equation}\label{eqn:conifoldcomm}
\underline{\mathcal{A}}_{u_0} \cong \C[x_1,x_2,x_3,x_4] / x_1x_3 - x_2 x_4.
\end{equation}
Thus one may view $\underline{\mathcal{A}}_{u}$ as a noncommutative conifold.
Let 
$$f=x_1 x_3 - x_2 x_4$$ 
be the defining equation of the conifold.

If we are given another function $g$ on $\C^4$, we can define a Poisson structure in the following way.
\begin{equation}\label{eqn:nambu2222}
\{x_i, x_j\} = \zeta \frac{dx_i \wedge dx_j \wedge df \wedge dg}{dx_1 \wedge dx_2 \wedge dx_3 \wedge dx_4} = \zeta \left( \frac{\partial f}{\partial x_k} \frac{\partial g}{\partial x_l} - \frac{\partial g}{\partial x_k} \frac{\partial f}{\partial x_l}\right).
\end{equation}
where $(i,j,k,l)$ is equivalent to $(1,2,3,4)$ up to an even permutation.
$\zeta$ in \eqref{eqn:nambu2222} will be some constant in our case though it could be a more complicated function in general.
Such a structure is in fact a special case of certain higher brackets among functions called the \emph{Nambu bracket}. (See for e.g. \cite{Ode}.) Choice of $g$ will be fixed shortly.

One can check that \eqref{eqn:nambu2222} satisfies a Jacobi relation, and $f$ and $g$ lie in the Poisson center. Therefore it descends to the quotient algebra $\mathcal{B}_{f,g}=\C [x_1, x_2, x_3, x_4] / \langle f,g \rangle$ which is a coordinate ring of  a hypersurface in the conifold defined by $g=0$. Now we provide some relation between our $\underline{\mathcal{A}}_u / W_{u}$ and the deformation quantization of $\mathcal{B}_{f,g}$ associated with the Poisson structure \eqref{eqn:nambu2222}.

First, we make the following specific choice of the second function $g$ in \eqref{eqn:nambu2222}. Restricted to the loops based at the first vertex of $Q$, world sheet potential $(W_{(\lambda,t)} )_1$ \eqref{eqn:w12222def} can be written as
\begin{equation}\label{eqn:WuquadC4}
W_u:=  \frac{K}{4 \pi i} \left(  a'(u)  (x_1 x_3 + x_3 x_1) +  b'(u)  (x_2 x_4 + x_4 x_2) +  c'(u)  ( x_2^2+ x_4^2) + d'(u) ( x_1^2 + x_3^2) \right) 
\end{equation}
Compared with Theorem \ref{thm:mir-2222}, this should reduce to 
$$g:= \phi (q_d) (x_1^2 + x_2^2 + x_3^2 + x_4^2) + \psi (q_d) x_1 x_3$$
at $u=u_0$, which we denote by $g$. Thus, if we set $v= u- u_0$ (so that $v=0$ gives the commutative conifold \eqref{eqn:conifoldcomm}) and expand $a,b,c,d$ in terms of $v$, we have
$$ a(v) = a(0) + a'(0) v + O(v), \quad b(v) = b(0) + b'(0) v + O(v), \quad c(v) =  2 \zeta \phi v + O(v), \quad d(v) =  2 \zeta \phi v + O(v)$$
with $a(0)=-b(0)$ and $a'(0) + b'(0) = \zeta \psi $ for some constant $\zeta$. Here, $a(0)$ and $a'(0)$ are evaluated at $v=0$ (i.e. $u=u_0$) and similar for $b,c,d$.

\begin{theorem} \label{dq2222}
The family of noncommutative algebra $\underline{\mathcal{A}}_{v}/(W_{v})$ near $v=0$ gives a quantization of the complete intersection given by two quadratic equations $f=0$ and $g=0$ in $\C^4$ in the sense of \cite{EG}.
\end{theorem}
We remark that $\{f=0\} \cap \{g=0\}$ defines the mirror elliptic curve in $\mathbb{P}^3$ after projectivization.  As in Theorem \ref{thm:tw-ring333} \cite{ATV, St1}, we expect that $\underline{\mathcal{A}}_{v}/(W_{v})$ is given as a twisted homogeneous coordinate ring of the mirror elliptic curve embedded in the quadric $\{x_1x_3=x_2x_4\}\subset \bP^3$ (which can be identified as $\bP^1 \times \bP^1$) by the line bundle $\mathcal{L} = \cO_{\bP^1}(1) \boxtimes \cO_{\bP^1}(1)$.

\begin{proof}
As $v \to 0$, we have $$\underline{\mathcal{A}}_{v}/(W_{v}) \to \C[x_1,x_2,x_3,x_4]/\langle f,g \rangle.$$
So it remains to check that the first order terms of the commutators for $A_{v}$ are equivalent to the Poisson bracket induced by $f$ and $g$ with the formula \eqref{eqn:nambu2222}. The first equation in \eqref{eqn:rel2} gives rise to
\begin{equation*}
r (x_3  x_4 - x_4  x_3) =v \left( a'(0)x_4  x_3 + b'(0) x_3  x_4 + c'(0) x_3  x_2 + d'(0) x_4  x_1   \right) + o(v), 
\end{equation*}
where $r:=a(0) =-b(0) \in \C$. Thus we obtain
\begin{equation}\label{eqn:2222x3x4}
\lim_{v \to 0} \frac{ x_3  x_4 - x_4  x_3}{v} = \frac{1}{r} \left( (a'(0) +b'(0) )x_3  x_4 + c'(0) x_3  x_2 + d'(0) x_4  x_1   \right)
\end{equation}
since $x_3$ and $x_4$ commute at $v=0$ (and so do other variables).

On the other hand,  the Poisson bracket between $x_3$ and $x_4$ from  \eqref{eqn:nambu2222} is given by
$$\{x_3,x_4\}= \zeta \left( \frac{\partial f}{\partial x_1} \frac{\partial g}{\partial x_2} - \frac{\partial g}{\partial x_1} \frac{\partial f}{\partial x_2} \right)=\zeta \left( \psi x_3 x_4+ 2 \phi (x_3 x_2 + x_4  x_1 )  \right),$$
which is proportional to the right hand side of \eqref{eqn:2222x3x4} since $a'(0) + b'(0) = \zeta \psi$ and $c'(0) = d'(0) = 2\zeta \phi$. Similarly, the first order commutator between $x_i$ and $x_{i+1}$ for other $i$'s agrees with the Poisson bracket of the corresponding variables.


Lastly, the remaining two commutators can be computed from \eqref{eqn:rel1}. Considering the difference between two equations in \eqref{eqn:rel1}, we have
\begin{equation}\label{eqn:rel3}
\begin{array}{l}
a\, x_3  x_1 - a\, x_1  x_3 +c\, x_4^2 - c\, x_2^2 =0\\
b\, x_4  x_2 - b\, x_2  x_4 + d\, x_3^2 - d \, x_1^2 =0.
\end{array}
\end{equation}
By the same trick, the above implies
$$r(x_1 x_3 - x_3 x_1) =  v( c'(0) x_4^2 - c'(0) x_2^2 ) +va'(0)( x_3 x_1 - x_1 x_3  )+ o(v)$$
$$r(x_2 x_4 - x_4 x_2)= v(d'(0) x_1^2- d'(0) x_3^2 ) +vb'(0) (x_4 x_2 - x_2 x_4) + o(v),$$
and hence,
$$\lim_{v\to0} \frac{x_1 x_3 - x_3 x_1}{v} = \frac{1}{r}( c'(0) x_4^2 - c'(0) x_2^2 ) = \frac{\zeta}{r} ( 2\phi x_4^2 - 2 \phi x_2^2)$$
$$\lim_{v\to0} \frac{x_2 x_4 - x_4 x_2}{v}= \frac{1}{r}(d'(0) x_1^2- d'(0) x_3^2 ) =\frac{\zeta}{r} ( 2 \phi x_1^2 - 2 \phi  x_3^2)$$
since variables commute at $v=0$. One can easily check that these coincide (up to scaling) with the Poisson brackets \eqref{eqn:nambu2222}, which in this case are $\zeta \left( \frac{\partial f}{\partial x_k} \frac{\partial g}{\partial x_l} - \frac{\partial g}{\partial x_k} \frac{\partial f}{\partial x_l}\right)$ for $(k,l) = (4,2)$ and $(k,l) = (3,1)$ respectively.
\end{proof}

From the relations \eqref{eqn:rel2} and \eqref{eqn:rel3}, we can capture some information on noncommutative deformation of $\C^4$ (or that of $\bP^3$ regarded as a graded algebra). Note that at the commutative point $v=0$, these six relations (\eqref{eqn:rel2} and \eqref{eqn:rel3}) boil down to the commutators among four variable $x_1,x_2,x_3,x_4$. Therefore, the algebra $A_v$ defined by 
$$A_v:= \frac{\C \{x_1, x_2, x_3, x_4\}}{ \left\langle 
\begin{array}{ll}
 ax_4  x_3 + bx_3  x_4 + cx_3  x_2  + dx_4  x_1,&  a x_3  x_2 +  b x_2  x_3 + cx_4  x_3 +  dx_1  x_2 \\
  a x_2  x_1 + b x_1  x_2 + c x_1  x_4 +  dx_2  x_3,&  a x_1  x_4 + b x_4  x_1 + cx_2  x_1 + dx_3  x_4\\
ax_3  x_1 - a x_1  x_3 +c x_4^2 - c x_2^2,& bx_4  x_2 - b x_2  x_4 + d x_3^2 - d x_1^2
\end{array}
\right\rangle}
$$
can be thought of as a noncommutative deformation of $\C^4$, where $a,b,c,d$
satisfy the two quadratic relations $ac+bd=0$ and $a^2 + c^2 - b^2
-d^2- \frac{\psi}{\phi} ac =0$.

$\underline{\cA}_v$ is obtained by quotienting $A_v$ by the conifold relation $C_1$ which can be taken as the first line in \eqref{eqn:rel1}.
Set $C_2$ to be the (pre)image of $W_v(\in \underline{\cA}_v)$ in $A_v$ defined by the same equation \eqref{eqn:WuquadC4}. From the proof of Theorem \ref{dq2222}, we see that the six relations in $A_v$ determine the Poisson bracket on $\C[x_1,x_2,x_3,x_4]$ \ref{eqn:nambu2222} with $f=C_1, g= C_2$.

If we equip $A_v$ with the grading by $\deg x_i =1$ ($i=1,2,3,4$), it may also be  regarded as a deformation of the homogeneous coordinate ring of $\bP^3$. We conjecture that $A_v$ is isomorphic to the 4-dimensional Sklyanin algebra, which was known as a noncommutative $\bP^3$ having a 2-dimensional center (see for instance \cite{SPS}).

\begin{conjecture}
$A_v$ is isomorphic to the 4-dimensional Sklyanin algebra. In addition, $C_1$ and $C_2$ are central in $A_v$, which generate the whole center of $A_v$.
\end{conjecture}

\begin{remark}
After the submission, Authors were informed by Chirvasitu and Smith that the conjecture is true only for a certain subfamily of the algebras for which $a,b,c,d$ satisfy some special condition. Unfortunately, we do not have a good geometric interpretation of this. See \cite[Proposition 5.2]{CS2017} for the precise statement.
\end{remark}

We give a brief summary of our deformation quantization construction in Chapter \ref{sec:333} and Chapter \ref{sec:4puncture2222} .
In the case of $\mathbb{P}^1_{3,3,3}$ (resp. $\mathbb{P}^1_{2,2,2,2}$), the commutative weak Maurer-Cartan space of the reference Lagrangian $\bL_0$ is given by $\C^3$ (resp. a conifold  $x_1x_3-x_2x_4=0$ in $\C^4$) with a potential function $W_0$. For $\mathbb{P}^1_{2,2,2,2}$, we first applied a quiver construction to obtain a non-commutative mirror LG model, and reduced it to a commutative LG model by restricting to loops based at one vertex of the mirror quiver.

We defined a Poisson structure on this weak Maurer-Cartan space using the function 
$W_0$ (resp.  $W_0$ and $f_0=x_1x_3-x_2x_4$ in the case of $\C^4$) so that
$W_0$ (resp. $f_0,W_0$) lies in the Poisson Kernel.
In particular, this Poisson structure descends to the quotient hypersurface 
$W_0=0$ in $\C^3$ (resp. the complete intersection $f_0=W_0=0$ in $\C^4$).
Geometrically the hypersurface $W_0=0$ gives the strict (or ``flat'') Maurer-Cartan space. Namely, it consists of unobstructed deformations whose corresponding potential values are $0$.

Then we obtain the deformation quantization of the affine spaces $\C^3$ (resp. conifold $f_0=0$ in $\C^4$) by considering non-commutative Maurer-Cartan equations for the family of deformed Lagrangians $\bL_t$. Moreover, the further quotient by the central element $W_t$ (in the deformation quantization algebra) gives the deformation quantization of the (singular) hypersurface $W_0=0$ in $\C^3$ (resp. the complete intersection $f_0=W_0=0$ in $\C^4$) in the sense of \cite{EG}. This should correspond to the deformation of the strict Maurer-Cartan space.

However, the Fukaya category obviously remains invariant under $t$ since we do not deform the symplectic structure, and hence we expect that the singularity category of the deformation quantization of the hypersurface which arises in this way remains invariant. (See for e.g., Corollary \ref{cor:etatwt}.) Presumably, this is related to the fact that the superpotential $W_t$ is central for any $t$.

%
%
%


\chapter{Extended mirror functor} \label{sec:ext-mir}
In this chapter we construct an extended noncommutative Landau-Ginzburg model $(A,d,W)$, which is a curved dg-algebra (see for instance \cite{Kapustin-Li,Segal}), from the Lagrangian Floer theory of $\bL$.  We also construct an $A_\infty$-functor from the Fukaya category of $X$ to the category of curved dg-modules of $(A,d,W)$.  In this sense $(A,d,W)$ is a generalized mirror of $X$ (and it is the mirror in the ordinary sense when the functor induces an equivalence between the corresponding derived categories).  

The construction uses Lagrangian deformations in all degrees (instead of odd-degree immersed deformations only).  More concretely, it uses $m_k^{(\bL,\tilde{b}),\ldots}$ where $\tilde{b}$ parametrizes all Lagrangian deformations.
In practice it is much more difficult to compute $(A,d,W)$ and the full general form of $(A,d,W)$ does not appear much in the study of mirror symmetry.  More frequently used is its simplified forms.  For instance when $d=0$, it reduces to a curved algebra which we use throughout this paper.  When $W=0$ it reduces to dg algebras.  For a Calabi-Yau threefold with certain additional grading assumptions on the reference Lagrangian $\bL$ (see Proposition \ref{prop:Ginzburg}), the construction produces Ginzburg algebra (with $W = 0$), which is explicit and was used in \cite{Smith} for studying stability conditions on Fukaya category of certain local Calabi-Yau threefolds.

\section{Extended Landau-Ginzburg mirror} 

Suppose $\bL=\bigcup_i L_i$ has transverse self-intersections.  In general the Lagrangian Floer theory is $\Z_2$-graded.  When we have a holomorphic volume form on the ambient space $X$ and $\bL$ is graded, then the Lagrangian Floer theory is $\Z$-graded.  We work with either case in this paper.  (In the $\Z_2$-graded case all degrees below take value in $\Z_2$.)

Let $\{X_e\}$ be the set of odd-degree immersed generators of $\bL$.  As explained in Chapter \ref{sec:MFseverallag}, we have a corresponding endomorphism quiver $Q=Q^\bL$ whose vertices are one-to-one corresponding to $L_i$ and whose arrows are one-to-one corresponding to the generators $X_e$.  The completed path algebra is denoted as $\widehat{\Lambda Q}$.

\begin{definition}[Extended quiver] \label{def:extQ}
Let $\{X_f\}$ be the set of even-degree immersed generators.  Moreover take a Morse function on each $L_i$ and let $\{T_g\}$ be the set of their critical points \emph{other than} the maximum points (which correspond to the fundamental classes $\one_{L_i}$).   The extended quiver $\tilde{Q}$ is defined to be a directed graph whose vertices are one-to-one corresponding to $\{L_i\}$, and whose arrows are one-to-one corresponding to $\{X_e\} \cup \{X_f\} \cup \{T_g\}$.
\end{definition}

Recall that $m_0^b$ takes the form 
$$m_0^b = \sum_i W_i \one_{L_i} + \sum_f P_f X_f + \sum_g P_g T_g.  $$
$W = \sum_i W_i \in \widehat{\Lambda Q}/R$ is a central element, where $R$ is the two-sided ideal generated by $P_f$ and $P_g$.  As explained in Chapter \ref{sec:MFseverallag}, we obtain a Landau-Ginzburg model $(\widehat{\Lambda Q}/R,W)$.

\begin{prop} \label{prop:W=0}
Suppose $X$ is Calabi-Yau and $\bL$ is $\Z$-graded such that $\deg X_e \geq 1$ for all odd-immersed generators $X_e$.  Then $W \equiv 0$.
\end{prop}
\begin{proof}
Since $\deg X_e \geq 1$, $\deg x_e \leq 0$ for all $e$, and hence $\deg W_i \leq 0$.  But $\deg W_i = 2$.  Thus $W \equiv 0$.
\end{proof}

Now we consider extended deformations in directions of $X_f$ and $T_g$.
\begin{definition}
Take
$$\tilde{b} = \sum_e x_e X_e + \sum_f x_f X_f + \sum_g t_g T_g.  $$
Define $\deg x_e = 1 - \deg X_e$, $\deg x_f = 1 - \deg X_f$ and $\deg t_g = 1 - \deg T_g$ such that $\tilde{b}$ has degree one.  
\end{definition}

The deformation parameter $\deg x_e$  is always even, but the other parameters may not be.
Hence we need to modify the sign rule of \eqref{eqn:pulloutrule} in the definition \ref{defn:nccoeff} to the following.
\begin{equation}\label{eqn:pulloutrule2}
m_k(f_1 e_{i_1},f_2 e_{i_2},\cdots, f_k e_{i_k} ) :=  (-1)^\epsilon f_k f_{k-1} \cdots f_2 f_1\cdot m_k(e_{i_1},\cdots, e_{i_k}).
\end{equation}
where $\epsilon$ is obtained from the usual Koszul sign convention.
For example, if $f_j$ is the dual variable of $e_{i_j}$ for all $j=1,\ldots k$, then we have
\begin{equation}\label{eqn:pulloutrule3}
m_k(f_1 e_{i_1},f_2 e_{i_2},\cdots, f_k e_{i_k} ) :=  (-1)^{\sum_{j=1}^k \deg f_j} f_k f_{k-1} \cdots f_2 f_1\cdot m_k(e_{i_1},\cdots, e_{i_k}),
\end{equation}
since one can pull out $f_k, f_{k-1},\ldots,f_1$ one at a time and by assumption $f_j e_{i_j}$ will have (shifted) degree 0.
One can check that this defines an $\AI$-structure. 

Then $m_0^{\tilde{b}}$ will additionally have odd-degree part, i.e., it is of the form
$$m_0^{\tilde{b}} = \sum_i W_i \one_{L_i} + \sum_e P_e X_e + \sum_f P_f X_f + \sum_g P_g T_g.  $$

The constructions in previous chapters generalize to this setting, to give the noncommutative ring $A$ modulo the relations $(P_e,P_f,P_g)$, and  an extended $\AI$ functor to matrix factorizations of an extended potential function $\widetilde{W}$.

Alternatively, we explain a construction which gives a curved dg-algebra $(A,d,\widetilde{W})$, and an $\AI$-functor to the category
of curved dg-modules over $(A,d,\widetilde{W})$.

We will construct a curved dg-algebra from the extended deformations.
\begin{definition}[\cite{Kapustin-Li}]
A curved dg-algebra is a triple $(A,d,F)$, where $A$ is a graded associative algebra, $d$ is a degree-one derivation of $A$, and $F \in A$ is a degree 2 element such that $dF = 0$ and $d^2 = [F,\cdot] = F (\cdot) - (\cdot) F$.  
\end{definition}

If $d^2=0$ (or equivalently $F$ is central), a curved dg-algebra structure
gives rise to a noncommutative Landau-Ginzburg model on its $d$-cohomology with the cohomology class $[F]$ as a (worldsheet) superpotential.  If the dg-algebra is formal, then $dg$-modules over $A$ (Definition \ref{def:cdg} below) boils down to matrix factorizations of $[F]$ over $H(A)$.  In this section we don't take this assumption and work with $dg$-modules.

Now take the path algebra $\Lambda \tilde{Q}$.  It carries a natural curved dg structure as follows.  

\begin{definition} 
Define the extended worldsheet superpotential to be $\tilde{W} = \sum_i \tilde{W}_i$, where $\tilde{W}_i$ are the coefficients of $\one_{L_i}$ in $m_0^{\tilde{b}}$.  
Define a derivation $d$ of degree one on $\Lambda \tilde{Q}$ by
$$d(a_e x_e + b_f x_f + c_g t_g) = a_e P_e + b_f P_f + c_g P_g$$
for $a_e, b_f, c_g \in \Lambda$, and extend it using Leibniz rule.
(Our sign convention is $d(xy) = x dy + (-1)^{\deg y} (dx) y$.)
\end{definition}
Then the $A_\infty$-relations for $\bL$ can be rewritten as follows.

\begin{prop}
$d \tilde{W} = 0$ and $d^2 = [\tilde{W}, \cdot]$.  In other words $(\Lambda \tilde{Q}, d, \tilde{W})$ is a curved dg algebra.
\end{prop}
\begin{remark}
It is well-known (see \cite{KS_Ainfty}) that $\AI$-algebra structure on $V$ is equivalent to a codifferential $\hat{d}$ on
the tensor coalgebra $TV$. And if $V$ is finite dimensional (as in our case), then one can take its dual algebra $(TV)^*$ 
with a differential $\tilde{d}$ (with $\tilde{d}^2=0$). 
We may regard $\{ x_e, x_f, x_g\}$ as dual generators of the dual algebra $(TV)^*$.  There is
one more generator $x_{\be}$ that we have suppressed in the previous formulation.  Turning it on simply results in adding a multiple of unit in $\tilde{W}$.  The desired identities can be obtained from $\tilde{d}^2=0$ by setting $x_{\be}=0$. But we show the direct proof below due to signs.  These two formulations are equivalent.
\end{remark}
\begin{proof}
Consider a deformed $\AI$-structure $\left\{ m_k^{\tilde{b}} \right\}_{k \geq 0}$ (see Definition \ref{def:bdeformAI}), and the first $\AI$-equation:
$$m_1^{\tilde{b}}(m_0^{\tilde{b}}) =0.$$
From the definition, we have $m_0^{\tilde{b}} =  \tilde{W} \be + \sum_X (dx) X$ for  $X$ are elements in $\{X_e\} \cup \{X_f\} \cup \{T_g\}$. Hence we have 
$$m_1^{\tilde{b}}( \tilde{W} \cdot \be) + m_1^{\tilde{b}}( \sum_X (dx) X)=0.$$
The first term becomes (writing $\tilde{b} = \sum_X xX$)
$$m_2( \tilde{W} \cdot \be, \tilde{b}) + m_2( \tilde{b}, \tilde{W} \cdot \be) = \sum_X (x \tilde{W} - \tilde{W} x) X = - \sum [\tilde{W},x]X.$$
The second term can be written as follows.

\begin{align}\label{cvmd1}
m_1^{\tilde{b}} \left( \sum_X (dx) X \right)  =& \sum_{k,i,X_j,X} m_{k+1}(x_1X_1,\cdots,(dx)X, x_iX_i, \cdots x_kX_k) \\
=& \sum_{k,i,X_j,X} (-1)^{|x_{i-1}|+\cdots+|x_1|} x_k \cdots x_i (dx) \cdots x_1 m_{k+1}(X_1,\cdots,X,\cdots X_k) \nonumber \\
=& \sum_{X_j,X} d(x_k,\ldots,x_1)   m_{k+1}(X_1,\cdots,X,\cdots X_k)  \nonumber \\
=& d(m_0^{\tilde{b}}) = d(  \tilde{W} \be + \sum_X (dx) X)  \nonumber \\
=& d(\tilde{W})\be + \sum_X (d^2x) X  \nonumber
\end{align}
Here, the degree of the dual variable $x_{j}$ is the negative of the shifted degree of the corresponding generator $X_j$.
Since the whole sum vanishes, we have $d \tilde{W} = 0$ and $d^2 x - [\tilde{W}, x] = 0$ for all $x$.
\end{proof}

The following proposition and its proof is an extended version of Proposition \ref{prop:W=0}.
\begin{prop} \label{prop:tildeW=0}
When $X$ is Calabi-Yau, and $\bL$ is graded such that $\deg X_e, \deg X_f, \deg T_g \geq 1$, we have $\tilde{W} = 0$.  Thus $(\Lambda \tilde{Q}, d)$ is a dg algebra.
\end{prop}
\begin{proof}
Since $\deg X \geq 1$ for $X \in \{X_e\} \cup \{X_f\} \cup \{T_g\}$, $\deg x \leq 0$, and hence $\deg \tilde{W}_i \leq 0$.  But $\deg \tilde{W}_i = 2$.  Thus $\tilde{W} \equiv 0$.
\end{proof}

In general it is difficult to write down $(\Lambda \tilde{Q}, d, \tilde{W})$ explicitly since it requires computing all $m_k$'s for $\bL$.  On the other hand, for Calabi-Yau threefolds with grading assumptions (see below), a lot of terms automatically vanish by dimension reason, and $(\Lambda \tilde{Q}, d, \tilde{W})$ can be recovered from the quiver $Q$ with potential $\Phi$ (and $\tilde{W} = 0$).  The resulting dg algebra $(\Lambda \tilde{Q}, d)$ is known as Ginzburg algebra.

\begin{definition}[\cite{Ginzburg}] \label{def:Ginzburg}
Let $(Q,\Phi)$ be a quiver with potential.  The Ginzburg algebra associated to $(Q,\Phi)$ is defined as follows.  Let $\tilde{Q}$ be the doubling of $Q$ by adding a dual arrow $\bar{e}$ (in reversed direction) for each arrow $e$ of $Q$ and a loop based at each vertex of $Q$.  Define the grading on the path algebra $\Lambda \tilde{Q}$ as
$$\deg x_e = 0, \deg x_{\bar{e}} = -1, \deg t_i = -2$$
where $t_i$ corresponds to the loop at a vertex.  Define the differential $d$ by
$$d t_i = \sum_e \pi_{v_i} \cdot [x_e, x_{\bar{e}}] \cdot \pi_{v_i}, \,\,\,\, d x_{\bar{e}} = \partial_{x_e} \Phi , \,\,\,\, d x_e = 0$$
where $\pi_{v_i}$ denotes the constant path at the vertex $v_i$.
Then the dg algebra $(\Lambda \tilde{Q},d)$ is called the \emph{Ginzburg algebra}.
\end{definition}

\begin{prop} \label{prop:Ginzburg}
Suppose $X$ is a Calabi-Yau threefold, $L_i$ are graded Lagrangian spheres (equipped with a Morse function with exactly one maximum point and one minimum point) such that $\deg X$ equals to either one or two for all immersed generators $X$.  Then $(\Lambda \tilde{Q}, d, \tilde{W}=0)$ produced from the above construction is the Ginzburg algebra associated to $(Q,\Phi)$.
\end{prop}

\begin{proof}
Since each branch $L_i$ is a sphere and we take the standard height function to be the Morse function which has exactly one maximum point and one minimum point, $T_i$ are the minimal points of the Lagrangian spheres $L_i$ with $\deg T_i = 3$.  Moreover there is an one-to-one correspondence between $\{X_e\}$ and $\{X_f\}$, and each pair is denoted as $(X_e ,X_{\bar{e}})$.  As a result, the extended quiver $\tilde{Q}$ is the doubling of $Q$ in Definition \ref{def:Ginzburg}.  We have $\deg x_e = 0$, $\deg x_{\bar{e}} = -1$ and $\deg t_i = -2$.  Since all elements in $\Lambda \tilde{Q}$ has $\deg \leq 0$, $d x_e$ which has degree one must vanish.  Moreover $d x_{\bar{e}}$ has degree zero, and hence must be a series in $\{x_e\}$.  Thus the coefficient of $x_{\bar{e}}$ in $m_0^{\tilde{b}}$ is the same as that in $m_0^b$, which is $\partial_{x_e} \Phi$.  As a consequence $d x_{\bar{e}} = \partial_{x_e} \Phi$.  Finally only constant disc can contribute to the coefficient of $T_i$ in $m_0^{\tilde{b}}$.  Moreover since $d t_i$ has degree -1, each term has exactly one $x_{\bar{e}}$ factor.  Thus each term must take the form $x_e x_{\bar{e}}$ or $x_{\bar{e}} x_e$ which is a path from $v_i$ back to itself.  As a result $d t_i = \sum_e \pi_{v_i} \cdot [x_e, x_{\bar{e}}] \cdot \pi_{v_i}$.  (The sign difference between $x_e x_{\bar{e}}$ and $x_{\bar{e}} x_e$ follows from the cyclicity of $(m_2(x_e,x_{\bar{e}}),\one_{L_i})$ and the property of unit $\one_{L_i}$.)
\end{proof}

\section{Extended mirror functor}
In this section we define an $A_\infty$-functor from $\Fuk(X)$ to the dg category of curved dg modules of $(\Lambda \tilde{Q}, d, \tilde{W})$ constructed in the previous section.  In this sense $(\Lambda \tilde{Q}, d, \tilde{W})$ can be regarded as a generalized mirror of $X$.

Curved dg-modules were well studied in existing literature and the definition is as follows.

\begin{definition} \label{def:cdg}
A curved dg-module over a curved dg-algebra $(A,d,F)$ is a pair $(M, d_M)$ where $M$ is a graded $A$-module and $d_M$ is a degree one linear endomorphism of $M$ such that 
\begin{equation} \label{eq:Leibniz}
d_M (am) = a (d_M m)+ (-1)^{\deg m} (da) m 
\end{equation}
and $d_M^2 = F$.
\end{definition}
Note that our convention for $d$ and $d_M$ is ``differentiating from the right", and so the sign is defined as above. The same sign convention applies below.

\begin{definition}
Given two curved dg-modules $(M,d_M)$ and $(N,d_N)$, the morphisms from $M$ to $N$ form a graded vector space $\Hom_A (M,N)$ with a degree-one differential defined as
$$ d_{M,N} f := d_N \circ f - (-1)^{\deg f} (f \circ d_M) $$
with $d_{M,N}^2 = 0$.  Thus the category of curved dg-modules is a dg category.  
\end{definition}

The construction of the functor is similar to the one in Section \ref{sec:FuktoMF}.  Namely, we consider intersections between an object $L$ of $\Fuk(X)$ with the reference Lagrangian $\bL$ and their Floer theory to construct the functor.  Here since we consider extended deformations of $\bL$, we need to take account of extra contributions from the even-degree immersed generators $X_f$ and also the critical points $T_g$.  Also we need to be careful with signs since we use deformations of all degrees (instead of just odd degrees).

Let $U$ be a Lagrangian which intersects transversely with $\bL$.  Given $Y \in U \cap \bL$, as in Section \ref{sec:FuktoMF}, we want to define the differential acting on $Y$ to be
$$ m_1^{\tilde{b},0} (Y) = \sum_{n=0}^\infty m_{n+1} (\tilde{b},\ldots,\tilde{b},Y) = \sum_n \sum_{(X_1,\ldots,X_n)} (-1)^{\sum_{j=1
}^n |x_j|} x_n \ldots x_1 m_{n+1} (X_1,\ldots,X_n,Y).  $$
Due to the sign difference in dg-category and $\AI$-category, we define the dg-module structure as follows.
\begin{definition}
Let $P_v := \Lambda \tilde{Q} \cdot \pi_v$, the left $\Lambda \tilde{Q}$-module consisting of all paths of $\tilde{Q}$ beginning at a vertex $v$ (we read from right to left).  Given a Lagrangian $U$ (which intersects transversely with $\bL$), the corresponding module is defined as $M = \cF(U) = \bigoplus_i \bigoplus_{Y \in U \cap L_i} P_{v_i} \langle Y \rangle$.  Then the differential $d_M$ is defined by 
$$d_M \, Y = (-1)^{\deg Y} m_1^{\tilde{b},0} (Y)$$ 
and extending to the whole module by linearity and Leibniz rule \eqref{eq:Leibniz}.
\end{definition}

\begin{prop} \label{prop:fctor-obj}
$d_M^2 = \tilde{W}$ and hence $(M,d_M)$ is a curved dg-module.
\end{prop}
\begin{proof}
We start by noting that $(d_M)^2 \neq (m_1^{\tilde{b},0})^2$ as we need to take $d$ of the coefficient by Leibniz rule
in the case of $(d_M)^2$. We need to prove $d_M(d_M(Y)) = \tilde{W} \cdot Y$ for any $Y \in \bL \cap U$.
We will use the following deformed curved $\AI$-bimodule equation:
$$(m_1^{\tilde{b},0})^2 + m_2^{\tilde{b}, \tilde{b}, 0}(m_0^{\tilde{b}}(1),Y) =0,$$
which implies that
\begin{eqnarray*}
-(m_1^{\tilde{b},0})^2 &=&  m_2(\tilde{W} \be, Y) + m_2^{\tilde{b}, \tilde{b}, 0}(\sum_X (dx)X, Y)\\
 &=&  \tilde{W}\cdot Y  + \sum_{k,X_i,X} m_{k+2}(x_1X_1, \cdots, (dx)X,\cdots, x_kX_k,Y) \\
 &=&   \tilde{W}\cdot Y + \sum_{k,X_i,} d(x_k \cdots x_1)m_{k+1}(X_1,\cdots, X_k,Y)
 \end{eqnarray*}
 where the last identify is explained in \eqref{cvmd1}. Now we can prove the desired identity as follows.
 \begin{eqnarray*}
 d_M^2(Y) &=&  d_M ( (-1)^{|Y|} m_1^{\tilde{b},0}(Y)) \\
                &=&  (-1)^{|Y|} (-1)^{|Y|+1 } (m_1^{\tilde{b},0})^2(Y) - \sum_{X_i,k}d(x_k\cdots x_1)m_{k+1}(X_1,\cdots,X_k,Y) \\
                &=&\tilde{W} \cdot Y. 
  \end{eqnarray*}

\end{proof}

In particular, when $\tilde{W} = 0$ Lagrangians are transformed to dg-modules over $\Lambda \tilde{Q}$, which can be identified as dg representations of $\tilde{Q}$.  By Proposition \ref{prop:tildeW=0} this is usually the case in Calabi-Yau situations.
Typically these representations have infinite dimensions since they consist of all paths of $\tilde{Q}$.  Taking the cohomologies of $d_M$, one obtains (graded) quiver representations.  

The functor on (higher) morphisms is defined as follows.  Given unobstructed Lagrangians $U_1,\ldots,U_k$ which intersect transversely with each other and also with $\bL$, we define a map $\Hom(U_1,U_2) \otimes \ldots \otimes \Hom(U_{k-1},U_k) \to \Hom(M_{U_1},M_{U_k})$ by sending $Y \in M_{U_1}$ to 
$$m^{\tilde{b},0,\cdots,0}_k(Y,\cdot,\ldots,\cdot) = \sum_{l=0}^\infty m_{k+l}(\tilde{b},\ldots,\tilde{b},Y,\cdot,\ldots,\cdot) = \sum_n \sum_{(X_i)} (-1)^{\sum_j |x_j|} x_n \ldots x_1 m_{k+n}(X_1,\ldots,X_n,Y,\cdot,\ldots,\cdot) \in M_{U_k}.$$  
Again the signs have to be modified to make it to be an $A_\infty$-functor.

\begin{definition}
Given unobstructed Lagrangians $U_1,\ldots,U_k$ which intersect transversely with each other and also with $\bL$, we define a map $\Hom(U_1,U_2) \otimes \ldots \otimes \Hom(U_{k-1},U_k) \to \Hom(M_{U_1},M_{U_k})$ by
$$\cF(A_1\otimes \ldots \otimes A_{k-1})(Y) := (-1)^{(|Y|+1) \cdot \sum_{i=1}^{k-1} (|A_i|+1)} m_k^{\tilde{b},0,\cdots,0}
(Y, A_1,\cdots, A_{k-1}).$$
\end{definition}

\begin{theorem}
The above definition of $\cF$ gives an $A_\infty$-functor $\Fuk(X) \to \mathrm{cdg}(\tilde{W})$
where $\mathrm{cdg}(\tilde{W})$ is the dg category of curved dg modules of $(\Lambda \tilde{Q}, d, \tilde{W})$.
\end{theorem}
\begin{proof}
The proof is almost the same as that of Theorem 2.18 \cite{CHL13} together with an argument that is used in the proof of Proposition \ref{prop:fctor-obj},
and we leave the proof as an exercise.
\end{proof}

%
%
%


\chapter{Mirror construction for punctured Riemann surface} \label{sec:Rs}
In \cite{Bocklandt}, Bocklandt considered a pair of ``mirror" dimers $Q$ and $Q^\vee$ equipped with perfect matchings, which are embedded in two different Riemann surfaces $\Sigma$ and $\Sigma^\vee$ (with the number of punctures $\geq 3$). He combinatorially defined a category from $Q$ which was proven to be isomorphic to the wrapped Fukaya category of $\Sigma - Q_0$, and found an embedding into the category of matrix factorizations associated to $Q^\vee$. In this way, Bocklandt generalized the construction of Abouzaid et al. \cite{AAEKO} of the homological mirror symmetry  for punctured spheres.

In this chapter, inspired by the work of Bocklandt, we apply our construction developed in Chapter \ref{sec:MFseverallag} to punctured Riemann surfaces and recover   the result of \cite{Bocklandt}, and also generalize them.  In particular, our approach gives a geometric explanation of the mysterious combinatorial construction of a mirror dimer in \cite{Bocklandt}. This gives a nice application of the theory developed in this paper.

\begin{definition}\label{defnquiv}
Let $\Sigma$ be a closed Riemann surface, with finitely many points $Q_0 = \{p_i\}$. We consider a polygonal decomposition of $(\Sigma,Q_0)$ with vertices only at $Q_0$. Choosing orientations of all edges defines a quiver $Q$
with the set of arrows  $Q_1$. The complement $\Sigma \setminus Q$ consists of polygons, and the closure of these open polygons are called faces of $Q$.
We assume that every vertex has valency at least 2.
Also  by $Q$ we denote  the quiver or  the associated polygonal decomposition of $\Sigma$. 
\end{definition}
\begin{theorem}
We can take a family of Lagrangians $\bL$ canonically for a given polygonal decomposition $Q$ of $\Sigma$ to obtain a noncommutative mirror LG model $(\cA,W)$ so that the following holds:
\begin{enumerate}
\item There exists a graded quiver $Q^\mathbb{L}$ (whose arrows are $\Z_2$-graded) with a spacetime superpotential $\Phi$ such that
$\cA =\Jac(Q^\bL,\Phi)$ and we have a $\Z_2$-graded $A_\infty$-functor 
$$\cF_{\mathbb{L}} : \Fuk^\wrap (\Sigma \setminus Q_0) \to \MF (\Jac(Q^\bL,\Phi),W).$$
\item If $Q$ is a dimer, then  the quiver $Q^\mathbb{L}$ is also a dimer (with degree zero arrows)
embedded in a Riemann surface $\Sigma^\vee$. Here $\Sigma^\vee$ can be constructed from the spacetime superpotential
$\Phi$.
\item If $Q$ is a zig-zag consistent dimer , then $\cF_{\mathbb{L}}$ defines an $\Z_2$-graded derived equivalence 
 $$ \cF_{\mathbb{L}} : D\Fuk^\wrap (\Sigma \setminus Q_0)
\to \mathrm{Hmf}(\Jac(Q^\bL),W) \subset \mathrm{H}\MF (\Jac(Q^\bL,\Phi),W).$$
\item When $Q$ has a perfect matching, then  $\cF_{\mathbb{L}}$ defines a $\Z$-graded functor.
\item When $Q$ is zig-zag consistent dimer with a perfect matching,  $\cF_{\mathbb{L}}$ defines a $\Z$-graded 
derived equivalence to
$\mathrm{Hmf}(\Jac(Q^\bL),W)$. This $\Z$-graded functor agrees with that of Bocklandt \cite{Bocklandt}.
\end{enumerate}
\end{theorem}
\begin{remark}
The previous elliptic curve examples can be regarded as an analogue of this construction for closed surfaces without punctures.
\end{remark}
Here, $ \Fuk^\wrap  (\Sigma \setminus Q_0)$ is the wrapped Fukaya category of \cite{AS10}, and 
the definitions of a dimer model, and the dg-category $\mathrm{mf}(\Jac(Q^\bL),W)$ will be given in Definition \ref{def:dimer},  \ref{def:mfc}. Proof of this theorem will occupy the rest of the chapter.

The result (5) was proved by Bocklandt \cite{Bocklandt}  by computing enough of both wrapped Fukaya category (of generating non-compact Lagrangians) and matrix factorizations 
and proving a uniqueness theorem of $\AI$-structures using such computations and $\Z$-grading. Without a perfect matching condition of a dimer (that is a $\Z$-grading), it is difficult to mach $\AI$-structures of both sides.  On the other hand, in our formulation an $A_\infty$-functor is given canonically, once we determine $\bL$ and its Maurer-Cartan equations.

The result (3), (4) are also partly based on the work of Bocklandt.
For (3) we use his explicit calculation of basis on $B$-side for zig-zag consistent dimers to show that it is an equivalence
 and for (4), we use the $\Z$-grading defined by Bocklandt's explicit choice of vector fields on punctured Riemann surfaces.
%

\begin{remark}
Throughout the paper, the endomorphism quiver $Q^\bL$ (Definition \ref{def:mirrorquiverseveral}) was often denoted as $Q$ for simplicity, but
just in this chapter, by $Q$ we denote a quiver embedded in a Riemann surface. We will see that $Q^\bL$ equals the mirror dimer $Q^\vee$ (as dimers).
\end{remark}

\section{Reference Lagrangian $\mathbb{L}$ from a polygonal decomposition}
The reference Lagrangian $\bL$ for a given polygonal decomposition $Q$ of $\Sigma$ will be
obtained from the inscribed polygons to every faces of $Q$. Namely, for each face $P$ of $Q$, we connect the midpoints of neighboring edges of $P$ to obtain $\check{P}$ which is inscribed to $P$. See Figure \ref{fig:rminusrev} (a). 
\begin{definition}\label{defnp'}
We may choose $\check{P}  \subset P$ for each face $P$ of $Q$ such that when two faces $P_1,P_2$ share an edge $e$,
then the edges of $\check{P}_1, \check{P}_2$ connect smoothly at the midpoint of $e$ (locally forming a triple intersection of smooth embedded curves). 
 In this way, edges of $\check{P}$ for faces $P$ of $Q$ define a family of (possibly immersed) curves  $\bL = \{L_1,\cdots, L_k\}$ in $\Sigma \setminus Q$. Each $L_i$ may be also obtained by starting from a mid-point of an edge of $Q$ and traveling to the midpoint of
the left  neighboring edge of a face, and next the right neighboring edge of the following face and continuing this process (see Figure \ref{fig:traversing_lag} (b)). 
Therefore, we call each $L_i$ a {\em zigzag } Lagrangian. 
\end{definition}

\begin{figure}[h]
\begin{center}
\includegraphics[height=2in]{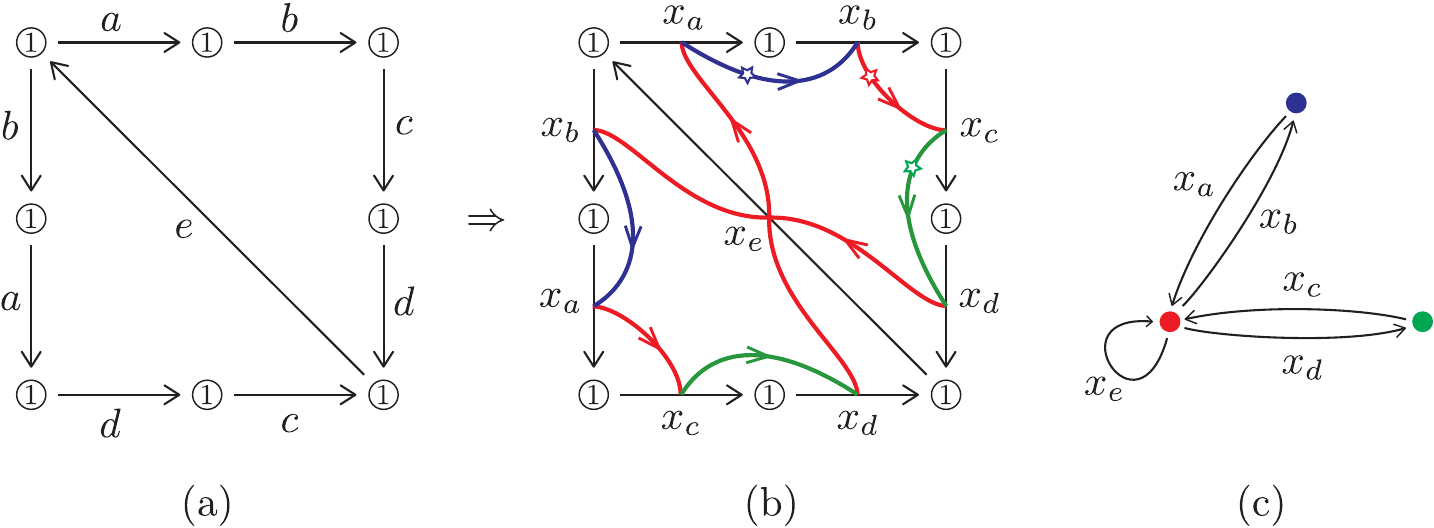}
\caption{(a) dimer $Q$ $\quad$  (b) $\mathbb{L} = \{L_1,L_2,L_3\}$ $\quad$  (c) $Q^{\mathbb{L}}$ }\label{fig:rminusrev}
\end{center}
\end{figure}

Note that  $\Sigma \setminus Q$ is an exact symplectic manifold, and we can also make $\bL$ exact.
\begin{lemma} $\mathbb{L}$ can be made an exact Lagrangian by a suitable smooth isotopy.
\end{lemma}
\begin{proof}
Let $\alpha$ be the fixed Liouville form, $d\alpha = \omega$.  For every $i$ we want to suitably deform $L_i$ such that $\int_{L_i} \alpha = 0$.
Near a puncture, $\alpha$ is locally given as $r d\theta$ where $r$ goes to $\infty$ approaching the puncture.  Hence the integral of $\alpha$ along a circular arc in counterclockwise (clockwise resp.) direction increases to $\infty$ (decreases to $-\infty$ resp.) when the arc moves closer to the puncture.  Since $L_i$ is a zigzag Lagrangian, it contains arcs in both directions.  Hence we can deform $L_i$ near a puncture to achieve $\int_{L_i} \alpha = 0$.  
\end{proof}

Let us recall the definition of a dimer.

\begin{definition}\label{def:dimer}
A quiver in the Definition \ref{defnquiv} is called a {\em dimer model} if it is embedded in a surface $\Sigma$, and for each face $P$ of $Q$, the boundary arrows of $Q$ are either all along with or all opposite to the boundary orientation of $\partial P$.  Hence the set of arrows bounding each face can be regarded as a cycle in $Q$ which we call a \emph{boundary cycle}. 
\end{definition}

\begin{definition}
The set of boundary cycles of $Q$ is decomposed into two sets $Q_2^+$ (positive cycles) and $Q_2^-$ (negative cycles).  We will call a face of $\Sigma$ bounded by a cycle in $Q_2^+$ (resp. $Q_2^-$) a positive (resp. negative) face.
\end{definition}
Each arrow in a dimer model $Q$ is contained in exactly one positive cycle and one negative cycle.

For a dimer model $Q$,  the inverse image $\tilde{Q}$ in the universal cover $\tilde{\Sigma}$ of $\Sigma$ gives rise to a quiver
in an obvious way. For a given arrow $\tilde{a}$ of $\tilde{Q}$, a zig ray (zag ray, resp.) is defined  as an infinite path in $\tilde{Q}$
$$ \cdots  \tilde{a}_2 \tilde{a}_1  \tilde{a}_0$$
where $\tilde{a}_0 =\tilde{a}$, $ \tilde{a}_{i+1} \tilde{a}_{i}$ is a part of a positive boundary cycle for even $i$ and is a part of a negative boundary cycle for odd $i$.  A zig or zag ray projects down to a cycle in the path algebra of $Q$ which we will call a zigzag cycle. The following lemma is easy to check. (See (a) of Figure \ref{fig:traversing_lag}.)

\begin{lemma}
For a dimer model $Q$, there is an one-to-one correspondence between zig-zag Lagrangian of $Q$ and zig-zag cycle in the path algebra of $Q$.
\end{lemma}

\begin{figure}
\begin{center}
\includegraphics[height=1.8in]{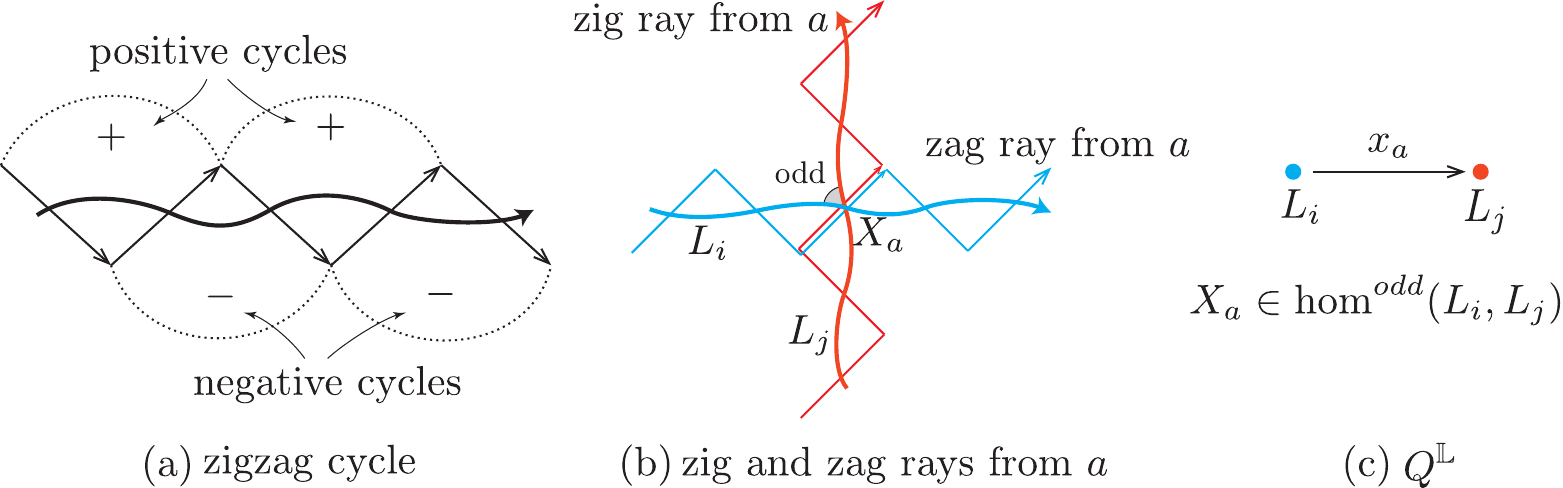}
\caption{Zigzag Lagrangians}\label{fig:traversing_lag}
\end{center}
\end{figure}

\section{Floer theory for $\bL$ and its Maurer-Cartan space for a dimer $Q$}
We will apply our formalism to the set $\bL$ of zigzag Lagrangians. Let us first discuss the case of a dimer $Q$, in which we can directly apply the result of
Chapter \ref{sec:MFseverallag}. For a quiver $Q$ which is not a dimer, we will consider somewhat extended deformation
theory to construct $Q^\bL$ later.

Lagrangians in $\bL$ are compact exact Lagrangians in  $\Sigma \setminus Q_0$, and we have an $\AI$-algebra
$\CF^\bullet(\bL, \bL)$  following \cite{Se}. Here, Lagrangians are oriented in a canonical way, following the orientation of
the corresponding zigzag cycles. 

\begin{lemma}\label{lem:dimmir1}
The endomorphism quiver $Q^\bL$ (in the Definition \ref{def:mirrorquiverseveral}) for a dimer $Q$ is described as follows. Its vertices $Q^\bL_0 = \{ v_1,\cdots, v_k\} $ corresponds to zigzag Lagrangians 
 $\bL= \{ L_1,\cdots, L_k\}$. 
There exist an one-to-one correspondence between the set of arrows $Q^\bL_1$ and $Q_1$, given by the following rule.
For an arrow $e \in Q_1$, there exist a (dual) arrow $x_e \in Q^\bL_1$ of degree zero from a zag ray of $e$ to the zig ray of $e$.
\end{lemma}
\begin{proof}
An arrow of $Q^\bL_1$ is given by an odd degree intersection between $L_i$ and $L_j$. 
Recall that the generator $p$ in $\CF^\bullet(L_i,L_j)$ is odd (resp. even) if  at $T_p\Sigma$, one can rotate by an angle less than $\pi$, the oriented unit tangent vector of $L_i$  in a counter-clockwise (resp. clockwise) direction  to obtain the oriented unit tangent vector of $L_j$. Since $Q$ is a dimer, for each edge $e$, we have an odd degree morphism $X_e$
from the zag ray $L_i$ of $e$ to the zig ray $L_j$ of $e$. See the Figure \ref{fig:traversing_lag} (b).
\end{proof}
From now on, we will identify $Q^\bL_1$ and $Q_1$ using this Lemma.

Due to the punctures, and to our choice of Lagrangian $\bL$, the following lemma is immediate
\begin{lemma}
The only non-trivial holomorphic polygon with boundary on $\bL$ is
given by the polygon $\check{P}$ inscribed to each face $P$ of $Q$ (see Definition \ref{defnp'}).
\end{lemma}

Following Chapter \ref{sec:MFseverallag}, we set the Maurer-Cartan element $b = \sum_{e \in Q_1} x_e X_e$ and obtain
$$ m_0^b = \sum_i W_i \one_{L_i} + \sum_{e} P_{e} \bar{X}_{e} $$
where $P_{e}$ is counting discs with insertions of $b$'s with an output  $\bar{X}_{e}$, and $W_i$ is counting discs $U$ with output to be the minimum point of $L_i$.  By definition $\cA_\bL$ is the (completed) path algebra of $Q^\bL$ quotient by $P_{e}$ and $W_\bL = \sum_i W_i$.
Since we work in the exact setting, we can work over $\C$-coefficient rather than the Novikov ring $\Lambda_0$.  Indeed completing the path algebra is not necessary in this case.

From now on,  a holomorphic polygon $\check{P}$ contributing to $m_0^b$ will be said to be {\em positive}, denoted as $\check{P}_+$ if the induced boundary orientation match with that of $\mathbb{L}$, and {\em negative}, denoted as $\check{P}_-$ otherwise.

We will put non-trivial spin structure on $\bL$ to match our construction to the combinatorial mirror description of Bocklandt \cite{Bocklandt}.

\begin{definition}\label{def:nsp}
Assign a spin structure on $\bL$ by placing a marked point on one of the edges of every negative  $k$-gon $\check{P}$ for odd $k$.
\end{definition}

Recall that a marked point may be thought of as a place where trivialization of tangent bundle is twisted. Thus the sign of a holomorphic disc with boundary on $\bL$ passing through $k$ marked points (for spin structures) will have an additional sign twist by $(-1)^k$.

\begin{lemma}\label{lem:spcom}
The coefficients $P_{e}$ can be defined as partial derivative $\partial_e \Phi$ of the spacetime superpotential $\Phi$. 
With the spin structure given in Definition \ref{def:nsp}, we have $$\Phi = \sum_{\check{P}_+} (\textrm{boundary word of} \;\check{P}_+)
- \sum_{\check{P}_- } (\textrm{boundary word of}\; \check{P}_-)\big)$$
Here $|\check{P}|$ is the number of vertices in the polygon $\check{P}$, and  a boundary word of $\check{P}$ is recorded in clockwise direction starting from any point (it is well-defined in the cyclic quotient).
\end{lemma}
\begin{proof}
Since the only non-trivial holomorphic polygons bounded by $\bL$ come from the polygons $\check{P}$, it
only remains to check the sign. A positive polygon contributes with a positive sign, and the negative $k$-gon
contributes with $(-1)^{k-1}$ according to the sign convention of \cite{Se}, but we added further twisting by $(-1)^k$.
Therefore, any negative polygon $\check{P}_-$ always contribute with the negative sign. 
\end{proof}

For example, the spacetime potential $\Phi$ for the dimer $Q$ in Figure \ref{fig:rminusrev} is given by
\begin{equation*}
\Phi =  x_a x_b x_e x_c x_d  - x_a x_e x_d x_c x_b.
\end{equation*}

The Lagrangian Floer potential $W$ can be also computed combinatorially as follows.
\begin{lemma}\label{lem:potrm}
The potential $W$ as an element in $\cA_\bL = \Jac(Q^\bL, \Phi)$ can be written as 
$$W = \sum_{i=1}^k W_i$$
where $W_i$ is the boundary word of the unique polygon $\check{P}$  passing through  the minimum $T_i \subset L_i$.
Here, the word is  recorded clockwise starting from $T_i$ (but $T_i$ itself is not recorded) and
with the spin structure given in Definition \ref{def:nsp}, all signs of $W_i$'s are positive.
\end{lemma}
\begin{proof}
It is enough to show that the signs of all holomorphic polygons contributing to $W_\bL$ are positive.  
A positive  polygon contribute positively. A negative $k$-gon contributing to $m_k$ with the output $\one_\bL$
contributes as $(-1)^k$ with the standard spin structure on $\bL$. Hence, the spin structure of 
 Definition \ref{def:nsp} additionally contributes $(-1)^k$ to the negative $k$-gon counting, and hence 
 even the negative $k$-gon is counted positively.
\end{proof}

The Lagrangian Floer potential for the example drawn in Figure \ref{fig:rminusrev} can be computed as
$$W =  x_c x_d x_e  x_a x_b +  x_b x_c x_d x_e x_a + x_d x_e x_a x_b x_c,$$
where the minimum $T_i$ is represented by $\star$ in the picture.

\section{Relation to the mirror dimer model $Q^\vee$}
We show that the combinatorial mirror dimer construction of Bocklandt \cite{Bocklandt} can be recovered from our construction $(Q^\bL, \Phi, W)$ (with the spin structure given in Definition \ref{def:nsp}).

We first recall the definition of the mirror quiver $Q^\vee$.
\begin{definition}[\cite{Bocklandt}] \label{def:mirrorquiver}
The dual quiver $Q^\vee$ of $Q$ is defined by the following data:
\begin{enumerate}
\item The vertices of $Q^\vee$ are zig-zag cycles in $Q$.
\item The sets of arrows for $Q$ and $Q^\vee$ are the same. $a \in Q_1$ gives the arrow of $Q^\vee$ with $h(a)$ the cycle coming from the zig ray and $t(a)$ that from the zag ray. 
\end{enumerate}
\end{definition}
From Lemma \ref{lem:dimmir1}, we have
\begin{cor}
The dual quiver $Q^\vee$ can be identified with $Q^\bL$.
\end{cor}

The dual dimer $Q^\vee$ embeds into a surface $\Sigma^\vee$.  Combinatorial construction of $\Sigma^\vee$ was given by physicists Feng-He-Kennaway-Vafa \cite{FHKV}.  

\begin{definition} \label{def:dualdimer}
The dual dimer $Q^\vee$ embedded in a surface $\Sigma^\vee$ consists of the following data:
\begin{enumerate}
\item The positive faces of $Q^\vee$ coincide with those of $Q$ in $\Sigma$;
\item The negative faces of $Q^\vee$ are obtained from the negative faces of $Q$ by flipping and reversing the order of their arrows. (See Figure \ref{fig:dualdimer}.)
\end{enumerate}
i.e. $\Sigma^\vee$ may be obtained as follows. First cut $\Sigma$ along the arrows of $Q$, and flip the negative faces, and reverse the direction of arrows in the flipped negative faces. One finally glues back the faces (the original positive faces and the flipped negative faces) by matching arrows 
(forgetting the information of vertices). 
This defines an involution on the set of dimer models.
\end{definition}

\begin{figure}[h]
\begin{center}
\includegraphics[height=1.2in]{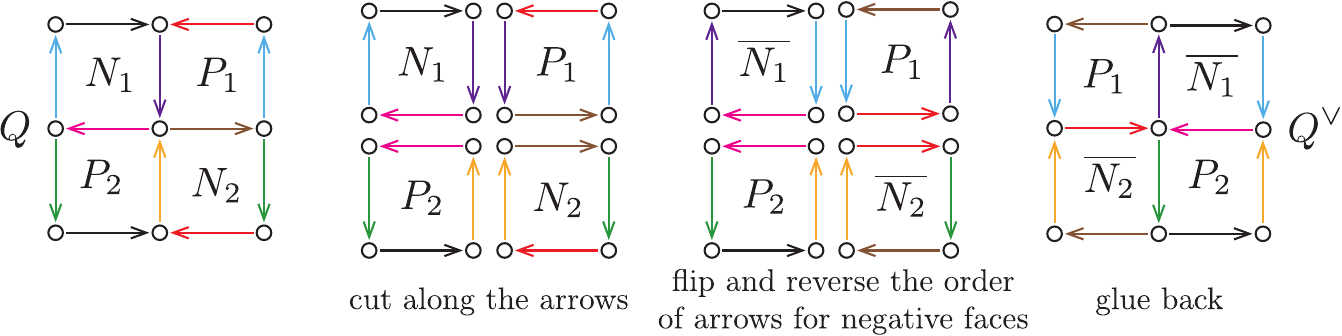}
\caption{Combinatorial construction of dual dimer}\label{fig:dualdimer}
\end{center}
\end{figure}

Let us explain how to obtain $\Sigma^\vee$ in our geometric construction. Roughly speaking, we patch together holomorphic
polygons $\check{P}$'s bounded by $\bL$ to obtain $\Sigma^\vee$.
But we need to attach $\check{P}$ after suitable modification.

Let $F$ be a face of $Q$, and denote by $\check{P}_F$ the holomorphic polygon in $F$ bounded by $\bL$. 
Each polygon $\check{P}_F$ will be replaced by a dual polygon $F^\vee$ in the following way.
An edge of $\check{P}_F$ is given by  Lagrangian $L_i$, and hence corresponds to a vertex $v_i$ (of the quiver $Q^\bL$).
A vertex of $\check{P}_F$ is given by the intersection point of $CF(L_i, L_j)$ and hence can be replaced by an edge $e_{ij}$ (
which are arrows of the quiver $Q^\bL$).
This process provides the dual polygon $F^\vee$, which can be attached to the quiver $Q^\bL$. If we attach $F^\vee$
to $Q^\bL$ for each face $F$ of $Q$, we obtain the surface $\Sigma^\vee$.

Note that this dual polygon $F^\vee$ has the same shape as the face $F$, in fact, if $F$ is a positive face, 
edges of $F$ can be identified with edges of $F^\vee$ (but vertices cannot be identified). If $F$ is a negative face,
edges of $F$ can be identified with edges of $F^\vee$ but with reverse orientation. Since branches of $\bL$ crosses 
at the edges the polygonal decomposition $Q$, we need to glue the modifications of the holomorphic polygons via flipping
negative faces, which identifies vertices and edges of the neighboring tiles.

Formally, this procedure of attaching 2-cells can be naturally read from the spacetime superpotential $\Phi$ (see \eqref{def:Phi}),
which records the holomorphic polygons $\check{P}$, and  the sign determines the orientation of the cell.
\begin{lemma}
For each positive  term $(r_+)_{cyc}$ in $\Phi$, we glue a polygonal face to $Q^\bL$ along the edges of $(r_+)$.
For each negative term $- (r_-)_{cyc}$ in $\Phi$, we glue a polygonal face to $Q^\bL$ along the edges of $(r_-)$ in the
reverse order (See Figure \ref{fig:rminusrev} (b)). Then the resulting Riemann surface can be identified with $\Sigma^\vee$.
\end{lemma}
We leave the proof of the lemma as an exercise. It is interesting that the ``space"-time superpotential $\Phi$ contains additional
information about the mirror ``space''.

From Lemma \ref{lem:spcom}, it is easy to  identify combinatorial Jacobi algebra of the mirror dimer $Q^\vee$
with the Jacobi algebra of $Q^\bL$ from $\Phi$, where the combinatorial Jacobi algebra of a dimer is defined as follows.

\begin{definition}
For a dimer $Q$, $\Jac(Q)$ is defined to be the quotient of the path algebra by the ideal generated by $r_e := r_{e,+} - r_{e,-}$ (for every arrows $e$) where $e \cdot r_{e,+}$ is the unique positive cycle in $Q$ containing $e$, and $e \cdot r_{e,-}$ is the unique negative cycle containing $e$. 
\end{definition}

\begin{cor}\label{cor:jac}
With the spin structure given in Definition \ref{def:nsp}, we have 
$$\Jac (Q^\bL, \Phi) = \Jac(Q^\vee).$$
\end{cor}
Hence, our formalism explains why Jacobi relation appears in the mirror quiver algebra in \cite{Bocklandt}.

\begin{lemma}
With the spin structure given in Definition \ref{def:nsp},
the potential $W$ of Lemma \ref{lem:potrm} agrees with a combinatorial central element $l$, defined in \cite{Bocklandt} 
\begin{equation}\label{eqn:deflcv}
l:= \sum_{v \in Q^\vee_0} C_v
\end{equation}
where for each $v \in Q^\vee_0$, $C_v$ is an a cycle starting from $v$. 

Therefore,  the generalized mirror $(\cA_\bL,W_\bL)$ coincides with the Landau-Ginzburg model $(\Jac(Q^\vee),l)$.
\end{lemma}
\begin{proof}
A vertex $v \in Q^\vee_0$ corresponds to a Lagrangian $L_v \in \bL$, and a cycle $C_v$ corresponds to a boundary word of a polygon $\check{P}$ with an edge in $L_v$. But there may be several polygons $\check{P}$ with an edge in $L_v$. Let $P_1, P_2$ be the two faces of $Q$ sharing an edge $a$, and suppose that $L_v$ passes through the mid-point of $a$. See figure \ref{fig:rminusrev}. Then it is easy to check that the boundary word of $\check{P}_1$
and that of $\check{P}_2$ are the same up to Jacobi relation of $\Phi$. Hence, by applying this fact several times if necessary, we may assume that $\check{P}$ that corresponds to a cycle $C_v$ meet the minimum $T_v$ of the Morse function on the Lagrangian $L_v$. Therefore $W_\bL$ can be identified with $l$.
\end{proof}

\section{The case of a non-dimer $Q$}

The construction in this case  will be a generalized version of that in Chapter \ref{sec:MFseverallag}, as we also consider even degree immersed generators like in the case of extended mirror functors, but we will not take the full extended directions.
We will take only half of immersed generators of $\CF^\bullet(\bL, \bL) = \oplus_{i,j} \CF^\bullet(L_i, L_j)$ to form 
a formal deformation space.

Let us first choose any orientation for each $L_i$ of $\bL$. 
An intersection point $p \in L_i \cap L_j$ at the midpoint of an arrow $e$ of $Q_1$ gives rise to two generators $X_e$ (odd degree ) and $\bar{X}_e$ (even degree)
in $\CF^\bullet(\bL,\bL)$.

\begin{definition}
Consider a face $P$ of $Q$, and a corresponding polygon $\check{P}$ bounded by $\bL$.
For each vertex $p$ of $\check{P}$, we choose an immersed generator $Y_e$ in $\{ X_e, \bar{X}_e\}$ using the polygon $\check{P}$:
here $Y_e$ corresponds to the intersection $\CF^\bullet(L_a,L_b)$ so that the order of the branch of $L_a, L_b$ at $p$ is
counterclockwise with respect to $\check{P}$. We denote by $\bar{Y}_e$ the other generator in $\{X_e, \bar{X}_e\}$.
If $Q$ is a dimer, we have $Y_e=X_e$.
\end{definition}

%

%
%
%
%

 We will only turn on $Y_e$  as a formal deformation direction, since $\bar{Y}_e$ do not play any significant role in the Floer theory of $\bL$. Since $Y_e$ is $\Z_2$-graded, the quiver $Q^\bL$ is $\Z_2$-graded.

\begin{definition}
We define the graded quiver $Q^\bL$ as follows. Its vertices $Q^\bL_0 = \{ v_1,\cdots, v_k\} $ corresponds to zigzag Lagrangians 
 $\bL= \{ L_1,\cdots, L_k\}$.
For each $Y_e$, we assign an (dual) arrow $x_e \in Q^\bL_1$ from the zag ray to the zig ray of $e$ of degree $$deg(e_p) = 1-deg(Y_p).$$
Here, zag ray (resp.  zig ray ) of an arrow $e$ is the Lagrangian which pass through the midpoint of the arrow $e$
and travel to the right (resp. left) side of the arrow.
\end{definition}

The following lemma holds by the definition.

\begin{lemma}
Since each midpoint of an edge in the original quiver $Q$ is an intersection point, 
there is an obvious one-to-one correspondence between the arrow set of $Q$ and $Q^\bL$,
\end{lemma}

\begin{definition}
We set $$b = \sum_{e \in Q_1} x_e Y_e.$$ As the degree of $x_e = 1-deg(Y_e)$,
$b$ has total degree one.
\end{definition}
Now, we may proceed as in \ref{sec:MFseverallag}, to obtain
$$ m_0^b = \sum_i W_i \one_{L_i} + \sum_{e} P_{e} \bar{Y}_{e}.$$
By definition $\cA_\bL$ is the (completed) path algebra of $Q^\bL$ quotient by $P_{e}$ and $W_\bL = \sum_i W_i$.
The computation of $\Phi$ and the central element $W_\bL$ can be carried out in a similar way as in the case of a dimer.
But the main difference from the case of dimers is that first, arrows in $Q^\bL$ may have non-trivial grading, and second
$\Phi$ and $W_\bL$ can be computed as in the Lemma \ref{lem:spcom}, \ref{lem:potrm}, but signs in the formula will not be simple  as dimers. 

\begin{figure}[h]
\begin{center}
\includegraphics[height=1.6in]{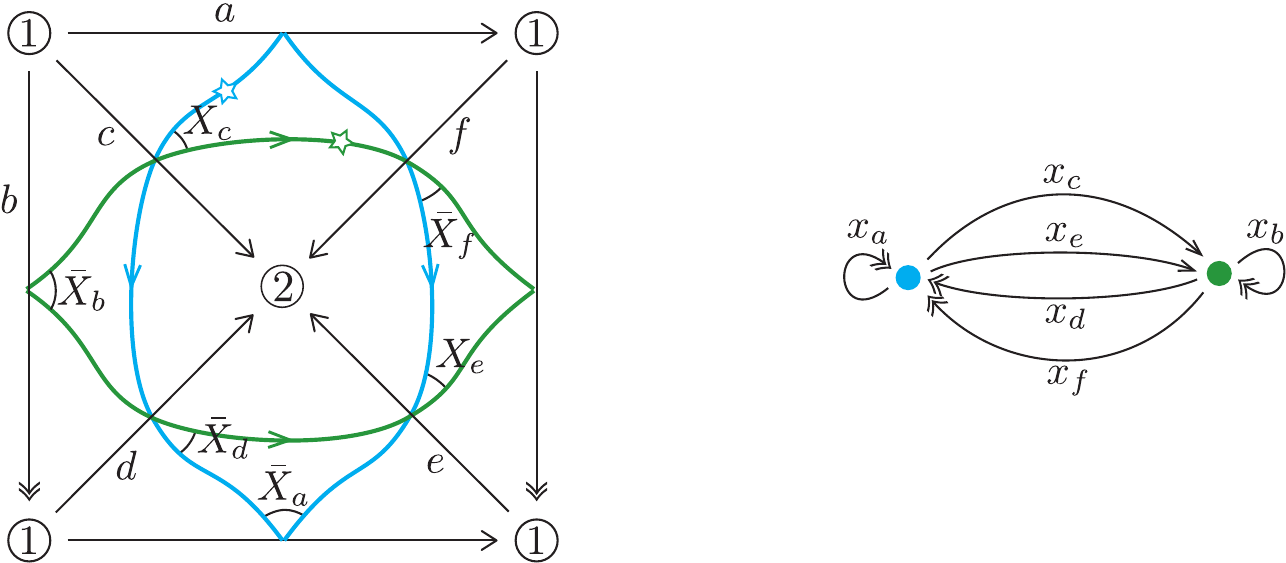}
\caption{Example of non-dimer}\label{fig:nondimer}
\end{center}
\end{figure}

In the non-dimer example of Figure \ref{fig:nondimer}, we turn on the half of immersed generators $\left\{\bar{X}_a, \bar{X}_b, X_c, \bar{X}_d, X_e, \bar{X}_f \right\}$ as formal deformation. Thus, four variables $x_a, x_b, x_d, x_f$ have odd degrees, and the other two variable $x_c$ and $x_e$ have even degrees. One can directly compute from the picture that
\begin{equation*}
\begin{array}{ll}
\Phi = &x_a x_f x_c  + x_c x_d x_b+x_e x_a x_d+x_e x_f x_b\\
W_\bL =& x_a x_f x_c + x_c x_a x_f
\end{array}
\end{equation*}
where $\star$ in the figure represents the location of minimum $T_i$. Here $\Phi$ defines a cyclic potential (in the graded sense).

By applying the (generalized) construction of Chapter \ref{sec:MFseverallag}, we obtain a $\Z_2$-graded $\AI$-functor $\mathcal{F}_\bL$ from the wrapped Fukaya category to the dg-category of matrix factorizations of $W_\bL$, and this finishes the proof of (1).

\section{Perfect matching and $\Z$-grading}

%
%

 We will  show that a {\em perfect matching} on $Q$ induces a $\mathbb{Z}$-grading on $\mathbb{L}$,
 and this also induces the $\Z$-grading for categories and functors.
 \begin{definition}
A perfect matching for a dimer $Q \subset \Sigma$ is a subset $\mathcal{P}$ of $Q_1$ such that every boundary cycle of a polygon contains exactly one arrow in $\mathcal{P}$.
\end{definition}
From a perfect matching, \cite{Bocklandt} constructed a nowhere-vanishing vector field $\mathcal{V}$ on $\Sigma \setminus Q_0$ which essentially plays the role of a Calabi-Yau volume form. 
Let $e$ be the unique edge of a positive face $F^+$ which belongs to $\mathcal{P}$. Over $F^+$, the vector field $\mathcal{V}$ is defined in such a way that its flow trajectories are loops starting and ending at $h(e)$ and turning in counterclockwise direction.  Over the negative face $F^-$ which has $e$ as the unique edge in $\mathcal{P}$, the flow trajectories are loops starting and ending at $t(e)$ and turning in clockwise direction.  See (a) of Figure \ref{fig:refexactness} for the illustration.

From the vector field $\mathcal{V}$, any Lagrangian $L$ (which is simply a curve) in $\Sigma \setminus Q_0$ has a phase function $L \to U(1)$ which is measuring the angle from $\mathcal{V}(p)$ to $T_pL$ at each $p \in L$.  A grading on $L$ is a lift of the phase function to $L \to \R$.   

\begin{lemma}
Suppose $Q$ is given a perfect matching.  Then each branch of $\mathbb{L}$ admits a grading. 
\end{lemma}

\begin{proof}
As illustrated in (a) of Figure \ref{fig:refexactness}, $\mathcal{V}(p)$ once makes a $2\pi$-rotation with respect to $T_p {L_i}$, and then rotates back to the opposite direction by $2 \pi$ as we travel along $L_i$. Therefore, after slight adjustment, the angles from $T_p {L_i}$ to $\mathcal{V}(p)$
can be chosen to lie in the interval $(-\pi, \pi)$, and hence the phase function can be lifted to $\R$.
\end{proof}


\begin{figure}
\begin{center}
\includegraphics[height=2.8in]{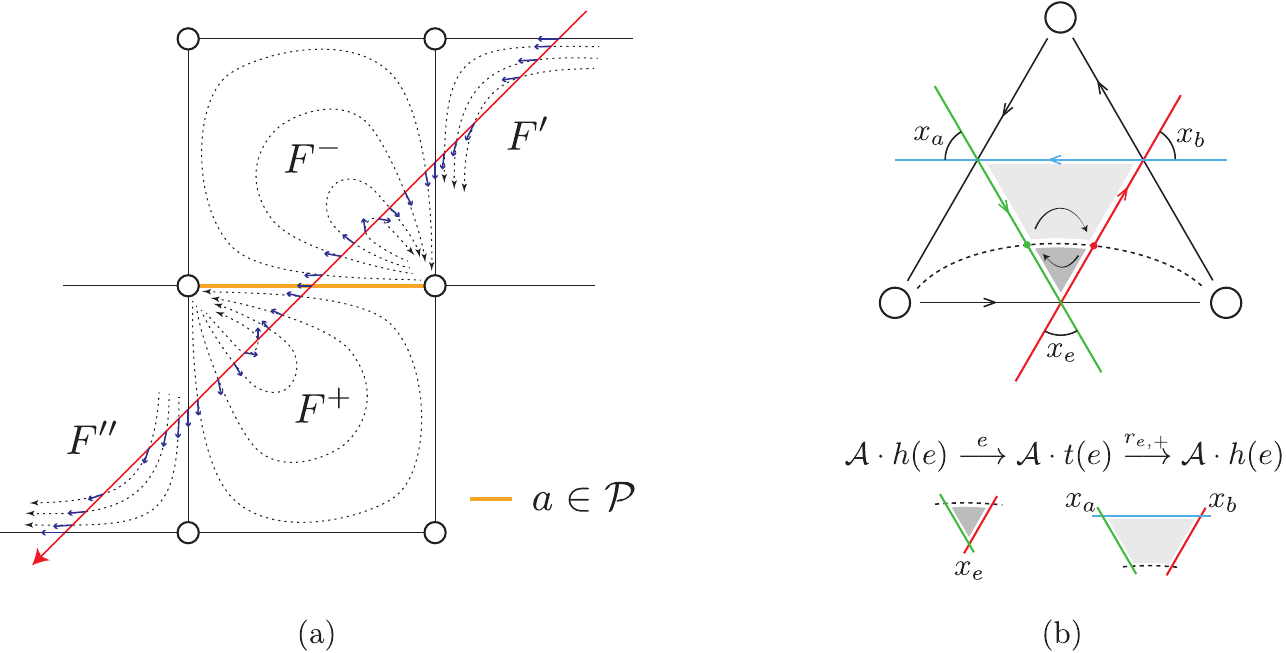}
\caption{(a) Perfect matching and the grading of $\bL$ $\quad$ (b) Mirror matrix factorization $P_e$}\label{fig:refexactness}
\end{center}
\end{figure}

\begin{cor}
We have an $A_\infty$-functor $\cF_{\mathbb{L}} : \Fuk^\wrap (\Sigma \setminus Q_0) \to \MF (\Jac(Q^\vee),W)$ which is graded if there exists a perfect matching on $Q$.
\end{cor}

\section{Mirror functor} \label{sec:punctured-fctor}
Recall that  the wrapped Fukaya category $\Fuk^\wrap (\Sigma \setminus Q_0)$  has classical generators $\{U_e\}$ given by the arrows $e \in Q_1$,
which are considered as non-compact Lagrangians (see Abouzaid \cite{Ab}). For simplicity, we will assume that $Q$ is a dimer.
We will compute the images of these generators under our mirror functor $\mathcal{F}_\bL$, and show that it matches with the
result of Bocklandt. We will also show how the morphism $\Hom(U_{a}, U_{b})$ in wrapped Fukaya category maps
to that of the corresponding matrix factorizations through our functor too.
Note that  $\Hom(U_{a}, U_{b})$ is nonzero if and only if the two arrows $a$ and $b$ are incident to the same vertex $v$ in $Q$.


Using the notation of Corollary \ref{cor:jac}, we define the relevant matrix factorizations of $W_\bL$.
\begin{definition}\label{def:mfc}
For any edge $e \in Q^\bL_1$, 
a matrix factorization $P_e$ of $(\cA = \Jac(Q^\bL, \Phi),W)$ is defined as follows.  The underlying module of $P_e$ is $\cA \cdot h(e) \oplus \cA \cdot t(e)$.  (Recall that $\cA \cdot v$ is spanned by paths starting from $v$.)  The matrix factorization is given by
$$ P_e: \cA \cdot h(e) \stackrel{\cdot e}{\longrightarrow} \cA \cdot t(e) \stackrel{\cdot r_{e,+}}{\longrightarrow} \cA \cdot h(e).$$
We denote by $\textrm{mf}( \Jac(Q^\bL, \Phi),W)$ the dg-subcategory of $MF( \Jac(Q^\bL, \Phi),W)$, generated by $\{P_e\}_{e \in Q^\bL}$.
\end{definition}
Note that $p \cdot e \cdot r_{e,+} = p \cdot W = W \cdot p$ for any $p \in \cA \cdot h(e)$, and $p' \cdot r_{e,+} \cdot e = p' \cdot W = W \cdot p'$ for any $p' \in \cA \cdot t(e)$, and hence $P_e$ defined above is a matrix factorization of $W$.  

\begin{prop}\label{prop:Bobjcorr}
We have $\cF(U_e) = P_e$ for every $e \in Q^\bL_1$.
\end{prop}

\begin{proof}
We compute $\cF(U_e)$.  
To avoid a triple intersection, we slightly perturb $U_e$ (to a dashed curve) as in (b) of Figure \ref{fig:refexactness}.
$U_e$ intersects two branches $\{L_{h(e)}, L_{t(e)}\}$ of $\mathbb{L}$ at exactly two points $p_{h(e)}$ and $p_{t(e)}$.  Hence the underlying module of $\cF(U_e)$ is $\cA \cdot h(e) \oplus \cA \cdot t(e)$, where the first and second summands are even and odd respectively by considering the degrees of $p_{h(e)}$ and $p_{t(e)}$.  

Let $U$ denote the holomorphic polygon whose boundary cycle is $e r_{e,+}$.  $U_e$ splits $U$ into two holomorphic strips, one from $p_{h(e)}$ to $p_{t(e)}$ which has a corner at $X_e$, and another one from $p_{t(e)}$ back to $p_{h(e)}$ which has corners at $X_{e'}$ where $e'$ are arrows contained in $r_{e,+}$.  These are the only two strips contributing to $m_1^{b,0}$, and hence the matrix factorization is the one given above. 
\end{proof}




Next, we consider a morphism level functor.
Bocklandt already has found a class of morphisms between matrix factorizations $P_a, P_b$ when $a,b$ lie
on a  zig-zag cycle of $Q^\vee$ (caution: this is not $Q$), which we recall first.

\begin{definition} \label{def:zeta}[ \cite{Bocklandt} ]
Fix a zigzag cycle in $Q^\vee$ and let $a$ and $b$ be two arrows in the cycle.  Without loss of generality suppose it is a zig cycle with respect to $a$.  Take a lifting of the arrow $a$ to the universal cover and also denote it by $a$, and denote by $b^{(i)}$ the liftings of $b$ which are contained in the zig ray of $a$ for $i \in \Z$, where $b^{(0)}$ is the first one met by the zig ray.  ($b^{(0)} = a$ when $b=a$.)  Let $a_0 \ldots a_k$ be the zig path with $a_0 = a$ and $a_k = b^{(i)}$ (where $k$ depends on $i$).  $a_{j+1} a_{j}$ for even $j$ is a part of a positive cycle $F^+_j$ in $Q^\vee$.  Hence the complement of $ a_{j+1} a_{j}$ in this cycle gives a path from $h(a_{j+1})$ to $t(a_j)$. By composing all these paths over all even $j$, we obtain a path $\textnormal{opp}_1$, which is from $h(b^{(i)})$ ($t(b^{(i)})$) to $t(a)$ when $k$ is odd (even). Likewise, we take the complement of $a_{j+1} a_{j}$ for odd $j$ in the unique negative cycle $F^-_j$ containing it, and compose them together to have a path $\textnormal{opp}_2$, which is from $t(b^{(i)})$ ($h(b^{(i)})$) to $h(a)$ when $k$ is odd (even). 


For any element $p$ in $(P_a)_0 = \cA \cdot h(a)$, $p \cdot \textnormal{opp}_2$ defines a path in $(P_b)_1 = \cA \cdot t(b)$ (or $(P_b)_0 = \cA \cdot h(b)$) when $k$ is odd (or even).  Similarly $\textnormal{opp}_1$ defines a map from $(P_a)_1= \cA \cdot t(a) $ to $(P_b)_0$ (or $(P_b)_1$) when $k$ is odd (or even). Thus we get the elements $\zeta^{(i)} = (\zeta_0^{(i)}, \zeta_1^{(i)})$ in $\Hom (P_a, P_b)$, where $\zeta_0: p \mapsto p \cdot \textnormal{opp}_2 $ and $\zeta_1: p \mapsto p \cdot \textnormal{opp}_1$.
\end{definition}

In the above definition, when $k$ is odd $\zeta^{(i)}$ takes the form
\begin{equation*}
\xymatrix{ (P_a)_0 \ar[dr]_{\zeta^{(i)}_0} \ar[r]& (P_a)_1 \ar[dr]^{\zeta^{(i)}_1}& \\ 
&(P_b)_1 \ar[r]& (P_b)_0
}
\end{equation*}
See the left side of Figure \ref{fig:imagewrapgen} for $\textnormal{opp}_1$ and $\textnormal{opp}_2$. 

\begin{figure}
\begin{center}
\includegraphics[height=4in]{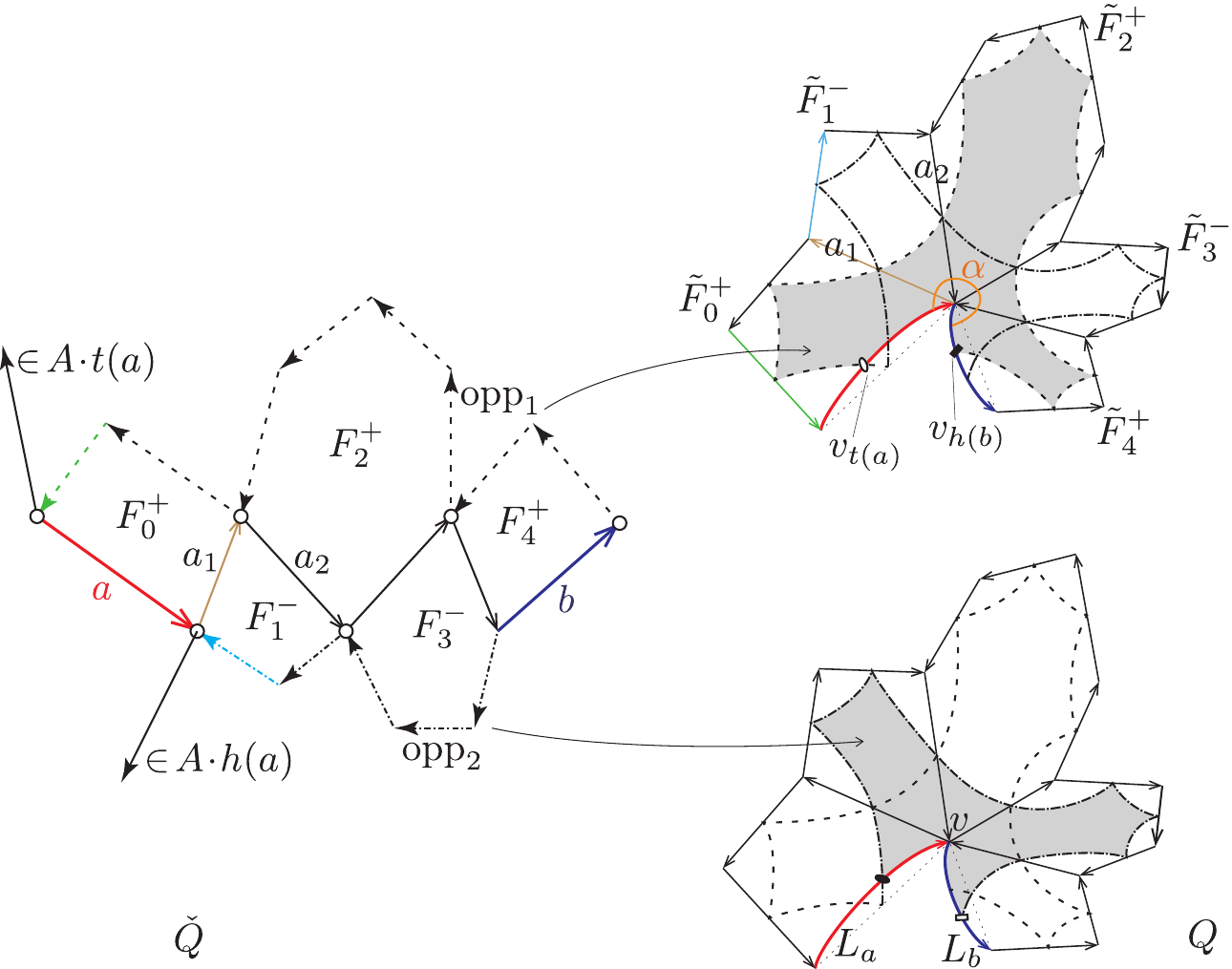}
\caption{Holomorphic polygons corresponding to paths ${\rm opp}_1$ and ${\rm opp}_2$ in $\check{Q}$}\label{fig:imagewrapgen}
\end{center}
\end{figure}

Using the fact that $(Q^\vee)^\vee = Q$, a zigzag cycle of $Q^\vee$ corresponds to a vertex $v$ in $Q$.
Therefore, the condition on $a,b$ to lie on a a zigzag cycle in $Q^\vee$ corresponds to the fact that 
the arrows $a$ and $b$ share the vertex $v$ in $Q$, and therefore, there exist a morphism in wrapped Fukaya category.

We show that generators of $\Hom(U_a,U_b)$ at that puncture are mapped to the above morphisms between matrix
factorizations.


\begin{prop} \label{prop:hom}
The functor $\cF^{\mathbb{L}}$ sends the generators of $\Hom (U_a, U_b)$ to $\pm \zeta^{(i)} \in \Hom (P_a, P_b)$.
\end{prop}

\begin{proof}
Here we only consider the chord $\alpha$ from $U_a$ to $U_b$ which has minimal winding number and assume that $h(a) = t(b)$ is the common vertex of $a$ and $b$ in $Q$. Then $\deg \alpha$ is odd.  The other cases are similar.

Set $\eta:=\cF^{\mathbb{L}}_1 (\alpha)$, the image of the morphism $\alpha$ under our functor. Since $\eta : \left(  A \cdot v_{h(a)} \oplus   A\cdot v_{t(a)} \right) \to \left( A \cdot v_{h(b)} \oplus   A\cdot v_{t(b)} \right)$ is linear over $A$ and so is $\zeta : \left(  A \cdot h(a)  \oplus   A \cdot t(a) \right) \to \left(   A \cdot h(b) \oplus  A \cdot t(b)  \right)$, it suffices to show that 
\begin{equation*}
\begin{array}{l}
\eta(v_{t(a)}) = \zeta (t(a)) \,\,(=  \textnormal{opp}_1, \,\,\, \mbox{identified with} \,\,\, \textnormal{opp}_1 \cdot v_{h(b)}) \quad \mbox{and} \\
 \eta(v_{h(a)}) = \zeta (h(a)) \,\,(= \textnormal{opp}_2 , \,\,\, \mbox{identified with} \,\,\, \textnormal{opp}_2  \cdot v_{t(b)} )
 \end{array}
 \end{equation*}
up to overall sign change. 
(Recall from the proof of Proposition \ref{prop:Bobjcorr} that the isomorphism $P_{L_{a}} \to P_a$ identifies $v_{h(a)}$ with $h(a)$ and $v_{t(a)}$ with $t(a)$, and does the similar for $b$.) We only show the first identity ``$\eta(v_{t(a)} ) =  \textnormal{opp}_1$" in detail as the proof for the second identity is essentially identical. Recall that $\eta (v_{t(a)}) = m_2^{b,0,0} (v_{t(a)}, \alpha) = \sum_k m_k (b, \cdots,b, v_{t(a)}, \alpha )$.

By definition of mirror quiver (Definition \ref{def:mirrorquiver}), the zigzag cycle containing $ a_k a_{k-1} \cdots a_1 a_0$ (with $a_0=a$ and $a_k=b$) in $Q^\vee$ corresponds to the vertex $v$ in $Q$, where we used the fact that $Q = (Q^\vee)^\vee$. In the mirror $Q$ of $Q^\vee$, the counter part of the zigzag path $a_0 \cdots a_k$ and $\textnormal{opp}_i$ ($i=1,2$) in $Q^\vee$ have the following configuration. The positive cycle $\tilde{F}^+_i$ in $Q$ mirror to each positive cycle $F_i^{+}$ (for even $i$) contains the vertex $v$, while no two such cycles in $Q$ share an arrow. The same happens for the mirror cycles $\tilde{F}^-_i$ to $F_i^-$ (for odd $i$). If we walk around the vertex $v$ in $Q$ counterclockwise, we meet $\tilde{F}^+_{k-1}$, $\tilde{F}^-_{k-2}$, $\tilde{F}^+_{k-3}$, $\cdots$, $\tilde{F}^+_2$, $\tilde{F}^-_1$, $\tilde{F}^+_{0}$ in order. The arrow $a_i$ in $Q$ sit between two consecutive cycles $\tilde{F}^{\pm}_{i}$, $\tilde{F}^{\mp}_{i-1}$ in this sequence. (See right side of Figure \ref{fig:imagewrapgen}).

Now, the path $\textnormal{opp}_1$ in $Q^\vee$ gives rise to a piecewise smooth curve $\gamma$  contained in $\mathbb{L}$ whose corners lie only on arrows of $\tilde{F}_{i}^+$ for even $i$. More precisely, $\gamma$ turns at each middle point of arrows in $\tilde{F}_{i}^+$ except $a_{i}$ and $a_{i+1}$. (See the dotted curve in upper right diagram in Figure \ref{fig:imagewrapgen}.) Observe that $\gamma$ intersects both $L_a$ and $L_b$ once. Checking the degrees of intersections, we see that they should be $v_{t(a)}$ and $v_{h(b)}$, respectively. Therefore, we have a bounded domain in $\Sigma$ enclosed by $\gamma$ and $L_a$ and $L_b$. Since $\alpha$ is a minimal Hamiltonian chord from $L_a$ to $L_b$ around the puncture $v$ in $\Sigma$, it lies in this bounded domain. Thus, after chopping off the corner around $v$ along $\alpha$, the domain become a holomorphic disc $T_1$ contributing to $m_2^{b,0,0} (v_{t(a)}, \alpha)$, and obviously there are no other contributions due to punctures. Note that the corners of $T_1$ along $\mathbb{L}$ is precisely $\textnormal{opp}_1$. Therefore, this disc defines a map
\begin{equation}\label{eqn:m2discspunctured1}
\eta : v_{t(a)} \mapsto \pm   \textnormal{opp}_1 \cdot v_{h(b)}.
\end{equation}

Similar procedure produces a holomorphic disc $T_2$ which gives rise to
\begin{equation}\label{eqn:m2discspunctured2}
\eta(v_{h(a)}) = \pm \textnormal{opp}_2 \cdot  v_{t(b)}  (=\pm \zeta (h(a)) ).
\end{equation}
We finally show that the signs in \eqref{eqn:m2discspunctured1} and \eqref{eqn:m2discspunctured2} are the same. We divide the computation for the sign difference in the following two steps.

\begin{enumerate}
\item[(i)] Ignoring special points,
the sign of $T_1$ is positive since its boundary orientation agrees with that of $L_a$, $L_b$ and $\mathbb{L}$. The boundary orientation of $T_2$ agree with those of Lagrangians only along $L_a$ and $L_b$. So, the sign of $T_2$ is $(-1)^{|\textnormal{opp}_2|}$ where $|\textnormal{opp}_2|$ is the number of arrows contained in $\textnormal{opp}_2$. 
 Observe that
\begin{equation}\label{eqn:morlevelsigndiff1}
|\textnormal{opp}_2| \equiv \sum_i |\tilde{F}_i^{-}|  \mod 2
\end{equation}
where $|\tilde{F}_i^{-}|$ is the number of edges of $\tilde{F}_i^{-}$. 

\item[(ii)] Now, we take the special points into account. Recall that precisely one special point is placed on the boundary of each negatively oriented polygon (for the potential) with the odd number of edges. Therefore, the sign difference between $T_1$ and $T_2$ due to special points is
\begin{equation}\label{eqn:morlevelsigndiff2}
(-1)^{\sum_{|\tilde{F}_i^{-}| \equiv 1 } 1}
\end{equation}
where the sum in the exponent is taken over all $\tilde{F}_i^{-}$ with the odd number of edges.
%
\end{enumerate}

It is easy to see that the sign differences \eqref{eqn:morlevelsigndiff1} and \eqref{eqn:morlevelsigndiff2} cancel with each other, which completes the proof. 
\end{proof}

In what follows we show that the functor is an embedding assuming that $Q^\vee$ is zigzag consistent using the following
result of Bocklandt on matrix factorization category.
We first recall the zigzag consistency condition introduced by Bocklandt.
 He showed that the Jacobi algebra of a zigzag consistent dimer is 3 Calabi-Yau \cite[Theorem 4.4]{Bocklandt}.  Similar results were obtained by Ishii and Ueda \cite{Ishii-Ueda}.  For the construction of generalized mirrors and functors we do not need this condition.  However, the condition is needed in the B-side to ensure the functor is an embedding (see Section \ref{sec:punctured-fctor}).
\begin{definition}[Zigzag consistency \cite{B-consistency}] \label{def:consistent}
A dimer model is called zigzag consistent if for each arrow $e$, the zig and the zag rays at $e$ in the universal cover only meet at $e$.
\end{definition}

\begin{prop}[\cite{Bocklandt}] \label{prop:basis-hom}
Suppose $Q^\vee$ is zigzag consistent.  Then for every pair of arrows $a$ and $b$, the set $\{\zeta^{(i)}: i \in \Z\}$ defined in Definition \ref{def:zeta} forms a basis of the cohomology $H^*(\Hom(P_a,P_b),\delta)$.  In particular $H^*(\Hom(P_a,P_b),\delta)$ vanishes when $a$ and $b$ do not share a common vertex in $Q$.
\end{prop}

As a consequence, when $a$ and $b$ share a common vertex in $Q$, the induced map on cohomologies $\Hom(U_a,U_b) \to H^*(\Hom(P_a,P_b),\delta)$ is an isomorphism.  When $a$ and $b$ do not share a common vertex in $Q$, the statement remains true since both sides are zero.  Hence we have

\begin{cor}
Suppose $Q^\vee$ is zigzag consistent.  Then the functor $\cF: \Fuk(\Sigma - Q_0) \to \MF(\Jac(Q^\vee),W)$ is an embedding.
\end{cor}

%
%
%


\chapter{Mirrors of Calabi-Yau threefolds associated with quadratic differentials} \label{sec:CY3}

The mirror construction and the natural functor formulated in this paper has important applications to Calabi-Yau threefolds associated to Hitchin systems constructed by Diaconescu-Donagi-Pantev \cite{DDP}.  Using the functor, Fukaya categories of the Calabi-Yau threefolds are transformed to derived categories of certain Calabi-Yau algebras known as the Ginzburg algebras \cite{Ginzburg}.  In this paper we will focus on the $\SL(2,\C)$ case.  Calabi-Yau algebra has the advantage of being more linear and was rather well-understood in existing literature, see for instance \cite{KS_Ainfty,KS-stab,CEG,Labardini}.  In particular stability conditions and wall-crossing of Donaldson-Thomas invariants of the Fukaya categories can be understood by using Ginzburg algebras and the functor.  Kontsevich-Soibelman \cite{KS-stab} developed a deep theory on the stability structures on Calabi-Yau algebras in dimension three and their relations with cluster algebras.

Recent mathematical studies in this direction were largely motivated by the works of Gaiotto-Moore-Neitzke \cite{GMN10,GMN13,GMN-network}, who studied wall-crossing of BPS states over the moduli spaces of Hitchin systems and the relations with hyper-K\"ahler metrics.  SYZ mirror symmetry for moduli spaces of Hitchin systems with structure groups $\SL_r$ and $\bP\SL_r$ was found by Hausel-Thaddeus \cite{HT}.  Mathematically the BPS states are stable objects of Fukaya categories of the associated (non-compact) Calabi-Yau threefolds.  Later Smith \cite{Smith} computed the Fukaya categories for $\SL(2)$ systems, and constructed an embedding by hand from the Fukaya categories to derived categories of Ginzburg algebras.  Moreover Bridgeland-Smith \cite{BS} provided a detailed study of stability conditions of derived categories of Ginzburg algebras.  Furthermore Fock-Goncharov \cite{FG1,FG2} provided a combinatorial approach to study $\SL(m)$ (or $\bP \GL(m)$) systems for general $m$ using techniques of cluster algebras. 

The approach taken in this paper is constructive and functorial.  We apply our general construction to produce the noncommutative mirror $\cA = \widehat{\Lambda Q} / (\partial \Phi)$.  Then it automatically follows from our general theory that there exists a god-given injective functor from the Fukaya category to the mirror derived category (which is in the reverse direction of the embedding by Smith).  
The functor is the crucial object for studying stability conditions and Donaldson-Thomas invariants for the Fukaya category.



The computations in this chapter rely essentially on \cite{Smith}.  From the viewpoint of this paper, the quiver algebra $\cA$ can be regarded as a mirror of $X$, in the sense that Lagrangian submanifolds of $X$ (symplectic geometry) are transformed to quiver representations (algebraic geometry).  

\begin{remark}
The notations for the superpotentials here are different from that in \cite{Smith}.  The spacetime superpotential $\Phi$ here was denoted by $W$ in \cite{Smith}.  In this paper $W$ denotes the worldsheet superpotential, which vanishes in this case. 
\end{remark}

\section{Non-compact Calabi-Yau threefolds associated with quadratic differentials}
We first briefly recall the Calabi-Yau geometry from \cite{DDP, Smith}.
Let $S$ be a closed Riemann surface of genus $g(S) > 0$, with a set $M \not= \emptyset$ of marked points.  An $\SL(m,\C)$ Higgs bundle on $(S,M)$ is the data $(E,D,\phi)$ where $E$ is a holomorphic vector bundle with a Hermitian metric, $D$ is a metric connection, and $\phi \in \Omega^1(\End E)$ satisfies that $D + \phi$ is a flat connection (with certain parabolic asymptotic behavior at the marked points) \cite{Hitchin,Simpson}.  The set of eigenvalues of the one-form valued skew-Hermitian matrix $\phi$ produce a spectral curve $\Sigma \subset K_S$, whose defining equation takes the form
\begin{equation} \label{eq:spectral}
\eta^m + \sum_{r=2}^m \phi_r \cdot \eta^{m-r} = 0
\end{equation}
for $\eta \in K_S$, where $\phi_r$ are meromorphic $r$-differentials, that is sections of $K_S^{\otimes r}$, with poles at $P$.  The projection $K_S \to S$ restricts to a ramified $m:1$ covering map $\Sigma \to S$. 

The \emph{associated (non-compact) Calabi-Yau threefold $X$ is a conic fibration} over the total space of $K_S$ with discriminant locus being $\Sigma \subset K_C$.  It is defined by the equation
$$ uv = \eta^m + \sum_{r=2}^m \phi_r \cdot \eta^{m-r}. $$
To make sense of the above equation, $u, v$ are taken to be sections of certain vector bundles.  Moreover since $\phi_r$ are meromorphic, $u,v$ are also meromorphic.  This means we need to modify the bundles by tensoring the divisor line bundles associated to $P$.  This was carefully carried out by Smith \cite{Smith} for $m=2$, which is recalled below for the readers' convenience.  From now on we focus on the case $m=2$, and so $\phi = \phi_2$ is a quadratic differential.

Assume that the meromorphic quadratic differential $\phi$ only has zeroes of order one, and has poles of order two at the set of marked points $M$.  The Calabi-Yau threefold is locally defined by a quadratic equation $uv - \eta^2 = \varphi(z)$ (away from poles) where the quadratic differential on $S$ is locally written as $\phi(z) = \varphi(z) dz^{\otimes 2}$.  

In order to make sense of this equation globally, one can consider a rank-two vector bundle $V$ over $S$ with the property that $\det V = K_S(M)$.  Denote $S^2 V$ to be the symmetric square of $V$.  Any map $V \to V^*$ induces a map $\det V \to \det V^*$.  Hence we have the determinant map 
$ \det: S^2 V \to (\det V)^{\otimes 2} =  K_S(M)^{\otimes 2}$
which is a quadratic function over each fiber of $S^2 V$ (and $u,v,\eta$ above are fiber coordinates of $S^2 V$).  Since $\varphi$ can be identified as a section of $K_S(M)^{\otimes 2}$, the equation $$\det = \pi^*\varphi$$ makes sense where $\pi$ is the bundle map $K_S(M)^{\otimes 2} \to S$.

In order to specify the poles of $u, v, \eta$ and ensure the Calabi-Yau property, Smith made a modification of the bundle $S^2 V$ as follows.  Let $W$ be a rank 3 vector bundle defined by the exact sequence of sheaves
$$ 0 \to W \to S^2 V \to (\iota_M)_* \C \to 0 $$
where $(\iota_M)_* \C$ is the skyscraper sheaf supported at the poles of $\phi$, and the map $S^2 V \to (\iota_M)_* \C$ around each $p \in M$ is defined by $(u,v,\eta) \mapsto u(p)$, where a choice of local trivialization of $V$ at each $p \in M$ is fixed and $(u,v,\eta)$ are the corresponding local fiber coordinates of $S^2 V$.  It essentially means that the local section $u$ is allowed to have a simple pole at $p$.  The map $\det$ can be pulled back from $S^2V$ to $W$ and hence the above equation also makes sense for $W$. 

\begin{prop}[Lemma 3.5 of \cite{Smith}]
$X := \{\det = \pi^* \varphi\} \subset W$ is a Calabi-Yau threefold, namely it has a trivial canonical line bundle.
\end{prop}

We have a natural fibration $X \to S$ which is the restriction of the bundle map $W \to S$.  A fiber over $p \in S$ is a smooth conic $u^2 + v^2 + \eta^2 = t$ for $t \not= 0$ when $p$ is neither a zero nor a pole; it is a degenerate conic $u^2 + v^2 + \eta^2 = 0$ when $p$ is a simple zero; it is a disjoint union of two complex planes $\C^2 \coprod \C^2$ defined by $u^2 = c$ when $p$ is a double pole.

\section{Generalized mirrors}

Recall that we need to fix a set of Lagrangians in order to construct the generalized mirror.  In the current situation the quadratic differential $\phi$ on $S$ gives a collection of Lagrangian spheres $\{L_i\}$ in $X$.  It is called to be a WKB collection of spheres whose associated $A_\infty$-algebra was computed by \cite{Smith}.  It is briefly summarized as follows.  See Figure \ref{fig:triangulation}.

\begin{figure}[htb!]
  \centering
   \begin{subfigure}[b]{0.3\textwidth}
   	\centering
    \includegraphics[width=\textwidth]{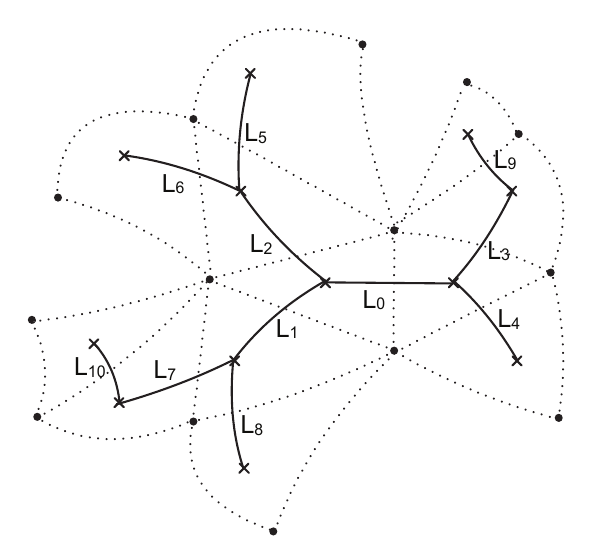}
    \caption{A triangulation induced from a quadratic differential.}
    \label{fig:triangulation}
   \end{subfigure}
   \hspace{10pt}
   \begin{subfigure}[b]{0.3\textwidth}
   	\centering
    \includegraphics[width=\textwidth]{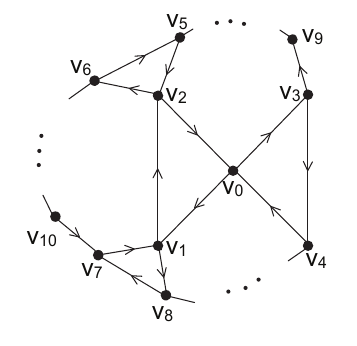}
    \caption{The corresponding quiver $Q$.}\label{fig:quiver}
   \end{subfigure}
	\caption{On the left, dots are double poles and crosses are simple zeroes of a quadratic differential.  The edges of the induced triangulation are shown by dotted lines.  The solid lines show the dual graph, which are images of Lagrangian spheres upstairs labeled by $L_k$.  $L_k$ correspond to the vertices $v_k$ of the quiver $Q$ on the right, and arrows in $Q$ correspond to degree-one morphisms among $L_k$'s.}
\end{figure}

It is classically known from the theory of quadratic differentials that a triangulation can be obtained from a quadratic differential.  This was a main ingredient in the work of Gaiotto-Moore-Neitzke \cite{GMN13} and Bridgeland-Smith \cite{BS}.

\begin{definition}[Triangulation induced from quadratic differential]
The triangulation induced from a quadratic differential $\phi$ is defined as follows.  The quadratic differential $\phi$ defines a horizontal foliation on $S$, which consists of horizontal trajectories $\gamma$ such that $\phi|_\gamma$ are real-valued.  Assume that $\phi$ is generically chosen such that it is saddle-free, which means that there is no horizontal trajectory connecting two zeroes.  Now consider the finite set of all separating trajectories, which are horizontal trajectories connecting double poles with zeroes.  The image divide the surface $S$ into finitely many chambers.  In each chamber fix one generic trajectory connecting two double poles.  The finite set of all such chosen generic trajectories gives a triangulation of $S$ whose vertices are double poles.  
\end{definition}

For simplicity assume that there is no self-folded triangle (see \cite[Lemma 3.21]{Smith}).  Then the WKB collection of spheres is defined as follows.

\begin{definition}[WKB collection of spheres] \label{def:L-CY3}
The WKB collection of Lagrangian spheres $\{L_i\}$ associated to a quadratic differential is defined as follows.  There is one simple zero in each triangle of the associated triangulation.  Take the embedded dual graph of the triangulation, whose vertices are simple zeroes and whose edges are paths connecting simple zeroes.  It gives the so-called Lagrangian cellulation of the Riemann surface $S$.  The edges are called to be matching paths, and each of them corresponds to a Lagrangian sphere in the threefold $X$.  The WKB collection is the set of all these Lagrangian spheres.
\end{definition}

By Lemma 4.4 of \cite{Smith}, $\{L_i\}$ are graded such that for each zero of $\phi$ corresponding to an intersection point of three Lagrangian spheres $(L_{i_1},L_{i_2},L_{i_3})$ ordered clockwise around the zero, the morphism from $L_{i_p}$ to $L_{i_{p+1}}$ has degree one and that from $L_{i_{p+1}}$ back to $L_{i_p}$ has degree two (where $p \in \Z_3$).  

Intuitively given a matching path, the Lagrangian is obtained by taking union of the real loci of the conic fibers over each point of the matching path.  The real loci are two-spheres over interior points of the path, which degenerate to points over the two boundary points of the path, and hence the resulting Lagrangian is topologically a three-sphere.  

Following Bridgeland-Smith \cite{BS}, define a quiver $Q$ below.  See Figure \ref{fig:quiver}.

\begin{definition}[The associated quiver $Q$] \label{def:quiverCY3}
For each triangle $\Delta$ of the triangulation mentioned above, take a triangle $\Delta'$ inscribed in $\Delta$.  Moreover the edges of $\Delta'$ are oriented clockwise.  Define a quiver $Q$ which is the union of all the inscribed triangles: a vertex of $Q$ is a vertex of one of the inscribed triangles, and an arrow of $Q$ is an oriented edge of an inscribed triangle.

For each zero $f$ of $\phi$, define $T(f)$ to be the boundary path (which has length three) of the corresponding inscribed triangle.  For each pole $p$ of $\phi$, define $C(p)$ to be the corresponding cycle of $Q$ which is counterclockwise oriented.
\end{definition}

Now take $\bL = \{L_i\}$ to be the reference Lagrangians and carry out the mirror construction given in Chapter \ref{sec:MFseverallag}.  The following essentially follows from the computations of \cite{Smith}.

\begin{theorem} \label{thm:mir-CY3}
The generalized mirror of the non-compact Calabi-Yau threefold $X$ corresponding to the above choice of reference Lagrangians $\bL = \{L_i\}$ is given by 
$$\cA = \widehat{\Lambda Q}/\langle \partial_e \Phi: e \textrm{ is an arrow of } Q \rangle$$
(with the worldsheet superpotential $W \equiv 0$) where $Q$ is defined in Definition \ref{def:quiverCY3} and 
$$ \Phi := \sum_{\textrm{zero} f} T(f) - \sum_{\textrm{pole } p} n_p \bT^{A_p} C(p) + \sum_{k \in K} n_k \bT^{A(U_k)} U_k , $$
$n_p, n_k \in \rat - \{0\}$ and $A_p, A(U_k) \in \R_{>0}$, $K$ is a certain index set, $U_k$ are loops in $\widehat{\Lambda Q}$ which are different from $C(p)$ and $T(f)$.
Note that $\Phi$ is cyclic.  It is understood that the above expression is given in terms of cyclic representatives, and differentiation is taken cyclically.
\end{theorem}

\begin{proof}
From the definition of $Q$ (Definition \ref{def:quiverCY3}) and the chosen grading on $\{L_i\}$ (Definition \ref{def:L-CY3}), it is clear that the vertices of $Q$ correspond to $L_i$ and the arrows of $Q$ correspond to degree-one morphisms among $\{L_i\}$.  Let's denote the generators of the Floer complex of $\bL$ by $X_e$ which are the degree-one morphisms corresponding to each arrow $e$ of $Q$, and $X_{\bar{e}}$ which are the degree-two morphisms corresponding to each reversed arrow $\bar{e}$, and $T_i$ which are the point classes of $L_i$, and $\one_{L_i}$ which are the fundamental classes of $L_i$.  The generalized mirror takes the form $(\widehat{\Lambda Q} / I,W)$, where $I$ is the ideal generated by weakly unobstructed relations, which are the coefficients $P_{\bar{e}}$ of $X_{\bar{e}}$ in 
$$m_0^{b = \sum_e x_e X_e} = \sum_{i} W_i \one_{L_i} + \sum_{e} P_{\bar{e}} X_{\bar{e}}$$
and $W = \sum_i W_i$.  In the above expression $e$ runs over all the arrows of $Q$, $X_e$ are the corresponding degree-one morphisms and $X_{\bar{e}}$ are the corresponding reversed degree-two morphisms.

In this Calabi-Yau situation, by Proposition \ref{prop:W=0} we have $W = 0$.  Fix an edge $e$ of $Q$ and consider 
$$P_{\bar{e}} = \sum_{k \geq 0} \sum_{(X_{e_1},\ldots,X_{e_k}))} x_{e_k} \ldots x_{e_1} \left(m_k(X_{e_1},\ldots,X_{e_k}), e \right).$$   
Suppose $\beta$ is a polygon class contributing to $m_k(X_{e_1},\ldots,X_{e_k})$.  The input morphisms have degree one and the output morphism has degree two.  If $\beta$ is a constant polygon class, then it is supported at certain an intersection point of three Lagrangian spheres in $\{L_i\}$ corresponding to a zero of $\phi$.  By degree reason, $\beta$ is a class of constant triangles with two inputs of degree one and an output of degree two, and the counting is one by \cite[Lemma 4.8]{Smith}.  It gives the term $x_{e_2} x_{e_1}$ in $P_{\bar{e}}$ where $(X_{e_1},X_{e_2},X_{e})$ are the three degree-one morphisms at the same zero $f$ of $\phi$ labeled clockwise.  It equals to $\partial_e T(f)$.

If $\beta$ is a non-constant polygon class, it is projected to a non-constant polygon class in the surface $S$ bounded by matching paths (since the intersection of $L_i$ with a fiber is the real locus in a smooth conic which do not bound any non-constant holomorphic disc).  In particular it includes the term $n_p \bT^{A_p} x_{e_k} \ldots x_{e_1}$, where $X_{e_1} \ldots X_{e_k} X_{e}$ is a cycle in $Q$ corresponding to a pole of $\phi$, $n_p$ is the counting and $A_p$ is the symplectic area of $\beta$.  By \cite[Lemma 4.10]{Smith}, $n_p \not= 0$.  The term equals to $\partial_e \, (n_p \bT^{A_p} C(p))$.  All the other terms are denoted by $U_k$.

Consequently, $P_{\bar{e}} = \partial_e \Phi$ and result follows.
\end{proof}

We can also carry out the extended mirror construction given by Chapter \ref{sec:ext-mir} which uses formal deformations in all degrees rather than just odd-degrees.

\begin{theorem} \label{thm:ext-mir-CY3}
The extended generalized mirror of the non-compact Calabi-Yau threefold $X$ is the Ginzburg dg algebra $(\Lambda \tilde{Q},d)$ where $\tilde{Q}$ is the doubling of $Q$ by adding to $Q$ a reversed edge $\bar{e}$ for each edge $e$ of $Q$, and adding a loop $t$ at each vertex of $Q$; $d$ is the derivation defined by
$$d t_i = \sum_e v_i [x_e, x_{\bar{e}}] v_i, \, d x_{\bar{e}} = \partial_{x_e} \Phi, \, d x_e = 0.$$
\end{theorem}

\begin{proof}
Let's denote the generators of the Floer complex of $\bL$ by $X_e$ which are the degree-one morphisms corresponding to each arrow $e$ of $Q$, $X_{\bar{e}}$ which are the degree-two morphisms corresponding to each reversed arrow $\bar{e}$, $T_i$ which are the point classes of $L_i$, and $\one_{L_i}$ which are the fundamental classes of $L_i$.  The extended generalized mirror uses deformations of $\bL$ in all directions except those along fundamental classes.  In particular from Definition \ref{def:extQ}, the corresponding quiver is $\tilde{Q}$ described above. 

Consider the deformation $\tilde{b}=\sum_e x_e X_e + \sum_e x_{\bar{e}} X_{\bar{e}} + \sum_i t_i T_i$ where $\deg x_e = 0$, $\deg x_{\bar{e}} = -1$ and $\deg t_i = -2$ such that $\tilde{b}$ has overall degree $1$.  We need to compute 
$$m_0^{\tilde{b}} = \sum_i \tilde{W}_i \one_{L_i} + \sum_i (d t_i) T_i + \sum_e (d x_{\bar{e}}) X_{\bar{e}} + \sum_e (d x_e) X_e $$
which has overall degree $2$.  By Proposition \ref{prop:tildeW=0}, $\tilde{W}_i = 0$.  Also by degree reason $d x_e = 0$.  

Consider the coefficient $d t_i$.  Only constant polygons can contribute to the output $T_i$ since the output marked point is unconstrained and no non-constant disc can be rigid.  Such constant polygons are supported at an intersection point of $L_i$ with another Lagrangian sphere.  By degree reason, only $m_2(X_e, X_{\bar{e}})$ or $m_2(X_{\bar{e}},X_e)$ contributes, where $e$ is an arrow from or to $v_i$ respectively, and hence $d t_i = \sum_e v_i [x_e, x_{\bar{e}}] v_i$.  

For $d x_{\bar{e}}$ which has degree zero, only degree-zero variables can appear.  This means we only need to consider polygons with degree one morphisms to be inputs (and the degree-two morphism $X_{\bar{e}}$ to be the output).  Thus we only need to consider $m_0^b$ and result follows from Theorem \ref{thm:mir-CY3}.
\end{proof}

The derived category on the mirror side can be described using either the Ginzburg algebra $(\Lambda \tilde{Q},d)$ or the Jacobi algebra (its $d$-cohomology) $\cA = \widehat{\Lambda Q} / \langle \partial_e \Phi \rangle$.  Namely by Lemma 5.2 of \cite{Keller-Yang}, the derived category of finite-dimensional dg modules of $(\Lambda \tilde{Q},d)$, which is a $CY^3$ triangulated category by \cite{Keller}, has a canonical bounded t-structure, whose heart is the category of nilpotent representations of $\cA$.

\section{Mirror functor}

From Section \ref{sec:FuktoMF} we have a mirror functor
$$ \Fuk(X) \to \dg(\cA) = \MF(\cA, W=0).$$
In this section we describe the mirror functor explicitly, namely we compute the quiver module mirror to a Lagrangian matching sphere $L_0$ in a WKB collection $\bL$.

The mirror quiver module is by definition $\cF((\bL,b),L_0)$ together with the differential $m_1^{(b,0)}$.  $L_0$ intersects with five Lagrangian spheres in the WKB collection, namely, $L_0$ itself, and four other Lagrangian spheres $L_1, \ldots, L_4$ as shown in Figure \ref{fig:triangulation}.  Thus the mirror module is 
$$ \check{L}_0 = P_{0} \cdot \one_{0} \oplus P_{0} \cdot T_{0} \oplus P_{2} \cdot X_{\overrightarrow{v_2v_0}} \oplus P_{4} \cdot X_{\overrightarrow{v_4v_0}} \oplus P_{1} \cdot X_{\overleftarrow{v_0v_1}} \oplus P_{3} \cdot X_{\overleftarrow{v_0v_3}} $$
where $P_i := \widehat{\Lambda Q} \cdot v_i$ for each vertex $v_i$, and $\overrightarrow{vw}$ denotes the edge from vertex $v$ to vertex $w$ in $Q$ and $\overleftarrow{vw}$ denotes the reversed arrow in the extended quiver $\tilde{Q}$.  Recall that $X_{\overrightarrow{vw}}$ has degree $1$ and $X_{\overleftarrow{vw}}$ has degree $2$.  Moreover $\one_0$ has degree $0$ and $T_0$ has degree $3$.  (Also recall that $x_e$ has degree zero for all edges $e$ of $Q$, and hence elements in $P_i$ has degree $0$.)

We need to compute the differential $d_{\check{L}_0}$ of the module $\check{L}_0$, which is defined by $m_1^{(\bL,b),L_0}$.

\begin{prop}
$L_0$ is transformed to the dg module $(\check{L}_0,d_{\check{L}_0})$ under the mirror functor, where $\check{L}_0$ is given above and
$d_{\check{L}_0}: \check{L}_0 \to \check{L}_0$ is given by
\begin{align*}
d_{\check{L}_0} (\one_0) &= x_{\overrightarrow{v_2v_0}} X_{\overrightarrow{v_2v_0}} + x_{\overrightarrow{v_4v_0}} X_{\overrightarrow{v_4v_0}} \\
d_{\check{L}_0} (X_{\overrightarrow{v_2v_0}}) &= (x_{\overrightarrow{v_1v_2}} + \alpha \cdot x_{\overrightarrow{v_1v_8}}) X_{\overleftarrow{v_0v_1}} - \mathfrak{p}_1 X_{\overleftarrow{v_0v_3}} \\
d_{\check{L}_0} (X_{\overrightarrow{v_4v_0}}) &= - \mathfrak{p}_2 X_{\overleftarrow{v_0v_1}} + (x_{\overrightarrow{v_3v_4}} + \beta \cdot x_{\overrightarrow{v_3v_9}}) X_{\overleftarrow{v_0v_3}} \\
d_{\check{L}_0} (X_{\overleftarrow{v_0v_1}}) &= x_{\overrightarrow{v_0v_1}} T_0 \\
d_{\check{L}_0} (X_{\overleftarrow{v_0v_3}}) &= x_{\overrightarrow{v_0v_3}} T_0 \\
d_{\check{L}_0} (T_0) &= 0
\end{align*}
for some $\alpha \in v_2 \cdot \cA \cdot v_8$, $\beta \in v_4 \cdot \cA \cdot v_9$, $\mathfrak{p}_1 \in v_2 \cdot \cA \cdot v_3$, $\mathfrak{p}_2 \in v_4 \cdot \cA \cdot v_1$.  The leading order terms of $\mathfrak{p}_i$ are $n_i \bT^{A_i} p_i$ respectively, where $n_i \in \rat$, $A_i \in \R_{>0}$, $p_1$ is the path from $v_3$ to $v_2$ such that $p_1 \cdot \overrightarrow{v_0v_3} \cdot \overrightarrow{v_2v_0}$ is a loop in $Q$ which contributes to a pole term in $\Phi$, and $p_2$ is the path from $v_1$ to $v_4$ such that $p_2 \cdot \overrightarrow{v_0v_2} \cdot {\overrightarrow{v_4v_0}}$ is a loop in $Q$ which contributes to a pole term in $\Phi$.
\end{prop}

\begin{proof}
The differential is computed using Theorem \ref{thm:ext-mir-CY3}, which gives an explicit expression of $m_0^{\tilde{b}}$.  The first equality follows from $m_2(X, \one_0) = X$.  The second and third equalities follows from the expressions of $m_k(X_{e_1},\ldots,X_{e_k})$.  The higher order terms of $\Phi$ correspond to the terms $\alpha \cdot x_{\overrightarrow{v_1v_8}}$, $\beta \cdot x_{\overrightarrow{v_3v_9}}$, $\mathfrak{p}_1, \mathfrak{p}_2$.
The fourth and fifth equalities come from $m_2(X_{\overrightarrow{v_0v_i}},\overleftarrow{v_0v_i}) = T_0$ for $i=1,3$.  The last equality is due to $\deg T_0 = 3$ which is the highest degree.
\end{proof}

Taking the cohomology, we have

\begin{prop} \label{prop:coho-mod}
The cohomology of the dg module $(\check{L}_0,d_{\check{L}_0})$ mirror to $L_0$ equals to
$$ H^i(d_{\check{L}_0}) = \left\{ \begin{array}{ll}
\Lambda \cdot T_0 & \textrm{ for } i = 3;\\
0 & \textrm{ otherwise.}
\end{array}
\right.$$
\end{prop}
\begin{proof}
First we compute $H^0$, which is the kernel of the map $d_{\check{L}_0}: P_0 \one_0 \to P_2 X_{\overrightarrow{v_2v_0}} \oplus  P_4 X_{\overrightarrow{v_4v_0}}$ defined by $c \one_0 \mapsto c x_{\overrightarrow{v_2v_0}} X_{\overrightarrow{v_2v_0}} + c x_{\overrightarrow{v_4v_0}} X_{\overrightarrow{v_4v_0}}$ for $c \in P_0$.  Suppose $c \not= 0$.  Since $c x_{\overrightarrow{v_2v_0}} = 0$ but $c \not= 0$ in $\widehat{\Lambda Q} / R$, $c x_{\overrightarrow{v_2v_0}} \in \widehat{\Lambda Q} \cdot \{ \partial_i \Phi \}$.  The only elements in $\{\partial_i \Phi\}$ containing terms of the form $\ldots x_{\overrightarrow{v_2v_0}}$ are 
$$ \partial_{x_{\overrightarrow{v_5v_2}}} \Phi = x_{\overrightarrow{v_6v_5}} x_{\overrightarrow{v_2v_6}} + p x_{\overrightarrow{v_2v_0}} + \ldots $$
and
$$ \partial_{x_{\overrightarrow{v_1v_2}}} \Phi = x_{\overrightarrow{v_0v_1}} x_{\overrightarrow{v_2v_0}} + q x_{\overrightarrow{v_2v_6}} + \ldots $$
where the first two terms correspond to polygons of the quiver.  In order to get rid of the term $x_{\overrightarrow{v_6v_5}} x_{\overrightarrow{v_2v_6}}$ or $q x_{\overrightarrow{v_2v_6}}$, we have to subtract by
$$ (\partial_{x_{\overrightarrow{v_2v_6}}} \Phi) \cdot x_{\overrightarrow{v_2v_6}} = x_{\overrightarrow{v_5v_2}} x_{\overrightarrow{v_6v_5}} x_{\overrightarrow{v_2v_6}} + x_{\overrightarrow{v_1v_2}} q x_{\overrightarrow{v_2v_6}} + \ldots$$
Hence $c x_{\overrightarrow{v_2v_0}}$ must be of the form $u \cdot \left(x_{\overrightarrow{v_5v_2}} \partial_{x_{\overrightarrow{v_5v_2}}} \Phi + x_{\overrightarrow{v_1v_2}} \partial_{x_{\overrightarrow{v_1v_2}}} \Phi - (\partial_{x_{\overrightarrow{v_2v_6}}} \Phi) x_{\overrightarrow{v_2v_6}} \right)$ for some $u$.  But 
$$x_{\overrightarrow{v_5v_2}} \partial_{x_{\overrightarrow{v_5v_2}}} \Phi + x_{\overrightarrow{v_1v_2}} \partial_{x_{\overrightarrow{v_1v_2}}} \Phi = (\partial_{x_{\overrightarrow{v_2v_6}}} \Phi) x_{\overrightarrow{v_2v_6}} + (\partial_{x_{\overrightarrow{v_2v_0}}} \Phi) x_{\overrightarrow{v_2v_0}} $$
since both the left and right hand sides are sums of all the polygons with one of its sides bounded by $L_2$.  Thus $c x_{\overrightarrow{v_2v_0}} = u \cdot (\partial_{x_{\overrightarrow{v_2v_0}}} \Phi) x_{\overrightarrow{v_2v_0}}$, implying that
$c$ itself is zero in $\widehat{\Lambda Q} / R$, a contradiction.  Hence $H^0 = 0$.

Now we compute $H^1$.  First consider the kernel of the map $d_{\check{L}_0}: P_2 \cdot X_{\overrightarrow{v_2v_0}} \oplus P_4 \cdot X_{\overrightarrow{v_4v_0}} \to P_{1} \cdot X_{\overleftarrow{v_0v_1}} \oplus P_{3} \cdot X_{\overleftarrow{v_0v_3}}$ given by 
$$h X_{\overrightarrow{v_2v_0}} + k X_{\overrightarrow{v_4v_0}} \mapsto (h (x_{\overrightarrow{v_1v_2}} + \alpha \cdot x_{\overrightarrow{v_1v_8}}) - k \mathfrak{p}_2)X_{\overleftarrow{v_0v_1}} + (- h \mathfrak{p}_1 + k (x_{\overrightarrow{v_3v_4}} + \beta \cdot x_{\overrightarrow{v_3v_9}})) X_{\overleftarrow{v_0v_3}}.  $$  
We only need to consider the case $(h,k) \not= (0,0) \in (\widehat{\Lambda Q} / R)^2$.  In order to be in the kernel, $h=0$ if and only if $k=0$.  Thus neither $h$ nor $k$ is zero.  We have $h (x_{\overrightarrow{v_1v_2}} + \alpha \cdot x_{\overrightarrow{v_1v_8}}) - k \mathfrak{p}_2, - h \mathfrak{p}_1 + k (x_{\overrightarrow{v_3v_4}} + \beta \cdot x_{\overrightarrow{v_3v_9}}) \in \widehat{\Lambda Q} \cdot \{ \partial_i \Phi \}$.  

Note that since $\mathfrak{p}_2 \in v_4 \cdot \widehat{\Lambda Q} \cdot v_1$, it can be written as
$$ \mathfrak{p}_2 = a \cdot x_{\overrightarrow{v_1v_8}} + b \cdot x_{\overrightarrow{v_1v_2}} $$
for some $a \in v_4 \cdot \widehat{\Lambda Q} \cdot v_8$ and $b \in v_4 \cdot \widehat{\Lambda Q} \cdot v_2$.  The leading order term of $a$ is the path $a_0$ such that $x_{\overrightarrow{v_4v_0}}a_0 x_{\overrightarrow{v_1v_8}}x_{\overrightarrow{v_0v_1}}$ corresponds to a polygon of $Q$.
The only elements in $\{\partial_i \Phi\}$ containing terms of the form $\ldots x_{\overrightarrow{v_1v_2}}$ or $\ldots x_{\overrightarrow{v_1v_8}}$ are 
\begin{align*}
\partial_{x_{\overrightarrow{v_0v_1}}} \Phi =& x_{\overrightarrow{v_2v_0}}x_{\overrightarrow{v_1v_2}} - x_{\overrightarrow{v_4v_0}} \mathfrak{p}_2 + x_{\overrightarrow{v_2v_0}} \cdot \alpha \cdot x_{\overrightarrow{v_1v_8}},\\
\partial_{x_{\overrightarrow{v_7v_1}}} \Phi =& x_{\overrightarrow{v_8v_7}} x_{\overrightarrow{v_1v_8}} - x_{\overrightarrow{v_{10}v_7}} \mathfrak{p} + x_{\overrightarrow{v_8v_7}} \gamma \cdot x_{\overrightarrow{v_1v_2}}.
\end{align*} 
Thus
$$ (h-kb) x_{\overrightarrow{v_1v_2}} + (-ka + h\alpha) x_{\overrightarrow{v_1v_8}} = \lambda (x_{\overrightarrow{v_2v_0}}x_{\overrightarrow{v_1v_2}} - x_{\overrightarrow{v_4v_0}} \mathfrak{p}_2 + x_{\overrightarrow{v_2v_0}} \cdot \alpha \cdot x_{\overrightarrow{v_1v_8}}) + \mu (x_{\overrightarrow{v_8v_7}} x_{\overrightarrow{v_1v_8}} - x_{\overrightarrow{v_{10}v_7}} \mathfrak{p} + x_{\overrightarrow{v_8v_7}} \gamma \cdot x_{\overrightarrow{v_1v_2}})$$
for some $\lambda \in \widehat{\Lambda Q} \cdot v_0$ and $\mu \in \widehat{\Lambda Q} \cdot v_7$.  Consider the terms $(\ldots x_{\overrightarrow{v_1v_8}})$ in the above equation.  It gives
$$ (k-\lambda x_{\overrightarrow{v_4v_0}}) a + (\lambda x_{\overrightarrow{v_2v_0}} - h) \alpha + \mu (x_{\overrightarrow{v_8v_7}} - x_{\overrightarrow{v_{10}v_7}} \mathfrak{p}') = 0 $$
where $\mathfrak{p} = \mathfrak{p}' x_{\overrightarrow{v_1v_8}} + \mathfrak{p}'' x_{\overrightarrow{v_1v_2}}$.  The leading order terms of $x_{\overrightarrow{v_8v_7}} - x_{\overrightarrow{v_{10}v_7}} \mathfrak{p}'$ is $x_{\overrightarrow{v_8v_7}}$, while the leading order term $a_0$ of $a$ does not begin with $x_{\overrightarrow{v_8v_7}}$.  If the leading order term of $\alpha$ does not begin with $x_{\overrightarrow{v_8v_7}}$, then the coefficient of $a$ has to be zero: $k = \lambda x_{\overrightarrow{v_4v_0}}$.  Otherwise $\mu=0$.

Now consider the terms $(\ldots x_{\overrightarrow{v_1v_2}})$.  In the first case that $k = \lambda x_{\overrightarrow{v_4v_0}}$, it gives 
$$h - \lambda x_{\overrightarrow{v_2v_0}} = \mu (x_{\overrightarrow{v_8v_7}} \gamma - x_{\overrightarrow{v_{10}v_7}}\mathfrak{p}'').$$
Substituting back to the above equation $(\lambda x_{\overrightarrow{v_2v_0}} - h) \alpha + \mu (x_{\overrightarrow{v_8v_7}} - x_{\overrightarrow{v_{10}v_7}} \mathfrak{p}') = 0$, we obtain
$$\mu (x_{\overrightarrow{v_8v_7}} \gamma - x_{\overrightarrow{v_{10}v_7}}\mathfrak{p}'') \alpha = \mu (x_{\overrightarrow{v_8v_7}} - x_{\overrightarrow{v_{10}v_7}} \mathfrak{p}').$$
The right hand side has the term $\mu \cdot x_{\overrightarrow{v_8v_7}}$ while the left hand side does not.  It forces $\mu = 0$.  As a result, we have
$$ k = \lambda x_{\overrightarrow{v_4v_0}}, h = \lambda x_{\overrightarrow{v_2v_0}}, \mu = 0. $$
Hence $h X_{\overrightarrow{v_2v_0}} + k X_{\overrightarrow{v_4v_0}} \in \mathrm{Im} (d_{\check{L}_0})$.

In the second case that $\mu=0$, it gives
$$ h - \lambda x_{\overrightarrow{v_2v_0}} = (k - \lambda x_{\overrightarrow{v_4v_0}}) b. $$
Substituting back to the above equation $(\lambda x_{\overrightarrow{v_2v_0}} - h) \alpha + (k - \lambda x_{\overrightarrow{v_4v_0}})a = 0$, we obtain
$$(k - \lambda x_{\overrightarrow{v_4v_0}}) b \alpha = (k - \lambda x_{\overrightarrow{v_4v_0}}) a.$$
But $b \alpha \not= a$ since the leading order term $a_0$ does not pass through $v_2$.  Hence $k - \lambda x_{\overrightarrow{v_4v_0}}$, and the same conclusion follows.

A similar argument shows that $H^2 = 0$ and we will not repeat the argument.  Now consider $H^3$, which is the cokernel of $d_{\check{L}_0}:P_{1} \cdot X_{\overleftarrow{v_0v_1}} \oplus P_{3} \cdot X_{\overleftarrow{v_0v_3}} \to P_0 T_0$ defined by $$d_{\check{L}_0} (a X_{\overleftarrow{v_0v_1}} + b X_{\overleftarrow{v_0v_3}}) = (a x_{\overrightarrow{v_0v_1}} + bx_{\overrightarrow{v_0v_3}}) T_0.$$
Any element in $P_0$ is of the form $c v_0 + a x_{\overrightarrow{v_0v_1}} + bx_{\overrightarrow{v_0v_3}}$for some $c \in \Lambda$ and $a,b \in \widehat{\Lambda Q}$.  Hence the cokernel is one dimensional, which is $H^3 = \Lambda \cdot T_0$.
\end{proof}

In summary, on cohomology level a Lagrangian sphere in a WKB collection is transformed to a simple module at the corresponding vertex under the mirror functor.  Combining this result together with Theorem \ref{thm:inj} on injectivity of the functor, and the result of \cite{Keller-Yang} on the morphism spaces of the simple modules, we can conclude that

\begin{theorem} \label{thm:WKB}
The mirror functor restricts to be an equivalence from the derived category generated by a WKB collection of Lagrangian spheres in the Fukaya category to the derived category of modules of the mirror.
\end{theorem}

\begin{proof}
By Theorem \ref{thm:inj}, the induced functor on ${\rm HF}^*(L_i,L_j) \to H^*(\check{L}_i,\check{L}_j)$ is injective.  By Proposition \ref{prop:coho-mod}, we have $\check{L}_i$ being resolutions of the corresponding simple module $S_i$ of the quiver (up to shifting).  By Lemma 2.15 of \cite{Keller-Yang}, the morphism space from $S_i$ to $S_j$ is one-dimensional when $v_i$ and $v_j$ are adjacent vertices in $Q$ with $i \not= j$, and is two-dimensional when $i=j$.  These dimensions agree with ${\rm HF}^*(L_i,L_j)$ where $L_i$ are unobstructed Lagrangian three-spheres.  Hence the induced maps on ${\rm HF}^*(L_i,L_j) \to H^*(\check{L}_i,\check{L}_j)$ are indeed isomorphisms.  Furthermore, the simple modules generate the derived category of the Ginzburg algebra corresponding to $Q$ (see the proof of Theorem 2.19 in \cite{Keller-Yang}.  Hence the functor gives an equivalence between the derived category of WKB Lagrangian spheres and the derived category of the Ginzburg algebra.
\end{proof}

\appendix
%
%
%


\chapter{Theta function calculations}
\section{Weak Maurer-Cartan relations for $\mathbb{P}^1_{3,3,3}$}\label{appsubsec:333theta}
To obtain simpler expressions of $A,B,C$ \eqref{eqn:333ABCbef} in terms of theta functions, observe first that
\begin{equation*}
\begin{array}{lcl}
A(\lambda,t) &=& q_0^{9t^2 + 6t+1}  \lambda^{\frac{1+3t}{2}}  \sum_{k\in \Z}  (q_0^{36})^{k^2}  \left( \lambda^3 (q_0^{36})^{(t+\frac{1}{3})} \right)^{k} \\
B(\lambda,t) &=& q_0^{9t^2 + 6t+1}  \lambda^{\frac{1+3t}{2}}  \sum_{k\in \Z}  (q_0^{36})^{(k+\frac{2}{3})^2}  \left( \lambda^3 (q_0^{36})^{(t+\frac{1}{3})} \right)^{ (k + \frac{2}{3})}  \\
C(\lambda,t) &=& q_0^{9t^2 + 6t+1}  \lambda^{\frac{1+3t}{2}}  \sum_{k\in \Z}  (q_0^{36})^{(k+\frac{1}{3})^2}  \left( \lambda^3 (q_0^{36})^{(t+ \frac{1}{3})} \right)^{(k + \frac{1}{3})} .
\end{array}
\end{equation*}
Since $\lambda = e^{2 \pi i s}$ and $q_{orb}^3 = e^{ 2 \pi i \tau}$ (for $q_{orb} = q_0^8$, we have 
$$q_0^{36} = e^{ \pi i (3 \tau)} \quad \mbox{and} \quad \lambda^3 \left(q_0^{36} \right)^{ (t+ \frac{1}{3}) } = e^{6 \pi i \left( s + \tau \frac{t}{2}  +\frac{\tau}{6} \right) } = e^{2\pi i (3u)}$$
where $u= s+ \tau\frac{t}{2} + \frac{\tau}{6}$. Therefore, dividing $A,B,C$ by a common factor, we obtain
\begin{equation*}
\begin{array}{llllll}
a(\lambda,t)&:=& q_0^{-(3t+1)^2} \lambda^{-\frac{1+3t}{2}} A(\lambda,t) &=& \sum_{k\in \Z}  e^{\pi i  (3 \tau) k^2}   e^{2\pi i k (3u)} &=\theta \left[0,0]\right (3u,3\tau)\\
b(\lambda,t)&:=&q_0^{-(3t+1)^2}  \lambda^{-\frac{1+3t}{2}}  B(\lambda,t) &=&  \sum_{k\in \Z}  e^{\pi i (3 \tau) (k+\frac{2}{3})^2}  e^{2\pi i  (k+\frac{2}{3}) (3u)} &=  \theta \left[\frac{2}{3},0\right] (3u,3\tau)\\
c(\lambda,t)&:=& q_0^{-(3t+1)^2}  \lambda^{-\frac{1+3t}{2}}  C(\lambda,t) &=& \sum_{k\in \Z}  e^{\pi i (3 \tau) (k+\frac{1}{3})^2}  e^{2\pi i  (k+\frac{1}{3}) (3u)} &= \theta \left[\frac{1}{3},0\right] (3u,3\tau).
\end{array}
\end{equation*}
Recall that the theta function above is defined as
$$\theta[a,b](w,\tau) = \sum_{m\in \Z} e^{\pi i \tau (m+a)^2} e^{2 \pi i (m+a)(w+b)}.$$

Define $\Theta_j (u, \tau)_s := \theta \left[ \frac{j}{s}, 0 \right] ( s u, s \tau)$ so that 
$$a(u) = \Theta_0 (u, \tau)_3, \quad b(u) = \Theta_2 (u, \tau)_3, \quad c(u) = \Theta_1 (u, \tau)_3.$$
See for e.g. \cite[Theorem 3.2]{Dol} to observe that $(a(u) : b(u) : c(u)) : E_\tau \to \mathbb{P}^2$ gives an embedding of the elliptic curve into the projective plane.

\section{Weak Maurer-Cartan relations for $\mathbb{P}^1_{2,2,2,2}$}\label{appsubsec:theta2222}
In order to simplify $A,B,C,D$ \eqref{eqn:ABCD2222bef},
we use the following meromorphic function
$$ f( z_1 ,z_2 ;\tau) = \displaystyle\sum_{ (\alpha(z_1) +l)(\alpha(z_2) +k) >0 } {\rm sign} (\alpha(z_1) + l) e^{2\pi i ( \tau kl + l z_1 + k z_2) }$$
where $\alpha(z) = {\rm Im} (z) /{\rm Im} (\tau)$. $f$ can be written as (constant multiple of) a ratio involving three theta functions as
\begin{equation}\label{eqn:meroftheta}
f(z_1,z_2;\tau) = \frac{\theta'(\tau+1/2,\tau)}{2 \pi i } \cdot \frac{\theta(z_1 + z_2 - (\tau+1)/2 , \tau)}{\theta(z_1 - (\tau+1)/2, \tau) \theta ( z_2 - (\tau+1)/2, \tau)}
\end{equation}
This function is well known from the literature (see for e.g. \cite{Kr}), and appeared in \cite{Pol} in connection with $m_3$-products on the Fukaya category of an elliptic curve. 

Observe that $A$ can be rewritten as
\begin{eqnarray*}
A(\lambda,t) &=& \displaystyle\sum_{k,l \geq 0}^{\infty} \lambda^{-2l-\frac{1}{2}} q_d^{(4k+1 -2t)(4l+1)}- \sum_{k,l \geq 0}^{\infty} \lambda^{2l + \frac{3}{2} } q_d^{(4k+3 +2t)(4l+3)} \\
 &=&\displaystyle\sum_{k,l \geq 0}^{\infty} \lambda^{-2l-\frac{1}{2}} q_d^{(4k+1 -2t)(4l+1)}- \sum_{k,l \leq -1}^{-\infty} \lambda^{-2l-\frac{1}{2}} q_d^{(4k+1 -2t)(4l+1)}
\end{eqnarray*}
We set $\tilde{\tau} = 2 \tau$ and $\tilde{u}'= -2s - \tau t = -2s - \tilde{\tau} \frac{t}{2}$. Then
\begin{equation*}
\begin{array}{lcl}
A(\lambda,t) &=&  q_d^{1-2t} e^{-\pi i s} \displaystyle\sum_{k , l \geq 0} \exp \left(2 \pi i \left( \tilde{\tau} k l + l \left( -2s - \tilde{\tau} \left( \frac{t}{2} - \frac{1}{4} \right) \right) + k \left(\frac{\tilde{\tau}}{4} \right) \right) \right) \\
&&-  q_d^{1-2t} e^{- \pi i s}  \displaystyle\sum_{k,l \leq -1} \exp \left(2 \pi i \left( \tilde{\tau} k l + l \left( -2s - \tilde{\tau} \left( \frac{t}{2} - \frac{1}{4} \right) \right) + k \left(\frac{\tilde{\tau}}{4} \right) \right) \right)  \\
&=& 
q_d e^{  \frac{\pi i}{2} \tilde{u}' } \displaystyle\sum_{k , l \geq 0} \exp \left(2 \pi i \left( \tilde{\tau} k l + l \left( \tilde{u}' + \frac{\tilde{\tau}}{4} \right)  + k \left(\frac{\tilde{\tau}}{4} \right) \right) \right) \\
&& - q_d e^{  \frac{\pi i}{2} \tilde{u}' }  \displaystyle\sum_{k,l \leq -1} \exp \left(2 \pi i \left( \tilde{\tau} k l + l \left( \tilde{u}' + \frac{\tilde{\tau}}{4} \right)  + k \left(\frac{\tilde{\tau}}{4} \right) \right) \right) \\
&=&  q_d e^{  \frac{\pi i}{2} \tilde{u}' } f\left(\tilde{u}' + \frac{\tilde{\tau}}{4},\frac{\tilde{\tau}}{4};\tilde{\tau} \right) 
\end{array}
\end{equation*}
since $\alpha\left(\tilde{u}' + \frac{\tilde{\tau}}{4} \right) = {\rm Im} \left(-2s - \tilde{\tau} \frac{t}{2} + \frac{\tilde{\tau}}{4} \right) / {\rm Im} (\tilde{\tau}) = \frac{1}{4}-\frac{t}{2}$ and $\alpha \left(\frac{\tilde{\tau}}{4} \right) = \frac{1}{4}$.
Now, we apply \eqref{eqn:meroftheta} to obtain
\begin{equation*}
\begin{array}{lcl}
A(\tilde{u}) &=& \tilde{K} \dfrac{ \theta ( \tilde{u}' - \frac{1}{2} , \tilde{\tau}) }{\theta ( \tilde{u}' - (\tilde{\tau} +2) /4 , \tilde{\tau})  \theta ( - (\tilde{\tau} +2) /4 , \tilde{\tau}) } \\
&=& \tilde{K} \dfrac{ \theta ( \tilde{u} , \tilde{\tau}) }{\theta \left( \tilde{u} - \frac{\tilde{\tau}}{4} , \tilde{\tau} \right)  \theta \left( \tilde{u}_0 -\frac{\tilde{\tau}}{4}, \tilde{\tau} \right) } \\
&=& \tilde{K} e^{-\pi i \tilde{\tau}} e^{-2\pi i (\tilde{u} + \tilde{u}_0)}  \dfrac{ \theta ( \tilde{u} , \tilde{\tau}) }{\theta \left( \tilde{u}_1 + \frac{3 \tilde{\tau}}{4} , \tilde{\tau} \right)  \theta \left( \tilde{u}_0 +\frac{\tilde{3 \tau}}{4}, \tilde{\tau} \right) } \\
&=& \tilde{K} e^{ \pi i (\frac{9}{16} \tilde{\tau} + \frac{3}{2} \tilde{u})} e^{ \pi i (\frac{9}{16} \tilde{\tau} + \frac{3}{2} \tilde{u}_0)} e^{-\pi i \tilde{\tau}} e^{-2\pi i (\tilde{u} + \tilde{u}_0)}  \dfrac{ \theta ( \tilde{u} , \tilde{\tau}) }{\theta \left[\frac{3}{4},0 \right]\left( \tilde{u} , \tilde{\tau} \right)  \theta \left[\frac{3}{4},0 \right] \left( \tilde{u}_0 , \tilde{\tau} \right) } \\
&=& \tilde{K} e^{ -\frac{\pi i}{2} (\tilde{u} + \tilde{u}_0)} e^{\frac{\pi i \tilde{\tau}}{8}}
\dfrac{ \theta ( \tilde{u}  , \tilde{\tau}) }{\theta \left[\frac{3}{4},0 \right]\left( \tilde{u} , \tilde{\tau} \right)  \theta \left[\frac{3}{4},0 \right] \left( \tilde{u}_0 , \tilde{\tau} \right) }= K' \dfrac{ \theta ( \tilde{u} , \tilde{\tau}) }{\theta \left[\frac{3}{4},0 \right]\left( \tilde{u} , \tilde{\tau} \right)  \theta \left[\frac{3}{4},0 \right] \left( \tilde{u}_0 , \tilde{\tau} \right) } 
\end{array}
\end{equation*}
where we made a change of coordinate by $\tilde{u} := \tilde{u}'-\frac{1}{2}= -2s - \tau t -\frac{1}{2} = -2s - \tilde{\tau} \frac{t}{2} -\frac{1}{2}$ and $\tilde{u}_0 := - 1/2 $.  Note that $u := \tilde{u}/2 = -s - \tau \frac{t}{2} - \frac{1}{4}$ is a variable on $E_\tau= \C / (\Z \oplus \Z \langle \tau \rangle)$. The precise constants are given by
$$\tilde{K} := -i \frac{q_d}{2 \pi}   \theta' ( (\tilde{\tau} +1)/2, \tilde{\tau}) e^{  \frac{\pi i \tilde{u}}{2}  },\quad  K':=\tilde{K} e^{ -\frac{\pi i}{2} (\tilde{u} + \tilde{u}_0)} e^{\frac{\pi i \tilde{\tau}}{8}} = \frac{1}{2\pi} e^{\frac{\pi i \tilde{\tau}}{4}}  \theta' ( (\tilde{\tau} +1)/2, \tilde{\tau}).$$


%

We set $\theta_0$ to be the constant $\theta \left[\frac{3}{4},0 \right] \left( \tilde{u}_0 , \tilde{\tau} \right)$ ($\tilde{u}_0 =-1/2$) to get
$$\theta_0 (K')^{-1} A=   \dfrac{\theta [0,0] ( \tilde{u} , \tilde{\tau}) }{ \theta \left[\frac{3}{4}, 0 \right] (\tilde{u}, \tilde{\tau}) }  $$
which we denote by $a$. Setting $K = \theta_0 (K')^{-1}$, we have $a := K A = \dfrac{\theta [0,0] ( \tilde{u} , \tilde{\tau}) }{ \theta \left[\frac{3}{4}, 0 \right] (\tilde{u}, \tilde{\tau}) }  $.
%
%
\begin{equation*}
\begin{array}{l}
b:=  K B= -i \dfrac{\theta [0,0]( \tilde{u} , \tilde{\tau}) }{ \theta \left[\frac{1}{4},0 \right] ( \tilde{u}  , \tilde{\tau})} \\
c:= K C=  i \dfrac{ \theta \left[\frac{2}{4},0 \right] \left( \tilde{u}  , \tilde{\tau} \right) }{ \theta  \left[\frac{1}{4}, 0 \right] \left( \tilde{u} , \tilde{\tau}\right) } \\
d:= K D= \dfrac{ \theta \left[ \frac{2}{4},0\right]\left( \tilde{u}  , \tilde{\tau} \right) }{ \theta \left[\frac{3}{4},0 \right] \left( \tilde{u}, \tilde{\tau} \right) }  
\end{array}
\end{equation*}
It directly follows that $ac + bd =0$.

\section{Embedding of $E_\tau$ into $\mathbb{P}^3$}\label{appsubsec:theta2222e}

We next show that $(a:b:c:d) : E_\tau = \C / (\Z \oplus \Z \langle \tau \rangle ) \to \mathbb{P}^3$ is an embedding of the elliptic curve. Define two maps $\Phi_1, \Phi_2 : E_\tau \to \mathbb{P}^1$ by
$$\Phi_1 := (b: c)=\left( \theta [0,0]( \tilde{u} , \tilde{\tau}) :  - \theta \left[1/2,0 \right](\tilde{u}, \tilde{\tau}) \right]  = (\Theta_0 (u, \tau)_2 : - \Theta_1 (u, \tau)_2 )$$
 and
\begin{eqnarray*}
\Phi_2 &:=& (a: b) = \left(  \theta [1/4,0]( \tilde{u} , \tilde{\tau}) :  -i \theta \left[3/4,0 \right] (\tilde{u}, \tilde{\tau}) \right) \\
&=& \left(  \theta [0,0] \left( \tilde{u} +\frac{\tilde{\tau}}{4}, \tilde{\tau} \right) :  -\theta \left[1/2,0 \right] \left(\tilde{u} +\frac{\tilde{\tau}}{4}, \tilde{\tau} \right) \right)=\left(\Theta_0 \left(u + \frac{\tau}{4}, \tau \right)_2 : -i\Theta_1 \left(u +\frac{\tau}{4} , \tau \right)_2 \right).
\end{eqnarray*}
(Recall $\Theta_j (u, \tau)_2 := \theta \left[ \frac{j}{2}, 0 \right] ( 2u, 2\tau)$ for $j=0,1$ and $\tilde{\tau} = 2 \tau$ and $\tilde{u} = 2u$.)

Consider the map $\Phi : \C / (\Z \oplus \Z \langle \tau \rangle ) \to \mathbb{P}^1 \times \mathbb{P}^1$ given by 
$$\Phi = (\Phi_1 : \Phi_2) =\left( (b:c),\,(a:b) \right).$$
Composing with the embedding
$$ \iota : \mathbb{P}^1 \times \mathbb{P}^1 \hookrightarrow \mathbb{P}^3 \qquad ([z_0 : z_1], [w_0:w_1] ) \mapsto (z_0 w_0 : z_0 w_1 : z_1 w_1 : -z_1 w_0) $$
we obtain back the original map $(a:b:c:d) $ since
$$ \iota \circ \Phi = (ba: bb : cb : -ca) =  ( a:b:c  :-ca / b) = (a:b:c:d) $$
where we used $ac = -bd$. Thus it suffices to show that $\Phi$ gives an embedding of $E_\tau$ into $\mathbb{P}^1 \times \mathbb{P}^1$. 

Note that $\Theta_j (-u) = \Theta_j (u)$ for all $j$. 
Using this property, it is easy to check that $\Phi_1 : E_\tau \to \mathbb{P}^1$ is a branched double cover whose fiver is $\{ u, 1+\tau - u\}$ where $u \in E_\tau$. On the other hand, $\Phi_2 : E_\tau \to \mathbb{P}^1$ is also a branched double cover by the same reason, but the fiber is $\{u, 1+ \frac{\tau}{2} -u\}$ in this case. This shows that $\Phi = (\Phi_1:\Phi_2)$ is injective.

\backmatter
\bibliographystyle{amsalpha}
\bibliography{geometry}
\printindex

\end{document}